\documentclass{amsbook}

\usepackage{amsmath,amsthm,verbatim,amssymb,amsfonts,amscd, graphicx,lmodern,color,braket,tikz,mathtools,enumitem,float}
\usepackage{afterpage}
\usepackage[toc,page]{appendix}
\usepackage{graphicx}  %demo is to compile without .png [demo]
\usepackage{tikz-cd}
\usepackage{pgfplots}
\usepackage{subfigure}
\usepackage{longtable}
\usetikzlibrary{arrows}
\usetikzlibrary{calc}
\usetikzlibrary{matrix}

\newtheorem{theorem}{Theorem}[chapter]
\newtheorem*{theorem*}{Theorem}
\newtheorem{lemma}[theorem]{Lemma}

\theoremstyle{definition}
\newtheorem{definition}[theorem]{Definition}
\newtheorem{example}[theorem]{Example}

\newtheorem{proposition}[theorem]{Proposition}

\newtheorem*{examples*}{Examples}

\newtheorem*{theoremA}{Theorem A}
\newtheorem*{theoremB}{Theorem B}
\newtheorem*{theoremC}{Theorem C}
\newtheorem*{theoremD}{Theorem D}
\newtheorem*{theoremE}{Theorem E}

\theoremstyle{remark}
\newtheorem{remark}[theorem]{Remark}

\def\R{\ensuremath\mathbb{R}}

\def\T{\mathbb{T}}

\def\id{\text{Id}}

\numberwithin{section}{chapter}
\numberwithin{equation}{chapter}

\usepackage{graphicx}

\usepackage{tikz-cd}
\usepackage{circuitikz}
\usetikzlibrary{shapes.geometric, arrows}

\newcommand\DrawGenus[7]{% (x, y), size, size*1.5, rotation, color, thickness
  \pgfmathsetmacro{\xstart}{#1 - (0.985*#4)}
  \pgfmathsetmacro{\ystart}{#2 + (0.2*#3)}
	\draw[color = #6, rotate around={#5:(#1,#2)}, #7] (\xstart, \ystart) arc (190:350:#4  and #3);
	\draw[color = #6, rotate around={#5:(#1,#2)}, #7] (\xstart, \ystart) arc (190:210:#4  and #3) arc (150:30:#4  and #3) arc (330:350:#4  and #3);
}

\newcommand\DrawFilledGenus[8]{% (x, y), size, size*1.5, rotation, color, thickness
  \pgfmathsetmacro{\xstart}{#1 - (0.985*#4)}
  \pgfmathsetmacro{\ystart}{#2 + (0.2*#3)}
	\draw[color = #6, rotate around={#5:(#1,#2)}, #7] (\xstart, \ystart) arc (190:350:#4  and #3);
	\draw[color = #6, rotate around={#5:(#1,#2)}, #7] (\xstart, \ystart) arc (190:210:#4  and #3) arc (150:30:#4  and #3) arc (330:350:#4  and #3);
	\draw[color = #6, rotate around={#5:(#1,#2)}, #7, fill = #8] (\xstart, \ystart) arc (190:210:#4  and #3) arc (150:30:#4  and #3) arc (-30:-150:#4  and #3);
}

\newcommand\DrawDonut[7]{% (x, y), size, size*1.5, rotation
  \pgfmathsetmacro{\fctr}{.08}
  \pgfmathsetmacro{\newwidth}{0.5*#4}
  \pgfmathsetmacro{\newheight}{0.5*#3}
  \draw[color = #6, rotate around={#5:(#1,#2)}, #7] (#1, #2) ellipse (#4  and #3);
  \DrawGenus{#1}{#2}{\newheight}{\newwidth}{#5}{#6}{#7}
}

\newcommand\DrawFilledDonutops[8]{% (x, y), size, size*1.5, rotation
  \pgfmathsetmacro{\fctr}{.08}
  \pgfmathsetmacro{\newwidth}{0.5*#4}
  \pgfmathsetmacro{\newheight}{0.5*#3}
  \draw[color = #6, rotate around={#5:(#1,#2)}, #7, fill = #6, opacity = .6] (#1, #2) ellipse (#4  and #3);
  \DrawFilledGenus{#1}{#2}{\newheight}{\newwidth}{#5}{#6}{#7}{#8}
}

\tikzstyle{mytheorembox} = [draw=vdgreen, fill=blue!20, very thick, rectangle, rounded corners, inner sep=10pt, inner ysep=15pt]
\tikzstyle{mytheoremfancytitle} =[fill=vdgreen, text=white]

\definecolor{vdblue}{rgb}{0,0,.3}
\definecolor{dblue}{rgb}{0,0,.7}
\definecolor{lblue}{rgb}{.3,.3,1}
\definecolor{vlblue}{rgb}{.7,.7,1}
\definecolor{vvlblue}{rgb}{.9,.9,1}

\definecolor{vdred}{rgb}{.3,0,0}
\definecolor{dred}{rgb}{.7,0,0}
\definecolor{lred}{rgb}{1,.3,.3}
\definecolor{vlred}{rgb}{1,.7,.7}

\definecolor{vdgreen}{rgb}{0,.2,0}
\definecolor{dgreen}{rgb}{0,.4,0}
\definecolor{lgreen}{rgb}{.3,1,.3}
\definecolor{vlgreen}{rgb}{.7,1,.7}

\definecolor{lyellow}{rgb}{1,1,.3}

\definecolor{gray1}{rgb}{0.22,0.22,0.22}
\definecolor{gray2}{rgb}{0.28,0.28,0.28}
\definecolor{gray3}{rgb}{0.36,0.36,0.36}
\definecolor{gray4}{rgb}{0.44,0.44,0.44}
\definecolor{gray5}{rgb}{0.52,0.52,0.52}
\definecolor{gray6}{rgb}{0.6,0.6,0.6}
\definecolor{gray7}{rgb}{0.68,0.68,0.68}
\definecolor{gray8}{rgb}{0.76,0.76,0.76}

\definecolor{color1}{rgb}{1,0,0}
\definecolor{color2}{rgb}{0.98,0,0.816}
\definecolor{color3}{rgb}{0.717,0,1}
\definecolor{color4}{rgb}{0,0,1}
\definecolor{color5}{rgb}{0,1,1}
\definecolor{color6}{rgb}{0,1,0}
\definecolor{color8}{rgb}{1,1,0}
\definecolor{color7}{rgb}{1,0.651,0}

\begin{document}
\frontmatter
\title{Action-angle coordinates and KAM theory for singular symplectic manifolds}

\author{Eva Miranda}
\address{Laboratory of Geometry and Dynamical Systems, Department of Mathematics \& IMTech, Universitat Polit\`ecnica de Catalunya, Barcelona and CRM, Centre de Recerca Matemàtica, Bellaterra}
%\curraddr{UPC-Edifici P, Avinguda del Doctor Maranón, 44-50, 08028, Barcelona, Spain}
\email{evamiranda@upc.edu}

\thanks{ Both authors are supported by the project PID2019-103849GB-I00 of the Spanish State Agency  AEI /10.13039/501100011033 and by the AGAUR Gencat project 2021 SGR 00603. Eva Miranda is supported by the Catalan Institution for Research and Advanced Studies via an ICREA Academia Prize 2021 and by the Alexander Von Humboldt foundation via a Friedrich Wilhelm Bessel Research Award.  Eva Miranda is also supported by the Spanish State
Research Agency, through the Severo Ochoa and Mar\'{\i}a de Maeztu Program for Centers and Units
of Excellence in R\&D (project CEX2020-001084-M). Eva Miranda also acknowledges partial support from the grant
“Computational, dynamical and geometrical complexity in fluid dynamics”, Ayudas Fundación BBVA a Proyectos de Investigación Científica 2021.  }

\author{Arnau Planas}
\address{Department of Mathematics, Universitat Polit\`ecnica de Catalunya, Barcelona}
\email{arnau.planas.bahi@gmail.com}

\date{May, 2023}
\keywords{\texttt{amsbook}, AMS-\LaTeX}

\maketitle
\dedicatory{Dedicated to the memory of Amelia Galcerán Sorribes.  }
\setcounter{page}{4}
\tableofcontents

%-----------------------------------------------------------------------------
% Beginning of preface.tex
%-----------------------------------------------------------------------------
%
% AMS-LaTeX 1.2 sample file for a monograph, based on amsbook.cls.
% This is a data file input by chapter.tex.
%%%%%%%%%%%%%%%%%%%%%%%%%%%%%%%%%%%%%%%%%%%%%%%%%%%%%%%%%%%%%%%%%%%%%%%%

\chapter*{Preface}

{\begin{flushright}{\begin{minipage}{5cm}{\emph{I confess I envy the planets — they've got their own orbits and nothing stands on their way.} 
\\
Intermezzo, Mykhailo Kotsiubynsky.}\end{minipage}
}\end{flushright}}

This monograph explores classification and perturbation problems for integrable systems on a class of Poisson manifolds called $b^m$-Poisson manifolds. This is the first class of Poisson manifolds for which perturbation theory is established outside the symplectic category. Even if the class of $b^m$-Poisson manifolds is not ample enough to represent the wild set of Poisson manifolds, this investigation can be seen as a first step for the study of perturbation theory for general Poisson manifolds. Whenever the Poisson manifold can be described as an $E$-symplectic manifold (\cite{enest} and \cite{evageoff}) the theory developed in this monograph yields more than a mild generalization in Poisson Geometry and, this \emph{toy example},  sets the path to consider KAM theory in the general realm of Poisson manifolds. Reduction theorems and $b^m$-symplectic manifolds have been recently explored in \cite{anastasiaeva}. This monograph contributes to the theory opening the investigation of perturbation theory on these manifolds thus completing other facets in the study of their dynamics as the recent work on the Arnold conjecture \cite{cedricjoaquimeva}.

Symplectic geometry has been the common language of physics as the position-momentum tandem can be modelled over a cotangent bundle. Cotangent bundles are naturally endowed with a symplectic form which is a non-degenerate closed $2$-form. The symplectic form of the cotangent bundle is given by the differential of the Liouville one-form.

$b^m$-Poisson manifolds are manifolds that are symplectic away from a hypersurface along which they satisfy some transversality properties. They often model problems on symplectic manifolds with boundary such as the study of their deformation quantization and celestial mechanics.  As on the complementary of the critical set the manifolds are symplectic, extending the investigation of Hamiltonian dynamics to this realm is key to understanding the Hamiltonian Dynamics on the compactification of symplectic manifolds. Several regularization transformations used in celestial mechanics (as McGehee or Moser regularization) provide examples of such compactifications.

One of the interesting properties of $b^m$-Poisson manifolds is that their investigation can be achieved considering the language of $b^m$-forms. That is to say, we can work with forms that are symplectic away from the critical set and admit a smooth extension as a \emph{form over a Lie algebroid} generalizing De Rham forms as form over the standard Lie algebroid of the tangent bundle of the manifold. To consider $b^m$-forms the standard tangent bundle is replaced by  the $b^m$-tangent bundle. This allows us to mimic symplectic geometry by replacing the cotangent bundle with the dual of the $b^m$-tangent bundle. However, Poisson geometry leaves its footprint and new invariants which can be identified as the modular class of the Poisson structure arise already at the semilocal level.

Contrary to the initial expectations,  several of the results for $b^m$-symplectic manifolds do not resemble the $b$-case so far. Considering these more general singularities yields a better understanding of the general Poisson case and the different levels of complexity. As an illustration of this phenomena: in the study of quantization of those systems an interesting pattern makes the quantization radically different in the even and odd case \cite{GMWbquant, GMWbmquant} and the resulting model is finite-dimensional in the $b$-case. Understanding how the different degrees $m$ are related is a hard task: The desingularization technique introduced by Guillemin-Miranda-Weitsman in \cite{GMW17} turned out to have important applications in the investigation of complexity properties of toric $b^m$-symplectic manifolds \cite{GMWbmconvexity} and to the study of the Arnold conjecture in this set-up \cite{cedricjoaquimeva}. In this monograph, we explore a new facet of these manifolds: that of perturbation theory.

%This monograph starts by recalling the equivariant classification of $b^m$-Poisson structures investigating, in particular, the analog of Moser's classification theorem for symplectic surfaces and their equivariant analogues.
%To provide some basic examples and fix ideas we include several results of classification of $b^m$-surfaces. The classification invariants in the case of surfaces are encoded in a cohomology called $b^m$-cohomology which has been deeply studied by \cite{Scott16}. Mazzeo-Melrose type formula for $b^m$-cohomology decomposes it in two pieces which can  be read off the De Rham cohomology of both the ambient manifold $M$ and the critical hypersurface.  As an outcome of this identification, the Poisson classification of these manifolds is given by the De Rham cohomology of the manifold and the hypersurface.

%This classification is extended to the equivariant setting if we assume that the singular forms are preserved by the group action of a compact Lie group.
%These techniques can be extended to the classification of $b^m$-Nambu structures which we do not consider in this monograph.

In the second part of the monograph, we consider integrable systems on these manifolds. Under mild conditions, integrable systems on these singular manifold have an associated "generalized" Hamiltonian action of tori in a neighbourhood of a Liouville torus.  We use this generalized Hamiltonian group action to prove the existence of action-angle coordinates in a neighborhood of a Liouville torus.  The action-angle coordinate theorem that we prove furnishes a semi-local normal form within the vicinity of a Liouville torus for the $b^m$-symplectic structure.  which depends on the modular weight of the connected component of the critical set in which the Liouville torus is lying as well as the modular weights of the associated toric action. This action-angle theorem allows us to identify a neighborhood of the Liouville torus with the $b^m$-cotangent lift of the action of a torus acting by translations on itself. This interpretation of the action-angle theorem as cotangent lift allows us to pinpoint the modular weight as their only semilocal invariant. In doing so, we compare this action-angle coordinate theorem with the classical action-angle coordinate theorems for symplectic manifolds and an action-angle theorem for folded symplectic manifolds (\cite{EvaRobert}).

In part 3 of the monograph we study perturbation theory in this new set-up and examine some potential applications to physical systems. In particular, we prove a KAM theorem for $b^m$-Poisson manifolds which clearly refines and improves the one obtained for $b$-Poisson manifolds in \cite{KMS16}. As an outcome of this result together with the extension of the desingularization techniques of Guillemin-Miranda-Weitsman to the realm of integrable systems, we obtain a KAM theorem for folded symplectic manifolds where KAM theory has never been considered before. In the way, we also obtain a brand-new KAM theorem for symplectic manifolds where the perturbation keeps track of a distinguished hypersurface. In celestial mechanics, this distinguished hypersurface can be the line at infinity or the collision set. Escape orbits in celestial mechanics can often be compactified as singular orbits associated with these singular structures.

\aufm{Barcelona, May 2023, Eva Miranda and Arnau Planas}

%-----------------------------------------------------------------------------
% End of preface.tex
%-----------------------------------------------------------------------------

\mainmatter
\part{Introduction and preliminaries}
\chapter{Introduction}
\label{ch:introduction}

Both symplectic and Poisson geometry emerge from the study of classical mechanics. Both are broad fields widely studied and with powerful results. But as Poisson structures are far more general than the symplectic ones, most outstanding results in symplectic geometry do not translate well to Poisson manifolds. $b^m$-Poisson structures, also known as $b^m$-symplectic structures, bridge this gap by extending symplectic structures in a controlled manner. This allows fundamental results in symplectic geometry to carry over to $b^m$-symplectic geometry. However, adapting theories such as deformation or Moser theory requires additional work, as discussed in \cite{GMPS} and other sources.

%Here is where $b^m$-Poisson structures come to play. $b^m$-Poisson structures (or $b^m$-symplectic structures) lie somewhere between these two worlds. They extend symplectic structures but in a really controlled way. This is why fundamental results in symplectic geometry still work in $b^m$-symplectic geometry. However, an adaptation of these theories like deformation or Moser theory requires some work (see \cite{GMPS} and others).

The study of $b^m$-Poisson geometry sparked from the study of symplectic manifold with boundary (see \cite{Melrose93} and \cite{NT96}). In the last years, the interest in this field increased after the classification result for $b$-Poisson structures obtained in \cite{Radko02}. Later on, \cite{GMP14} translated these structures to the language of forms and started applying symplectic tools to study them. A lot of papers in the following years studied different aspects of these structures: \cite{GMP10}, \cite{GMP14}, \cite{GMP15}, \cite{GMW17}, \cite{MO2} and  \cite{GUAL14} are some examples.

Inspired by the study of manifolds with boundary, we work on a pair of manifolds $(M,Z)$ where $Z$ is a hypersurface and call this pair $b$-manifold

In this context, \cite{Scott16} generalized the $b$-symplectic forms by allowing higher degrees of degeneracy of the Poisson structures. The $b^m$-symplectic structures inherit most of the properties of $b$-symplectic structures. This booklet focuses on  different aspects of the investigation of $b^m$-symplectic structures covering mainly integrable systems and KAM theory. First, we present some preliminary notions necessary to address the problem of perturbation.  We present an action-angle theorem for $b^m$-Poisson structures and state and prove the KAM theory equivalent in manifolds with $b^m$-symplectic structures.

\section{Structure and results of this monograph}

\subsection{Part $1$: Introduction and Preliminaries}

In the preliminaries, we give the basic notions that lead to the questions we are addressing in this booklet. In the first part, we introduce the concept of $b$-Poisson manifolds or $b$-symplectic manifolds, a class of Poisson manifold which is symplectic outside a critical hypersurface. It study comes motivated by the investigation of manifolds with boundary. Next, we talk about a generalization of these structures, that allows a higher degree of degeneracy of the structure: the $b^m$-symplectic structures. These structures are the main focus of our investigations. A key concept that will play an important role in this book is the study of the desingularization of these singular structures. Finally, we give a short introduction to KAM theory, a theory that will be generalized in the setting of $b^m$-manifolds in the last chapter.

Motivation comes from several examples of singular symplectic structures appearing naturally in classical problems of celestial mechanics which are discussed in the last chapter of the monograph. We also describe the difficulties of finding these examples and the subtleties of dealing with these singular structures in the exploration of conservative systems.

\subsection{Part $2$: Action-angle coordinates and cotangent models for $b^m$-integrable systems}

In this Chapter we define the concept of $b^m$-functions and $b^m$-integrable systems. We present several examples of $b^m$-integrable systems that come from classical mechanics. After all this, we present a version of the action-angle theorem for $b^m$-symplectic manifolds.

\begin{theoremA}
Let $(M,x,\omega,F)$ be a $b^m$-integrable system, where $F = (f_1 = a_0 \log(x) + \sum_{j=1}^{m-1} a_j\frac{1}{x^j}, f_2,\ldots,f_n)$. Let $m\in Z$ be a regular point, and such that the integral manifold through $m$ is compact. Let $\mathcal{F}_m$ be the Liouville torus through $m$.
Then, there exists a neighborhood $U$ of $\mathcal{F}_m$ and coordinates $(\theta_1,\ldots,\theta_n,\sigma_1,\ldots,\sigma_n):\mathcal{U}\rightarrow\mathbb{T}^n\times B^n$ such that:

\begin{enumerate}
\item We can find an equivalent integrable system $F = (f_1 = a_0'\log(x) + \sum_{j=1}^{m-1} a_j'\frac{1}{x^j})$ such that $a_0',\ldots, a_{m-1}' \in \mathbb{R}$,
\item $$\omega|_\mathcal{U} = \left(\sum_{j=1}^m c_j'\frac{c}{\sigma_1^j}d\sigma_1\wedge d\theta_n\right) + \sum_{i=2}^{n} d \sigma_i\wedge d\theta_i$$ where $c$ is the modular period and $c_j' = -(j-1)a_{j-1}'$, also
\item the coordinates $\sigma_1,\ldots,\sigma_n$ depend only on $f_n,\ldots f_n$.
\end{enumerate}

\end{theoremA}

\subsection{Part $3$: KAM theory on $b^m$-symplectic manifolds and applications to Celestial Mechanics}
In this chapter we provide several  KAM theorems for (singular) symplectic manifolds including $b^m$-symplectic manifolds.

 We begin by considering perturbation theory for $b^m$-symplectic manifolds. Then we give an outline of how to construct the $b^m$-symplectomorphism that will be the main character of the proof of the KAM theorem for $b^m$-symplectic manifolds. After this, we show some technical results that are needed for the proof. These technical results even if quite similar to the standard KAM equivalents, have some subtleties that need to be addressed. We end the chapter with the proof of the $b^m$-KAM theorem and several applications to establish KAM theorems in other singular situations (folded symplectic manifolds) and on symplectic manifolds with prescribed invariant hypersurfaces.

The first KAM theorem is the following:

\textcolor{black}{
\begin{theoremB}
Let $\mathcal{G} \subset \mathbb{R}^n$, $n\geq 2$ be a compact set.
Let $H(\phi, I) = \hat h (I) + f(\phi,I)$, where $\hat h$ is a $b^m$-function $\hat h (I) = h(I) + q_0 \log(I_1) + \sum_{i=1}^{m-1} \frac{q_i}{I_1^i}$ defined on $\mathcal{D}_\rho(G)$, with $h(I)$ and $f(\phi,I)$ analytic.
Let $\hat u = \frac{\partial \hat h}{\partial I}$ and $u = \frac{\partial h}{\partial I}$.
Assume $|\frac{\partial u}{\partial I}|_{G,\rho_2} \leq M$, $|u|_\xi \leq L$.
Assume that $u$ is $\mu$ non-degenerate ($|\frac{\partial u}{\partial I}|\geq \mu|v|$ for some $\mu \in \mathbb{R}^+$ and $I \in \mathcal{G}$. Take $a = 16M$.
Assume that $u$ is one-to-one on $\mathcal{G}$ and its range $F = u(\mathcal{G})$ is a $D$-set.
Let $\tau>n-1,\gamma>0$ and $0 < \nu < 1$. Let
\begin{enumerate}
\item \begin{equation}\label{eq:int_kam1}
\varepsilon:=\|f\|_{\mathcal{G}, \rho} \leq \frac{\nu^2 \mu^2 \hat \rho^{2\tau+2}}{2^{4\tau+32}L^6M^3} \gamma^2,
\end{equation}
\item \begin{equation}\label{eq:int_kam2}
\gamma \leq \min(\frac{8LM\rho_2}{\nu \hat \rho^{\tau+1}}, \frac{L}{\mathcal{K}'})
\end{equation}
\item \begin{equation}\label{eq:int_kam3}
\mu \leq \min(2^{\tau+5}L^2 M,2^7\rho_1 L^4 K^{\tau+1},\beta\nu^{\tau+1}2^{2\tau+1}\rho_1^\tau),
\end{equation}
\end{enumerate}
where $\hat \rho := \min \left(\frac{\nu\rho_1}{12(\tau+2)},1\right)$.
Define the set $\hat G = \hat G_\gamma := \{I \in  \mathcal{G}-\frac{2\gamma}{\mu} | u(I) \text{ is } \tau,\gamma,c,\hat q- Dioph.\}$.
Then, there exists a real continuous map $\mathcal{T}: \mathcal{W}_{\frac{\rho_1}{4}}(\mathbb{T}^n)\times \hat G \rightarrow \mathcal{D}_\rho(\mathcal{G})$ analytic with respect the angular variables such that
\begin{enumerate}
\item\label{kam:point1} For all $I \in \hat G$ the set $\mathcal{T}(\mathbb{T}^n\times \{I\})$ is an invariant torus of $H$, its frequency vector is equal to $u(I)$.
\item\label{kam:point2} Writing $\mathcal{T}(\phi,I)=(\phi + \mathcal{T}_\phi(\phi,I), I + \mathcal{T}_I(\phi,I))$ with estimates
$$|\mathcal{T}_\phi(\phi,I)| \leq \frac{2^{2\tau + 15} M L^2}{\nu^2 \hat \rho^{2\tau+1}}\frac{\varepsilon}{\gamma^2}$$
$$|\mathcal{T}_I(\phi,I))| \leq \frac{2^{10+\tau} L (1+M)}{\nu \hat \rho^{\tau+1}}\frac{\varepsilon}{\gamma}$$
\item\label{kam:point3} $\text{meas} [(\mathbb{T}^n\times \mathcal{G})\setminus\mathcal{T}(\mathbb{T}^n\times \hat G)] \leq C \gamma$ where $C$ is
a really complicated constant depending on $n$,  $\mu$,  $D$,  $\text{diam} F$,  $M$, $\tau$, $\rho_1$, $\rho_2$, $K$ and $L$.
\end{enumerate}
\end{theoremB}
}

Also, we obtain a way to associate a standard symplectic integrable system or a folded integrable system to a $b^m$-integrable system, depending on the parity of $m$. This is done in such a way that the dynamics of the desingularized system are the same as the dynamics of the original one. So it defines a \emph{honest} desingularization of the integrable system.

\begin{theoremC}
The desingularization transforms a $b^m$-integrable system into an integrable system  on a symplectic manifold for even $m$. For $m$ odd, the desingularization associates to it  a folded integrable system. The integrable systems satisfy:
$$X_{f_j}^\omega = X_{f_{j\epsilon}}^{\omega_\epsilon}.$$
\end{theoremC}

By employing this desingularization technique in conjunction with the previous $b^m$-KAM theorem, we are able to derive two novel KAM theorems. The first of these theorems is a KAM theorem for standard symplectic manifolds, where the perturbation has a particular expression. This result is more restrictive than the standard KAM theorem but allows us to guarantee that the perturbations leave a given hypersurface invariant. This means that the tori belonging to that hypersurface remain on the hypersurface after the perturbation. There are various scenarios and contexts where this can be advantageous, such as in Celestial Mechanics problems where it is desirable to monitor a specific hypersurface, such as the line at infinity. The higher-order singularities allow us to consider perturbations that are tangent to the hypersurface up to a certain order.

\begin{theoremD}
Consider a neighborhood of a Liouville torus of an integrable system $F_\varepsilon$ as in \ref{eq:desingularized_even} of a symplectic manifold $(M, \omega_\varepsilon)$ semilocally endowed with coordinates $(I,\phi)$, where $\phi$ are the angular coordinates of the torus, with $\omega_\varepsilon = c' dI_1 \wedge d\phi_i + \sum_{j= 1}^n dI_j\wedge d\phi_j$. Let $H=(m-1)c_{m-1}c' I_1 + h(\tilde I) + R(\tilde I,\tilde \phi)$ be a nearly integrable system where
$$
\left\{
\begin{array}{rcl}
\tilde I_1 & = & c'\frac{I_1^{m+1}}{m+1},\\
\tilde \phi_1 & = & c' I_1^m \phi_1 ,
\end{array}
\right.
$$
and
$$
\left\{
\begin{array}{rcl}
\tilde I & = & (\tilde I_1, I_2, \ldots, I_n),\\
\tilde \phi & = & (\tilde \phi_1, \phi_2, \ldots, \phi_n).
\end{array}
\right.
$$
Then the results for the $b^m$-KAM theorem \ref{th:bm_kam} applied to $H_{\text{sing}} = \frac{1}{I_1^{2k-1}} + h(I) + R(I,\phi)$ hold also for this desingularized system.
\end{theoremD}

The second one is a KAM theorem for folded-symplectic manifolds, where KAM theory has not been considered to date.

\begin{theoremE}
Consider a neighborhood of a Liouville torus of an integrable system $F_\varepsilon$ as in \ref{eq:desingularized_odd} of a folded symplectic manifold $(M, \omega_\varepsilon)$ semilocally endowed with coordinates $(I,\phi)$, where $\phi$ are the angular coordinates of the Torus, with $\omega_\varepsilon = 2cI_1 dI_1 \wedge d\phi_1 + \sum_{j=2}^m dI_j \wedge d\phi_j$.
Let $H = (m-1)c_{m-1} cI_1^2 + h(\tilde I) + R(\tilde I, \tilde \phi)$ a nearly integrable system with
$$
\left\{
\begin{array}{rcl}
\tilde I_1 & = &  2c\frac{I_1^{m+2}}{m+2},\\
\tilde \phi_1 & = &  2c I_1^{m+1} \phi_1 ,
\end{array}
\right.
$$
and
$$
\left\{
\begin{array}{rcl}
\tilde I & = & (\tilde I_1, I_2, \ldots, I_n),\\
\tilde \phi & = & (\tilde \phi_1, \phi_2, \ldots, \phi_n).
\end{array}
\right.
$$
Then the results for the $b^m$-KAM theorem \ref{th:bm_kam} applied to $H_{\text{sing}} = \frac{1}{I_1^{2k}} + h(I) + R(I,\phi)$ also hold for this desingularized system.
\end{theoremE}

Last but not least, we illustrate the connection between $b^m$-symplectic structures and classical mechanics by providing several examples. Several potential applications to celestial mechanics and fluid dynamics are discussed.

\chapter{A primer on singular symplectic manifolds}
\label{ch:preliminaries}

In this first chapter of the booklet we introduce basic notions
 on singular symplectic structures, as well as some concepts on standard KAM theory. Those are the two main pillars of this monograph.

{Let $M$ be a smooth manifold, a \textbf{Poisson structure} on $M$ is a bilinear map
$\{\cdot, \cdot\}:C^\infty(M)\times C^\infty(M) \rightarrow C^\infty(M)$
which is skew-symmetric and satisfies both the Jacobi identity and the Leibniz rule. It is possible to express $\{f,g\}$ in terms of a bivector field via the following equality $\{f,g\}=\Pi(df\wedge dg)$ with $\Pi$ a section of $\Lambda^2(TM)$. $\Pi$ is the associated \textbf{Poisson bivector}. We will use indistinctively the {terminology} of Poisson structure when referring to the bracket or the Poisson bivector. }

A {\emph{$b$-Poisson {bivector field}}} on a manifold $M^{2n}$ is a Poisson bivector such that the map
\begin{equation}\label{eq:transverse}
F: M \rightarrow \bigwedge^{2n} TM: p \mapsto (\Pi(p))^n
\end{equation}
is transverse to the zero section. Then, a pair $(M,\Pi)$ is called a \textbf{$b$-Poisson manifold} and the vanishing set $Z$ of $F$ is called the \textbf{critical hypersurface}. {Observe that $Z$ is an embedded hypersurface}.

This class of Poisson structures was studied by Radko \cite{Radko02} in dimension two and considered in numerous papers in the last years: \cite{GMP10}, \cite{GMP14}, \cite{GMP15}, \cite{GMW17}, \cite{MO2} and  \cite{GUAL14}  among others.

\section{$b$-Poisson manifolds}

Next, we recall classification theorem of $b$-Poisson surfaces as presented by Olga Radko and  the cohomological re-statement and proof given by Guillemin, Miranda and Pires in \cite{GMP14}.

{ In what follows, $(M,\Pi)$  will be a closed smooth surface with a $b$-Poisson structure on it, and  $Z$  its critical hypersurface.}

{Let $h$ be the distance function to $Z$ as in \cite{MO2}}\footnote{{Notice the difference with \cite{Radko02} where $h$ is assumed to be a global defining function.}}.

\begin{definition}\label{lvol} The {\textbf{Liouville volume of $(M, \Pi)$}}  is the following limit: $ V(\Pi )\coloneqq \lim _{\epsilon \to 0}\int _{|h|>\epsilon }\omega^n $\footnote{{For surfaces $n = 1$.}}.
\end{definition}

{
The previous limit exists and it is independent of the choice of the defining function $h$ of $Z$ (see \cite{Radko02} for the proof).}

\begin{definition}
{For any $(M,{\Pi})$ oriented Poisson manifold, let $\Omega$} {be} a volume form on it, and let $u_f$ denote the Hamiltonian vector field of a smooth function $f:M\rightarrow\mathbb{R}$. The \textbf{modular vector field} {$X^{\Omega}$} is the derivation defined as follows:
	$$f\mapsto \frac{\mathcal{L}_{u_f}\Omega}{\Omega}.$$
\end{definition}	
	
	\begin{definition} Given $\gamma$ a connected component of {  the critical set $Z(\Pi )$ of a closed $b$-Poisson manifold $(M,\Pi)$},   the \textbf{modular period} of $\Pi$ around $\gamma$ is defined as:

	$$T_{\gamma}(\Pi )\coloneqq \textrm{period of }\, {X^{\Omega}}|_{\gamma }. $$
	
\end{definition}

{
\begin{remark}
The modular vector field $X^\Omega$ of the $b$-Poisson manifold $(M,Z)$ does not depend at $Z$ on the choice of $\Omega$ because for different choices for volume form the difference of modular vector fields  is a Hamiltonian vector field. Observe that this Hamiltonian vector field vanishes on the critical set as $\Pi$ vanishes there too.
\end{remark}
}

\begin{definition}
{
Let $\mathcal{M}_n(M)=\mathcal{C}_n(M)/\sim$ where $\mathcal{C}_n(M)$ is the space of disjoint oriented curves and $\sim$ identifies two sets of curves if there is an orientation-preserving diffeomorphism mapping the first one to the  second one and preserving the orientations of the curves.
}
\end{definition}

The following theorem classifies $b$-symplectic structures on surfaces using these invariants:
\begin{theorem}[\textbf{Radko} \cite{Radko02}]\label{Radko}
{Consider two $b$-Poisson structures $\Pi$, $\Pi'$ on {a closed} {orientable} surface $M$. Denote its critical {hypersurfaces} by $Z$ and $Z'$. These two $b$-Poisson structures are globally equivalent (there exists a global {orientation preserving} diffeomorphism sending $\Pi$ to $\Pi'$) if and only if the following coincide}:
	\begin{itemize}
	    \item the equivalence classes of $[Z]$ and $[Z']\in\mathcal{M}_n(M)$,
	    \item their modular periods around the connected components of $Z$ and $Z'$,
	    \item their  Liouville volume.
	\end{itemize}
\end{theorem}

An appropriate formalism to deal with these structures was introduced in \cite{GMP10}.

\begin{definition}
A \textbf{$b$-manifold}\footnote{The `$b$' of $b$-manifolds stands for `boundary', as initially considered by Melrose {(Chapter 2 of \cite{Melrose93})} for the study of pseudo-differential operators on manifolds with boundary.} is a pair $(M,Z)$ of a manifold and an {embedded} hypersurface.
\end{definition}

 In this way, the concept of $b$-manifold previously introduced by Melrose is generalized to consider additional geometric structures on the manifold.

\begin{definition}
A \textbf{$b$-vector field} on a $b$-manifold $(M,Z)$ is a vector field tangent to the hypersurface $Z$ at every point $p\in Z$.
\end{definition}

{
\begin{definition}
A \textbf{$b$-map} from $(M,Z)$ to $(M',Z')$ is a {smooth} map $\phi:M \rightarrow M'$ such that $\phi^{-1}(Z') = Z$ and $\phi$ is transverse to $Z'$.
\end{definition}
}

Observe that if $x$ is a local defining function for $Z$ {and $(x, x_1, \ldots, x_{n-1})$ are local coordinates in a neighborhood of $p \in Z$} then {the $C^\infty(M)$-module of $b$-vector fields has the following local basis}
\begin{equation}\label{eq:generatebvectors}
\{x \frac{\partial}{\partial{x}}, \frac{\partial}{\partial{x_1}},\ldots, \frac{\partial}{\partial{x_{n-1}}}\}.
\end{equation}

In contrast to \cite{GMP10}, in this monograph we are not requiring the existence of a global defining function for $Z$ and orientability of $M$. However,  we require the existence of a defining  function in a neighborhood of each {point} of $Z$. {By relaxing this condition, the normal bundle of $Z$  need not  be trivial.}

{Given $(M,Z)$ a $b$-manifold, \cite{GMP10} shows that there exists a vector bundle,
denoted by $^b TM$  whose smooth sections are $b$-vector fields. This bundle
is called the \textbf{$b$-tangent bundle} of $(M,Z)$.}

The \textbf{$b$-cotangent bundle} $^b T^*M$ is defined using duality. A \textbf{$b$-form} is a  section of the $b$-cotangent bundle. {Around a point $p\in Z$ the $C^\infty(M)$-module of {these} sections has the following local basis:}
\begin{equation}\label{eq:generatebforms}
\{\frac{1}{x} dx, d x_1,\ldots, d x_{n-1}\}.
\end{equation}
In the same way we define a \textbf{$b$-form of degree $k$} to be a section of the bundle $\bigwedge^k(^b T^*M)$, the set of these forms is denoted $^b\Omega^k(M)$. Denoting by $f$ the distance function\footnote{Originally in \cite{GMP10} $f$ stands for a global function, but for non-orientable manifolds we may  use the distance function instead.} to the {critical hypersurface} $Z$, we may write the following decomposition as in \cite{GMP10} {for any $\omega \in ^b\Omega^k(M)$} :

\begin{equation}\label{eq:decomposition}
\omega=\alpha\wedge\frac{df}{f}+\beta, \text{ with } \alpha\in\Omega^{k-1}(M) \text{ and } \beta\in\Omega^k(M).
\end{equation}

This decomposition allows to extend the differential of the {de} Rham complex $d$ to $^b\Omega(M)$ by setting $d\omega=d\alpha\wedge\frac{df}{f}+d\beta.$ 

Degree $0$ functions are called $b$-functions and and near $Z$ can be written as
$$ c \log |x| + g, $$
where $c\in \R, g\in C^\infty,$ and $x$ is a local defining function.

\begin{figure}
\centering

\begin{tikzpicture}

\pgfmathsetmacro{\ellbasex}{7}
\pgfmathsetmacro{\ellbasey}{3.5}

\pgfmathsetmacro{\majoraxis}{2}
\pgfmathsetmacro{\minoraxis}{.5}

\pgfmathsetmacro{\rlineymid}{4.25}
\def\R{1.6}
\pgfmathsetmacro{\circlex}{1.5}

\draw[dashed, very thick, color = red] (\circlex + \R, \rlineymid) arc (0:180:{\R} and {\R * .2});
\draw[very thick, fill = magenta, opacity = .6] (\circlex, \rlineymid) circle (\R);
\draw[very thick] (\circlex, \rlineymid) circle (\R);
\draw[rotate around={-45:(\circlex, \rlineymid - .3)},dblue, fill = dblue] (\circlex, \rlineymid - .3) ellipse (.4 and .2);
\draw[very thick, color = red] (\circlex + \R, \rlineymid) arc (0:-180:{\R} and {\R * .2});

\node (circaround) at (\circlex, \rlineymid - .3) [circle,draw=none,thick, minimum size=1cm] {};
												
\draw[ultra thick, ->] (circaround)  to [bend right = 25] (\ellbasex - 1.2*\majoraxis, \ellbasey - .3);

\draw[very thick, fill = dblue] (\ellbasex - \majoraxis, \ellbasey) arc (-180:180:{\majoraxis} and {\minoraxis});
\node (baseone) at (\ellbasex - \majoraxis * .27, \ellbasey - \minoraxis * 1.25) {};
\node (basetwo) at (\ellbasex + \majoraxis * .27, \ellbasey + \minoraxis * 1.25) {};
\draw[very thick, red] (baseone) -- (basetwo);

\draw[draw = none, fill = orange] (\ellbasex - \majoraxis, \ellbasey + 1) arc (180:230:{\majoraxis} and {\minoraxis}) arc (-90:-47.4:.75) node (cylh) {} -- ++(0.95, 0.875) arc(-47.4:-90:.75) arc(100:180:{\majoraxis} and {\minoraxis}) -- cycle;
\draw[draw = none, fill = purple] (cylh)  arc (-47.4:0:.75) -- ++(0.95, 0.875) arc(0:-47.4:.75) -- cycle;

\draw[very thick] (\ellbasex - \majoraxis, \ellbasey + 1) arc (180:230:{\majoraxis} and {\minoraxis}) arc (-90:0:.75);
\draw[very thick] (\ellbasex - \majoraxis, \ellbasey + 1) arc (180:100:{\majoraxis} and {\minoraxis}) arc (-90:0:.75);

\draw[draw = none, fill = green] (\ellbasex + \majoraxis, \ellbasey + 1.1) arc (0:50:{\majoraxis} and {\minoraxis}) arc (270:180:.75) -- ++(-0.95, -0.875) arc(180:270:.75) arc (-80:0:{\majoraxis} and {\minoraxis}) -- cycle;

\draw[very thick] (\ellbasex +\majoraxis, \ellbasey + 1.1) arc (0:50:{\majoraxis} and {\minoraxis}) arc (270:180:.75);
\draw[very thick] (\ellbasex +\majoraxis, \ellbasey + 1.1) arc (0:-80:{\majoraxis} and {\minoraxis}) arc (270:180:.75);

\end{tikzpicture}
\caption{Artistic representation of a $b$-function on a $b$-manifold  near the critical hypersurface.} \label{fig:L1}
\end{figure}
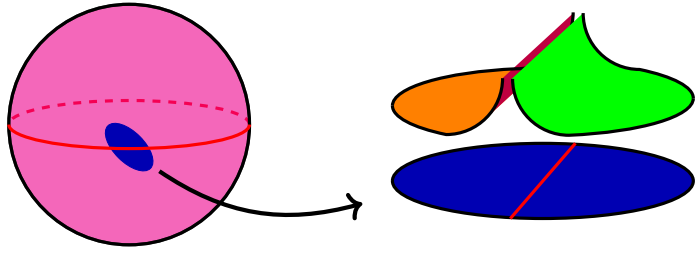

The associated cohomology is called \textbf{$b$-cohomology} and it is denoted by \textbf{$^b H^*(M)$}.

\begin{definition}
A \textbf{$b$-symplectic} form on a $b$-manifold $(M^{2n},Z)$ is defined as a non-degenerate closed $b$-form of degree $2$ (i.e., $\omega_p$ is of maximal rank as an element of $\Lambda^2(\,^b T_p^* M)$ for all $p\in M$).
\end{definition}

The notion {of} $b$-symplectic forms is dual to {the} notion of $b$-Poisson structures. The advantage of using forms rather than bivector fields is that symplectic tools can be `easily' exported.

Radko's classification theorem \cite{Radko02} can be translated  into this language. {This translation was already formulated in} \cite{GMP10}:

 \begin{theorem}[\textbf{Radko's theorem in $b$-cohomological language, \cite{GMP14}}] Let $S$ be a {closed} orientable surface and  let  $\omega_0$ and $\omega_1$ be two $b$-symplectic forms on {$(S,Z)$}  defining the same $b$-cohomology class (i.e.,$[\omega_0]= [\omega_1]$).  Then there exists a diffeomorphism {$\phi:S\rightarrow S$} such that $\phi^*\omega_1 = \omega_0$.
\end{theorem}

\section{On $b^m$-Symplectic manifolds}
\subsection{Basic definitions}
By relaxing the transversality condition allowing higher order singularities (\cite{ARNOLD85} and \cite{ARNOLD75}) we may consider other symplectic structures with singularities as done by
Scott \cite{Scott16} with $b^m$-symplectic structures.
{
 Let $m$ be a positive integer a \textbf{$b^m$-manifold} is a \textbf{$b$-manifold} $(M,Z)$ together with
 a $b^m$-tangent bundle attached to it. The $b^m$-tangent bundle is (by Serre-Swan theorem \cite{SWAN}) a vector bundle, $^{b^m} TM $ whose sections are given by, $$\Gamma(^{b^m}TM)=\{v \in \Gamma(TM): v(x) \quad \text{is tangent to order } m \text{ to } Z \},$$
where $x$ is a defining function for the critical set $Z$  in a neighborhood of each connected component of $Z$ and can be defined as
$x:M\setminus Z \rightarrow (0,\infty), x \in C^\infty(M)$ such that:
\begin{itemize}
\item $x(p) = d(p)$ a distance function from $p$ to $Z$ for $p: d(p) \leq 1/2$
\item $x(p) = 1$ on $M\setminus\{p\in M \text{ such that }d(p) < 1\}$.\footnote{{Then a $b^m$-manifold will be a triple $(M, Z, x)$, but for the sake of simplicity we refer to it as a pair $(M,Z)$ and we tacitly assume that the function $x$ is fixed.}}
\end{itemize}
(This definition of $x$ allows us to extend the construction in  \cite{Scott16}  to the non-orientable case as in \cite{MO2}.)
 We may define the notion of a $b^m$-map as a map in this category (see \cite{Scott16}).

The sections of this bundle are referred to as \textbf{$b^m$-vector fields} and their flows define $b^m$-maps.
In local coordinates, the sections of the $b^m$-tangent bundle  are generated by:
\begin{equation}\label{eq:generatebmvectors}
\{x^m \frac{\partial}{\partial{x}}, \frac{\partial}{\partial{x_1}},\ldots, \frac{\partial}{\partial{x_{n-1}}}\}.
\end{equation}
\noindent
}

Proceeding \emph{mutatis mutandis} as in the $b$-case one defines  the  $b^m$-cotangent bundle ($^{b^m} T^*M $), the $b^m$-{de} Rham complex and the $b^m$-symplectic structures.

A {\bf Laurent Series} of a closed $b^m$-form $\omega$ is a decomposition of $\omega$ in a tubular neighborhood $U$ of $Z$
of the form
\begin{equation}\label{eqn:laurent0}
\omega = \frac{dx}{x^m} \wedge (\sum_{i = 0}^{m-1}\pi^*(\alpha_{i})x^i) + \beta
\end{equation}

\noindent with $\pi: U \to Z$ the projection of the tubular neighborhood onto $Z$, $\alpha_{i}$ a closed smooth {de} Rham form on $Z$ and $\beta$  a {de} Rham form on $M$.

In \cite{Scott16} it is proved that in  a  neighborhood of $Z$, every   closed $b^m$-form  $\omega$ can be written in a  Laurent form of type (\ref{eqn:laurent0}) having fixed a (semi)local defining function.

 $b^m$-Cohomology is related to de Rham cohomology via the following theorem:
\begin{theorem}[\textbf{$b^m$-Mazzeo-Melrose}, \cite{Scott16}]\label{thm:Mazzeo-Melrose}
{Let $(M,Z)$ be a $b^m$-manifold, then:}
\begin{equation}\label{eqn:Mazzeo-Melrose}
 ^{b^m}H^p(M) \cong H^p(M)\oplus(H^{p-1}(Z))^m.
\end{equation}
\end{theorem}

{The isomorphism constructed in the proof of the theorem above is non-canonical (see \cite{Scott16}).}

The Moser path method can be generalized to $b^m$-symplectic structures (see \cite{evageoff} for the generalization from surfaces in \cite{Scott16} to general manifolds):

{\begin{theorem}[\textbf{Moser path method}]\label{mpm} Let $\omega_t$ be a path of $b^m$-symplectic forms defining the same $b^m$-cohomology class  $[\omega_t]$ on $(M^{2n}, Z)$ with $M^{2n}$ closed and orientable  then there exist  a $b^m$-symplectomorphism $\varphi: (M^{2n}, Z)\longrightarrow (M^{2n}, Z)$ such that $\varphi^*(\omega_1) = \omega_0$. \end{theorem}}

An outstanding consequence of Moser path method is a global classification of {closed} orientable $b^m$-symplectic surfaces \`{a} la Radko in terms of $b^m$-cohomology classes.

\begin{theorem}[\textbf{Classification of {closed} orientable $b^m$-{surfaces}, \cite{Scott16}}]\label{thm:scott}
Let $\omega_0$ and $\omega_1$ be two $b^m$-symplectic forms on a {closed orientable} connected $b^m$-surface $(S,Z)$.
Then, the following conditions are equivalent:
\begin{itemize}

  \item their $b^m$-cohomology classes coincide $[\omega_0] = [\omega_1]$,
  \item the surfaces are globally $b^m$-symplectomorphic,
  \item the Liouville volumes of $\omega_0$ and $\omega_1$ and the numbers $$\int_{\gamma} \alpha_{i}$$ for all connected
  components $\gamma \subseteq Z$ and all $1 \leq i \leq m$ coincide (where
  $\alpha_{i}$ are the one-forms appearing in the Laurent decomposition
  of the two $b^m$-forms {of degree 2}, $\omega_0$ and $\omega_1$).
\end{itemize}
\end{theorem}

\begin{definition}
The numbers $[\alpha_i] = \int_\gamma \alpha_i$ are called \emph{modular weights} for the connected components $\gamma \subset Z$.
\end{definition}

{A relative version of Moser's path method is proved in \cite{GMW17}. As a corollary we obtain the following local description of a $b^m$-symplectic manifold:}

\begin{theorem}[\textbf{$b^m$-Darboux theorem, {\cite{GMW17}}}]\label{theorem:Darbouxbn}
Let $\omega$ be a $b^m$-symplectic form on $(M,Z)$ and $p\in Z$. Then we can find a coordinate chart $(U,x_1,y_1,\ldots,x_n,y_n)$ centered at
$p$ such that on $U$ the hypersurface $Z$ is locally defined by $x_1=0$ and
$$\omega=\frac{d x_1}{x_1^m}\wedge {d y_1}+\sum_{i=2}^n d x_i\wedge d y_i.$$
\end{theorem}

{
\begin{remark}
 For the sake of simplicity sometimes we will omit any explicit reference to the critical set  $Z$ and we will talk directly about $b^m$-symplectic structures on manifolds $M$ implicitly assuming that $Z$ is the vanishing locus of  $\Pi^n$  where $\Pi$ is the Poisson vector field dual to the $b^m$-symplectic form.
\end{remark}
}

Next, we present two lemmas that allow us to talk about $b^m$-symplectic structures and $b^m$-Poisson as two different presentations of the same geometrical structure on a $b$-manifold. The lemma below shows that they are dual to each other and, thus, in one-to-one correspondence.

\begin{lemma}
Let $\omega$ be a $b^m$-symplectic and $\Pi$ its dual vector field, then $\Pi$ is a $b^m$-Poisson structure.
\end{lemma}

\begin{proof}
The quickest way to do this is to take the inverse, which is a bivector field, and observe that it is a Poisson structure (because $d\omega=0$ implies $[\Pi, \Pi]=0$). To see that it is $b^m$-Poisson
it is enough to check it locally for any point along the critical set. Take a point $p$ on the critical set $Z$ and apply the $b^m$-Darboux theorem  to get $\omega= dx_1/x_1^m\wedge dy_1+ \sum_{i>1} dx_i\wedge dy_i$
This means that in the new coordinate system $$\Pi=x_1^m\frac{\partial}{\partial x_1}\wedge \frac{\partial}{\partial y_1}+\sum_{i>1} \frac{\partial}{\partial x_i}\wedge \frac{\partial}{\partial y_i}$$
and thus $\Pi$ is a $b^m$-Poisson structure.

\end{proof}

Conversely,
\begin{lemma}
Let $\Pi$ be $b^m$-Poisson and $\omega$ its dual vector field, then $\omega$ is a $b^m$-symplectic structure.
\end{lemma}

\begin{proof}
 If $\Pi$ transverse à la Thom on $Z$ with singularity of order m then because of Weinstein's splitting theorem we can locally write

$$\Pi=x_1^m\frac{\partial}{\partial x_1}\wedge \frac{\partial}{\partial y_1}+\sum_{i>1} \frac{\partial}{\partial x_i}\wedge \frac{\partial}{\partial y_i}$$

now its inverse is $\omega= dx_1/x_1^m\wedge dy_1+ \sum_{i>1} dx_i\wedge dy_i$  which is a $b^m$-symplectic form.
\end{proof}

Hence we have a correspondence from $b^m$-symplectic structures to $b^m$-Poisson structures.

\section{{Desingularizing $b^m$-Poisson manifolds}}\label{sec:deblogging}

{In \cite{GMW17} Guillemin, Miranda and Weitsman presented a desingularization procedure for $b^m$-symplectic manifolds proving that we may associate a family of folded symplectic or symplectic forms to a given $b^m$-symplectic structure depending on the parity of $m$. Namely,}

\begin{theorem}[\textbf{Guillemin-Miranda-Weitsman}, \cite{GMW17}]\label{thm:deblogging}
 {Let $\omega$ be a $b^m$-symplectic structure  on a {closed orientable manifold} $M$  and let $Z$ be its critical hypersurface.
\begin{itemize}
\item If $m=2k$, there exists  a family of symplectic forms ${\omega_{\epsilon}}$ which coincide with  the $b^{m}$-symplectic form
    $\omega$ outside an $\epsilon$-neighborhood of $Z$ and for which  the family of bivector fields $(\omega_{\epsilon})^{-1}$ converges in
    the $C^{2k-1}$-topology to the Poisson structure $\omega^{-1}$ as $\epsilon\to 0$ .
\item If $m=2k+1$, there exists  a family of folded symplectic forms ${\omega_{\epsilon}}$ which coincide with  the $b^{m}$-symplectic form
    $\omega$ outside an $\epsilon$-neighborhood of $Z$.
\end{itemize}}

\end{theorem}

{As a consequence of Theorem \ref{thm:deblogging}, any closed orientable manifold that
supports a $b^{2k}$-symplectic structure necessarily supports a symplectic structure.}

{ In \cite{GMW17} explicit formulae are given for even and odd cases. Let us refer here to the even-dimensional case as these formulae will be used later on.}

Let us  briefly recall how the desingularization is defined and  the main result in \cite{GMW17}. Recall that we can express the $b^{2k}$-form as:

\begin{equation}\label{eq:laurent2}
\omega = \frac{dx}{x^{2k}}\wedge \left(\sum_{i=0}^{2k-1}x^i\alpha_i\right) + \beta.
\end{equation}

This expression holds on a $\epsilon$-tubular neighborhood of a given connected component of $Z$. This expression comes directly from equation \ref{eqn:laurent0}, to see a proof of this result we refer to \cite{Scott16}.

\begin{definition}\label{dfn:deblogging} {Let $(S,Z,x)$, be a $b^{2k}$-manifold, where $S$ is a closed orientable manifold and let $\omega$ be a $b^{2k}$-symplectic form. Consider  the decomposition given by the expression (\ref{eq:laurent2}) on an $\epsilon$-tubular neighborhood $U_\epsilon$ of a connected component of $Z$.}

	Let $f \in \mathcal{C}^\infty(\mathbb{R})$ be an odd smooth function satisfying $f'(x) > 0$ for all $x \in \left[-1,1\right] $ and satisfying outside that
\begin{equation}
f(x) = \begin{cases}
\frac{-1}{(2k-1)x^{2k-1}}-2& \text{for} \quad x < -1,\\
\frac{-1}{(2k-1)x^{2k-1}}+2& \text{for} \quad x > 1.\\
\end{cases}
\end{equation}
	Let $f_\epsilon(x)$ be defined as $\epsilon^{-(2k -1)}f(x/\epsilon)$.

The \textbf{$f_\epsilon$-desingularization} $\omega_\epsilon$ is {a form that is defined on $U_\epsilon$ by the following expression:}
	$$\omega_\epsilon = df_\epsilon \wedge \left(\sum_{i=0}^{2k-1}x^i\alpha_i\right) + \beta.
	$$
%\begin{figure}[h!]
%	\centering
%	\includegraphics[scale=.7]{Deblogging.eps}
%	\caption{Smooth and odd extension of $f$ inside the interval {$[-1,1]$ such that $f' > 0$}.}
%	\label{fig:deblogging}
%\end{figure}
\end{definition}

{This desingularization procedure is also known as \textbf{deblogging} in the literature.}

{
\begin{remark}
Though there are infinitely many choices for $f$, we will assume that we choose one, and assume it fixed through the rest of the discussion. It would be interesting to discuss the existence of an  isotopy of forms under a change of function $f$.
\end{remark}
}

{
\begin{remark}
Because $\omega_\epsilon$ can be trivially extended  to the whole $S$ in such a way that it agrees with $\omega$ (see  \cite{GMW17}) outside a neighborhood of $Z$, we can  talk about the $f_\epsilon$-desingularization of $\omega$  as a form on $S$.
\end{remark}
}

\chapter{A crash course on KAM theory}

The last part of this monograph is entirely dedicated to prove a KAM theorem for $b^m$-symplectic structures and to find applications. So the aim of this section is to give a quick overview of the traditional KAM theorem. The setting of the KAM theorem is a symplectic manifold with action-angle coordinates and an integrable system in it. The theorem says that under small perturbations of the Hamiltonian "most" of the Liouville tori survive.

Consider $\mathbb{T}^n\times G \subset \mathbb{T}^n \times \mathbb{R}^n$ with action-angle coordinates in it $(\phi_1,\ldots,\phi_n,I_1,\ldots,I_n)$ and the standard symplectic form $\omega$ in it. And assume the Hamiltonian function of the system is given by $h(I)$ a function only depending on the action coordinates. Then the Hamilton equations of the system are given by

$$\iota_{X_h} \omega = dh$$

where $X_h$ is the vector field generating the trajectories. Because $h$ does not depend on $\phi$ the angular variables the system is really easy to solve, and the equations are given by

$$x(t) = (\phi(t), I(t)) = (\phi_0 + ut, I_0),$$

where $u = \partial h/\partial I$ is called the frequency vector. These motions for a fixed initial condition are inside a Liouville torus, and are called quasi-periodic.

The KAM theorem studies what happens to such systems when a small perturbation is applied to the Hamiltonian function, i.e. we consider the evolution of the system given by the Hamiltonian $h(I) + R(I,\phi)$, where we think of the term $R(I,\phi)$ as the small perturbation in the system. With this in mind, the Hamilton equations can be written as

$$
\begin{array}{l}
\dot{\phi} = u(I) + \frac{\partial}{\partial I}R(I,\phi),
\dot{I} = -\frac{\partial}{\partial \phi}R(I,\phi),
\end{array}
$$

Another important concept to have in mind is the concept of \emph{rational dependency}. A frequency $u$ is rationally dependent if $\langle u , k \rangle= 0$ for some $k \in \mathbb{Z}^n$, if there exists no $k$ satisfying the condition then the vector $u$ is called rationally independent. There is a stronger concept of being rationally independent and that is the concept of being Diophantine. A vector $u$ is $\gamma$,$\tau$-diophantine if $\langle u, k\rangle \geq \frac{\gamma}{|k|_1^\tau}$ for all $k\in \mathbb{Z}^n\setminus\{0\}$. $\gamma > 0$ and $\tau > n-1$.

The KAM theorem states that the Liouville tori with frequency vector satisfying the diophantine condition survive under the small perturbation $R(I,\phi)$. There are conditions relating the size of the perturbation with $\gamma$ and $\tau$. Also, the set of tori satisfying the Diophantine condition has measure $1 - C\gamma$ for some constant $C$.

Now we give a proper statement of the theorem as was given in \cite{D}.

\begin{theorem}[Isoenergetic KAM theorem]\label{theorem:standard_KAM}
Let $\mathcal{G} \subset \mathbb{R}^n$, $n > 2$, a compact, and let $H(\phi,I) = h(I) + f(\phi,I)$ real analytic on $\mathcal{D}_\rho(\mathcal{G})$. Let $\omega = \partial h/\partial I$, and assume the bounds:

$$\left|\frac{\partial^2 h}{\partial I^2}\right|_{\mathcal{G}, \rho_2} \leq M, \quad |\omega|_{\mathcal{G}} \leq L \quad \text{and} \quad |\omega_n(I)| \geq l \forall I \in \mathcal{G}.$$

Assume also that $\omega$ is $\mu$-isoenergetically non-degenerate on $\mathcal{G}$. For $a = 16M/l^2$, assume that the map $\Omega = \Omega_{\omega,h,a}$ is one-to-one on $\mathcal{G}$, and that its range $F = \Omega(\mathcal{G})$ is a $D$-set. Let $\tau > n-1$, $\gamma > 0$ and $0 < \nu < 1$ given, and assume:

$$\varepsilon := \|f\|_{\mathcal{G},\rho} \leq \frac{\nu^2 l^6 \mu^2 \hat \rho^{2\tau+2}}{2^{4\tau+32}L^6 M^3}\cdot \gamma^2, \quad \gamma \leq \min\left(\frac{8L M\rho_2}{\nu l \hat \rho^{\tau+1}}, l\right),$$

where we write $\rho := \min \left(\frac{\nu \rho_1}{12(\tau+2)},1\right)$. Define the set

$$\hat G = \hat G_\gamma := \left\{I \in \mathcal{G} - \frac{2\gamma}{\mu} : \omega(I) \text{is} \tau,\gamma-\text{Diophantine} \right\}.$$

Then, there exists a real continuous map $\mathcal{T}:\mathcal{W}_{\frac{\rho_1}{4}}(\mathbb{T}^n)\times \hat G \rightarrow \mathcal{D}_\rho(\mathcal{G})$, analytic with respect to the angular variables, such that:

\begin{enumerate}
\item For every $I \in \hat G$, the set $\mathcal{T}(\mathbb{T}^n\times\{I\})$ is an invariant torus of $H$, its frequency vector is colinear to $\omega(I)$ and its energy is $h(I)$.
\item Writing
$$\mathcal{T}(\phi,I) = (\phi + \mathcal{T}_\phi(\phi,I), I + \mathcal{T}_I(\phi,I)),$$
one has the estimates
$$|\mathcal{T}_\phi|_{\hat G, (\frac{\rho_1}{4},0), \infty} \leq \frac{2^{2\tau+15} L^2 M}{\nu^2 l^2 \hat \rho^{2\tau+1}}\frac{\varepsilon}{\gamma^2}, \quad |\mathcal{T}_I|_{\hat G, (\frac{\rho_1}{4},0)} \leq \frac{2^{\tau+16} L^3 M}{\nu l^3 \mu \hat \rho^{\tau+1}}\frac{\varepsilon}{\gamma}$$
\item $\text{meas}[(\mathbb{T}^n\times \mathcal{G}) \setminus \mathcal{T}(\mathbb{T}^n\times \hat G)] \leq C\gamma$, where $C$ is a very complicated constant depending on $n$, $\tau$, $\text{diam} F$, $D$, $\hat{\rho}$, $M$,  $L$,  $l$,  $\mu$.

\end{enumerate}

\end{theorem}

\begin{remark}
This version of the KAM theorem is the isoenergetic one, this version ensures that the energy of the Liouville Tori identified by the diffeomorphism after the perturbation remains the same as before the perturbation. Our version of the $b^m$-KAM is not isoenergetic for the sake of simplifying the computations.
\end{remark}

Also, we should outline that the KAM theorem has already been explored in singular symplectic manifolds before. In \cite{KMS16} the authors proved a KAM theorem for $b$-symplectic manifolds, for a particular kind of perturbations.

\begin{theorem}[KAM Theorem for $b$-Poisson manifolds]

Let $\T^n \times B_r^n$ be endowed with standard coordinates $(\varphi,y)$ and the $b$-symplectic structure. Consider a $b$-function
$$H = k \log|y_1| + h(y)$$
 on this manifold, where $h$ is analytic. Let $y_0$ be a point in $B_r^n$ with first component equal to zero, so that the corresponding level set $\T^n \times \{y_0\}$ lies inside the critical hypersurface $Z$.

Assume that the frequency map
$$\tilde \omega: B^n_r \to \R^{n-1}, \quad \tilde \omega( y):= \frac{\partial h}{\partial \tilde y}(y)$$
 has a Diophantine value $\tilde \omega := \tilde \omega(y_0)$ at $y_0 \in B^n$ and that it is non-degenerate at $y_0$ in the sense that the Jacobian $ \frac{\partial \tilde \omega}{\partial \tilde y} (y_0) $ is regular.

Then the torus $\T^n \times \{ {y_0}\}$  persists under sufficiently small perturbations of $H$ which have the form mentioned above, i.e. they are given by $\epsilon P$, where $\epsilon \in \R$ and $P \in ^b \! C^\infty(\T^n \times B_r^n)$ has the form
\begin{align*}
 P(\varphi, y) &= k' \log |y_1| + f(\varphi,y) \\
 f(\varphi, y) &= f_1(\tilde \varphi, y ) + y_1 f_2(\varphi, y) + f_3(\varphi_1,y_1).
\end{align*}
More precisely, if $|\epsilon|$ is sufficiently small, then the  perturbed system
  $$ H_\epsilon =H + \epsilon P$$
admits an invariant torus $\mathcal{T}$.% close to $\T^n \times \{ {y_0}\}$.

Moreover, there exists a diffeomorphism $\T^n \to \mathcal{T}$ close\footnote{By saying that the diffeomorphism is ``$\epsilon$-close to the identity'' we mean that, for given $H, P$ and $r$, there is a constant $C$ such that $\|\psi - \id\| < C \epsilon.$} to the identity taking the flow $\gamma^t$  of the perturbed system on $\mathcal{T}$ to the linear flow  on $\T^n$ with frequency vector
$$ \left(\frac{k+\epsilon k'}{c}, \tilde \omega \right).$$
\end{theorem}

%
%\include{chap1}
%-----------------------------------------------------------------------
% Beginning of chap1.tex
%-----------------------------------------------------------------------
%
%  AMS-LaTeX sample file for a chapter of a monograph, to be used with
%  an AMS monograph document class.  This is a data file input by
%  chapter.tex.
%
%  Use this file as a model for a chapter; DO NOT START BY removing its
%  contents and filling in your own text.
%
%%%%%%%%%%%%%%%%%%%%%%%%%%%%%%%%%%%%%%%%%%%%%%%%%%%%%%%%%%%%%%%%%%%%%%%%

\part[Action-angle coordinates and cotangent models]{Action-angle coordinates and cotangent models for $b^m$-integrable systems}

In this part,  we consider the semilocal classification for any $b^m$-Poisson manifold in a neighbourhood of an invariant compact submanifold.
The compact submanifolds under consideration are the compact invariant leaves of the distribution $\mathcal D$ generated by the Hamiltonian vector fields $X_{f_i}$ of an integrable system. An integrable system is given by a set of $n$ functions on a $2n$-dimensional symplectic manifold  which we can order in a map $F=(f_1, \dots, f_n)$. Historically, integrable systems were introduced to actually \emph{integrate} Hamiltonian systems $X_H$ using the \emph{first-integrals $f_i$} and, classically, we identify $H=f_1$. It turns out that in the symplectic context the compact regular orbits of the distribution $\mathcal D$ coincide with the fibers $F^{-1}(F(p))$ for any point $p$ on these orbits/fibers.
The fact that the orbit coincides with the connected fiber is part of the magic of symplectic duality.

The same picture is reproduced for singular symplectic manifolds of $b^m$-type or $b^m$-Poisson manifolds as we will see in this chapter.

The study of action-angle coordinates has interest from this geometrical point of view of the classification of geometric structures in a neighbourhood of a compact submanifold of a $b^m$-Poisson manifold. It also has interest from a dynamical point of view as these compact submanifolds now coincide with invariant subsets of the Hamiltonian system under consideration.

From a geometric point of view, the existence of action-angle coordinates determines a \emph{unique} geometrical model for the $b^m$-Poisson (or $b^m$-symplectic) structure in a neighbourhood of the invariant set. From a dynamical point of view, the existence of action-angle coordinates provides a normal form theorem that can be used to study stability and perturbation problems of the Hamiltonian systems (as we will see in the last chapter of this monograph).

An important ingredient that makes our action-angle coordinate theorem \emph{brand-new } from the symplectic perspective is that the system under consideration is more general than Hamiltonian, it is $b^m$-Hamiltonian as the first-integrals of the system can be $b^m$-functions which are not necessarily smooth functions. Dynamically, this means that we are adding to the set of Hamiltonian invariant vector fields, the \emph{modular vector field} of the integrable system.

In contrast to the standard action-angle coordinates for symplectic manifolds, our action-angle theorem comes with $m$ additional invariants associated with the modular vector field which can be interpreted in cohomological terms as the projection of the $b^m$-cohomology class determined by the modular vector field on the first cohomology group of the critical hypersurface under the Mazzeo-Melrose correspondence.

The strategy of the proof of the action-angle coordinate systems is the search of a toric action (so this takes us back to the motivation of the use of \emph{symmetries} in this monograph). In contrast to the symplectic case, it is not enough that this action is Hamiltonian as then a direction of the Liouville torus would be missing. We need the toric action to be $b^m$-Hamiltonian.
The structure of this proof looks like the one in \cite{KMS16} but encounters serious technical difficulties as in order to check that the \emph{natural} action to be considered is
$b^m$-Hamiltonian we need to go deeper inspired by \cite{Scott16}  in the relation between the geometry of the modular vector field and the coefficients of the Taylor series $c_i$ of one of the first-integrals. This allows us to understand new connections between the geometry and analysis of $b^m$-Poisson structures not explored before.

Once we prove the existence of this $b^m$-Hamiltonian action the proof looks very close to the one in \cite{KMS16}.

In the second chapter of this part
we re-state the action-angle theorem as a cotangent lift theorem with the following mantra:

 \emph{Every integrable system on a $b^m$-Poisson manifold looks like a $b^m$-cotangent lift in a neighborhood of a Liouville torus.}

\chapter{An action-angle theorem for $b^m$-symplectic manifolds}

\section{Basic definitions}
\subsection{On $b^m$-functions}
 The definition of the analogue of $b$-functions in the $b^m$-setting is somewhat delicate. The set of $^{b^m}\mathcal{C}^\infty(M)$ needs to be such that for all the functions $f\in ^{b^m}\mathcal{C}^\infty(M)$, its differential $df$ is a $b$-form, where $d$ is the $b^m$-exterior differential.
Recall that a form in $^{b^m}\Omega^k(M)$ can be locally written as
$$\alpha\wedge\frac{dx}{x^m} + \beta$$
where $\alpha \in \Omega^{k-1}(M)$ and $\beta \in \Omega^{k}(M)$. Recall also that
$$d\left(\alpha\wedge\frac{dx}{x^m} + \beta\right) = d\alpha\wedge\frac{dx}{x^m} + d\beta.$$

We need $df$ to be a well-defined $b^m$-form of degree $1$. Let $f = g\frac{1}{x^{k-1}}$, then $df = dg \frac{1}{x^{k-1}} - g\frac{k-1}{x^k}dx$. This from can only be a $b^m$-form if and only if $g$ only depends on $x$. If $f = g\log(x)$, then $dg\log(x) + g\frac{1}{x}dx$, which imposes $dg=0$ and hence $g$ to be constant.

With all this in mind, we make the following definition.
\begin{definition}
The set of $b^m$-functions is defined recursively according to the formula $$~^{b^m} \mathcal{C}^\infty(M)= x^{-(m-1)}\mathcal{C}^\infty(x) + ~^{b^{m-1}} \mathcal{C}^\infty(M)$$

\noindent  {with $\mathcal{C}^\infty(x)$ the set of smooth functions in the defining function $x$} and $$^{b}\mathcal{C}^\infty(M)=\{ g \log\vert x\vert+ h, g \in \mathbb{R}, h\in \mathcal{C}^\infty(M)\}.$$

\end{definition}

\begin{remark}
A $^{b^m} \mathcal{C}^\infty (M)$-function can be written as
$$f = a_0 \log x + a_1\frac{1}{x} + \ldots + a_{m-1}\frac{1}{x^{m-1}} + h$$
where $a_i, h \in \mathcal{C}^\infty(M)$.
\end{remark}

%\begin{remark}
%The Hamiltonian vector field associated to a $b^m$-function $H$ is a smooth vector field. Because locally we can take the expressions:
%$$\Pi = x_1^m\frac{\partial}{\partial x_1} \wedge \frac{\partial}{\partial y_1} + \sum_{i = 2}^{m} \frac{\partial}{\partial x_i}\wedge\frac{\partial}{\partial y_i} \text{ and }H = c_0 \log{x_1} + \sum_{i = 1}^{m-1}c_i \frac{1}{x_1^i} + f.$$
%Then if we compute $X_H = \Pi(dH,\cdot) = c_0 y^{m-1}\frac{\partial}{\partial y_1} + \sum_{i = 1}^{m-1} c_i i y^{m-i-1} \frac{\partial}{\partial y_1} + \Pi(f,\cdot)$, we obtain a smooth vector field.
%\end{remark}

\begin{remark}
From this chapter on we are only considering $b^m$-manifolds $(M,x,Z)$ with $x$ defined up to order $m$. I.e. we can think of $x$ as a jet of a function that coincides up to order $m$ to some defining function.
This is the original viewpoint of Scott in \cite{Scott16} which we adopt from now on. The difference with respect to the other chapters is that we do not fix an specific function.
%The reader may as well adjust her/his reading to assume the function has been fixed in case this is more convenient.
\end{remark}

\begin{definition}
We say that two $b^m$-integrable systems $F_1, F_2$ are equivalent if there exists $\varphi$, a $b^m$-symplectomorphism, i.e. a diffeomorphism preserving both $\omega$ and the critical set $Z$ (``up to order $m$''\textcolor{black}{\footnote{I.e. it preserves the jet $x$}}), such that  $\varphi \circ F_1 = F_2$.
\end{definition}

\begin{remark}
The Hamiltonian vector field associated to a $b^m$-function $f$ is a smooth vector field. Let us compute it  locally using the $b^m$-Darboux theorem:
$$\Pi = x_1^m\frac{\partial}{\partial x_1} \wedge \frac{\partial}{\partial y_1} + \sum_{i = 2}^{m} \frac{\partial}{\partial x_i}\wedge\frac{\partial}{\partial y_i} \text{ and } f = a_0 \log{x_1} + \sum_{i = 1}^{m-1}a_i \frac{1}{x_1^i} + h.$$
Then if we compute

$$\begin{array}{rcl}
\displaystyle df & = & \displaystyle \overbrace{a_0}^{c_1}\frac{1}{x_1} + \sum_{i=1}^{m-1} \overbrace{(a_i' - (i-1)a_{i-1})}^{c_i}\frac{1}{x_1^i} dx_1 \\
& & \quad -\overbrace{(m-1)a_{m-1}}^{c_m}\frac{1}{x_1^m}dx_1 + dh\\
\\
 & = & \displaystyle \sum_{i = 1}^{m}\frac{c_i}{x_1^i} dx_1 + dh.
\end{array}$$
Then,

\begin{equation}\label{eq:bmhamiltonianvf}
X_f = \Pi(df,\cdot) = \sum_{i=1}^{m} c_i x_1^{m-i}\frac{\partial}{\partial y_1} + \Pi(dh,\cdot),
\end{equation}

 we obtain a smooth vector field.

\end{remark}

%
%
%Let $(M^{2n},Z,x)$ be a $b^m$-manifold. The set of $b^m$-functions $^{b^m} \mathcal{C}^\infty (M)$ is defined by:
%$$^{b^m} \mathcal{C}^\infty (M) = x^{-(m-1)}\mathcal{I}\oplus ^b\mathcal{C}^\infty(M)$$
%where $^b\mathcal{C}^\infty(M) = \{c\log(x) + h \quad h \in \mathcal{C}^\infty(M), c \in \mathbb{R}\}$ and $\mathcal{I}$ are the functions that locally around $Z$ only depend on $x$}

\section{On $b^m$-integrable systems}

In this section we present the definition of a $b^m$-integrable system as well as some observations about these objects.

\begin{definition}
Let $(M^{2n},Z,x)$ be a $b^m$-manifold, and let $\Pi$ be a $b^m$-Poisson structure on it.
$F = (f_1, \ldots, f_n)$\footnote{$f_i$ are $b^m$-functions.} is a $b^m$-integrable system\footnote{ In this monograph we only consider integrable systems of maximal rank $n$.} if:

\begin{enumerate}%[i)]
\item $df_1,\ldots,df_n$ are independent on a dense subset of $M$ \textcolor{black}{and in all the points} of $Z$ where independent means that the form $df_1\wedge\ldots\wedge df_n$ is non-zero as a section of $\Lambda^n(^{b^m} T^{*}(M))$,
\item  the functions $f_1,\ldots, f_n$ Poisson commute pairwise.
\end{enumerate}
\end{definition}

\begin{definition}
 The points of $M$ where $df_1,\ldots, df_n$ are independent  are called \textbf{regular} points.
\end{definition}

The next remarks will lead us to a normal form for the first function $f_1$.
\begin{remark}\label{independence}
Note that $df_1,\ldots, df_n$ are independent on a point if and only if $X_{f_1},\ldots,X_{f_n}$ are independent at that point. This is because the map
$$^{b^m}TM \rightarrow ^{b^m}T^*M:u\mapsto \omega_p(u,\cdot)$$
is an isomorphism.
\end{remark}

\begin{remark}
The condition of $d f_1, \ldots, d f_n$ being independent must be understood as $d f_1 \wedge \ldots \wedge df_n$ being a non-zero section of $\bigwedge^n (\enskip ^{b^m}T^*M)$.

\end{remark}

\begin{remark}\label{rk:hamiltonian_vf}
By remark \ref{independence}  the vector fields $X_{f_1},\ldots, X_{f_n}$ have to be independent. This implies that one of the $f_1,\ldots, f_n$ has to be a singular (non-smooth) $b^m$-function with a singularity of maximal degree. If we write $f_i = c_{0,i}\log(x_1) + \sum_{j=1}^{m-1} \frac{c_{j,i}}{x_1^j} + \tilde{f}_1$
$$X_{f_i} = \sum_{j = 1}^{m} x_1^{m-j} \hat{c}_{j,i} \frac{\partial}{\partial y_1} +  X_{\tilde{f}_i}$$
where $\hat c_{j,i}(x) = \frac{d(c_{j,i})}{dx}-(j-1)c_{j-1,i}$. If there is no $b^m$-function with a singularity of maximum degree all the terms in the $\partial/\partial y_1$ direction become 0 at $Z$. And hence $X_{f_1},\ldots, X_{f_n}$ cannot have maximal rank at $Z$.
% begaus hamiltonian VF can only have up to rank n
\end{remark}

%\begin{remark}
%If one of the $f_i$ has a singularity of maximal degree, no other function can have a singularity, because $d f_1 \wedge \ldots \wedge d f_n$ would not be a $b^m$-form and in particular would not be a non-zero section of $\bigwedge^n \enskip ^{b^m}TM$.
%\end{remark}

%\begin{remark}
%From now on we will assume that given a $b^m$-integrable system $(f_1, \ldots, f_n)$, $f_1,\dots,f_{n-1}$ are smooth functions and $f_n = c_{0}\log(x_1) + \sum_{j=1}^{m-1} \frac{c_{j}}{x_1^j}$. Observe that $f_n$ we are taking the smooth part of $f_n$ equal to 0, $\tilde f_n = 0$. This is not necessarily true. But we impose this as a necessary condition to a $b^m$-integrable system

%\end{remark}

\textcolor{black}{
\begin{lemma}
Let $F = (f_1,\ldots,f_n)$ a $b^m$-integrable system. If $f_1$ has a singularity of maximal degree, there exists an equivalent integrable system $F' = (f_1',\ldots,f_n')$ where $f_1'$ has a singularity of maximal degree and \textbf{no other} $f_i'$ has singularity of any degree.
\end{lemma}
\begin{proof}
Let $f_i = \underbrace{c_{0,i}\log(x_1) + \sum_{j=1}^{m-1}\frac{c_{j,1}}{x_1^j}}_{\zeta_i(x_1)} + \tilde f_i = \zeta_i(x_1) + \tilde f _i$.
By remark \ref{rk:hamiltonian_vf}\footnote{Here have used the $b^m$-Darboux theorem to do the computations.},
$$X_{f_i} = \underbrace{\sum_{i=1}^{m} x_1^{m-j} \hat c_{j,i}}_{g_i(x_1)}\frac{\partial}{\partial y_1} + X_{\tilde f_i} = g_i(x_1)\frac{\partial}{\partial y_1} + X_{\tilde f_i}.$$
Note that $g_i(x_1) = g_i(0) = \hat c_{m,i}$ at $Z$. Let us look at the distribution given by the Hamiltonian vector fields $X_{f_i} = g_i(x_1)\frac{\partial}{\partial y_1} + X_{\tilde f_i}$. This distribution is the same that the one given by:
\begin{equation}\label{eq:distribution}
\{X_{f_1}, X_{f_2} - \frac{g_2(x_1)}{g_1(x_1)} X_{f_1},\ldots, X_{f_n} - \frac{g_n(x_1)}{g_1(x_1)} X_{f_1}\}.
\end{equation}
Observe that for $i > 1$, $X_{f_i} - \frac{g_i(x_1)}{g_1(x_1)} X_{f_1} = X_{\tilde f_i} + \frac{g_2(x_1)}{g_1(x_1)} X_{\tilde f_1}$. Also $g_1(x_1)$ is different from $0$ close to $Z$ because at $Z$ $g_1(x_1) = \hat c_{m,1}$.
Since the distribution given by these vector fields is the same, an integrable system that has Hamiltonian vector fields \ref{eq:distribution} would be equivalent to $F$. From the expression \ref{eq:distribution} it is clear that the new vector fields commute. And it is also true that this new vector fields are Hamiltonian. Let us take $F'$ the set of functions that have as Hamiltonian vector fields \ref{eq:distribution}.
\end{proof}
}

From now on we will assume the integrable system to have only one singular function and this function to be $f_1$.

\begin{remark}
Because we asked  $X_{f_1},\ldots,X_{f_n}$ to be linearly independent at all the points of $Z$ and using the previous remarks  $c_m := c_{m,1} \neq 0$ at all the points of $Z$.
\end{remark}

Furthermore, we can assume $f_1$  to have a smooth part equal to zero as subtracting the smooth part of $f_1$ to all the functions gives an equivalent system.
Also, we can assume that $c_m$ is 1 because dividing all the functions of the $b^m$-integrable system by $c_m$ also gives us an equivalent system.

\textbf{As a summary, we can assume $f_1 = a_0 \log(x) + a_1 1/x + \ldots + a_{m-2}1/x^{m-2} + 1/x^{m-1}$ and $f_2, \ldots, f_n$ to be smooth, $a_0\in \mathbb{R}$ and $a_1, \ldots, a_{m-2} \in \mathcal{C}^\infty(x)$.}

Also we are going to state lemma 3.2 in \cite{GMP17}, because we are going to use it later in this section. The result states that if we have a toric action on a $b^m$-symplectic manifold (which we will prove in a neighbourhood of a Liouville torus), then we can assume the coefficients $a_2, \ldots, a_{m-2}$ to be constants. More precisely

\begin{lemma}\label{lemma:ctt_coefs}
There exists a neighborhood of the critical set $U = Z\times (-\varepsilon, \varepsilon)$ where the moment map $\mu: M \rightarrow \mathfrak{t}^*$ is given by
$$\mu = a_1 \log|x| + \sum_{i=2}^{m}a_i \frac{x^{-(i-1)}}{i-1} + \mu_0$$
with $a_i \in \mathfrak{t}^0_L$ and $\mu_0$ is the moment map for the $T_L$-action on the symplectic leaves of the foliation.
\end{lemma}

\section{Examples of $b^m$-integrable systems}

The following example illustrates why it is necessary to use the definition of $b^m$-function as considered above. There are natural examples of changes of coordinates in standard integrable systems on symplectic manifolds that yield  $b^m$-symplectic manifolds but do not give well-defined $b^m$-integrable systems.

\begin{example}
Consider a time change in the two body problem, to obtain a $b^2$-integrable system. In the classical approach to solve the $2$-body problem  the following two conserved quantities are obtained:
$$\begin{array}{rcl}
f_1 & = & \frac{\mu y^2}{2} + \frac{l^2}{2\mu r^2} - \frac{k}{r},\\
f_2 & = & l,
\end{array}
$$
with symplectic form $\omega = dr\wedge dy + dl \wedge d\alpha$, where $r$ is the distance between the two masses and $l$ is the angular momentum.
We also know that $l$ is constant along the trajectories. Because $l$ is a constant of the movement, we can do a symplectic reduction on its level sets. The form on the symplectic reduction becomes $dr \wedge dy$. To simplify the notation, we will use $x$ instead of $r$.
Then $\omega = dx\wedge dy$. With hamiltonian function given by $f = \frac{\mu}{2}y^2 + \frac{l}{2\mu}\frac{1}{x^2} - k \frac{1}{x}$. Hence, the equations are:
$$
\begin{array}{rcl}
 \dot{x} & = & \frac{\partial f}{\partial y},\\
 \dot{y} & = & -\frac{\partial f}{\partial x}.
\end{array}
$$
Doing a time change $\tau = x^3 t$ then $\frac{d x}{d\tau} = \frac{1}{x^3} \frac{dx}{dt}$. With this time coordinate, the equations become:
$$
\begin{array}{rcl}
 \dot{x} & = & \frac{1}{x^3}\frac{\partial f}{\partial y},\\
 \dot{y} & = & -\frac{1}{x^3}\frac{\partial f}{\partial x}.
\end{array}
$$
These equations can be viewed as the motion equations given by a $b^3$-symplectic form $\omega = \frac{1}{x^3} dx\wedge dy$.

Let us check that this is actually a $b^m$-integrable system.
\begin{itemize}%[a)]
\item All the functions Poisson commute is immediate because we only have one.
\item $df = \mu y dy + (\frac{k}{x^2} - \frac{l}{\mu}\frac{1}{x^3})dx$ is a $b^3$-form because the term with $dx$ does not depend on $y$.
\item All the functions are independent, this is true because $df$ does not vanish as a $b^3$-form.
\end{itemize}
\end{example}

\begin{example}
In the paper \cite{Marle} the author builds an action of $SL(2,\mathbb{R})$ over $(P,\omega_P)$ where
$P = \{\xi\in \mathbb{C} | i(\bar{\xi} - \xi)>0\}$ is the complex semi-plane, with moment map $J_P(\xi) = \frac{R}{\xi_{im}}((|\xi|^2 + 1),2\xi_r,\pm (|\xi|^2 + 1))$, where the $\pm$ sign depends on the choice of the hemisphere projected by the stereographic projection. From now on we will take the sign $+$. Also the symplectic form $\omega_P$ has the following expression:

$$\omega_P = \pm \frac{R}{\xi_{im}^2} d\xi_r\wedge d\xi_{im}$$

In order to simplify the notation we identify $P$ with the real half-plane $P = \{x,y\in\mathbb{R}^2|y>0\}$. With this identification, the moment map becomes $J_p(x,y)=\frac{R}{y}(x^2 + y^2 + 1, 2x, x^2 + y^2 + 1)$. Obviously, this moment map does not give an integrable system. The symplectic form writes as:

$$\omega_P = \frac{R}{y^2} dy\wedge dx.$$

This form can be viewed as a $b^2$-form if we extend $P$ including the line $\{y=0\}$ as its singular set.
Let us consider only one of the components of $J_P$ as $b^m$-function and let us see if it gives a $b^m$-integrable system. First we will try with $f_1 = \frac{R}{y}(x^2 + y^2 + 1) $ and then $f_2 = \frac{R}{y}(2x)$.
\begin{enumerate}%[i)]
\item $f_1 = \frac{R}{y}(x^2 + y^2 + 1)$
We have to check three things to see if this gives a $b^2$-integrable system.
\begin{enumerate}%[a)]
\item All the functions Poisson commute is immediate because we only have one.
\item All the functions are $b^m$-functions. This point does not hold because $d f_1= \frac{R}{y^2}(2xydx + (y^2-x^2 - 1)dy)$ and the first component makes no sense as a section of $\Lambda^1 (^{b^2}T^*M)$.
\item All the functions are independent. In this case, we need to check that $df_1$ does not vanish, but since it is not a $b^m$-form it makes no sense  to be a non-zero section of $\Lambda^1 (^{b^2}T^*M)$.
\end{enumerate}
\item $f_2 = \frac{R}{y}(2x)$
\begin{enumerate}

%[a)]
\item Same as before.
\item All the functions are $b^m$-functions. This point does not hold because $d f_2= \frac{2R}{y}dx - \frac{2Rx}{y^2}dy$ and the first component makes no sense as a section of $\Lambda^1 (^{b^2}T^*M)$.
\item Same as before.
\end{enumerate}
\end{enumerate}

\end{example}

\begin{example} Toric actions give natural examples of integrable systems where the component functions are given by the moment map. In the case of surfaces:
    $S^1$-actions on surfaces give natural examples of $b^m$-integrable systems. Only torus and spheres admit circle actions. 
    
    In the picture below two integrable systems on the $2$-sphere depending on the degree $m$. On the right the image of the moment map that defines the integrable system. The action is by rotations along the central axis.

    Namely consider the sphere $S^2$ as a  $b^m$-symplectic manifold having as critical set the equator:
        
        \[(S^2, Z = \{h = 0\}, \omega=\frac{d h}{h^m}\wedge d\theta),\] with $h\in\left[-1,1\right]$ and $\theta\in\left[0,2\pi\right)$.

For $m=1$: The computation $\iota_{\frac{\partial}{\partial \theta}}\omega=- \frac{d h}{h}=-d( \log |h|),$ tells us that the function $\mu(h,\theta)= \log |h|$ is the moment map and defines a $b$-integrable system.

For higher values of $m$: $\iota_{\frac{\partial}{\partial \theta}}\omega=- \frac{d h}{h^m}=-d(-\frac{1}{(m-1)h^{m-1}}),$ and the  moment map is $\mu(h,\theta)= -\frac{1}{(m-1)h^{m-1}}$ which defines a $b^m$-integrable system.
    
 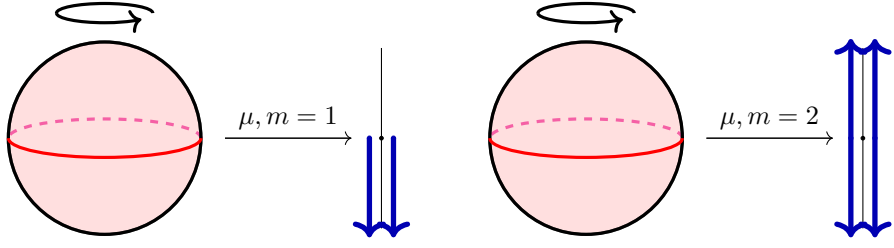
\begin{figure}[h!]
            \begin{tikzpicture}[scale=0.8]
                \pgfmathsetmacro{\rlinex}{6}
                \pgfmathsetmacro{\baseptd}{8}
                \pgfmathsetmacro{\rlineybottom}{2.75}
                \pgfmathsetmacro{\rlineymid}{4.25}
                \pgfmathsetmacro{\rlineytop}{5.75}
                \pgfmathsetmacro{\vertstretch}{1.9}	
                \pgfmathsetmacro{\yshift}{4.25}	

                \def\R{1.6}
                \pgfmathsetmacro{\circlex}{1.4}
                \draw[dashed, very thick, color = magenta] (\circlex + \R, \rlineymid) arc (0:180:{\R} and {\R * .2});
                \draw[very thick, fill = pink, opacity = .5] (\circlex, \rlineymid) circle (\R);
                \draw[very thick] (\circlex, \rlineymid) circle (\R);
                \draw[very thick, color = red] (\circlex + \R, \rlineymid) arc (0:-180:{\R} and {\R * .2});
                
                \draw (\circlex + 2, \rlineymid) edge node[above] {$\mu, m = 1$} (\rlinex - 0.5, \rlineymid);
                \draw[->] (\circlex + 2, \rlineymid) -- (\rlinex - 0.5, \rlineymid);

                \draw (\rlinex, \rlineybottom) -- (\rlinex, \rlineytop) node[right] {};

                \draw[black, fill = black] (\rlinex, \rlineymid) circle(.3mm);

                \draw[dblue, fill = dblue] (\rlinex - 0.2, \rlineymid) circle (1pt);
                \draw[dblue, fill = dblue] (\rlinex + 0.2, \rlineymid) circle (1pt);

                \draw[line width = 2pt, join = round, dblue, <-] (\rlinex - 0.2, \rlineybottom - 0.15) -- (\rlinex - 0.2, \rlineymid );
                \draw[line width = 2pt, join = round, dblue, <-] (\rlinex + 0.2, \rlineybottom - 0.15) -- (\rlinex + 0.2, \rlineymid);
                \draw [very thick, ->] (\circlex, \rlineymid + 1.3*\R) ++(\R * -.5, 0) arc (180:320: {\R * .5} and {\R * .1});
                \draw [very thick] (\circlex, \rlineymid + 1.3*\R) ++(\R * -.5, 0) arc (180:0: {\R * .5} and {\R * .1});

                \coordinate (shift) at (8,0);
                \begin{scope}[shift=(shift)]
                \pgfmathsetmacro{\rlinex}{6}
                \pgfmathsetmacro{\baseptd}{8}
                \pgfmathsetmacro{\rlineybottom}{2.75}
                \pgfmathsetmacro{\rlineymid}{4.25}
                \pgfmathsetmacro{\rlineytop}{5.75}
                \pgfmathsetmacro{\vertstretch}{1.9}	
                \pgfmathsetmacro{\yshift}{4.25}	

                \def\R{1.6}
                \pgfmathsetmacro{\circlex}{1.4}
                \draw[dashed, very thick, color = magenta] (\circlex + \R, \rlineymid) arc (0:180:{\R} and {\R * .2});
                \draw[very thick, fill = pink, opacity = .5] (\circlex, \rlineymid) circle (\R);
                \draw[very thick] (\circlex, \rlineymid) circle (\R);
                \draw[very thick, color = red] (\circlex + \R, \rlineymid) arc (0:-180:{\R} and {\R * .2});
                
                \draw (\circlex + 2, \rlineymid) edge node[above] {$\mu, m = 2$} (\rlinex - 0.5, \rlineymid);
                \draw[->] (\circlex + 2, \rlineymid) -- (\rlinex - 0.5, \rlineymid);

                \draw (\rlinex, \rlineybottom) -- (\rlinex, \rlineytop) node[right] {};

                \draw[black, fill = black] (\rlinex, \rlineymid) circle(.3mm);

                \draw[dblue, fill = dblue] (\rlinex - 0.2, \rlineymid) circle (1pt);
                \draw[dblue, fill = dblue] (\rlinex + 0.2, \rlineymid) circle (1pt);

                \draw[line width = 2pt, join = round, dblue, <-] (\rlinex - 0.2, \rlineybottom - 0.15) -- (\rlinex - 0.2, \rlineymid );
                \draw[line width = 2pt, join = round, dblue, ->] (\rlinex - 0.2, \rlineymid) -- (\rlinex - 0.2, \rlineytop + 0.15);
                
                \draw[line width = 2pt, join = round, dblue, <-] (\rlinex + 0.2, \rlineybottom - 0.15) -- (\rlinex + 0.2, \rlineymid );
                \draw[line width = 2pt, join = round, dblue, ->] (\rlinex + 0.2, \rlineymid) -- (\rlinex + 0.2, \rlineytop + 0.15);
                
                \draw [very thick, ->] (\circlex, \rlineymid + 1.3*\R) ++(\R * -.5, 0) arc (180:320: {\R * .5} and {\R * .1});
                \draw [very thick] (\circlex, \rlineymid + 1.3*\R) ++(\R * -.5, 0) arc (180:0: {\R * .5} and {\R * .1});
                \end{scope}
            \end{tikzpicture}
            \caption{Integrable systems associated to the moment map of an $S^1$-action by rotations  on a $b^m$-symplectic $2$-sphere $S^2$.}
            \label{fig:S2}
        \end{figure}
    \end{example}

\begin{example}
    Consider now as $b^2$-symplectic manifold the $2$-torus
\[
(\mathbb{T}^2, Z = \{\theta_1 \in \{0, \pi\}\}, \omega=  \frac{d\theta_1}{\sin^2\theta_1}\wedge d\theta_2)
\]
with standard coordinates: $\theta_1, \theta_2 \in \left[0, 2\pi \right)$.  Observe that the critical  hypersurface $Z$ in this example is not connected. It is the union of two disjoint circles. Consider  the circle action of rotation on the $\theta_2$-coordinate with fundamental vector field $\frac{\partial}{\partial\theta_2}$. As the following computation holds,
$$\iota_{\frac{\partial}{\partial\theta_2}}\omega = - \frac{d \theta_1}{\sin^2 \theta_1} = d\left(\frac{\cos\theta_1}{\sin\theta_1}\right).$$
The fundamental vector field of the $S^1$-action defines $^{b^2}C^{\infty}$-integrable system given by the function $-\frac{\cos\theta_1}{\sin\theta_1}$.

\end{example}

\begin{figure}[ht]
\begin{center}

\begin{tikzpicture}[scale=0.8]

\pgfmathsetmacro{\rlinex}{6}
\pgfmathsetmacro{\rlineybottom}{2.75}
\pgfmathsetmacro{\rlineymid}{4.25}
\pgfmathsetmacro{\rlineytop}{5.75}

\def\R{1.6}
\pgfmathsetmacro{\donutx}{1.5}
\pgfmathsetmacro{\sizer}{1.8}

%top of rotation arrow
\draw [very thick] (\donutx - 0.3*\sizer, \rlineymid) arc (180:0: {\sizer * .4} and {\sizer * .1});

%draw the donut on the left
\draw [red, very thick, dashed] (\donutx - 0.5, \rlineymid - .55*\sizer) arc (90:270:{.16*\sizer} and {0.32*\sizer});
\draw [red, very thick, dashed] (\donutx - 0.5, \rlineymid + .55*\sizer) arc (270:90:{.16*\sizer} and {0.32*\sizer});

\DrawFilledDonutops{\donutx - 0.5}{\rlineymid}{.6*\sizer}{1.2*\sizer}{-90}{yellow!30}{very thick}{white}
\DrawDonut{\donutx - 0.5}{\rlineymid}{.6*\sizer}{1.2*\sizer}{-90}{black}{very thick}

\draw [red, very thick] (\donutx - 0.5, \rlineymid - .55*\sizer) arc (90:-90:{.16*\sizer} and {0.32*\sizer});
\draw [red, very thick] (\donutx - 0.5, \rlineymid + .55*\sizer) arc (-90:90:{.16*\sizer} and {0.32*\sizer});

\draw [very thick] (\donutx - .3*\sizer, \rlineymid) arc (180:132: {\sizer * .4} and {\sizer * .1});
\draw [very thick, ->] (\donutx - .3*\sizer, \rlineymid) arc (180:325: {\sizer * .4} and {\sizer * .1});

%draw the arrow
\draw     (\donutx + 1, \rlineymid) edge node[above] {$\mu$} (\rlinex - 0.5, \rlineymid);
\draw[->] (\donutx + 1, \rlineymid) -- (\rlinex - 0.5, \rlineymid);

\draw (\rlinex, \rlineybottom) -- (\rlinex, \rlineytop) node[right] {};

\draw[black, fill = black] (\rlinex, \rlineymid) circle(.3mm);

\draw[line width = 2pt, join = round, dblue, <->] (\rlinex - 0.2, \rlineybottom - 0.15) -- (\rlinex - 0.2, \rlineytop + 0.15);
\draw[line width = 2pt, join = round, dblue, <->] (\rlinex + 0.2, \rlineybottom - 0.15) -- (\rlinex + 0.2, \rlineytop + 0.15);
 
\end{tikzpicture}

\end{center}
% \end{center}
 \caption{Integrable system given by an $S^1$-action on a $b^2$-torus $\mathbb{T}^2$ and its associated moment map.}
 \label{fig:torus}
 \end{figure}
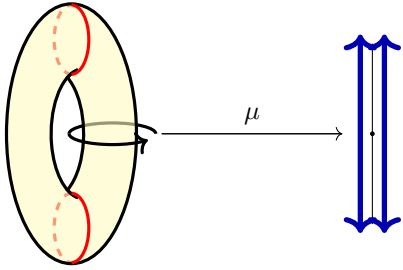
\begin{example} \label{ex:bmtorus}
The former example can be made general to produce examples of $b^m$-integrable systems on a $b^m$-symplectic manifold for any integer $m$
\[
(\mathbb{T}^2, Z = \{\theta_1 \in \{0, \pi\}\}, \omega=  \frac{d\theta_1}{\sin^m\theta_1}\wedge d\theta_2).
\]
Then 
$$\iota_{\frac{\partial}{\partial\theta_2}}\omega = - \frac{d \theta_1}{\sin^m \theta_1} = d\left(\frac{|\cos\theta_1|}{\cos\theta_1} \frac{_2F_1 \left ( \frac{1}{2}, \frac{1 - m}{2}; \frac{3 - m}{2}; \sin^2(\theta_1) \right )}{(1-m) \sin^{m - 1} \theta_1}\right),$$
with $_2F_1$  the hypergeometric function. 

Thus, the associated $S^1$-action has as  $^{b^m}C^{\infty}$-Hamiltonian the function 
$$- \frac{|\cos\theta_1|}{\cos\theta_1} \frac{_2F_1 \left ( \frac{1}{2}, \frac{1 - m}{2}; \frac{3 - m}{2}; \sin^2(\theta_1) \right )}{(1-m) \sin^{m - 1} \theta_1}$$ \noindent which defines a $b^m$-integrable system.
\end{example}
Now we give a couple  of examples of $b^m$-integrable systems.

\begin{example}
This example uses the product of $b^m$-integrable systems on a $b^m$-symplectic manifold with an integrable system on a symplectic manifold. Given $(M_1^{2n_1}, Z,x,\omega_1)$ a $b^m$-symplectic manifold with $f_1,\ldots,f_{n_1}$ a $b^m$-integrable system and $(M_2^{2n_2},\omega_2)$ a symplectic manifold with $g_1,\ldots,g_{n_2}$ an integrable system. Then $(M_1\times M_2, Z\times M_2, x, \omega_1 + \omega_2)$ is a $b^m$-symplectic manifold and $(f_1,\ldots,f_{n_1},g_1,\ldots,g_{n_2})$ is a $b^m$-integrable system on the higher dimensional manifold.

In particular by combining the former examples of $b^m$-integrable systems on surfaces and arbitrary integrable systems on symplectic manifolds we obtain examples of $b^m$-integrable systems in any dimension.
\end{example}

\begin{example}\textbf{(From integrable systems on cosymplectic manifolds to $b^m$-integrable systems:)}

Using the extension theorem (Theorem 50) of \cite{GMP14} we can extend any integrable system $(f_2,\dots, f_n)$ to an integrable system in a neighbourhood of a cosymplectic manifold $(Z, \alpha, \omega) $ by just adding a $b^m$-function $f_1$ to the integrable system so that the new integrable system is $(f_1, f_2,\dots, f_n)$  and considering the associated $b^m$-symplectic form:

\begin{equation}\label{eq:normalform}\tilde{\omega}=p^*\alpha\wedge\frac{dt}{t^m}+p^*\omega. \end{equation}

(t is the defining function of $Z$).

\end{example}

\section{Looking for a toric action}

In this section we pursue the proof of action-angle coordinates for $b^m$-integrable systems by recovering a torus group action. This action is associated to the Hamiltonian vector fields associated to $X_{f_i}$.

This is the same strategy used for $b$-integrable systems in \cite{KMS16}- One of the main difficulties is  to prove that the coefficients $a_1,\ldots, a_n$ can be considered  as constant functions. This makes it more difficult to prove the existence of a $\mathbb{T}^n$-action in the general $b^m$-case than in the $b$-case, but once we have it we can use the results in \cite{GMW17} to assume that the coefficients $a_1,\ldots, a_n$ are constant functions.

%\begin{remark}
%$df_1,\ldots,df_n$ are independent if and only if $X_{f_1},\ldots, X_{f_n}$ are independent. This holds because the map $^{b^m}TM\rightarrow ^{b^m} T^{*}M:u\mapsto \omega(u,\cdot)$ is bijective.
%\end{remark}

%\begin{remark}
%Because $X_{f_1},\ldots,X_{f_n}$ have to be linearly independent, one of the functions has a singularity of maximal degree. Looking at the expression \ref{eq:bmhamiltonianvf} it is easy to see that without the term of maximal degree all the terms in the component $\partial/\partial y_i$ are $0$ at $Z$.
%\end{remark}

In this section we provide some preliminary material that will be needed later:
\begin{proposition}\label{prop:mod_period}
Let $(M,Z,x,\omega)$ be a $b^m$-symplectic manifold such that $Z$ is connected with modular period $k$.
Let $\pi:Z\rightarrow S^1 \simeq \mathbb{R}/k\mathbb{Z}$ be the projection to the base of the corresponding mapping torus.
Let $\gamma: S^1 = \mathbb{R}/k\mathbb{Z} \rightarrow Z$ be any loop such that $\pi\circ\gamma$ is positively oriented and has constant velocity 1. Then the following are equal:
\begin{enumerate}%[i)]
\item The modular period of $Z$,

\item $\int_\gamma \iota_{\mathbb{L}} \omega$,

\item The value $a_{m-1}$ for any $^{b^m}\mathcal{C}^\infty(M)$-function
$$f = a_0 \log(x) + \sum_{j= 1}^{m-1} a_j\frac{1}{x^j} + h$$
such that the hamiltonian vector field $X_f$ has 1-periodic orbits homotopic in $Z$ to some $\gamma$.
\end{enumerate}
\end{proposition}

\begin{proof}

Let us first prove that (1)=(2) and then that (2)=(3).
\begin{itemize}
\item[(1)=(2)] Let us denote by $\mathcal{V}_{mod}$ the modular vector field. Recall from \cite{GMW17} that $\iota_{\mathbb{L}}(\mathcal{V}_{mod})$ is the constant function 1. Let $s:[0,k]\rightarrow Z$ be the trajectory of the modular vector field. Because the modular period is $k$, $s(0)$ and $s(k)$ are in the same leaf $\mathcal{L}$. Let $\hat s :[0,k+1]\rightarrow Z$ a smooth extension of $s$ such that $s|_{[k,k+1]}$ is a path in $\mathcal{L}$ joining $\hat s (k) = s (k)$ to $\hat s (k+1) = s (0)$. This way $\hat s$ becomes a loop. Then,

$$k=\int_0^k 1 dt = \int_S \iota_{\mathbb{L}} \omega = \int_{\hat s } \iota_{\mathbb{L}}\omega=\int_\gamma\iota_{\mathbb{L}}\omega$$

\item[(2)=(3)] Let $r:[0,1] \mapsto Z$ be the trajectory of $X_f$ the hamiltonian vector field of $f$. Recall that $X_f$ satisfies
$$\iota_{X_f}\omega = \sum_{j=1}^m c_j \frac{dx}{x^i} + dh.$$
Let $x^m\frac{\partial}{\partial x}$ be a generator of the linear normal bundle $\mathbb{L}$.
We know that $X_f$ is 1-periodic and its trajectory is homotopic to $\gamma$. Hence,
$$
\begin{array}{rcl}
 k = \int_r \iota_{\mathbb{L}}\omega & = & \displaystyle \int_0^1 \iota_{x^m\frac{\partial}{\partial x}} \omega(X_f|_{r(t)})dt\\
 \\
 & = & \displaystyle \int_0^1 -(\sum_{j=1}^{m}c_i\frac{dx}{x^i} + dh)\cdot(x^m\frac{\partial}{\partial x})|_{r(t)}dt\\
 \\
 & = & -c_m = -a_{m-1}\\

\end{array}
$$
\end{itemize}

\end{proof}

We will also need a Darboux-Carathéodory theorem for $b^m$-symplectic manifolds:
\begin{theorem}[Darboux-Carath\'{e}odory ($b^m$-version)]\label{bmdarbouxcaratheodory}
Let
$$(M^{2n},x,Z, \omega)$$
 be a $b^m$-symplectic manifold and $m$ be a point on $Z$. Let $f_1,\ldots,f_n$ be a $b^m$-integrable system. Then there exist \textcolor{black}{$b^m$}-functions $(q_1,\ldots,q_n)$ around $m$ such that
$$\omega = \sum_{i=1}^n df_i\wedge dq_i$$
and the vector fields $\{X_{f_i},X_{q_j}\}_{i,j}$ commute.
If $f_1$ is not smooth (recall that $f_1 = a_0\log(x) + \sum_{j=1}^{m-1}a_j\frac{1}{x^i}$ with $a_n \neq 0$ on $Z$ and $a_0 \in \mathbb{R}$) the $q_i$ can be chosen to be smooth functions, and $(x,f_2,\ldots,f_n,q_1,\ldots,q_n)$ is a system of local coordinates.
\end{theorem}

\begin{proof}
The first part of this proof is exactly as in \cite{KMS16}.
Assume now $\displaystyle f_1 = a_0\log(x) + \sum_{j=1}^{m-1}a_j\frac{1}{x^i}$.
We modify the induction requiring also that $\mu_i$ (in addition to be in $K_i$) is also in $T^* M \subseteq ^bT^* M$.
We can also ask this extra condition while asking $\mu_i(X_{f_i})= 1$, we only have to check that $X_{f_i}$ does not vanish in $TM$. This is clear because $X_{f_i}$ does not vanish at $^b TM$ and
$$0 = \{f_n,f_i\} = \left(\sum_{i=1}^m \tilde{a}_i\frac{dx}{x^i}\right)(X_{f_i}) = \left(\frac{dx}{x^m} \sum_{i=1}^{m}a_i x^i\right)(X_{f_i}).$$

All the terms in the last expression vanish except for the one of degree $m$.

Then $dx/x^m$ is in the kernel of $X_{f_i}$, hence $X_{f_i}$ does not vanish on $TM$ and the $q_i$ can be chosen to be smooth.

$\{X_x,X_{f_2},\ldots,X_{f_{n}},X_{q_1},\ldots X_{q_n}\}$ commute because $\{X_{f_i},X_{q_i}\}_{i,j}$ commute. Then
$$dx\wedge df_2\ldots\wedge df_n\wedge dq_1 \wedge \ldots \wedge d q_n$$
is a non-zero section of $\bigwedge^n(^b TM)$. And hence
$$(x,f_2,\ldots, f_{n-1},q_1,\ldots,q_n)$$
 are local coordinates.

\end{proof}

Before proceeding with the proof of the action-angle coordinates, we need to prove that in a neighbourhood of a Liouville torus the fibration is semilocally trivial:

\begin{lemma}[Topological Lemma]\label{lemma:topological}
Let $m \in Z$ be a regular point of a $b^m$-integrable system $(M,x,Z,\omega,F)$. Assume that the integral manifold $\mathcal{F}_m$ through $m$ is compact. Then there exists a neighborhood $U$ of $\mathcal{F}_m$ and a diffeomorphism
$$\phi:U \simeq \mathbb{T}^n\times B^n$$
which takes the foliation $\mathcal{F}$ to the trivial foliation $\{\mathbb{T}^n\times\{b\}\}_{b\in B^n}$.
\end{lemma}

\begin{proof}
We follow the steps of \cite{LMV11}. In this case, the only extra step that must be checked is that the foliation given by the $b^m$-hamiltonian vector fields of $F = (f_1,f_2, \ldots, f_n)$ is the same as the one given by the level sets of $\tilde F := (x, f_2,\ldots, f_n)$. In our case $f_1 = a_0\log(x) + \sum_{u=1}^{m-1} a_i\frac{1}{x^i}$, where $a_0 \in \mathbb{R}$, $a_i \in \mathcal{C}^\infty(x)$, $a_{m-1} = 1$. Hence the foliations are the same.
Then as in \cite{LMV11}, we take an arbitrary Riemannian metric on $M$ and this defines a canonical projection $\psi:U \rightarrow \mathcal{F}_m$. Let us define $\phi := \psi\times \tilde F$. We obtain the commutative diagram (Figure \ref{fig:diagram_topological}).

\begin{figure}[H]
\centering
\begin{tikzcd}
U \arrow[rr, dashed, "\phi"] \arrow[rrd, "\tilde F"]& & \mathbb{T}^n\times B^n \arrow[d,"p"]\\
 & & B^n
\end{tikzcd}
\caption{Commutative diagram of the construction of the isomorphism of $b^m$-integrable systems.}
\label{fig:diagram_topological}
\end{figure}

which provides the necessary equivalence of $b^m$-integrable systems.
\end{proof}

%\begin{lemma}[$b^m$-Poincaré]\label{lemma:poincare}
%Let $\psi$ be a $b^m$-form such that $i^*\psi = 0$ where $i$ is the immersion of the
%\end{lemma}
%\begin{proof}
%The proof of this lemma works exactly as in the standard Poincaré lemma, because the homotopy formula works the same way for $b^m$-forms.
%\end{proof}

\section{Action-angle coordinates on $b^m$-symplectic \protect\\ manifolds}

In a neighbourhood of one of our Liouville tori all we can assume about the form of our $b^m$-symplectic structure is  that is given by the Laurent series defined in \cite{Scott16}.

That is to say, we can assume that in a tubular neighborhood $U$ of $Z$
$$\omega = \sum_{j=1}^{m-1}\frac{dx}{x^i}\wedge\pi^*(\alpha_i) + \beta,$$
where $\pi:U\rightarrow Z$ is the projection of the tubular neighborhood onto $Z$, $\alpha_i$ are closed smooth de Rham forms on $Z$ and $\beta$ a de Rham form on $M$ of degree $2$.

In \cite{BKM,anastasiaeva} normal forms are given for group actions in a neighbourhood of the orbit. Below we provide a normal for the integrable system in a neighbourhood of an orbit of the torus action associated to the integrable system. This theorem is finer than the $b^m$-symplectic slice theorem provided in \cite{anastasiaeva} as it also gives information about the first integrals. 

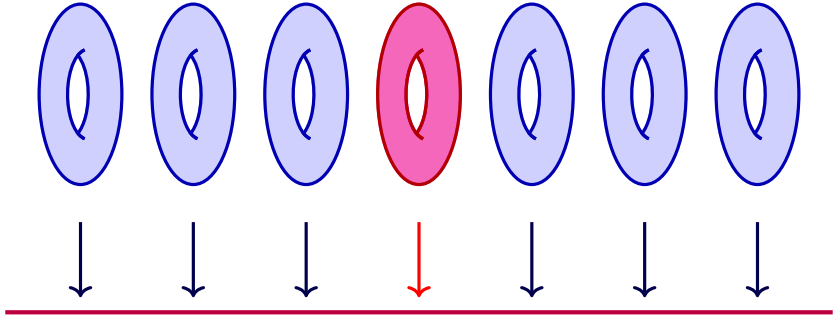
\begin{figure}
\centering
\begin{tikzpicture}
%%%%%%%%%regprogramming

\pgfmathsetmacro{\sizer}{1}
\pgfmathsetmacro{\basept}{5}	

\pgfmathsetmacro{\xone}{0}
\pgfmathsetmacro{\xtwo}{1.5}
\pgfmathsetmacro{\xthree}{3}
\pgfmathsetmacro{\xfour}{4.5}
\pgfmathsetmacro{\xfive}{6}
\pgfmathsetmacro{\xsix}{7.5}
\pgfmathsetmacro{\xseven}{9}

\pgfmathsetmacro{\ymid}{4.25}
\pgfmathsetmacro{\ytop}{5.65}
\pgfmathsetmacro{\ybottom}{2.85}

\DrawFilledDonutops{\xone}{\ymid}{.55*\sizer}{1.2*\sizer}{-90}{vlblue}{very thick}{white}
\DrawFilledDonutops{\xtwo}{\ymid}{.55*\sizer}{1.2*\sizer}{-90}{vlblue}{very thick}{white}
\DrawFilledDonutops{\xthree}{\ymid}{.55*\sizer}{1.2*\sizer}{-90}{vlblue}{very thick}{white}
\DrawFilledDonutops{\xfour}{\ymid}{.55*\sizer}{1.2*\sizer}{-90}{magenta}{very thick}{white}
\DrawFilledDonutops{\xfive}{\ymid}{.55*\sizer}{1.2*\sizer}{-90}{vlblue}{very thick}{white}
\DrawFilledDonutops{\xsix}{\ymid}{.55*\sizer}{1.2*\sizer}{-90}{vlblue}{very thick}{white}
\DrawFilledDonutops{\xseven}{\ymid}{.55*\sizer}{1.2*\sizer}{-90}{vlblue}{very thick}{white}

\DrawDonut{\xone}{\ymid}{.55*\sizer}{1.2*\sizer}{-90}{dblue}{very thick}
\DrawDonut{\xtwo}{\ymid}{.55*\sizer}{1.2*\sizer}{-90}{dblue}{very thick}
\DrawDonut{\xthree}{\ymid}{.55*\sizer}{1.2*\sizer}{-90}{dblue}{very thick}
\DrawDonut{\xfour}{\ymid}{.55*\sizer}{1.2*\sizer}{-90}{dred}{very thick}
\DrawDonut{\xfive}{\ymid}{.55*\sizer}{1.2*\sizer}{-90}{dblue}{very thick}
\DrawDonut{\xsix}{\ymid}{.55*\sizer}{1.2*\sizer}{-90}{dblue}{very thick}
\DrawDonut{\xseven}{\ymid}{.55*\sizer}{1.2*\sizer}{-90}{dblue}{very thick}

\draw[very thick,magenta,vdblue, ->](\xone, \ybottom - 0.3) -- +(0, -1);
\draw[very thick, vdblue,  ->] (\xtwo, \ybottom - 0.3) -- +(0, -1);
\draw[very thick, vdblue, ->] (\xthree, \ybottom - 0.3) -- +(0, -1);
\draw[very thick, red, ->] (\xfour, \ybottom - 0.3) -- +(0, -1);
\draw[very thick,vdblue, ->] (\xfive, \ybottom - 0.3) -- +(0, -1);
\draw[very thick,vdblue, ->] (\xsix, \ybottom - 0.3) -- +(0, -1);
\draw[very thick,vdblue, ->] (\xseven, \ybottom - 0.3) -- +(0, -1);
\draw[ultra thick, purple](\xone - 1, \ybottom - 1.5) -- (\xseven + 1, \ybottom - 1.5);
\end{tikzpicture}
\caption{Fibration by Liouville tori: The middle fiber of the point $p \in Z$ in magenta, the neighbouring Liouville tori in blue.}
\label{fig:tori}
\end{figure}

One of the non-trivial steps of the proof is to associate a toric action to the integrable system. The connection to normal forms of group actions will become even more evident when we discuss the associated cotangent models.

\begin{theoremA}[Action-angle coordinates for $b^m$-symplectic manifolds]
Let $(M,x,\omega, F)$ be a $b^m$-integrable system, where $F = (f_1 = a_0 \log(x) + \sum_{j=1}^{m-1} a_j\frac{1}{x^j}, \ldots,f_n)$ with $a_j$  for $j>1$ functions in $x$. Let $m\in Z$ be a regular point and let us assume that the integral manifold of the distribution generated by the $X_{f_i}$ through $m$ is compact. Let $\mathcal{F}_m$ be the Liouville torus through $m$.
Then, there exists a neighborhood $U$ of $\mathcal{F}_m$ and coordinates $(\theta_1,\ldots,\theta_n,\sigma_1,\ldots,\sigma_n):\mathcal{U}\rightarrow\mathbb{T}^n\times B^n$ such that:

\begin{enumerate}
\item We can find an equivalent integrable system $F = (f_1 = a_0'\log(x) + \sum_{j=1}^{m-1} a_j'\frac{1}{x^j}, \ldots, f_n)$ such that the coefficients $a_0',\ldots, a_{m-1}' $ of $f_1$ are constants $\in \mathbb{R}$,
\item $$\omega|_\mathcal{U} = \left(\sum_{j=1}^m c_j'\frac{c}{\sigma_1^j}d\sigma_1\wedge d\theta_1\right) + \sum_{i=2}^{n} d \sigma_i\wedge d\theta_i$$ where $c$ is the modular period and $c_j' = -(j-1)a_{j-1}'$, also
\item the coordinates $\sigma_1,\ldots,\sigma_n$ depend only on $f_1,\ldots f_n$.
\end{enumerate}

%Also, \textcolor{black}{$\langle c\cdot c_j', X \rangle = \alpha_j(X)$ for $X \in \mathfrak{t}$}.
\end{theoremA}

\begin{proof}
The idea of this proof is to construct an equivalent $b^m$-integrable system whose fundamental vector fields define a $\mathbb{T}^n$-action on a neighborhood of $\mathbb{T}^n\times\{0\}$.
It is clear that all the vector fields $X_{f_1},\ldots,X_{f_n}$ define a torus action on each Liouville tori $\mathbb{T}^n\times\{b\}$ where $b\in B^n$, but this does not guarantee that their flow defines a toric action on all $\mathbb{T}^n\times B^n$.
The proof is structured in three steps. The first one is the uniformization of the periods, i.e. we define an $\mathbb{R}^n$-action on a neighborhood of $\mathbb{T}^n\times\{0\}$ such that the lattice defined by its kernel at every point is constant. This allows to induce an actual action of a torus (as the periods are constant) of rank n: A $\mathbb T^n$ action by taking quotients. The second step consists in checking that this action is actually $b^m$-Hamiltonian. And in the final step we apply theorem \ref{bmdarbouxcaratheodory} to obtain the expression of $\omega$.
\begin{enumerate}

\item Uniformization of periods.

Let $\Phi_{X_F}^s$ be defined as the joint flow by the Hamiltonian vector fields of the action:
\begin{equation}\label{eq:action1}
\begin{array}{rcl}
 \Phi: \mathbb{R}^n\times(\mathbb{T}^n\times B^n)& \rightarrow & (\mathbb{T}^n\times B^n)\\
 ((s_1,\ldots,s_n),(x,b)) & \mapsto & \Phi_{X_{f_1}}^{s_1}\circ\cdots\circ\Phi_{X_{f_n}}^{s_n}((x,b))\\

\end{array}
\end{equation}

this defines an $\mathbb{R}^n$-action on $\mathbb{T}^n\times B^n$.
For each $b\in B^n$ at a single orbit $\mathbb{T}^n\times\{b\}$ the kernel of this action is a discrete subgroup of $\mathbb{R}^n$. We will denote the lattice given by this kernel $\Lambda_b$. Because the orbit is compact, the rank of $\Lambda_b$ is maximal i.e. $n$. This lattice is known as the period lattice of $\mathbb{T}^n\times\{b\}$ as we know by standard arguments in group theory that the lattice has to be of maximal rank so as to have a torus as a quotient.
In general we can not assume that $\Lambda_b$ does not depend on $b$. The process of uniformization of the periods modifies the action \ref{eq:action1} in such a way that $\Lambda_b = \mathbb{Z}^n$ for all $b$.
Let us consider the following Hamiltonian vector field $\sum_{i=1}^n k_iX_{f_i}$. The $b^m$-function that generates this Hamiltonian vector field is:
$$k_1\left(a_0\log(x) + \sum_{j=1}^{m-1}a_j\frac{1}{x^j}\right) + \sum_{i=2}^n k_i f_i$$
where recall that $a_{m-1}$ is constant equal 1. Observe that the coefficient multiplying $1/x^{m-1}$ is $k_1$. By proposition \ref{prop:mod_period} $k_1 = c$ the modular period. In this case $c = [\alpha_m]$.

Hence, for $b\in B^{n-1}\times\{0\}$ the lattice $\Lambda_b$ is contained in $\mathbb{R}^{n-1}\times c \mathbb{Z} \subseteq \mathbb{R}^n$.
Pick $(\lambda_1,\ldots,\lambda_n): B^n\rightarrow \mathbb{R}^n$ such that:
\begin{itemize}%[a)]
\item $(\lambda_1(b),\ldots,\lambda_n(b))$ is a basis of $\Lambda_b$ for all $b\in B^n$,
\item \textcolor{black}{$\lambda_i^n$ vanishes along $B^{n-1}\times\{0\}$ at order $m$ for $i<n$ and $\lambda_i$ is equal to $c$ along $B^{n-1}\times\{0\}$.}
\end{itemize}
In the previous points, $\lambda_i^j$ denotes the $j$-th component of $\lambda_i$. The first condition can be satisfied by using the implicit function theorem. That is because $\Phi(\lambda,m) = m$ is regular with respect to the $s$ coordinates. The second condition is automatically true because $\Lambda_b \subseteq \mathbb{R}^{n-1}\times c\mathbb{Z}$. We define the uniformed flow as:
\begin{equation}\label{eq:action2}
\begin{array}{rcl}
\tilde \Phi: \mathbb{R}^n\times(\mathbb{T}^n\times B^n)& \rightarrow & (\mathbb{T}^n\times B^n)\\
((s_1,\ldots,s_n),(x,b))& \mapsto & \Phi(\sum_{i=1}^n s_i \lambda_i,(x,b))\\

\end{array}
\end{equation}
\item The $\mathbb{T}^n$-action is $b^m$-Hamiltonian.
The objective of this step is to find $b^m$-functions $\sigma_1,\ldots,\sigma_n$ such that $X_{\sigma_i}$ are the fundamental vector fields of the $\mathbb{T}^n$-action $Y_i = \sum_{j=1}^n \lambda_i^j X_{f_j}$.

By using the Cartan formula for a $b^m$-symplectic form, we obtain:

$$
\begin{array}{rcl}
\mathcal{L}_{Y_i}\mathcal{L}_{Y_i}\omega & = & \mathcal{L}_{Y_i}(d(\iota_{Y_i}\omega) + \iota_{Y_i}d\omega)\\
 & = &  \mathcal{L}_{Y_i}(d(-\sum_{j=1}^{n} \lambda_i^j df_i))\\
 & = &  -\mathcal{L}_{Y_i}(\sum_{j=1}^n d\lambda_i^j\wedge df_j) = 0\\
\end{array}
$$

Note that $\lambda_i^j$ are constant on the level sets of $F$ as $\Phi(\lambda, m) = m$  and the level sets of $F$ are invariant by $\Phi$.

Recall that if $Y$ is a complete periodic vector field and $P$ is a bivector such that $\mathcal{L}_Y\mathcal{L}_Y P = 0$, then $\mathcal{L}_Y P = 0$.
So, the vector fields $Y_i$ are Poisson vector fields.
To show that each $\iota_{Y_i}\omega$ has a $^{b^m}\mathcal{C}^\infty$ primitive we will see that $[\iota_{Y_i}\omega] = 0$ in the $b^m$-cohomology.

One one hand, if $i >1$, $\iota_{Y_i}\omega$ vanishes at $Z$. This holds because $Y_i$ has not any component $\partial/\partial Y$.

Recall Proposition 6 from \cite{GMP14}:

\begin{proposition} If $\omega \in ^b \Omega(M)$ with $ \omega|_Z=0$, then $\omega\in \Omega(M)$.
\end{proposition}
In  a similar way for $b^m$-forms we have,

\begin{proposition} If $\omega \in ^{b^m} \Omega(M)$ with $\omega|_Z$ vanishing up to order $m$, then $\omega \in \Omega(M)$.
\end{proposition}

Thus as $\iota_{Y_i}\omega$ vanishes at $Z$, the $b^m$-forms  $\iota_{Y_i}\omega$ are indeed  smooth.
Thus  we can now apply the standard Poincaré lemma and as these forms are closed they are locally exact. This proves that all the vector fields $Y_i$ with $i>1$ are indeed Hamiltonian.

On the other hand, the fact that $\iota_{Y_1} \omega = c df_1$ is obvious.

Then, because we have a toric action that is Hamiltonian, we can use lemma 3.2 in \cite{GMP17}, and  we get an equivalent system such that $a_i$ are all constant and moreover $\langle a_i', X\rangle = \alpha_i(X^\omega)$. Note that by dividing by $a_{m-1}'$, we can still assume $a_{m-1}'=1$ to be consistent with our notation, but  we then have to multiply $f_1\cdot c$ in the next step.

\item Apply Darboux-Carathéodory theorem.

The construction above gives us some candidates $\sigma_1 = c f_1,\sigma_2,\ldots,\sigma_n$ for the action coordinates.

We now apply the Darboux-Carathéodory theorem and express the form in terms of $x$:
$$\omega = \left(\sum_{j= 1}^{m} c \frac{c_j}{x^j} dx\wedge d q_1\right)+ \sum_{i=2}^{n}d\sigma_i\wedge dq_i.$$

Since the vector fields $X_{\sigma_i} = \frac{\partial}{\partial q_i}$ are fundamental fields of the $\mathbb{T}^n$-action the flow \ref{eq:action2} gives a linear action on the $q_i$ coordinates.

Observe that the coordinate system is only defined in $\mathcal{U}$. It may not be valid at points outside $\mathcal{U}$ that may be in the orbit of points in $\mathcal{U}$. Let us see that the charts can be extended to these points.

Define $\mathcal{U}'$ the union of all tori that intersect $\mathcal{U}$.
We will see that the coordinates are valid at $\mathcal{U}'$.

Let $\{p_i,\theta_j\}$ be the extension of $\{\sigma_i,q_j\}$. It is clear that $\{p_i,\theta_j\} = \delta_{ij}$ by its construction in the Darboux-Carathéodory theorem.

To see that $\{\theta_i,\theta_j\} = 0$ we take the flows by $X_{p_k}$ and extend the expression to the whole $\mathcal{U}'$:

$$X_{p_k}(\{\theta_i, \theta_j\}) = \{\{\theta_i, \theta_j\},p_k\} = \{\theta_i,\delta_{ij}\} - \{\theta_j, \delta_{jk}\} = 0.$$

The fact that $\omega$ is preserved is obvious because $X_{p_k}$ are hamiltonian vector fields and thus they preserve the $b^m$-symplectic forms. Moreover, $t,\theta_1,p_2,\theta_2,\ldots,p_n,\theta_n$ are independent on $\mathcal{U}'$ and hence are a coordinate system in a neighbourhood of the torus.

\end{enumerate}
\end{proof}

\begin{remark}
In the proof  we have seen that there exists an equivalent integrable system where the coefficients of the singular function are indeed constant. From now on, when considering a $b^m$-integrable system we are going to make this assumption.
\end{remark}

\begin{remark}
    By means of the desingularization transformation we may obtain an action-angle coordinate theorem for folded manifolds as we do in Part 3 for the KAM theorem for folded symplectic manifolds. This folded action-angle theorem is a particular case of the one obtained in \cite{EvaRobert}.
\end{remark}

\chapter[Action-angle coordinates and cotangent lifts]{Reformulating the action-angle coordinate via cotangent lifts}

The action-angle theorem for symplectic manifolds (also known as action-angle coordinate theorem) can be reformulated in terms of a cotangent lift.

Recall that given a Lie group action on any manifold its cotangent lifted action is automatically Hamiltonian. By considering the action of a torus on itself by translations this action can be lifted to its cotangent bundle and give a semilocal normal form theorem as the Arnold-Liouville-Mineur theorem for symplectic manifolds. If we now replace this cotangent lift to the cotangent bundle to a lift to the $b^m$-cotangent bundle we obtain the semilocal normal form of the main theorem of this chapter.

Let start recalling the symplectic and $b$-symplectic case following \cite{KM17}.

\section{Cotangent lifts and Arnold-Liouville-Mineur in Symplectic Geometry}

Let $G$ be a Lie group and let $M$ be any smooth manifold. Given a group action $\rho:G\times M\longrightarrow M$, we define its cotangent lift as the action on $T^\ast M$ given by $\hat{\rho_g}:=\rho^\ast_{g^-1}$ where $g\in G$. We then have a commuting diagram

\begin{figure}[h]
\centering
\begin{tikzcd}
T^\ast M  \arrow[rr, "\hat{\rho_g}"] \arrow[d,"\hat{\pi}"] & & T^\ast M \arrow[d,"\pi"]\\
M  \arrow[rr, "\rho_g"]  & & M
\end{tikzcd}
\caption{Commutiative diagram of the construction of the isomorphism of $b^m$-integrable systems.}
\label{fig:diagram_cotangent_lift}
\end{figure}

where $\pi$ is the canonical projection from $T^\ast M$ to $M$.

 The cotangent bundle $T^*M$ is a symplectic manifold endowed
with the exact symplectic form given by the differential of the Liouville one-form $\omega=-d\lambda$. The Lioville one-form can be defined intrinsically:
\begin{equation}\label{liouvilleform}
 \langle \lambda_p, v\rangle:= \langle p, (\pi_p)_\ast (v)\rangle
\end{equation}
with  $v\in T(T^*M), p\in T^*M$.

A standard argument (see for instance \cite{GS90}) shows that the cotangent lift $\hat{\rho}$  is Hamiltonian with moment map $\mu:T^*M \to \mathfrak{g}^*$ given by
\begin{equation*}\label{eqn:lift}
\langle\mu(p),X \rangle := \langle \lambda_p ,X^\#|_{p} \rangle =\langle p,X^\#|_{\pi(p)}\rangle,
\end{equation*}
where  $p\in T^*M$, $X$ is an element of the Lie algebra $\mathfrak{g}$ and we use the same symbol $X^\#$ to denote the fundamental vector field of $X$ generated by the action on $T^\ast M$ or $M$.
This  construction is known  as the {\bf cotangent lift}.

 In the special case where the manifold $M$ is a torus $\T^n$ and the group is $\T^n$ acting by translations, we obtain the following explicit structure: Let $\theta_1,\ldots,\theta_n$  be the standard ($S^1$-valued) coordinates on $\T^n$ and let
\begin{equation}\label{co}
\underbrace{\theta_1,\ldots,\theta_n}_{=:\theta}, \underbrace{t_1, \ldots, t_n}_{=:t}
\end{equation}
be the corresponding chart on $T^\ast \T^n$, i.e. we associate to the coordinates \eqref{co} the cotangent vector $\sum_i t_i d \theta_i \in T^\ast_\theta \T^n$.
The Liouville one-form is given in these coordinates by
$$ \lambda = \sum_{i=1}^n t_i d \theta_i $$
and its negative differential is the standard symplectic form on $T^\ast \T^n$:
\begin{equation}\label{eq:omegacan}
\omega_{can} = \sum_{i=1}^n d \theta_i \wedge d t_i .
\end{equation}
Denoting by $\tau_\beta$ the translation by $\beta \in \T^n$ on $\T^n$, its lift to $T^\ast \T^n$ is given by
$$ \hat \tau_\beta: (\theta, t) \mapsto (\theta + \beta, t).$$
The moment map $\mu_{can}: T^\ast \T^n \to \mathfrak{t^\ast} $ of the lifted action with respect to the canonical symplectic form is
\begin{equation}\label{eq:mucan}
\mu_{can}(\theta,t) = \sum_i t_i d\theta_i,
\end{equation}
where the $\theta_i$ on the right hand side are understood as elements of $ \mathfrak{t^\ast}$ in the obvious way. Even simpler, if we identify $ \mathfrak{t^\ast}$ with $\R^n$ by choosing the standard basis $\frac{\partial}{\partial \theta_i}$ of  $\mathfrak{t}$ then the moment map is just the projection onto the second component of $T^\ast \T^n \cong \T^n \times \R^n$.  Note that the components of $\mu$ naturally define an integrable system on $T^\ast \T^n$.

 We can rephrase the Arnold-Liouville-Mineur theorem in terms of the symplectic cotangent model:

\begin{theorem} Let $F=(f_1,\ldots,f_n)$ be an integrable system on the symplectic manifold $(M,\omega)$. Then semilocally around a regular Liouville torus the system is equivalent to the cotangent model $(T^\ast \T^n)_{can}$ restricted to a neighbourhood of the zero section $(T^\ast \T^n)_0$ of $T^\ast \T^n$.
\end{theorem}

\section{The case of $b^m$-symplectic manifolds}

Let us start by introducing the twisted $b^m$-cotangent model for torus actions. This model has additional invariants: the modular vector field of the connected component of the critical set and the modular weights of the associated toric action.
Consider $T^\ast \T^n$ be endowed with the standard coordinates $(\theta, t)$, $\theta \in \T^n$, $t \in \R^n$ and consider again the action on $T^\ast \T^n$ induced by lifting translations of the torus $\T^n$. We will now  view this action as a $b^m$-Hamiltonian action with respect to a suitable $b^m$-symplectic form. In analogy to the classical Liouville one-form we define the following  non-smooth one-form away from the hypersurface $Z=\{t_1 = 0\}$~:

$$ \left(c c_1 \log|t_1| + \sum_{i=2}^{m}c c_i \frac{t_1^{-(i-1)}}{-(i-1)}\right) d \theta_1 + \sum_{i=2}^n t_i d\theta_i.$$

When differentiating this form we obtain a $b^m$-symplectic form on $T^\ast \T^n$ which we call (after a sign change) the {\bf twisted $b^m$-symplectic form  }on $T^\ast \T^n$ with invariants $(c c_1, \dots, c c_m)$:
\begin{equation}\label{eq:twistedform}
 \omega_{tw, c}:=\left(\sum_{j=1}^m c_j\frac{c}{t_1^j}d t_1\wedge d\theta_1\right) + \sum_{i=2}^{n} d t_i\wedge d\theta_i,
\end{equation}
where $c$ is the modular period.
The moment map of the lifted action  is then given by
\begin{equation}\label{eq:bmucan}\mu_{tw, q_0, \dots, q_{m-1})}:=( q_0 \log|t_1| + \sum_{i=2}^{m}q_i t_1^{-(i-1)} ,t_2, \ldots, t_n),
\end{equation}
where we are identifying $\mathfrak{t^\ast}$ with $\R^n$ and  $c_j = -(j-1)q_{j-1}$.

We call this lift together with the $b^m$-symplectic form \ref{eq:twistedform} the {\bf twisted $b^m$-cotangent lift} with modular period $c$ and invariants $(c_1, \dots, c_m)$. Note that the components of the moment map define a $b^m$-integrable system on $(T^\ast \T^n, \omega_{tw, (c c_1, \dots, c c_m)})$.

The model of twisted $b^m$-cotangent lift allows us to express the action-angle coordinate theorem for $b^m$-integrable systems in the following way:

\begin{theorem} Let $F=(f_1,\ldots,f_n)$ be a $b^m$-integrable system on the $b^m$-symplectic manifold $(M,\omega)$.  Then semilocally around a regular Liouville torus $\mathbb T$, which lies inside the critical hypersurface $Z$ of $M$, the system is equivalent to the cotangent model $(T^\ast \T^n)_{tw, (c c_1, \dots, c c_m)} $ restricted to a neighbourhood of $(T^\ast \T^n)_0$. Here $c$ is the modular period of the connected component of $Z$ containing  $\mathbb T$ and the constants $(c_1, \dots, c_m)$ are the invariants associated to the integrable system and its associated toric action.
\end{theorem}

%-----------------------------------------------------------------------
% Beginning of chap1.tex
%-----------------------------------------------------------------------
%
%  AMS-LaTeX sample file for a chapter of a monograph, to be used with
%  an AMS monograph document class.  This is a data file input by
%  chapter.tex.
%
%  Use this file as a model for a chapter; DO NOT START BY removing its
%  contents and filling in your own text.
% 
%%%%%%%%%%%%%%%%%%%%%%%%%%%%%%%%%%%%%%%%%%%%%%%%%%%%%%%%%%%%%%%%%%%%%%%%

\part{A KAM theorem for $b^m$-symplectic manifolds} 

The KAM theorem explains how integrable systems behave under small perturbations. More precisely, it studies how an integrable system in action-angle coordinates responds to a small perturbation on its Hamiltonian. The trajectories of an integrable system in action-angle coordinates can be seen as linear trajectories over a torus. The KAM theorem finds a way to transform these original trajectories to other linear trajectories over some transformed torus. The KAM theorem states that most of these tori, and the linear solutions of the system on these tori, survive if the perturbation is small enough.

In this part, we give a new KAM theorem for $b^m$-symplectic manifolds with detailed proof. This is contained in the first chapter of this part. 
Moreover, we devote three more chapters to applications:

\begin{enumerate}
\item \textbf{Desingularization of $b^m$-integrable systems.} We present a way to use the desingularization of $b^m$-symplectic manifolds presented in \cite{GMW17} to construct standard smooth integrable systems from $b^m$-integrable systems. This desingularized integrable system is uniquely defined.
\item \textbf{Desingularization of the KAM theorem on $b^m$-symplectic manifolds.} In this section we use the desingularization of $b^m$-integrable systems in conjunction with the KAM theorem for $b^m$-symplectic manifolds to deduce the original KAM theorem as well as a completely new KAM theorem for folded symplectic forms.
\item \textbf{Potential applications to Celestial mechanics.} We overview a list of motivating examples from Celestial mechanics where regularization transformations give rise to $b^m$-symplectic forms. We discuss some potential applications of perturbation theory in this set-up.
 \end{enumerate}
\chapter{A new KAM theorem}

The objective of this chapter is to give a construction of KAM theory in the setting of $b^m$-symplectic manifolds and with $b^m$-integrable systems. The core of the chapter is the construction of the proper statement and the proof of the equivalent of the KAM theorem on $b^m$-symplectic manifolds.

This chapter is  divided different sections:

\begin{enumerate}
\item \textbf{On the structure of the proof.} On this section we are going to present the main ideas that are going to appear in the proper statement and proof of the main theorem. The idea of the theorem is to build a sequence of $b^m$-symplectomorphisms such that its limit transforms the hamiltonian to only depend on the action coordinates.
\item \textbf{Technical results and definitions.} On this section we present some technical results and definitions that are key for the proof of the main theorem.
\item \textbf{KAM theorem on $b^m$-symplectic manifolds.} On this section we present the statement and the proof of the main result of this chapter. The proof is structured in 6 parts. In the first part we define the parameters that are going to be used to define the sequence of $b^m$-symplectomorphisms. In the second part we build precisely this sequence of $b^m$-symplectomorphisms. In the third part we see that the sequence of frequency maps of the transformed Hamiltonian functions at every step converges. In the fourth part we see that the sequence of $b^m$-symplectomorphisms converges. In the fifth part we obtain results on the stability of the trajectories under the original perturbation. In the sixth part, we find bounds to explain how close the invariant tori are from the unperturbed.
Finally, we obtain a bound for the measure of the set of invariant tori.

\end{enumerate}

\section{On the structure of the proof}

The first thing we do is to reduce our study to the case the perturbation is not a $b^m$-function but an analytic one. This is because any purely singular perturbation only affects the component in the direction of the modular vector field and can be easily controlled.

The idea of the proof is really similar to the classical KAM case.
We want to build a diffeomorphism such that its transformed hamiltonian only depends on the action coordinates. But it is not possible to build this diffeomorphism in one step. What we do, as it is done in the classical case, it is to build a sequence of diffeomorphisms such that the part of the hamiltonian depending on the angular variables decreases at every step. The idea is to remove the first $K$ terms of its Fourier expression at every step while making $K$ rapidly increase. This is done by assuming the diffeomorphism comes as the flow at time 1 generated by a Hamiltonian function. In this way one can use the Lie Series in conjunction with the Fourier series to find the expression for the hamiltonian function that generates our diffeomorphism. The final diffeomorphism will be the composition of all the diffeomorphisms obtained at each step. One of the main difficulties of the proof, as in the classical case, is to prove that these diffeomorphisms converge and to prove some bounds of its norm.

We also note that for our $b^m$-symplectic setting, the diffeomorphisms we consider leave the defining function of the critical set invariant up to order $m$, this will have an important role later. Also observe in particular that the critical set can not be transformed by any perturbation given by a $b^m$-function.\\
Next we give some technical definitions and results. We define the norms we are going to use to do all the estimates. We set the notation for the proof and the statement of the theorem. We define the notion of non-resonance for a neighborhood of the critical set of the $b^m$-symplectic manifold. We study the set of all possible non-resonant vectors. And we state the inductive lemma, which gives us estimates and constructions for every step of our sequence of diffeomorphisms.

After all this discussion we are in conditions to properly state the $b^m$-version of the KAM theorem. One important difference to the classical KAM theorem is that we have to guarantee that at $Z$ the set of non-resonant vectors does not become the whole set of frequencies. This condition can be understood as the perturbation being smaller than some constant multiplied by the inverse of the modular period.

The proof of the theorem is done in six different steps by following the structure on \cite{D}. Since we are going to use the inductive lemma at every step, first we define the parameters and sets to which  we are going apply such lemma. Then we check that we can actually apply the lemma and obtain some extra estimates for the results of the lemma. After this we see that the sequence of frequency vectors converges. We do the same with the sequence of canonical transformations. Then we get some bounds for the size of the components of the final diffeomorphism. Next we characterize the tori that survive by the perturbation. Finally we give some estimates for the measure of the set of these tori.

Note that our version of the $b^m$-KAM theorem improves the one in \cite{KMS16} in several ways. Firstly it is applicable to $b^m$-symplectic structures not only for $b$-symplectic. Also we give several estimates that are not obtained in \cite{KMS16}, this estimates have sense in a neighborhood of the critical set $Z$, while \cite{KMS16} only studied the behavior at $Z$. Finally the type of perturbation we consider is far more general, since we do not have any condition of the form of the perturbation but only on its size.

\subsection{Reducing the problem to an analytical perturbation.}

In the standard KAM, we assume to have an analytic Hamiltonian $h(I)$ depending only on the action coordinates and we add to it a small analytical perturbation $R(\phi,I)$. This perturbed system receives the name of \emph{nearly integrable system}. And then find a new coordinate system such that $h(I) + R(\phi,I) = \tilde h(\tilde I)$ where most of the quasi-periodic orbits are preserved and can be mapped to the unperturbed quasi-periodic orbits by means of the coordinate change.

In our setting we may assume $h(I)$ to not be analytical and be a $b^m$-function. Also the  perturbation $R(\phi, I)$ may as well be considered a $b^m$-function. In the following lines we justify without loss of generality that actually we can assume the perturbation to be analytical.

Let us state this more precisely. Let $(M,x,Z,\omega,F)$ be a $b^m$-manifold with a $b^m$ integrable system $F$ on it. Consider action angle coordinates on a neighborhood of $Z$.
Then we can assume the expressions:

$$\omega = \left(\sum_{j=1}^{m}\frac{c_j}{I_1^j}\right) d I_1 \wedge d\phi_1 + \sum_{i=2}^n dI_i\wedge d\phi_i, \text{ and }$$
$$F = (q_0' \log I_1 + \sum_{i=1}^{m-1} q_i'\frac{1}{I_1^i} + h(I), f_2, \ldots, f_n)$$

where $h,f_2,\ldots,f_n$ are analytical.

Let the Hamiltonian function of our system be the first component of the moment map $\hat h' = q'_0\log I_1 + \sum_{i=1}^{m-1} q_i'\frac{1}{I_1^i} + h = \zeta' + h$, where $\zeta' := q_0' \log I_1 + \sum_{i=1}^{m-1} q_i' \frac{1}{I_1^i}$. Note that $d\zeta' = \sum_{i=1}^m \hat q_i'\frac{1}{I_1'}$, where $\hat q_i ' = -(i-1)q_{i-1}'$. Note that by the result of the previous chapter $c_j/\hat q_j ' = \mathcal{K}$ the modular period. In particular $c_m/\hat q_m ' = \mathcal{K}$.

The hamiltonian system given by $\hat h'$ can be easily solved by $\phi = \phi_0 + u' t, I = I_0$ where $u'$ is going to be defined in the following sections.
Consider now a perturbation of this system: $\hat H' = \hat h'(I) = \hat R(I,\phi)$, where $\hat R$ is a $b^m$-function $\hat R(I,\phi) = R_{\zeta}(I_1) + R(I,\phi)$ where $R_\zeta(I_1) = (r_0 \log I_1 + \sum_{i=1}^{m-1} r_i \frac{1}{I_1^i})$ is the singular part. Then we can consider the perturbations $R_\zeta(I_1)$ and $R(I,\phi)$ separately. This way, we may consider $R_\zeta(I)$ as part of $\hat h'(I)$. Then we have a new hamiltonian
$$\hat h(I) = (q_0' + r_0) \log I_1 + \sum_{i=1}^{m-1}(q_i' + r_i)\frac{1}{I_1^i} + h = q_0 \log I_1 + \sum_{i=1}^{m-1} q_i \frac{1}{I_1^i} + h.$$

Now, instead of the identity $\mathcal{K} \hat q_j' = c_j$ we will have $\mathcal{K} (\hat q_j - \hat r_j) = c_j$, which implies $\mathcal{K}\left(1 - \frac{\hat r_j}{\hat q_j ' + \hat r_j}\right) = \frac{c_j}{\hat q_j}$. In particular

$$\mathcal{K}\left(1 - \frac{\hat r_m}{\hat q_m ' + \hat r_m}\right) = \frac{c_m}{\hat q_m}$$

Let us define $\mathcal{K}' = \mathcal{K}\left(1 - \frac{\hat r_m}{\hat q_m ' + \hat r_m}\right)$.
So from now on we assume $\hat h = q_0 \log I_1 + \sum_{i=1}^{m-1} q_i \frac{1}{I_1^i} + h$, that the perturbation $R(\phi,I)$ is analytical, and we have the condition $\frac{c_m}{\hat q_m} = \mathcal{K}'$. Observe that this system with only the singular perturbation is still easy to solve in the same way that the system previous to this perturbation was.

\subsection{Looking for a $b^m$-symplectomorphism}

Assume we have a Hamiltonian function $H = \hat{h}(I) + R(\phi,I)$ in action-angle coordinates. Where $\hat{h}(I)$ is the singular component of the $b^m$-integrable system, i.e.
\begin{equation}\label{eq:bm-hamiltonian}
\hat{h}(I) = h(I) + q_0 \log(I_1) + \sum_{i = 1}^{m-1}q_i\frac{1}{I_1^i},
\end{equation}
where $h(I)$ is analytical\footnote{If another component of the moment map is chosen to be the hamiltonian of the system, the result still holds: the computations can be replicated assuming $\hat{h}(I) = h(I)$.}.
Assume also that the $b^m$-symplectic form $\omega$\footnote{In classical KAM, $\omega$ is used to denote the frequency vector $\frac{\partial h}{\partial I}$. We need $\omega$ to denote the $b^m$-symplectic form so we are going to use $u$ to denote the frequency vector.} in these coordinates is expressed as:
\begin{equation}\label{eq:bm-symplectic}
\omega = \left(\sum_{j = 1}^m\frac{c_{j}}{I_1^j}\right)dI_1\wedge d \phi_1 + \sum_{i=2}^{n}dI_i\wedge d\phi_i.
\end{equation}

And finally, the expression for the frequency vector is:

$$\hat{u} = \frac{\partial \hat{h}}{\partial I} = \frac{\partial(h(I) + q_0 \log(I_1) + \sum_{i = 1}^{m-1}q_i\frac{1}{I_1^i})}{\partial I}$$ $$= \left(u_1 + \sum_{i = 1}^{m}\frac{\hat{q}_i}{I_1^i}, u_2, \ldots, u_n\right),
$$

where $\hat{q}_1 = q_0$ and $\hat{q}_{i-1} = -iq_i$ if $i \neq 0$.

The objective is to follow the steps of the usual KAM construction (the steps followed are highly inspired in \cite{D}) replacing the standard symplectic form for $\omega$ and taking as hamiltonian the $b^m$-function $\hat{h}$.

\begin{remark}
The objective of the construction is to find a diffeomorphism (actually a $b^m$-symplectomorphism) $\psi$ such that $H\circ\psi=h(\tilde{I})$. This is done inductively, by taking $H\circ\psi=H \circ \phi_1\circ\ldots\circ \phi_q \circ\ldots$, while trying to make $R(\phi,I)$ smaller at every step.
\end{remark}

\textbf{Let us focus in one single step}

Recall the classical formula:

\begin{lemma}\label{lemma:lie_taylor} See \cite{D}.
$$f\circ\phi_t=\sum_{j=0}^{\infty}\frac{t^j}{j!}L_W^jf,\quad L_W^jf=\{L_W^{j-1}f,W\}$$
Where $W$ is the Hamiltonian that generates the flow $\phi_t$, and $\{\cdot,\cdot\}$ is the corresponding Poisson bracket.
\end{lemma}

We will denote $r_k(H,W,t)=\sum_{j=k}^\infty\frac{t^j}{j!}L_W^jH$.

\begin{equation}\label{eq:new_hamiltonian_expression}
\begin{array}{rcl}
\displaystyle{H\circ\phi = \left.H\circ\phi\right|_{t=1}}& = & \displaystyle{\sum_{j=0}^{\infty}\frac{t^j}{j!}\left.L_W^j\underbrace{H}_{\hat{h}+R}\right|_{t=1}}\\
\\
&=& \displaystyle{ \hat{h}+R\{\hat{h}+R,W\}+r_2(H,W,1)}\\
\\
&=& \displaystyle{ \hat{h}+R+\{\hat{h},W\}+\{R,W\}+r_2(\hat{h},W,1)}\\
\\
& & \quad \displaystyle{+r_2(R,W,1)}\\
\\
&=& \displaystyle{ \hat{h}+\underbrace{R+\{\hat{h},W\}}_{\begin{subarray}{c}\text{We want to cancel} \\ \text{this term as} \\ \text{fast as we can}\end{subarray}}+r_2(\hat{h},W,1)+r_2(R,W,1)}
\end{array}
\end{equation}

We want $\{\hat{h},W\}+R_{\leq k} =0$, equivalently $\{W,\hat{h}\}= R_{\leq k} $, where $R_{\leq k} $ means the Fourier expression of $R$ up to order $K$:
$$R_{\leq k} = \sum_{\begin{subarray}{c} k\in\mathbb{R}^n \\ |k|_1 \leq K\end{subarray}} R_k(I)e^{ik\cdot\phi}$$

Let us impose the condition $\{W,\hat{h}\}=R_{\leq K}$. Let us write the expression of the Poisson bracket associated to the $b^m$-symplectic form.

$$
\begin{array}{rcl}
\{W,\hat{h}\} & = & \displaystyle \left(\frac{1}{\sum_{j = 1}^m \frac{c_j}{I_1^j}}\right)\left(\frac{\partial W}{\partial \phi_1}\frac{\partial \hat{h}}{\partial I_1} - \frac{\partial W}{\partial I_1}\frac{\partial \hat{h}}{\partial \phi_1}\right)\\
& & + \displaystyle{\sum_{i = 2}^n \left(\frac{\partial W}{\partial \phi_i}\frac{\partial \hat{h}}{\partial I_i} - \frac{\partial W}{\partial I_i}\frac{\partial \hat{h}}{\partial \phi_i}\right)}
\end{array}
$$

%Using the two expanded expressions:

%$$
%\begin{array}{rcl}
% \displaystyle{\sum_{i = 2}^n \left(\frac{\partial W}{\partial \phi_i}\frac{\partial h}{\partial I_i} - \frac{\partial W}{\partial I_i}\frac{\partial h}{\partial \phi_i}\right)} + I_1^m\left(\frac{\partial W}{\partial \phi_i}\frac{\partial h}{\partial I_i} - \frac{\partial W}{\partial I_i}\frac{\partial h}{\partial \phi_i}\right) & = & \displaystyle{\sum_{\begin{subarray}{c} k\in\mathbb{R}^n \\ |k|_1 \leq K\end{subarray}} R_k(I)e^{ik\cdot\phi}}
%\end{array}
%$$

Because $\hat{h}$ depends only on $I$, $\frac{\partial \hat{h}}{\partial \phi_i} = 0$ for all $i$. Moreover, the singular part of the $b^m$-function only depends on $I_1$ and hence its derivatives with respect to the other variables are also 0. Using that $\frac{\partial \hat{h}}{\partial I} = u + \sum_{i = 1}^{m}\frac{\hat{q}_i}{I_1^i}$ the previous expression can be  simplified:

$$
\begin{array}{rcl}
\{W,\hat{h}\} & = & \displaystyle \left(\frac{u_1 + \sum_{i = 1}^{m}\frac{\hat{q}_i}{I_1^i}}{\sum_{j = 1}^m \frac{c_j}{I_1^j}}\right)\frac{\partial W}{\partial \phi_1} + \displaystyle\sum_{i = 2}^n \frac{\partial W}{\partial \phi_i}u_i
\end{array}
$$

To expand the expression further we develop $W$ in its Fourier expression: $W=\sum_{\begin{subarray}{c} k\in\mathbb{R}^n \\ |k|_1 \leq K\end{subarray}}W_k(I)e^{ik\phi}$. The Fourier expansion is added up to order $K$, because it is only necessary for the expressions to agree up to order $K$. With this notations the condition becomes:

\begin{longtable}{rcl}
$\{W,\hat{h}\}_{\leq K}$ & $=$ & $\displaystyle
 \left(\frac{u_1 + \sum_{i = 1}^{m}\frac{\hat{q}_i}{I_1^i}}{\sum_{j = 1}^m \frac{c_j}{I_1^j}}\right)\frac{\partial }{\partial \phi_1}\left(\sum_{\begin{subarray}{c} k\in\mathbb{R}^n \\ |k|_1 \leq K\end{subarray}}W_k(I)e^{ik\phi}\right)$\\
 \\
& & $\qquad \qquad \qquad  + \displaystyle\sum_{j = 2}^n u_j \frac{\partial }{\partial \phi_j}\left(\sum_{\begin{subarray}{c} k\in\mathbb{R}^n \\ |k|_1 \leq K\end{subarray}}W_k(I)e^{ik\phi}\right)$\\
\\
& $=$ & $\displaystyle\left(\frac{u_1 + \sum_{i = 1}^{m}\frac{\hat{q}_i}{I_1^i}}{\sum_{j = 1}^m \frac{c_j}{I_1^j}}\right)\left(\sum_{\begin{subarray}{c} k\in\mathbb{R}^n \\ |k|_1 \leq K\end{subarray}}W_k(I)e^{ik\phi}ik_1\right)$\\
\\
& &  $\qquad \qquad \qquad + \displaystyle\sum_{j = 2}^n u_j \left(\sum_{\begin{subarray}{c} k\in\mathbb{R}^n \\ |k|_1 \leq K\end{subarray}}W_k(I)e^{ik\phi}ik_j\right)$
\\
\\
& = & $\displaystyle \sum_{\begin{subarray}{c} k\in\mathbb{R}^n \\ |k|_1 \leq K\end{subarray}} W_k(I)e^{ik\phi}\cdot\left( i k_1 \left(\frac{u_1 + \sum_{i = 1}^{m}\frac{\hat{q}_i}{I_1^i}}{\sum_{j = 1}^m \frac{c_j}{I_1^j}}\right)
+ \sum_{j=2}^nik_ju_j\right)$\\
\\
$ = R_{\leq K}$\\
\end{longtable}
Then, it is possible to make the two sides of the equation equal by imposing the condition term by term:

\begin{equation}\label{eq:solve_coefs}
\begin{array}{rcl}
W_k(I) & = & \displaystyle R_k(I)\frac{1}{i\left(k_1 \left(\frac{u_1 + \sum_{i = 1}^{m}\frac{\hat{q}_i}{I_1^i}}{\sum_{j = 1}^m \frac{c_j}{I_1^j}}\right)  + \sum_{j = 2}^n k_ju_j\right)}
\\
\\
& = & \displaystyle R_k(I)\frac{1}{i\left(k_1 \left(\frac{u_1 + \sum_{i = 1}^{m}\frac{\hat{q}_i}{I_1^i}}{\sum_{j = 1}^m \frac{c_j}{I_1^j}}\right) + \bar{k}\bar{u}\right)},
\end{array}
\end{equation}

where we adopted the notation $\sum_{j = 2}^n k_ju_j = \bar{k}\bar{u}$.

\begin{remark}
Observe that the expression \ref{eq:solve_coefs} has no sense when $k = \vec{0}$ and hence $\{W,h\}_0=R_0$\footnote{The zero term of the Fourier series can be seen as the angular average of the function} can not be solved. Let $W_0(I)=0$, then $\{h,W\}_{\leq K} = R_{\leq K} -R_0$.
\end{remark}

Plugging the results above into the equation  \ref{eq:new_hamiltonian_expression}, one obtains:

$$
H\circ\phi = \hat{h}  + R_0 + R_{\geq K} + r_2(\hat{h},W,1) + r_1(R,W,1)
$$

With this construction the diffeomorphism $\phi$ is found. But this is  only the first of many steps. If $q$ denotes the number of the iteration of this procedure, in general, we obtain:

\begin{equation}
\begin{array}{rcl}
H^{(q)}  =  H^{(q-1)}\circ\phi^{(q)} & = &  \hat{h}^{(q-1)}  + R_0^{(q-1)} + R_{\geq K}^{(q-1)} \\
& & + r_2(h^{(q-1)},W^{(q)},1) + r_1(R^{(q-1)},W^{(q)},1),
\end{array}
\end{equation}

and at every step:

\begin{equation}\label{eq:iterative_h_and_R}
\begin{cases}
\hat{h}^{(q)}  =  \hat{h}^{(q-1)} + R_0^{(q-1)}\\
R^{(q)}  =  R_{>K}^{(q-1)} + r_2(\hat{h}^{(q-1)}, W^{(q)},1) + r_1(R^{(q-1)}, W^{(q)},1)\\
\end{cases}
\end{equation}

\subsection{On the change of the defining function under \\ $b^m$-symplectomorphisms}

Note that since we are considering $b^m$-manifolds it only makes sense to consider $I_1$ up to order $m$, see \cite{Scott16}. When talking about defining functions we are interested in $[I_1]$, its jet up to order $m$.
By definition $b^m$-maps preserve $I_1$ up to order $m$ and $b^m$-vector fields $X$ are such that $\mathcal{L}_X(I_1) = g\cdot I_1^m$ for $g\in\mathcal{C}^\infty(M)$.

\begin{lemma}
Let $\phi_t$ be the integral flow of $X$ a $b^m$-vector field, then $\phi_t$ is a $b^m$-map.
\end{lemma}
\begin{proof}
We want
$$I_1\circ \phi_t = I_1 + I_1^m\cdot g$$ for some $g\in \mathcal{C}^\infty(M)$.
We will use \ref{lemma:lie_taylor}.
$$I_1 \circ \phi_t = \sum_{j=0}^\infty \frac{t^j}{j!}L_X^j I_1 = I_1 + \mathcal{L}_X(I_1) + \sum_{j=2}^\infty\frac{t^j}{j!}L_X^j I_1$$
$$= I_1 +  I_1^m + \sum_{j=2}^\infty\frac{t^j}{j!}L_X^j I_1.$$
On the other hand, let us prove by induction $L_X^k I_1 = g^{(k)} I_1^m$. The first case is obvious, assume the case $k$ holds and let us prove the case $k+1$.

$$
\begin{array}{rcl}
L_X^{k+1} I_1 & = & \{L_X^k I_1, X\} \\
& = & \{g^{(k)} I_1^m, X\} \\
& = & (L_X g^{(q)}) I_1^m + g^{(k)}\cdot m I_1^{m-1}L_X I_1 \\
& = & (L_X g^{(k)} + g^{(k)}\cdot m \cdot I_1^{m-1}\cdot g) I_1^m \\
& = & g^{(k+1)} I_1^m \\
\end{array}
$$

where $g^{(k+1)} = L_X g^{(k)} + g^{(k)}\cdot m \cdot I_1^{m-1}\cdot g$.

\end{proof}

\begin{lemma}
The Hamiltonian vector flow of some smooth hamiltonian function $h$ is a $b^m$-vector field.
\end{lemma}
\begin{proof}
At each point of $Z$ the following identity holds $\mathcal{L}_{X_h} I_1 = I_1^m\frac{\partial f}{\partial \phi_1}$. The result can be extended at a neighborhood of $Z$.
\end{proof}

Observe that combining the two previous results we get that the hamiltonian flow of a function preserves $I_1$ up to order $m$.

\section{Technical results}

As the non-singular part of our functions we will be considering analytic functions on $\mathbb{T}\times G$, $G \subset \mathbb{R}^n$. The easiest way to work with these functions is to consider them as holomorphic functions on some complex neighborhood. Let us define formally this neighborhood.

$$\mathcal{W}_{\rho_1}(\mathbb{T}^n) := \{ \phi : \Re\phi \in \mathbb{T}^n, |\Im \phi|_{\infty} \leq \rho_1\},$$
$$\mathcal{V}_{\rho_2}(G) := \{I \in \mathbb{C}^n : |I - I'|\leq \rho_2 \text{ for some } I' \in G\},$$
$$\mathcal{D}_\rho (G) := \mathcal{W}_{\rho_1}(\mathbb{T}^n) \times \mathcal{V}_{\rho_2}(G),$$

where $|\cdot|_\infty$ denotes the maximum norm and $|\cdot|_2$ denotes de Euclidean norm.
Now it is necessary to clarify the norms that are going to be used on these sets.

\begin{definition}
Let $f$ be an action function (only depending on the $I$-coordinates), and $F$ an action vector field.
$$
\begin{array}{rl}
|f|_{G,\eta} := \sup_{I\in\mathcal{V}_\eta(G)} |f(I)|,  & |f|_G := |f|_{G,0}\\
|F|_{G,\eta,p} := \sup_{I\in\mathcal{V}_\eta(G)} |F(I)|_p,  & |F|_{G,\eta} := |F|_{G,\eta,2}\\

\end{array}
$$
Now, assume $f(I,\phi)$ to be an action-angle function written using its Fourier expansion as $\sum_{k\in\mathbb{Z}^n} f_k(I)e^{ik\cdot\phi}$, and $F$ to be an action-angle vector field.
$$
\begin{array}{rcl}
|f|_{G,\rho} := \sup_{(\phi,I)\in\mathcal{D}_\rho(G)} |f(I)|, & \|f\|_{G,\rho}:=\sum_{k\in\mathbb{Z}^n} |f_k|_{G,\rho_2} e^{|k|_1\rho_1}\\
|F|_{G,\rho,p} := \sum_{k\in\mathbb{Z}^n} |F_k|_{G,\rho_2,p} e^{|k|_1\rho_1}, & \|F\|_{G,\rho} = \|F\|_{G,\rho,2}\\
\end{array}
$$
\end{definition}

\begin{lemma}[Cauchy Inequality]
$$\left\|\frac{\partial f}{\partial \phi}\right\|_{G,(\rho_1-\delta_1, \rho_2),1}\leq \frac{1}{e\delta_1}\left\|f\right\|_{G,\rho}$$
$$\left\|\frac{\partial f}{\partial I}\right\|_{G,(\rho_1, \rho_2-\delta_2),\infty}\leq \frac{1}{\delta_2}\|f\|_{G,\rho}$$
\end{lemma}

\begin{definition}
If $Df=(\frac{\partial f}{\partial \phi}, \frac{\partial f}{\partial I})$,
$$\|Df\|_{G,\rho,c}:= \max\left(\|\frac{\partial f}{\partial \phi}\|_{G,\rho,1},c\|\frac{\partial f}{\partial I}\|_{G,\rho,\infty}\right)$$
\end{definition}

\begin{definition}To simplify our notation, let us define:
$$\mathcal{A}(I_1) = \frac{\sum_{j = 1}^{m}\frac{\hat{q}_j}{I_1^j}}{\sum_{j=1}^m \frac{c_j}{I_1^j}} \quad \text{ and } \quad \mathcal{B}(I_1) = \frac{1}{\sum_{j=1}^m \frac{c_j}{I_1^j}}.$$
\end{definition}

\begin{remark}
With this notation, equation \ref{eq:solve_coefs} can be written as:
$$W_k(I) = \frac{R_k(I)}{i(k_1 \mathcal{B}(I_1) u_1 + \bar{k}\bar{u} + k_1 \mathcal{A}(I_1))}$$
\end{remark}

Observe that $\mathcal{A}(I_1)$ and $\mathcal{B}(I_1)$ are analytic (holomorphic on the complex extended domain) where the denominator does not vanish. We can assume that this does not happen by shrinking the domain $G$ in the direction of $I_1$. Observe, in particular, that when $I_1 \rightarrow 0$, $\mathcal{A}(I_1) \rightarrow \hat q_m/c_m = 1/\mathcal{K}'$ the inverse of the modular period and $\mathcal{B}(I_1) \rightarrow 0$. In this way,  the norms of $\mathcal{A}(I_1)$ and $\mathcal{B}(I_1)$ are bounded and well defined. We will denote these norms by $K_{\mathcal{A}}$ and $K_{\mathcal{B}}$ respectively. Also, since $\mathcal{A}(I_1)$ and $\mathcal{B}(I_1)$ are analytic, their derivatives will also be bounded, and we will denote the norms of these derivatives by $K_{\mathcal{A}'}$ and $K_{\mathcal{B}'}$.

To further simplify  the notation in the following computations we introduce the definition:

\begin{definition}

$$
\bar{\mathcal{A}} = \left(\begin{array}{c}
\mathcal{A}\\
0 \\
\end{array}\right)
\quad \text{ and } \quad
\bar{\mathcal{B}} = \left(\begin{array}{cc}
\mathcal{B} & 0\\
0 & \text{Id}_{n-1,n-1}\\
\end{array}\right)
$$

\end{definition}

\begin{remark}
With this notation, equation \ref{eq:solve_coefs} can be written as:
\begin{equation}\label{eq:solve_coefs_simplified}
W_k(I) = \frac{R_k(I)}{i(k \bar{\mathcal{B}}(I_1) u + k \bar{\mathcal{A}}(I_1))}
\end{equation}

\end{remark}

\begin{definition}
Having fixed $\omega$, a $b^m$-symplectic form (as in equation \ref{eq:bm-symplectic}) and $\hat{h}$ a $b^m$-function (as in equation \ref{eq:bm-hamiltonian}) as a hamiltonian. Given an integer $K$ and $\alpha > 0$, $F \subset \mathbb{R}^n$ (or $\mathbb{C}^n$) the space of frequencies is said to be $\alpha, K$-non-resonant with respect to $(c_1,\ldots,c_m)$ and $(\hat{q}_1,\ldots, \hat{q}_m)$ if
$$|k \bar{\mathcal{B}}(I_1) u + k \bar{\mathcal{A}}(I_1)|\geq \alpha, \forall k \in \mathbb{Z}\setminus \{0\}, |k|_1\leq K, \forall u \in F.$$
We are going to use the notation $\alpha, K, c,\hat{q}$-non-resonant.
\end{definition}

\begin{remark}
The non-resonance condition is established on $u = \partial h/\partial I$, not on $\hat{u} = \partial \hat h/\partial I$, because our non-resonance condition already takes into account the singularities. In this way we can use the analytic character of $u$.
\end{remark}

\begin{remark}
If $\left|\frac{\partial u}{\partial I}\right|_{G,\rho_2}$ is bounded by $M'$,  then $\left|\frac{\partial }{\partial I}\left( \bar{\mathcal{B}}u + \bar{\mathcal{A}}\right)\right|_{G,\rho_2}$ is also bounded:

\begin{equation}\label{eq:M_def}
\begin{array}{rcl}
\left|\frac{\partial }{\partial I}\left( \bar{\mathcal{B}}u + \bar{\mathcal{A}}\right)\right|_{G,\rho_2} &
 \leq &
 \left|\frac{\partial \bar{\mathcal{B}}}{\partial I} u
  + \bar{\mathcal{B}} \frac{\partial u}{\partial I}
   + \frac{\partial \bar{\mathcal{A}}}{\partial I}\right|_{G,\rho_2}\\
\\
& \leq & K_{\mathcal{B}'} |u|_{G,\rho_2} + K_{\mathcal{B}}M' + K_{\mathcal{A}} =: M.\\
\\
\end{array}
\end{equation}

\end{remark}

\begin{remark}
When we consider the standard KAM theorem, the frequency vector $u$ is relevant because the solution to the Hamilton equations of the unperturbed problem has the form:

$$I = I_0, \quad \phi = \phi_0 + ut.$$

Let us see what plays the role of $u$ in our $b^m$-KAM theorem. Let us find the coordinate expression of the solution to $\iota_{X_{\hat{h}}} \omega=  d \hat{h}$, where $\omega$ is a $b^m$-symplectic form in action-angle coordinates.

$$X_{\hat{h}} = \dot{I}_1\frac{\partial}{\partial I_1} + \ldots + \dot{I}_n\frac{\partial}{\partial I_n},$$

where $\dot{I}_1,\ldots, \dot{I}_n$ are the functions we want to find.

$$d\hat{h} = \left(\sum_{j=1}^m \hat{q}_i \frac{1}{I_1^j}\right) dI_1 + dh,$$

and hence,

$$X_{\hat{h}} = \Pi(d \hat{h},\cdot ) = \frac{\sum_{i=1}^m \frac{\hat q_i}{I_1^i}}{\sum_{i=1}^m \frac{c_j}{I_1^j}}\frac{\partial}{\partial \phi_i} + X_h.$$

Hence $\phi = \phi_0 + (\underbrace{\bar{\mathcal{B}} u + \bar{\mathcal{A}}}_{u'})t$. So the frequency vector that we are going to be concerned about is going to be $u'$ instead of $\hat{u} = \frac{\partial}{\partial I} \hat{h}$.
\end{remark}
\begin{lemma}
If $u$ is one-to-one from $\mathcal{G}$ to its image then $u' = \bar{\mathcal{B}} u + \bar{\mathcal{A}}$ is also one-to-one from $\mathcal{G}'$ to its image in a neighborhood of $Z$, while at $Z$ it is the projection of $u$ such that the first coordinate is sent to $\frac{\hat q_m}{c_m} = 1/\mathcal{K}'$ the inverse of the modular period, were  $\mathcal{G}'\subseteq \mathcal{G}$.
\end{lemma}
\begin{proof}
Because $$u' = \left(\frac{1}{\sum_{j=1}^m \frac{c_j}{I_1^j}} u_1 + \frac{\sum_{j=1}^m \frac{\hat{q}_j}{I_1^j}}{\sum_{j=1}^m \frac{c_j}{I_1}},u_2,\ldots,u_n\right),$$
and  $\mathcal{B}$ is invertible outside $I_1 = 0$, shrinking $\mathcal{G}$ if necessary in the first dimension the map is one-to-one. But at the critical set $\{I_1 = 0\}$, $u'$ is a projection of $u$ where the first component is sent to the constant value $\frac{\hat q_m}{c_m} = \frac{1}{\mathcal{K}'}$.
\end{proof}
\begin{lemma}
If $u(G)$ is $\alpha,K, c,\hat{q}$-non-resonant, then $u(\mathcal{V}_{\rho_2}(G))$ is $\frac{\alpha}{2},K, c,\hat{q}$-non-resonant, assuming that $\rho_2\leq \frac{\alpha}{2MK}$ and $\left|\frac{\partial u}{\partial I}\right|_{G,\rho_2} \leq M'$
\end{lemma}
\begin{proof}
Fix $k\in\mathbb{Z}\setminus\{0\}$, we want to bound $|k \bar{\mathcal{B}}(I_1) v + k \bar{\mathcal{A}}(I_1)|$ where $v\in u(\mathcal{V}_{\rho_2}(G))$ as a function on $v$.
Given $v \in u(\mathcal{V}_{\rho_2}(G))$ we  ask whether there is any bound for the distance to some $v'\in u(G)$.
$$v \in u(\mathcal{V}_{\rho_2}(G)) \Rightarrow v = u(x), x \in \mathcal{V}_{\rho_2}(G)\Rightarrow\exists y \in G \text{ such that } |x-y| \leq \rho_2.$$
Take $v'=u(y)$.

$$|v-v'|\leq|x-y|\left|\frac{\partial u}{\partial I}\right|_{G,\rho_2} \leq \rho_2 M' \leq \rho_2  M/K_\mathcal{B} \leq \frac{\alpha}{2MK} M/K_\mathcal{B} = \frac{\alpha}{2K K_\mathcal{B}}.$$

Where we used equation \ref{eq:M_def} in the third inequality.

$$
\begin{array}{rcl}
\displaystyle |k_1 \mathcal{B}(I_1) v_1 + \bar{k}\bar{v} + k_1 \mathcal{A}(I_1)| & \geq & \displaystyle \underbrace{|k_1 \mathcal{B}(I_1) v'_1 + \bar{k}\bar{v'} + k_1 \mathcal{A}(I_1)|}_{\geq \alpha} \\
& & -|k_1 \mathcal{B}(I_1) (v_1-v'_1) + \bar{k}(\bar{v} - \bar{v}')|\\
\\
& \geq & \displaystyle \alpha - K_\mathcal{B}\underbrace{|k\cdot(v-v')|}_{\leq K \alpha /(2K K_\mathcal{B})} \\
\\
 & \geq & \alpha -\alpha/2 = \alpha/2
\end{array}
$$

\end{proof}

\begin{proposition}\label{prop:14}
Let $\hat{h}(I)$ be a $b^m$-function as in equation \ref{eq:bm-hamiltonian}. Assume $h(I)$ and $R(\phi, I)$ be real analytic on $\mathcal{D}_\rho(G)$, $u(G) = \frac{\partial h}{\partial I}(G)$ is $\alpha, K, c, \hat{q}$-non-resonant.
Assume also that
$|\frac{\partial}{\partial I} u|_{G,\rho_2} \leq M'$
and $\rho_2 \leq \frac{\alpha}{2MK}$. Let $c > 0$ given.
Then $R_0(\phi, I)$, $W_{\leq K}(\phi, I)$ given by the previous construction are both real analytic on $\mathcal{D}_\rho(G)$ and the following bounds hold
\begin{enumerate}
\item\label{prop:14:item:1} $||D R_0||_{G,\rho,c} \leq ||D R||_{G,\rho,c}$
\item\label{prop:14:item:2} $||D (R - R_0)||_{G,\rho,c} \leq ||D R_0||_{G,\rho,c}$
\item\label{prop:14:item:3} $||D W||_{G,\rho,c} \leq \frac{2A}{\alpha}||D R_0||_{G,\rho,c}$
\end{enumerate}
Where $A = 1 + \frac{2Mc}{\alpha}$
\end{proposition}

\begin{proof}
Inequalities \ref{prop:14:item:1} and \ref{prop:14:item:2} are obvious because of the Fourier expression.
Let us prove inequality \ref{prop:14:item:3}. Let us expand $R(\phi,I)$ and $W(\phi, I)$ in their Fourier expression:
$$R = \sum_{\begin{subarray}{c} k\in\mathbb{R}^n\end{subarray}} R_k(I)e^{ik\cdot\phi}, \quad W = \sum_{\begin{subarray}{c} k\in\mathbb{R}^n \end{subarray}} W_k(I)e^{ik\cdot\phi}.$$
We will bound this expression finding  term-by-term bounds.

$$\frac{\partial R}{\partial \phi} = \sum_{\begin{subarray}{c} k\in\mathbb{R}^n\end{subarray}} R_k(I)e^{ik\cdot\phi}ik.$$

Hence, if we denote  $[ \frac{\partial R}{\partial \phi} ]_k $ the $k$-th term of the Fourier expansion of $\frac{\partial R}{\partial \phi}$, we have:

$$\left[ \frac{\partial R}{\partial \phi} \right]_k = R_k ik.$$

Let us compute the derivative of $W_k$ with respect to the $I$ variables:

\begin{longtable}{rcl}
$\displaystyle \frac{\partial W_k}{\partial I}$ & $=$ & $\displaystyle\frac{\partial}{\partial I}\left(\frac{R_k}{i(k \bar{\mathcal{B}}(I_1) u  + k \bar{\mathcal{A}}(I_1))}\right)$\\
\\
& $=$ & $\displaystyle \frac{\partial R_k/\partial I }{i(k \bar{\mathcal{B}}(I_1) u  + k \bar{\mathcal{A}}(I_1)))} - \frac{R_k i \frac{\partial}{\partial I }(k \bar{\mathcal{B}}(I_1) u  + k \bar{\mathcal{A}}(I_1)))}{[i(k \bar{\mathcal{B}}(I_1) u  + k \bar{\mathcal{A}}(I_1)))]^2}$\\
\\
& $=$ & $\displaystyle \frac{\partial R_k/\partial I }{i(k \bar{\mathcal{B}}(I_1) u  + k \bar{\mathcal{A}}(I_1)))} + \frac{R_k i k \frac{\partial}{\partial I }(\bar{\mathcal{B}}(I_1) u  + \bar{\mathcal{A}}(I_1)))}{[(k \bar{\mathcal{B}}(I_1) u  + k \bar{\mathcal{A}}(I_1)))]^2}$\\
\\
& $=$ & $\displaystyle \frac{\partial R_k/\partial I }{i(k \bar{\mathcal{B}}(I_1) u  + k \bar{\mathcal{A}}(I_1)))} + \frac{[\frac{\partial R_k}{\partial \phi}]_k \frac{\partial}{\partial I }(\bar{\mathcal{B}}(I_1) u  + \bar{\mathcal{A}}(I_1)))}{[(k \bar{\mathcal{B}}(I_1) u  + k \bar{\mathcal{A}}(I_1)))]^2}.$\\
\end{longtable}

Then, we take norms ($|\cdot|_{G,\rho_2,\infty}$) on each side of the equation.

$$
\begin{array}{rcl}
\displaystyle \left|\frac{\partial W_k}{\partial I}\right|_{G,\rho_2, \infty} & \leq & \displaystyle \frac{2}{\alpha}\left|\frac{\partial R_k}{\partial I}\right|_{G,\rho_2, \infty} +
\displaystyle \frac{4M}{\alpha^2}\left|\left[\frac{\partial R_k}{\partial \phi}\right]_k\right|_{G,\rho_2, \infty}\\
\\
& \leq & \displaystyle \frac{2}{\alpha}\left|\frac{\partial R_k}{\partial I}\right|_{G,\rho_2, \infty} +
\displaystyle \frac{4M}{\alpha^2}\left|\left[\frac{\partial R_k}{\partial \phi}\right]_k\right|_{G,\rho_2, 1}.
\end{array}
$$

Taking the supremum at the whole domain:

$$
\begin{array}{rcl}
\displaystyle \left\|\frac{\partial W_k}{\partial I}\right\|_{G,\rho_2, \infty} & \leq & \displaystyle \frac{2}{\alpha}\left\|\frac{\partial R_k}{\partial I}\right\|_{G,\rho_2, \infty} +
\displaystyle \frac{4M}{\alpha^2}\left\|\left[\frac{\partial R_k}{\partial \phi}\right]_k\right\|_{G,\rho_2, 1}.
\end{array}
$$

Moreover,

$$
\begin{array}{rcl}
\displaystyle \frac{\partial W(I)}{\partial \phi} & = & \displaystyle \frac{\partial}{\partial \phi}\left(\sum_{\begin{subarray}{c} k\in\mathbb{R}^n \end{subarray}} W_k(I)e^{ik\cdot\phi}\right)\\
\\
 & = & \displaystyle \frac{\partial}{\partial \phi}\left(\sum_{\begin{subarray}{c} k\in\mathbb{R}^n \end{subarray}} ik W_k(I)e^{ik\cdot\phi}\right).\\
\end{array}
$$

Hence, the $k$-th term of the Fourier series of $\frac{\partial W}{\partial \phi}$ is

$$\left[\frac{\partial W}{\partial \phi}\right]_k = W_k ik = \frac{R_k}{i(k \bar{\mathcal{B}}(I_1) u  + k \bar{\mathcal{A}}(I_1)))}ik $$
$$= \frac{1}{i(k \bar{\mathcal{B}}(I_1) u  + k \bar{\mathcal{A}}(I_1)))}\left[\frac{\partial R}{\partial \phi}\right]_k.$$

Taking norms ($\|\cdot\|_{G,\rho,1}$) at each side:

$$ \left\|\frac{\partial W}{\partial \phi}\right\|_{G,\rho,1} \leq \frac{2}{\alpha}\left\|\frac{\partial W}{\partial \phi}\right\|_{G,\rho,1}.$$

Then,

\begin{longtable}{rcl}
$\displaystyle \left\|D W\right\|_{G,\rho, c}$ & $=$ & $\displaystyle \max\left(\left\|\frac{\partial W}{\partial \phi}\right\|_{G,\rho,1}, c\left\|\frac{\partial W}{\partial I}\right\|_{G,\rho,\infty}\right)$\\
\\
& $\leq$ & $\displaystyle \max\left(\frac{2}{\alpha}\left\|\frac{\partial R}{\partial \phi}\right\|_{G,\rho,1}, c  \frac{2}{\alpha}\left\|\frac{\partial R}{\partial I}\right\|_{G,\rho_2, \infty} +
c\frac{4M}{\alpha^2}\left\|\frac{\partial R}{\partial \phi}\right\|_{G,\rho_2, 1} \right)$\\
\\
& $\leq$ & $\displaystyle \max\left(\frac{2}{\alpha}\left\|\frac{\partial R}{\partial \phi}\right\|_{G,\rho,1}, \frac{2}{\alpha}\left\|DR\right\|_{G,\rho_2, c} +
c \frac{4M}{\alpha^2}\left\| DR\right\|_{G,\rho_2, c} \right)$\\
\\
& $=$ & $\displaystyle \max\left(\frac{2}{\alpha}\left\|\frac{\partial R}{\partial \phi}\right\|_{G,\rho,1}, \frac{2}{\alpha}\left(1 + \frac{2M}{\alpha}c\right)\left\| DR\right\|_{G,\rho_2, c} \right)$\\
\\
& $\leq$ & $\displaystyle \frac{2}{\alpha}\left(1 + \frac{2M}{\alpha}c\right)\left\| DR\right\|_{G,\rho_2, c}$\\
\\
& $\leq$ & $\displaystyle \frac{2}{\alpha}A\left\| DR\right\|_{G,\rho_2, c}$,\\
\end{longtable}

where $A$ is as desired.

\end{proof}

Recall the Cauchy inequalities, see \cite{JP93}:

\begin{equation}\label{eq:cauchy_ineq}
\begin{array}{llr}
\displaystyle \left\|\frac{\partial f}{\partial \phi}\right\|_{G,(\rho_1, \rho_2),1} & \leq & \displaystyle \frac{1}{e \delta_1}\|f\|_{G,\rho} \\
\displaystyle \left\|\frac{\partial f}{\partial I}\right\|_{G,(\rho_1, \rho_2-\delta_2),\infty}& \leq & \displaystyle \frac{1}{\delta_2}\|f\|_{G,\rho}
\end{array}
\end{equation}

\begin{lemma}\label{lemma:1.1}
Let $f,g$ be analytic functions on $\mathcal{D}_\rho(G)$, where
$0<\delta = (\delta_1, \delta_2) < \rho = (\rho_1, \rho_2)$ and $c > 0$.
Define $\hat\delta_c := \min(c\delta_1, \delta_2)$. The following inequalities hold:
\begin{enumerate}
\item $\|Df\|_{G,\rho- \delta,c} \leq \frac{c}{\hat\delta_c}\|f\|_{G,\rho}$
\item $\|\{f,g\}\|_{G,\rho}\leq \frac{2}{c}\|Df\|_{G,\rho, c}\cdot \|D g\|_{G, \rho,c}$
\item $\|D(f_{>K})\|_{G, (\rho-\delta_1, \rho_2),c} \leq e^{-K\delta_1}\|Df\|_{G,\rho, c}$
\end{enumerate}
\end{lemma}

\begin{proof}
Let us prove each point separately.
\begin{enumerate}
\item %$$\|Df\|_{G,\rho-\delta, c} = \max \left\{\left\|\frac{\partial f}{\partial \phi}\right\|_{G,\rho-\delta, 1}, c\left\|\frac{\partial f}{\partial I}\right\|_{G,\rho-\delta,\infty}\right\}$$

Using the Cauchy inequalities one obtains the following:
$$\left\|\frac{\partial f}{\partial \phi}\right\|_{G,\rho-\delta,1} = \left\|\frac{\partial f}{\partial \phi}\right\|_{G,(\rho_1-\delta_1, \rho_2- \delta_2), 1} $$
$$\leq \left\|\frac{\partial f}{\partial \phi}\right\|_{G,(\rho_1-\delta_1, \rho_2), 1} \leq \frac{1}{e\delta_1} \|f\|_{G,\rho},$$
$$\left\|\frac{\partial f}{\partial I}\right\|_{G,\rho-\delta,\infty} = \left\|\frac{\partial f}{\partial I}\right\|_{G,(\rho_1-\delta_1, \rho_2- \delta_2), \infty} $$
$$\leq \left\|\frac{\partial f}{\partial I}\right\|_{G,(\rho_1, \rho_2-\delta_2), \infty} \leq \frac{1}{\delta_1} \|f\|_{G,\rho}.$$
Putting the two inequalities inside the definition of the norm:
$$
\begin{array}{rcl}
\|D f\|_{G,\rho-\delta, c} &=& \displaystyle \max\left\{\left\|\frac{\partial f}{\partial \phi}\right\|_{G,\rho-\delta, 1}, c\left\|\frac{\partial f}{\partial I}\right\|_{G, \rho-\delta, \infty}\right\} \\
\\
& \leq & \displaystyle \max \left\{\frac{1}{e \delta_1}\|f\|_{G,\rho}, \frac{c}{\delta_2}\|f\|_{G, \rho}\right\}\\
\\
& \leq & \displaystyle \max \left\{\frac{1}{e\delta_1}\frac{c}{c}, \frac{c}{\delta_2}\right\}\|f\|_{G,\rho}\\
\\
& \leq & \displaystyle \max \left\{\frac{c}{e \hat\delta_c},\frac{c}{\hat\delta_c}\right\}\|f\|_{G,\rho},\\
\end{array}
$$

where the last inequality holds because $\hat\delta_c = \min(c\delta_1,\delta_2)$.

\item
Let us find the expression of $\{f,g\}$ for a $b^m$-symplectic structure.
$\{f,g\} = \omega(X_f, X_g)$ where $X_f$ and $X_g$ are such that $\iota_{X_f}\omega=df$ and $\iota_{X_g}\omega=dg$. Let restrict the computations only to $f$.
$$df = \sum_{i=1}^n\frac{\partial f}{\partial \phi_1}d\phi_1, \quad X_f = \sum_{i= 1}^n a_i \frac{\partial}{\partial \phi_i} + \sum_{i=1}^n b_i \frac{\partial}{\partial \phi_i}.$$

Where $a_i$ and $b_i$ are coefficients to be determined by imposing the following condition:

$$\iota_{X_f}\omega = \left(\sum_{j=1}^m\frac{c_j}{ I_1^j}\right)(a_1 dI_1 - b_1 d\phi_1) + \sum_{i = 2}^n (a_i dI_i - b_i d\phi_i) = df.$$

Then, solving for the coefficients the following expressions are obtained:

$$a_1  = \frac{1}{\left(\sum_{j=1}^m\frac{c_j}{ I_1^j}\right)}\frac{\partial f}{\partial \phi_1} \quad \text{and} \quad a_i = \frac{\partial f}{\partial \phi_i} \text{ for } i \neq 1,$$
$$b_1  = -\frac{1}{\left(\sum_{j=1}^m\frac{c_j}{ I_1^j}\right)}\frac{\partial f}{\partial \phi_1} \quad \text{and} \quad b_i = -\frac{\partial f}{\partial \phi_i} \text{ for } i \neq 1.$$

Hence, the expression for the hamiltonian vector fields becomes:

$$
X_f = \frac{1}{\left(\sum_{j=1}^m\frac{c_j}{ I_1^j}\right)}\left(\frac{\partial f}{\partial \phi_1}\frac{\partial }{\partial \phi_1}- \frac{\partial f}{\partial I_1}\frac{\partial }{\partial I_1}\right) + \sum_{i = 1}^n\left(\frac{\partial f}{\partial \phi_i}\frac{\partial }{\partial \phi_i}- \frac{\partial f}{\partial I_i}\frac{\partial }{\partial I_i}\right),$$
$$
X_g =  \frac{1}{\left(\sum_{j=1}^m\frac{c_j}{ I_1^j}\right)}\left(\frac{\partial g}{\partial \phi_1}\frac{\partial }{\partial \phi_1}- \frac{\partial g}{\partial I_1}\frac{\partial }{\partial I_1}\right) + \sum_{i = 1}^n\left(\frac{\partial g}{\partial \phi_i}\frac{\partial }{\partial \phi_i}- \frac{\partial g}{\partial I_i}\frac{\partial }{\partial I_i}\right).
$$

Then the Poisson bracket applied to the two functions:

$$
\begin{array}{rcl}
\{f,g\}=\omega(X_f,X_g) & = & \displaystyle \frac{1}{\left(\sum_{j=1}^m\frac{c_j}{ I_1^j}\right)}\left(\frac{\partial f}{\partial I_1}\frac{\partial g}{\partial \phi_1} - \frac{\partial f}{\partial \phi_1}\frac{\partial g}{\partial I_1}\right)
\\
& & \displaystyle \quad + \sum_{i = 2}^n \left(\frac{\partial f}{\partial I_i}\frac{\partial g}{\partial \phi_i} - \frac{\partial f}{\partial \phi_i}\frac{\partial g}{\partial I_i}\right).
\end{array}
$$

And hence the norm of the Poisson bracket becomes:

$$
\begin{array}{rcl}
\|\{f,g\}\|_{G, \rho} & = & \displaystyle \left\|\frac{1}{\left(\sum_{j=1}^m\frac{c_j}{ I_1^j}\right)}\left(\frac{\partial f}{\partial I_1}\frac{\partial g}{\partial \phi_1} - \frac{\partial f}{\partial \phi_1}\frac{\partial g}{\partial I_1}\right)\right. \\
& & \quad \displaystyle+ \left.\sum_{i = 2}^n \left(\frac{\partial f}{\partial I_i}\frac{\partial g}{\partial \phi_i} - \frac{\partial f}{\partial \phi_i}\frac{\partial g}{\partial I_i}\right) \right\|_{G, \rho} \\
& \leq & \displaystyle \left\|\sum_{i = 1}^n \left(\frac{\partial f}{\partial I_i}\frac{\partial g}{\partial \phi_i} - \frac{\partial f}{\partial \phi_i}\frac{\partial g}{\partial I_i}\right) \right\|_{G, \rho} \\
\end{array}
$$

Where we assumed $\left|\sum_{j=1}^m\frac{c_j}{ I_1^j}\right| \geq 1$. This assumption makes sense, because we are interested in the behaviour close the critical set $Z$. Close enough to the critical set this expression holds. Then,

\begin{longtable}{rcl}
$\|\{f,g\}\|_{G, \rho}$ & $\leq$ & $\displaystyle \sum_{i = 1}^n \left\|\frac{\partial f}{\partial I_i}\right\|_{G, \rho}\left\|\frac{\partial g}{\partial \phi_i}\right\|_{G, \rho} +\sum_{i = 1}^n \left\|\frac{\partial f}{\partial \phi_i}\right\|_{G, \rho}\left\|\frac{\partial g}{\partial I_i}\right\|_{G, \rho}$\\
\\
& $\leq$ & $\displaystyle \left|\frac{\partial f}{\partial I}\right|_{G,\rho,\infty} \left|\frac{\partial g}{\partial I}\right|_{G,\rho,1} + \left|\frac{\partial f}{\partial I}\right|_{G,\rho,1} \left|\frac{\partial g}{\partial I}\right|_{G,\rho,\infty}$\\
\\
& $\leq$ & $\displaystyle \frac{1}{c}|Df\|_{G,\rho,c}\|Dg\|_{G,\rho,c} + \frac{1}{c}|Df\|_{G,\rho,c}\|Dg\|_{G,\rho,c}$\\
\\
& $\leq$ & $\displaystyle \frac{2}{c}\|Df\|_{G,\rho,c}\|Dg\|_{G,\rho,c}.$
\end{longtable}

\item Lastly,

$$
\left\|D(f_{>K})\right\|_{G,(\rho_1 - \delta_1,\rho_2),1} $$
$$= \max \left\{ \left\|\frac{\partial f_{>K}}{\partial \phi}\right\|_{G, (\rho_1-\delta_1, \rho_1), 1}, c\left\|\frac{\partial f_{>K}}{\partial I}\right\|_{G, (\rho_1-\delta_1, \rho_1), \infty}\right\}.
$$

We will proceed by bounding each term separately. On one hand:

\begin{longtable}{rcl}
$\displaystyle\left\|\frac{\partial f}{\partial \phi}\right\|_{G,(\rho_1,\rho_2),1}$ & $=$ & $\displaystyle \left\|\sum_{k\in\mathbb{Z}^n} ikf_k(I)e^{ik\phi}\right\|_{G,(\rho_1,\rho_2),1}$\\
\\
& $\geq$ & $\displaystyle \sum_{k\in\mathbb{Z}^n} k\left\|f_k(I)\right\|_{G,\rho_2,1}e^{|k|_1\rho_1}$\\
\\
& $\geq$ & $\displaystyle \sum_{\substack{k\in\mathbb{Z}^n \\ |k|_1 > K}} k\left\|f_k(I)\right\|_{G,\rho_2,1}e^{|k|_1(\rho_1 + \delta_1 - \delta_1)}$\\
\\
& $\geq$ & $\displaystyle e^{K \delta_1}\sum_{\substack{k\in\mathbb{Z}^n \\ |k|_1 > K}} k\left\|f_k(I)\right\|_{G,\rho_2,1}e^{|k|_1(\rho_1 - \delta_1)}$\\
\\
& $=$ & $\displaystyle e^{K \delta_1}\left\|\frac{\partial f_{>K}}{\partial \phi}\right\|_{G,(\rho_1-\delta_1,\rho_2),1}.$\\
\end{longtable}

On the other hand:

\begin{longtable}{rcl}
$\displaystyle \left\|\frac{\partial f}{\partial I}\right\|_{G,(\rho_1,\rho_2),\infty}$& $=$ & $\displaystyle \left\|\sum_{k \in \mathbb{Z}^n} \frac{\partial f_k(I)}{\partial I} e^{ik\phi}\right\|_{G,(\rho_1,\rho_2),\infty}$\\
\\
& $\geq$ & $\displaystyle \sum_{k \in \mathbb{Z}^n} \left\|\frac{\partial f_k(I)}{\partial I}\right\|_{G,\rho_2,\infty} e^{|k|_1\rho_1}$ \\
\\
& $\geq$ & $\displaystyle \sum_{\substack{k\in\mathbb{Z}^n \\ |k|_1 > K}} \left\|\frac{\partial f_k(I)}{\partial I}\right\|_{G,\rho_2,\infty} e^{|k|_1(\rho_1 + \delta_1 - \delta_1)}$ \\
\\
& $\geq$ & $\displaystyle e^{K\delta_1}\sum_{\substack{k\in\mathbb{Z}^n \\ |k|_1 > K}} \left\|\frac{\partial f_k(I)}{\partial I}\right\|_{G,\rho_2,\infty} e^{|k|_1(\rho_1 - \delta_1)}$\\
\\
& $\geq$ & $\displaystyle e^{K\delta_1}\left\|\frac{\partial f_{>K}}{\partial I}\right\|_{G,(\rho_1-\delta_1,\rho_2),\infty}.$\\
\end{longtable}

Hence $\|D(f_{>k})\|_{G,(\rho_1-\delta_1,\rho_2),c} \leq e^{-K\delta_1}\|Df\|_{G,\rho,c}$.

\end{enumerate}
\end{proof}

Now we define a norm that indicates how close a map $\Phi$ is to the identity.

\begin{definition}
Let $x = (\phi,I) \in \mathbb{C}^{2n}$, then
$$|x|_c := \max(|\phi|_1, c|I|_\infty)$$
\end{definition}

\begin{definition}
For a map $\Upsilon :\mathcal{D}_\rho(G)\rightarrow \mathbb{C}^{2n}$ its norm and the norm of its derivative its defined as:
$$ |\Upsilon|_{G,\rho,c}:=\sup_{x\in\mathcal{D}_\rho(G)}|\Upsilon(x)|_c,$$
$$ |D\Upsilon|_{G,\rho,c}:=\sup_{x\in\mathcal{D}_\rho(G)}|D\Upsilon(x)|_c,$$

where $\displaystyle |D\Upsilon(x)|_c = \sup_{\substack{y\in\mathbb{R}^{2n} \\ |y|_c = 1}}|D\Upsilon(x)\cdot y|_c$
\end{definition}

\begin{lemma}
If $\Upsilon$ is analytic on $\mathcal{D}_\rho(G)$, then $|D\Upsilon|_{G,\rho-\delta,C} \leq \frac{|\Upsilon|_{G,\rho,c}}{\hat\delta_c}$
\end{lemma}

\begin{proof}
Observe that if we consider $\|.\|$ any norm on $\mathbb{C}^n$ and a matrix $A$ of size $n\times n$, and $\|A\|$ defines the induced norm of matrices  i.e.
$$\|A\| = \sup_{\substack{y\in\mathbb{C}^{2n} \\ \|y\| = 1}}\|A\cdot y\|$$
then one has that $\|(\|a_1\|',\ldots,\|a_n\|')\| \leq \|A\|$ where $a_j$ denotes the $j$-th row of $A$. Also note that $\|\cdot\|'$ can be a any norm consider the infinity norm. This can be easily proven in the following way:
$$\|A\cdot y\| = \left\|\left(\begin{array}{c}a_1\cdot y \\ \vdots \\ a_n\cdot y\end{array}\right)\right\| \leq \left\|\left(\begin{array}{c}
\|a_1\|'\|y\|' \\ \vdots \\ \|a_n\|'\|y\|'
\end{array}\right)\right\|$$
Where $\forall y \in \mathbb{C}^n \text{ such that } \|y\| = 1$.
Let $a_j$ be the rows of $D\Upsilon(x)$,
$$ a_j = \left(\frac{\partial \Upsilon_j}{\partial \phi},\frac{\partial \Upsilon j}{\partial I}\right),$$
and be $\|a_j\|'$ its norm.
With this property in mind we proceed as follows:
\begin{longtable}{rcl}
$|D\Upsilon|_{G,\rho-\delta,c}$ & $=$ & $\displaystyle\sup_{x\in \mathcal{D}_{\rho-\delta}(G)} |D\Upsilon(x)|_c$ \\
\\
& $\leq$ & $\displaystyle\sup_{x\in \mathcal{D}_{\rho-\delta}(G)} |(|a_1|_\infty,\ldots,|a_n|_\infty)|_c$\\
%\\
%& $\leq$ & $ \left|\left(\max\left(\sup_{x\in\mathcal{D}_{\rho-\delta}}\left|\frac{\partial \Upsilon_1}{\partial \phi}\right|_1, c\sup_{x\in\mathcal{D}_{\rho-\delta}}\left|\frac{\partial \Upsilon_1}{\partial I}\right|_\infty\right),\ldots\right.\right.$\\
%\\
%& & $\qquad\left.\left.\ldots,\max\left(\sup_{x\in\mathcal{D}_{\rho-\delta}}\left|\frac{\partial \Upsilon_{2n}}{\partial \phi}\right|_1, c\sup_{x\in\mathcal{D}_{\rho-\delta}}\left|\frac{\partial \Upsilon_{2n}}{\partial I}\right|_\infty\right)\right)\right|_c$\\
%\\
%& $=$ & $ \left|\left(\max\left(\left\|\frac{\partial \Upsilon_1}{\partial \phi}\right\|_{G,\rho-\delta,1}, c\left\|\frac{\partial \Upsilon_1}{\partial I}\right\|_{G,\rho-\delta,\infty}\right),\ldots\right.\right.$\\
%\\
%& & $\qquad\left.\left.\ldots,\max\left(\left\|\frac{\partial \Upsilon_{2n}}{\partial \phi}\right\|_{G,\rho-\delta,1}, c\left\|\frac{\partial \Upsilon_{2n}}{\partial I}\right\|_{G,\rho-\delta,\infty}\right)\right)\right|_c$\\

\\
& $\leq$ & $ \left|\left(\sup_{x\in \mathcal{D}_{\rho-\delta}}\left\|D \Upsilon_1\right\|_{\infty},\ldots,\sup_{x\in \mathcal{D}_{\rho-\delta}} \left\|D \Upsilon_{2n}\right\|_{\infty}\right)\right|_c$\\
\\
& $=$ & $ \left|\left(\left\|D \Upsilon_1\right\|_{G,\rho-\delta,\infty},\ldots,\left\|D \Upsilon_{2n}\right\|_{G,\rho-\delta,\infty}\right)\right|_c$\\
\\
& $\leq$ & $ \left|\left(\frac{1}{\delta_1}\left\| \Upsilon_1\right\|_{G,\rho},\ldots,\frac{1}{\delta_1}\left\| \Upsilon_{2n}\right\|_{G,\rho}\right)\right|_c$\\
%\\
%& $\leq$ & $ \left|\left(\max\left(\frac{c}{e\hat\delta_c}\left\| \Upsilon_1\right\|_{G,\rho}, \frac{c}{\hat\delta_c}\left\| \Upsilon_1\right\|_{G,\rho}\right),\ldots\right.\right.$\\
%\\
%& & $\qquad\left.\left.\ldots,\max\left(\frac{c}{e\hat\delta_c}\left\| \Upsilon_{2n}\right\|_{G,\rho}, \frac{c}{\hat\delta_c}\left\| \Upsilon_{2n}\right\|_{G,\rho}\right)\right)\right|_c$\\
%\\
%& $\leq$ & $\left|\frac{c}{\hat\delta_c}\left\|\Upsilon_1\right\|_{G,\rho},\ldots,\frac{c}{\hat\delta_c}\left\|\Upsilon_{2n}\right\|_{G,\rho}\right|_c$\\
\\
& $\leq$ & $\frac{1}{\hat\delta_c}\left|\left\|\Upsilon_1\right\|_{G,\rho},\ldots,\left\|\Upsilon_{2n}\right\|_{G,\rho}\right|_c$\\
\\
& $=$ & $\frac{1}{\hat\delta_c}\sup_{x\in\mathcal{D}_\rho(G)}\left|\Upsilon_1,\ldots,\Upsilon_{2n}\right|_c = \frac{1}{\hat\delta_c}\sup_{x\in\mathcal{D}_\rho(G)}\left|\Upsilon\right|_c$\\
\\
& $=$ & $ \frac{1}{\hat\delta_c}\left|\Upsilon\right|_{G,\rho,c}$\\

\end{longtable}

\end{proof}

\begin{lemma}\label{lemma:1.2} Let $W$ be an analytic function on $\mathcal{D}_\rho(G)$, $\rho > 0$ and let $\Phi_t$ be its Hamiltonian flow at time $t$ ($t>0$). Let $\delta=(\delta_1,\delta_2)>0$ and $c>0$ given. Assume that $\|DW\|_{G,\rho,c}\leq \hat\delta_c$. Then, $\Phi_t$ maps $\mathcal{D}_{\rho-t\delta}(G)$ into $\mathcal{D}_{\rho}(G)$ and one has:
\begin{enumerate}
\item $|\Phi_t-\id|_{G,\rho-t\delta,c}\leq t \|DW\|_{G,\rho,c}$,
\item $\Phi(\mathcal{D}_{\rho}(G)) \supset \mathcal{D}_{\rho-t\delta}(G)$ for $\rho'\leq \rho-t\delta$,
\item Assuming that $\|DW\|_{G,\rho,c} < \hat\delta_c/2e$, for any given function $f$ analytic on $\mathcal{D}_\rho(G)$, and for any integer $m\geq 0$, the following bound holds:
$$
\begin{array}{lcl}
\|r_m(f,W,t)\|_{G,\rho-t\delta} \\
 \leq  \displaystyle \sum_{l= 0}^\infty \left[\frac{1}{\binom{l+m}{m}}\cdot \left(\frac{2e\|DW\|_{G,\rho,c}}{\hat\delta_c}\right)^l\right]\frac{t^m}{m!}\|L_W^m f\|_{G,\rho}\\
\\
 =  \displaystyle\gamma_m\left(\frac{2e\|DW\|_{G,\rho,c}}{\hat\delta_c}\right)\cdot t^m \|L_W^m f\|_{G,\rho},\\
\end{array}
$$

where for $0\leq x \leq 1$ we define

$$\gamma_m (x) \coloneqq \sum_{l=0}^\infty \frac{l!}{(l+m)!} x^l$$

\end{enumerate}
\end{lemma}
\begin{proof}

During the proof we are going to denote $\Phi_s(\phi_0,I_0)$ by $(\phi(s),I(s))$.

%\begin{figure}
%\begin{center}
%\begin{tikzpicture}[line cap=round,line join=round,>=triangle 45,x=1.0cm,y=1.0cm,scale=0.4]
%\clip(0,-4.5) rectangle (8,4);
%\definecolor{qqqqff}{rgb}{0.,0.,1.}
%\definecolor{ccwwff}{rgb}{0.0,0.1,1.0}
%\definecolor{qqzzff}{rgb}{0.0,0.6,1.0}
%\definecolor{xdxdff}{rgb}{0.0,0.8,1.0}
%\draw [rotate around={47.7:(3.65,-0.2)},line width=1.2pt,color=xdxdff] (3.65,-0.2) ellipse (4.04cm and 2.85cm);
%\draw [rotate around={47.686:(3.65,-0.2)},line width=1.2pt,color=qqzzff] (3.65,-0.2) ellipse (3.5cm and 2.47cm);
%\begin{small}
%\draw [color=ccwwff] (3,-2.1)-- ++(-5pt,-5pt) -- ++(10pt,10pt) ++(-10pt,0) -- ++(10pt,-10pt);
%\draw[color=ccwwff] (3.26,-1.27) node {$(\phi_0,I_0)$};
%\draw[color=xdxdff] (5.74,3.9) node {$\mathcal{D}_{\rho}(G)$};
%\draw[color=qqzzff] (4.6,1.4) node {$\mathcal{D}_{\rho-t\delta}(G)$};
%\end{small}
%\end{tikzpicture}
%\caption{Diagram of the sets of the proof and the starting point of the Hamiltonian flow of $W$.}\label{fig:1}
%\end{center}
%\end{figure}

Let us find the coordinate expression of the hamiltonian flow for the expression \ref{eq:bm-symplectic} of a $b^m$-symplectic form. Recall that the equation for the hamiltonian flow is $\frac{d}{ds}\phi_i(s)=\{\phi_i,W\}$ and $\frac{d}{ds}I_i(s)=\{I_i,W\}$.
$$\{\phi_i,W\} = \frac{1}{\left(\sum_{j=1}^m \frac{c_j}{I_1^j}\right)}\left(\frac{\partial \phi_i}{\partial I_1}\cdot\frac{\partial W}{\partial \phi_1} - \frac{\partial \phi_i}{\partial \phi_1}\cdot\frac{\partial W}{\partial I_1}\right)$$ 
$$+ \sum_{j=2}^n\left(\frac{\partial \phi_i}{\partial I_j}\cdot\frac{\partial W}{\partial \phi_j} - \frac{\partial \phi_i}{\partial \phi_j}\cdot\frac{\partial W}{\partial I_j}\right).$$
Hence,
$$
\displaystyle \frac{d}{ds}\phi_i(s) = -\frac{1}{\left(\sum_{j=1}^m \frac{c_j}{I_1^j}\right)}\frac{ \partial W}{\partial I_1} \text{ if } i = 1 \text{ and }
\displaystyle \frac{d}{ds}\phi_i(s) = -\frac{ \partial W}{\partial I_i} \text{ if } i \neq 1.
$$
On the other side,
$$\{I_i,W\} = \frac{1}{\left(\sum_{j=1}^m \frac{c_j}{I_1^j}\right)}\left(\frac{\partial I_i}{\partial I_1}\cdot\frac{\partial W}{\partial \phi_1} - \frac{\partial I_i}{\partial \phi_1}\cdot\frac{\partial W}{\partial I_1}\right) $$
$$+ \sum_{j=2}^n\left(\frac{\partial I_i}{\partial I_j}\cdot\frac{\partial W}{\partial \phi_j} - \frac{\partial I_i}{\partial \phi_j}\cdot\frac{\partial W}{\partial I_j}\right).$$
Hence,
$$
\frac{d}{ds}I_i(s) = \frac{1}{\left(\sum_{j=1}^m \frac{c_j}{I_1^j}\right)}\frac{ \partial W}{\partial \phi_1} \text{ if } i = 1 \text{ and } \frac{d}{ds}I_i(s) = \frac{ \partial W}{\partial \phi_i} \text{ if } i \neq 1.
$$

\begin{enumerate}
\item

Assume now that $0<s_0 \leq t$. Then,

\begin{longtable}{rcl}
$|\phi(s_0)-\phi_0|_\infty $ & $\leq $ & $s_0\sup_{0<s\leq s_0}|\phi'(s)|_\infty $ \\
& $=$ & $s_0\sup_{0<s\leq s_0}\left(\max(|\phi_1'(s)|,\ldots,|\phi_n'(s)|)\right) $ \\
\\
& $=$ & $\displaystyle s_0\sup_{0<s\leq s_0}\left(\max\left(\left|\frac{1}{\left(\sum_{j=1}^m\frac{c_j}{I_1^j}\right)}\frac{\partial W}{\partial I_1}\right|,\left|\frac{\partial W}{\partial I_2}\right|,\ldots,\left|\frac{\partial W}{\partial I_n}\right|\right)\right)$ \\
\\
& $\leq$ & $s_0\sup_{0<s\leq s_0}\left|\frac{\partial W}{\partial I}\right|_\infty  \leq s_0\left\|\frac{\partial W}{\partial I}\right\|_{G,\rho,\infty}$ \\

\end{longtable}
Where we have used again that on the domain $\mathcal{D}_\rho(G)$ the inequality $\left|\sum_{j=1}^m\frac{c_j}{I_1^j}\right| \geq 1$ holds.
%The last inequality is true if because $s_0\leq t$ then for $0<s\leq s_0$ the flow stays in $\mathcal{D}_\rho(G)$.
Similarly, $|I(s_0) - I_0| \leq s_0\|\frac{\partial W}{\partial \phi}\|_{G,\rho,1}$, and hence $|\Phi_t - \id|_{G,\rho-t\delta,c} \leq t\|D W\|_{G,\rho,c}$.

Because
\begin{equation}\label{eq:bound_flow}
\begin{array}{lr}
|\phi(s)-\phi_0|_\infty \leq t \|\frac{\partial W}{\partial I}\|_{G,\rho,\infty} \leq t \frac{\hat\delta_c}{c}\leq t\delta_1\frac{c}{c} = t\delta_1 & \forall 0 < s \leq s_0\\
|I(s)-I_0|_1 \leq t \|\frac{\partial W}{\partial \phi}\|_{G,\rho,1} \leq t \hat\delta_c\leq t\delta_2 & \forall 0 < s \leq s_0\\
\end{array}
\end{equation}

hence, $(\phi(s),I(s)) \in \mathcal{D}_{\rho-t\delta + t\delta}(G) = \mathcal{D}_\rho(G)$ for all $0 < s \leq s_0$.
\item

Repeat the same argument as in \ref{eq:bound_flow} with $\phi(-s)$. If $(\phi_0,I_0)\in \mathcal{D}_{\rho'-t\delta}$, then $(\phi(-s),I(-s)) \in \mathcal{D}_{\rho'-t\delta+t\delta}(G) = \mathcal{D}_{\rho'}$.
Hence, $$\mathcal{D}_{\rho'}(G) \supset \Phi^{-1}(\mathcal{D}_{\rho'-t\delta}(G)),$$ then $\Phi(\mathcal{D}_{\rho'}(G)) \supset \mathcal{D}_{\rho'-t\delta}(G)$.

\item
Consider $f$ an analytical function.
By the previous construction $f\circ\Phi_t$ is defined in $\mathcal{D}_{\rho -t\delta}(G)$.
Because $W$ is analytic we also have that $f\circ\Phi_t$ is analytic and we can expand its Lie series.
Let $m \in \mathbb{Z}, l \geq m+1, j = m+1,\ldots,l$ then
$$
\begin{array}{rcl}
\|L_W^jf\|_{G,\rho-(j-m)t\eta} & \leq & \frac{2}{c}\|D(L_W^{j-1} f)\|_{G,\rho-(j-m)t\eta,c}\|DW\|_{G,\rho,c}\\
 & \leq & \frac{2}{t\hat\eta_c}\|L_W^{j-1} f\|_{G,\rho-(j-1-m)t\eta}\|D W\|_{G,\rho,c},\\
\end{array}
$$

where we used lemma \ref{lemma:1.1} and defined $\eta = \frac{\delta}{(l-m)}$ and $\hat\eta_c = \min(c\eta_1,\eta_2)$.

Then,

$$\begin{array}{rcl}
\|L_W^l f\|_{G,\rho-t\delta} & \leq & \left(\frac{2\|DW\|_{G,\rho,c}}{t\hat\eta_c}\right)^{l-m}\\
&\leq& e^{l-m}\cdot(l-m)!\left(\frac{2\|DW\|_{G,\rho,c}}{\hat\delta_c}\right)^{l-m}\|L_W^m f\|_{G,\rho},
\end{array}$$

where we used that $\hat\eta_c = \frac{\hat\delta_c}{l-m}$ and $(l-m)^{(l-m)}\leq e^{l-m}\cdot(l-m)!$
And hence, the bound for $\|r_m(f,W,t)\|_{G,\rho-t\delta}$ is

$$\sum_{l=m}^\infty \frac{t^l}{l!}\|L_W^l f\|_{G,\rho-t\delta} \leq \left[\sum_{l= m}^\infty \frac{(l-m)!}{l!}\left(\frac{2e \|DW\|_{G,\rho,c}}{\hat\delta_c}\right)^{l-m}\right]\cdot t^m \|L_W^m f\|_{G,\rho}$$

and this series converges if $\|DW\|_{G,\rho,c} \leq \frac{\hat\delta_c}{2e}$.

\end{enumerate}
\end{proof}

\begin{theorem}\label{lemma:iterative}[Iterative Lemma] $H(\phi,I) = \hat h(I) + R(\phi,I)$ where $\hat h (I)$ is as in equation \ref{eq:bm-hamiltonian} defined on $\mathcal{D}_\rho(G)$. Let $\hat u = \frac{\partial \hat h }{\partial I}$ and $u = \frac{\partial h }{\partial I}$, and assume $u$ is $\alpha,K,c,\hat q$-non-resonant. Assume that $\left|\frac{\partial}{\partial I } u\right|_{G,\rho_2} \leq M'$. Let $\delta < \rho$ and $c > 0$, $A = 1 + \frac{2Mc}{\alpha}$. Assume that $\rho_2 \leq \frac{\alpha}{2MK}$, $\|DR\|_{G,\rho,c} \leq \frac{\alpha \hat \delta_c}{74A}$.
Then, there exists a real analytic map $\Phi:\mathcal{D}_{\rho-\frac{\delta}{2}}(G) \rightarrow \mathcal{D}_\rho(G)$, such that $H\circ \Phi = \hat h + \tilde{R}$,with:

\begin{enumerate}
\item $\|D\tilde{R}\|_{G,\rho-\delta,c} \leq e^{-K\delta_1}\|DR\|_{G,\rho,c} + \frac{14A}{\alpha\hat\delta_c} \|DR\|^2_{G,\rho,c}$,
\item $|\Phi-\id|_{G,\rho-\frac{\delta}{2},c} \leq \frac{2A}{\alpha}\|DR\|_{G,\rho,c}$,
\item $\Phi(\mathcal{D}_{\rho'}(G)) \supset \mathcal{D}_{\rho'-\frac{\delta}{2}}(G)$ for $\rho'\leq \rho-\frac{\delta}{2}$
\end{enumerate}

\end{theorem}

\begin{proof}
Recall that $\left|\frac{\partial}{\partial I } u\right|_{G,\rho_2} \leq M'$  implies $\left|\frac{\partial}{\partial I}( \bar{\mathcal{B}}u + \bar{\mathcal{A}})\right|_{G,\rho_2}\leq M$ by equation \ref{eq:M_def}.
By equation \ref{eq:iterative_h_and_R}

$$R^{(q)} = R^{(q-1)}_{>K} + r_2(\hat h^{(q-1)}, W^{(q)},1) + r_1(R^{(q-1)},W^{(q)},1).$$

To simplify the notation we are going to omit the index of the iteration:

\begin{equation}\label{eq:R_tilde}
\tilde R = R_{>K} + r_2(\hat h^, W,1) + r_1(R,W,1).
\end{equation}

Where $W$ is defined in terms of its Fourier expressions by equation \ref{eq:solve_coefs_simplified}:

$$
W_k(I) = \frac{R_k(I)}{i(k \bar{\mathcal{B}}(I_1) u + k \bar{\mathcal{A}}(I_1))}
$$

By proposition \ref{prop:14}: $\|DW\|_{G,\rho,c}\leq \frac{2A}{\alpha}\|DR\|_{G,\rho,c} \leq \frac{2A}{\alpha}\frac{\alpha\hat\delta_c}{74A} = \frac{\hat\delta_c}{37}$.
And $\Phi$ is defined as in lemma \ref{lemma:1.1}: $\Phi:\mathcal{D}_{\rho-\frac{\delta}{2}}(G) \rightarrow \mathcal{D}_{\rho}(G)$.
\begin{enumerate}
\item

Differentiating equation \ref{eq:R_tilde} we obtain:
$$D \tilde R = D R_{>K} + D r_2(\hat h, W,1) + D r_1(R,W,1).$$
Taking norms at every side of the expression:
\begin{longtable}{rcl}
$\|D\tilde R\|_{G,\rho-\delta,c}$ & $=$ & $\|D R_{>K} + D r_2(\hat h, W,1) + D r_1(R,W,1)\|_{G,\rho-\delta,c}$\\
& $\leq$ & $\|D R_{>K}\|_{G,\rho-\delta,c} + \|D r_2(\hat h, W,1)\|_{G,\rho-\delta,c}$ \\
& & $ + \|D r_1(R,W,1)\|_{G,\rho-\delta,c}$\\
& $\leq$ & $e^{-K\delta_1}\|DR\|_{G,\rho,c} $\\
& & $+ \frac{2c}{\hat\delta_c}\left(\|r_2(\hat h, W,1)\|_{G,\rho-\frac{\delta}{2},c} + \|r_1(R,W,1)\|_{G,\rho-\frac{\delta}{2},c}\right)$
\end{longtable}

Let us further develop the two last terms of the previous expression, by using lemma \ref{lemma:1.2}:

$$
\begin{array}{rcl}
 \|r_2(\hat h, W,1)\|_{G,\rho-\frac{\delta}{2},c} & \leq & \gamma_2\left(\frac{2e\|DW\|_{G,\rho,c}}{\hat\delta_c/2}\right)\|L_W^2 h\|_{G,\rho}\\
 & \leq & \gamma_2\left(\frac{4e\|DW\|_{G,\rho,c}}{\hat\delta_c}\right)\|\{\{h,W\},W\}\|_{G,\rho},\\
\end{array}
$$

$$
\begin{array}{rcl}
 \|r_1(\hat h, W,1)\|_{G,\rho-\frac{\delta}{2},c} & \leq & \gamma_1\left(\frac{2e\|DW\|_{G,\rho,c}}{\hat\delta_c/2}\right)\|L_W^1 R\|_{G,\rho}\\
 & \leq & \gamma_1\left(\frac{4e\|DW\|_{G,\rho,c}}{\hat\delta_c}\right)\|\{R,W\}\|_{G,\rho}.\\
\end{array}
$$

Then, using the second statement of lemma \ref{lemma:1.1} and that $\{W,h\} = R_{\leq K}$:
$$\|\{R,W\}\|_{G,\rho}\leq\frac{2}{c}\|DR\|_{G,\rho,c}\|DW\|_{G,\rho,c}, \text{ and }$$
$$\begin{array}{rcl}
|\{\{h,W\},W\}\|_{G,\rho} & = & \displaystyle\|\{R_{\leq K},W\}\|_{G,\rho} \\
\\
&\leq& \displaystyle\frac{2}{c}\|DR_{\leq K}\|_{G,\rho,c}\|DW\|_{G,\rho,c} \\
\\
&\leq& \displaystyle\frac{2}{c}\|DR\|_{G,\rho,c}\|DW\|_{G,\rho,c}.
\end{array}$$

Moreover, it is easy to see that $\gamma_1(x) = \frac{-\log(1-x)}{x}$ and $\gamma_2(x) = \frac{x + (1-x)\log(1-x)}{x^2}$. Observe that these functions are monotonously increasing in $x$. Recall that $\|DW\|_{G,\rho,c} \leq \frac{2A}{\alpha}\|DR\|_{G,\rho,c}$. Then,

\begin{longtable}{rcl}
$\|r_1(\hat h, W,1)\|_{G,\rho-\frac{\delta}{2},c} $\\
$ \qquad +  \|r_2(\hat h, W,1)\|_{G,\rho-\frac{\delta}{2},c}$ & $\leq$ & $\gamma_1\left(\frac{4e\|DW\|_{G,\rho,c}}{\hat\delta_c}\right)\|\{R,W\}\|_{G,\rho}$\\
& & $+ \gamma_2\left(\frac{4e\|DW\|_{G,\rho,c}}{\hat\delta_c}\right)\|\{\{h,W\},W\}\|_{G,\rho}$\\
& $\leq$ & $\gamma_1\left(\frac{4e\|DW\|_{G,\rho,c}}{\hat\delta_c}\right)\frac{2}{c}\|DR\|_{G,\rho,c}\|DW\|_{G,\rho,c}$\\
& & $+ \gamma_2\left(\frac{4e\|DW\|_{G,\rho,c}}{\hat\delta_c}\right)\frac{2}{c}\|DR\|_{G,\rho,c}\|DW\|_{G,\rho,c}$\\
& $\leq$ & $\gamma_1\left(\frac{4e\|DW\|_{G,\rho,c}}{\hat\delta_c}\right)\frac{2}{c}\frac{2A}{\alpha}\|DR\|^2_{G,\rho,c}$\\
& & $+ \gamma_2\left(\frac{4e\|DW\|_{G,\rho,c}}{\hat\delta_c}\right)\frac{2}{c}\frac{2A}{\alpha}\|DR\|^2_{G,\rho,c}$\\
& $\leq$ & $\frac{2}{c}[\gamma_1(\frac{4e}{37}) + \gamma_2(\frac{4e}{37})]\frac{2A}{\alpha}\|DR\|^2_{G,\rho,c}$\\
& $=$ & $\frac{4A}{\alpha c}[\gamma_1(\frac{4e}{37}) + \gamma_2(\frac{4e}{37})]\|DR\|_{G,\rho,c}^2$.\\
\end{longtable}

Moreover $\gamma_1(\frac{4e}{37}) + \gamma_2(\frac{4e}{37}) \approx 1.741\ldots < \frac{7}{4}$.

Then,

$$\begin{array}{rcl}
\|D\tilde R\|_{G,\rho-\delta,c} & \leq & e^{-K\delta_1}\|DR\|_{G,\rho,c} + \frac{2c}{\hat\delta_c}\frac{4A}{\alpha c}\frac{7}{4}\|\|_{G,\rho,c}^2\\
 & \leq & e^{-K\delta_1}\|DR\|_{G,\rho,c} + \frac{14A}{\hat\delta_c\alpha}\|DR\|^2_{G,\rho,c},\\
\end{array}
$$

as we wanted to prove.

\item Direct from lemma \ref{lemma:1.2}:
$$|\Phi - \id|_{G,\rho.\frac{\delta}{2},c}\leq\|DW\|_{G,\rho,c}\leq\frac{2A}{\alpha}\|DR\|_{G,\rho,c}$$
\item Also direct from lemma \ref{lemma:1.2}:
$$\Phi(\mathcal{D}_\rho(G)) \supset \mathcal{D}_{\rho'-\frac{\delta}{2}}(G), \text{ for } \rho' \leq \rho - \delta/2$$
\end{enumerate}
\end{proof}

\begin{definition} $\Delta_{c,\hat q}(k,\alpha) = \{J \in \mathbb{R}^n \text{ such that } |k \bar{\mathcal{B}}(I_1) J + k \bar{\mathcal{A}}(I_1)| < \alpha\}$
\end{definition}

\begin{lemma}\label{lemma:mesure_resonances}With the previous definitions we have the following bounds.

Outside of $Z$:
$$\text{meas}\left(F \cap \Delta_{c,\hat q}(k,\alpha)\right) \leq (\textnormal{diam} F)^{n-1}\frac{2\alpha}{|k|_{2,\omega}}.$$
At $Z$:
$$\text{meas}\left(F \cap \Delta_{c,\hat q}(k,\alpha)\right)
\left\{
\begin{array}{lcl}
=0 & & \text{if } \alpha \leq \frac{|k_1|}{\mathcal{K}'}\\
\leq (\text{diam} F)^n & & \text{if } \alpha > \frac{|k_1|}{\mathcal{K}'} \\
\end{array}
\right.
$$
\end{lemma}

\begin{proof}
It is important to understand the geometry of the set $\Delta_{c,\hat q}(k,\alpha)$. Recall that $k \bar{\mathcal{B}}(I_1) J = k_1 \mathcal{B}(I_1) J_1 + \bar{k}\bar{J}$, hence this part of the expression can be interpreted as the scalar product of the vector $J$ with the vector $(k_1 \mathcal{B}(I_1),k_2,\ldots,k_n)$. Then the set $\{J \in \mathbb{R}^n \text{ such that } |k \bar{\mathcal{B}}(I_1) J| < \alpha\}$ is the space between two hyperplanes orthogonal to $(k_1 \mathcal{B}(I_1),k_2,\ldots,k_n)$. Adding the term $k \bar{\mathcal{A}}(I_1)$ only applies a transition to the previous set. Let us find what is the separation between the hyperplanes.
Assume $J$ is parallel to $(k_1 \mathcal{B}(I_1),k_2,\ldots,k_n)$ with lengths $a$:
$$J = a\frac{(k_1 \mathcal{B}(I_1),k_2,\ldots,k_n)}{|k|_{2,\omega}},$$
where $|k|_{2,\omega} = \sqrt{B(I_1)^2k_1^2 + k_2^2 + \ldots k_n^2}$.
Then,
$$
\begin{array}{rcl}
 J\cdot (B(I_1),k_1,\ldots,k_n) & = & c(B(I_1)k_1^2 + k_2^2 + \ldots k_n^2)\frac{1}{|k|_{2,\omega}}\\
 & = & a|k|_{2,\omega} \leq \alpha \Leftrightarrow a \leq \frac{\alpha}{|k|_{2,\omega}}.
\end{array}
$$
And finally,
$$\text{meas}\left(F \cap \Delta_{c,\hat q}(k,\alpha)\right) \leq (\textnormal{diam} F)^{n-1}\frac{2\alpha}{|k|_{2,\omega}}.$$
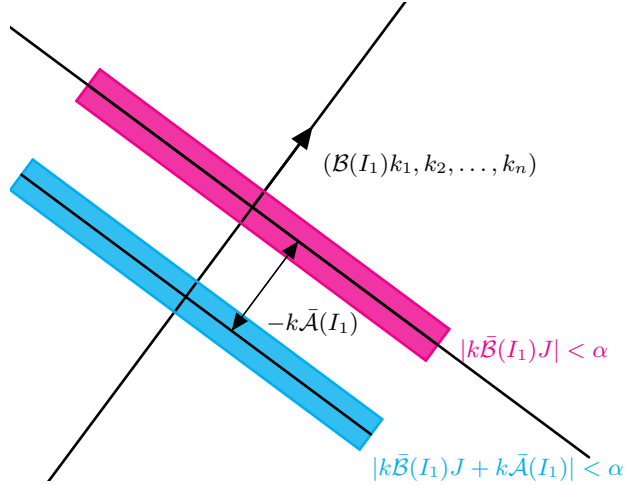
\begin{figure}[h!]
\begin{center}
%\definecolor{magenta}{rgb}{0.,0.2,0.6}
\begin{tikzpicture}[line cap=round,line join=round,>=triangle 45,x=1.0cm,y=1.0cm, scale = 0.6]
\clip(-6.3,-4.3) rectangle (8,6.3);
\fill[line width=1.pt,color=cyan,fill=cyan,fill opacity=0.8] (-6.305315071273844,2.143755827017319) -- (-5.80474032812592,2.817028856551277) -- (1.9374647098374522,-2.939257417027438) -- (1.4368899666895276,-3.612530446561396) -- cycle;
\fill[line width=1.pt,color=magenta,fill=magenta,fill opacity=0.8] (-4.838555792187489,4.116547057388468) -- (-4.318304103434274,4.8162855787615415) -- (3.4239009345290983,-0.9400006948171737) -- (2.9036492457758833,-1.639739216190248) -- cycle;
\draw [line width=1.pt,domain=-6.7128271268301205:6.518034851701254] plot(\x,{(--5.7156-4.*\x)/5.38});
\draw [line width=1.pt,domain=-6.7128271268301205:6.518034851701254] plot(\x,{(--12.088--5.38*\x)/4.});
\draw [->,line width=1.pt] (-0.9382935360133854,1.7599951940619964) -- (0.4,3.56);
\draw [line width=1.pt,color=cyan] (-6.305315071273844,2.143755827017319)-- (-5.80474032812592,2.817028856551277);
\draw [line width=1.pt,color=cyan] (-5.80474032812592,2.817028856551277)-- (1.9374647098374522,-2.939257417027438);
\draw [line width=1.pt,color=cyan] (1.9374647098374522,-2.939257417027438)-- (1.4368899666895276,-3.612530446561396);
\draw [line width=1.pt,color=cyan] (1.4368899666895276,-3.612530446561396)-- (-6.305315071273844,2.143755827017319);
\draw [line width=1.pt,color=magenta] (-4.838555792187489,4.116547057388468)-- (-4.318304103434274,4.8162855787615415);
\draw [line width=1.pt,color=magenta] (-4.318304103434274,4.8162855787615415)-- (3.4239009345290983,-0.9400006948171737);
\draw [line width=1.pt,color=magenta] (3.4239009345290983,-0.9400006948171737)-- (2.9036492457758833,-1.639739216190248);
\draw [line width=1.pt,color=magenta] (2.9036492457758833,-1.639739216190248)-- (-4.838555792187489,4.116547057388468);
\draw [line width=1.pt] (-6.055027699699883,2.4803923417842997)-- (1.687177338263491,-3.275893931794417);
\draw [line width=1.pt] (-4.578429947810884,4.466416318075007)-- (3.163775090152491,-1.2898699555037108);
\draw [->,line width=.5pt] (-1.3994012200506503,-0.9810399850924533) -- (0.07719653183835051,1.004983991198252);
\draw [->,line width=.5pt] (0.07719653183835062,1.0049839911982523) -- (-1.3994012200506503,-0.9810399850924532);
\begin{small}
\draw[color=black] (3,2.7) node {$(\mathcal{B}(I_1)k_1,k_2,\ldots,k_n)$};
\draw[color=cyan] (4.5,-4) node {$|k \bar{\mathcal{B}}(I_1) J + k \bar{\mathcal{A}}(I_1)| < \alpha$};
\draw[color=magenta] (5.3,-1.4) node {$|k \bar{\mathcal{B}}(I_1) J| < \alpha$};
\draw[color=black] (0.4,-0.8) node {$-k\bar{\mathcal{A}}(I_1)$};
\end{small}
\end{tikzpicture}
\end{center}
\caption{Graphical representation of the set $\Delta_{c,\hat q}(\alpha)$}

\end{figure}
The previous formula can not be applied if when we are at $Z$ and $k = (k_1,0,\ldots,0)$. At $Z$, 
$$\Delta_{c,\hat q}(K,\alpha) = \{J \in \mathbb{R}^n \text{ such that } |\bar K \bar J + k_1\frac{\hat q_m}{c_m}| < \alpha\}.$$
And if $k=(k_1,0,\ldots,0)$ then 
$$\Delta_{c,\hat q}(K,\alpha) = \{J \in \mathbb{R}^n \text{ such that } |k_1 \frac{\hat q_m}{c_m}| < \alpha\}.$$
Then

$$
\Delta_{c,\hat q}(k,\alpha) =
\left\{
\begin{array}{rcl}
\mathbb{R}^n & & \text{ if } |k_1| < \alpha\frac{c_m}{\hat q_m} = \alpha \mathcal{K}',\\
\{\emptyset\} & & \text{ if } |k_1| \geq \alpha\frac{c_m}{\hat q_m} = \alpha \mathcal{K}'.
\end{array}
\right.
$$

Using this last identity, the statement we wanted to prove is immediate.

\end{proof}

\begin{definition}
$G-b := \{I\in G \text{ such that } \mathcal{U}_b(I)\subset G\}$, where $\mathcal{U}_b(I)$ is the ball of radius $b$ centered at $I$.
\end{definition}
\begin{definition}
$F$ is a $D$-set if $\textnormal{meas}[(F-b_1)\setminus (F-b_2)] \leq D(b_2 - b_1)$.
\end{definition}

\begin{lemma}\label{lemma:measures_nonresonant}
Let $F \subset \mathbb{R}^n$ be a $D$-set for $d \geq 0$, $\tau > 0$, $\beta \geq 0$ and $k \geq 0$ an integer. Consider the set
$$F(d,\beta,K) := (F-d)\setminus \bigcup_{\substack{k\in\mathbb{Z}^n\setminus\{0\} \\ |k|_1 \leq K}} \Delta_{c,\hat q}\left(k,\frac{\beta}{|k|_1^\tau}\right).$$
Then, outside of $Z$:
\begin{enumerate}
\item\label{eq:meas_nonr_1} If $d'\geq d$,  $\beta'\geq \beta$, $k' \geq k$, then
$$
\textnormal{meas}[F(d,\beta,k)\setminus F(d',\beta',k')] \leq
 $$
$$ D(d' - d) + 2(\textnormal{diam}F)^{n-1}
\left(\sum_{\substack{k\in\mathbb{Z}^n\setminus\{0\}
\\ |k|_1 \leq K}}\frac{\beta' - \beta}{|k|_1^\tau|k|_{2,\omega}}
 + \sum_{\substack{k\in\mathbb{Z}^n\setminus\{0\} \\ 0 < |k|_1 \leq K}}
\frac{\beta'}{|k|_1^\tau|k|_{2,\omega}}\right)
$$
\item\label{eq:meas_nonr_2} For every $b \geq 0$
$$\textnormal{meas}[F(d,\beta,K)\setminus(F(d,\beta,K)-b)]\leq (D + 2^{n+1}(\dim F)^{n-1}K^n)b$$
\end{enumerate}
And inside of $Z$, if we assume $\beta \leq \frac{1}{\mathcal{K}'}$, the equation \ref{eq:meas_nonr_1} holds adding only the terms $\bar k \neq 0$ and \ref{eq:meas_nonr_2} holds without any change.
\end{lemma}

\begin{proof}
Recall that
$$\Delta_{c,\hat q}\left(k,\frac{\beta}{|k|_1^\tau}\right) = \left\{J\in\mathbb{R}^n \text{ such that } \left|k\bar{\mathcal{B}}(I_1)J + k\bar{\mathcal{A}}(I_1)\right| < \frac{\beta}{|k|_1^\tau}\right\}.$$

First we will prove the results outside of $Z$ and then
\begin{enumerate}
\item Let us expand the expression of $\textnormal{meas}[F(d,\beta,k)\setminus F(d',\beta',k')]$:
$$\left[(F-d)\setminus \bigcup_{\substack{k\in\mathbb{Z}^n\setminus\{0\} \\ |k|_1 \leq K}} \Delta_{c,\hat q}\left(k,\frac{\beta}{|k|_1^\tau}\right)\right]\setminus\left[(F-d)\setminus \bigcup_{\substack{k\in\mathbb{Z}^n\setminus\{0\} \\ |k|_1 \leq K}} \Delta_{c,\hat q}\left(k,\frac{\beta}{|k|_1^\tau}\right)\right].$$
Now we use the following property on the previous expression:
$$\begin{array}{rcl}
(A\setminus B)\setminus(C\setminus D) &  = &  [(A\setminus B)\setminus C]\cup [(A\setminus B)\cap D]\\
 & \subset & (A\setminus C)\cup[(A\setminus B)\cap D] \\
 & = & (A\setminus C)\cup(A\cap(D\setminus B)),
\end{array}
$$
where the last equality holds true because $D \supset B$. Using this property we have that $\textnormal{meas}[F(d,\beta,k)\setminus F(d',\beta',k')]$  is included in

$$
\begin{array}{rcl}

[(F-d)\setminus(F-d')]&\cup&\displaystyle\left[(F-d)\cap\left[\left(\bigcup_{\substack{k\in\mathbb{Z}^n\setminus\{0\} \\ |k|_1 \leq K'}} \Delta_{c,\hat q}\left(k,\frac{\beta'}{|k|_1}\right)\right)\right.\right.\\
\\
& & \displaystyle \qquad \qquad \qquad
\setminus\left.\left.\left(\bigcup_{\substack{k\in\mathbb{Z}^n\setminus\{0\} \\ |k|_1 \leq K}} \Delta_{c,\hat q}\left(k,\frac{\beta}{|k|_1}\right) \right)\right]\right].

\end{array}
$$

And this expression is equivalent to:

$$
\begin{array}{rcl}
[(F-d)\setminus (F-d')] & \cup &  \displaystyle \bigcup_{\substack{k\in\mathbb{Z}^n\setminus\{0\} \\ |k|_1 \leq K}} \left((F-d)\cap\left(\Delta_{c,\hat q}\left(k,\frac{\beta'}{|k|_1^\tau}\right)
\right.\right.
\\
& & \displaystyle \qquad \qquad \qquad \qquad \qquad \left.\left.
\setminus\Delta_{c,\hat q}\left(k,\frac{\beta}{|k|_1^\tau}\right)\right)\right)\\
\\
 & \cup & \displaystyle \bigcup_{\substack{k\in\mathbb{Z}^n\setminus\{0\} \\ K < |k|_1 \leq K'}} \left((F-d)\cap\Delta_{c,\hat q}\left(k,\frac{\beta'}{|k|_1^\tau}\right)\right).
\end{array}
$$

Now, using lemma \ref{lemma:mesure_resonances} we obtain:

$$
\textnormal{meas}(F(d,\beta,K)\setminus F(d',\beta',K')) \leq
$$
$$
\leq D(d'-d) + (\textnormal{diam}F)^{n-1}\left(\sum_{\substack{k\in\mathbb{Z}^n\setminus\{0\} \\ |k|_1 \leq K}}\frac{2(\beta'-\beta)}{|k|_1^\tau|k|_{2,\omega}}+ \sum_{\substack{k\in\mathbb{Z}^n\setminus\{0\} \\ K < |k|_1 \leq K'}}\frac{2\beta'}{|k|_1^\tau|k|_{2,\omega}}\right)
$$

\item
Observe that:

\begin{longtable}{rcl}
 $F(d,\beta,K) - b$ & $=$& $\displaystyle\left[(F-d) \setminus \bigcup_{\substack{k\in\mathbb{Z}^n\setminus\{0\} \\ |k|_1 \leq K}} \Delta_{c,\hat q} \left(k,\frac{\beta}{|k|_1^\tau}\right)\right] - b$\\
 & $\supset$ & $\displaystyle (F - (d+b))\setminus \bigcup_{\substack{k\in\mathbb{Z}^n\setminus\{0\} \\ |k|_1 \leq K}} \Delta_{c,\hat q} \left(k,\frac{\beta}{|k|_1^\tau} + b|k|_{2,\omega}\right).$ \\
\end{longtable}

Then,

\begin{longtable}{rcl}
& & $\textnormal{meas}[(F(d,\beta,K))\setminus(F(d,\beta,K)-b)]$ \\
& $\leq$ & $ \textnormal{meas}
\left[\left((F-d)\setminus\bigcup_{\substack{k\in\mathbb{Z}^n\setminus\{0\} \\ |k|_1 \leq K}}\Delta_{c,\hat q} \left(k,\frac{\beta}{|k|_1^\tau}\right)\right)\setminus
\right.$\\
& & \qquad $ \left.
\left((F-(d+b))\setminus\bigcup_{\substack{k\in\mathbb{Z}^n\setminus\{0\} \\ |k|_1 \leq K}}\Delta_{c,\hat q} \left(k,\frac{\beta}{|k|_1^\tau}\right)\right)\right]$ \\
& $\leq$ & $ \textnormal{meas}\left[(F-d)\setminus(F-(d+b))\cup \right.
$\\
& & \qquad $ \left.
\bigcup_{\substack{k\in\mathbb{Z}^n\setminus\{0\} \\ |k|_1 \leq K}}\left((F-d)\cap\left(\Delta_{c,\hat q} \left(k,\frac{\beta}{|k|_1^\tau}+ b|k|_{2,\omega}\right)\right)\right)\setminus\left(\Delta_{c,\hat q} \left(k,\frac{\beta}{|k|_1^\tau}\right)\right)\right]$  \\
& $\leq$ & $ Db + \sum_{\substack{k\in\mathbb{Z}^n\setminus\{0\} \\ |k|_1 \leq K}} (\textnormal{diam} F)^{n-1}\frac{2b|k|_{2,\omega}}{|k|_{2,\omega}}$  \\
& $\leq$ & $ Db + 2^n K^n(\textnormal{diam} F)^{n-1}\cdot 2 = Db + 2^{n+1}K^n(\textnormal{diam} F)^{n-1}$,  \\
\end{longtable}

where in the last inequality we used that the number of vectors $k$ such that $|k|_1\leq K$ is less or equal than $2^n K^n$.

\end{enumerate}

The previous identities worked outside of $Z$. Let us understand the set $F(d,\beta,K)$ when we are ate $Z$.
\begin{longtable}{rcl}
$F(d,\beta,K)$ & $:=$ & $(F-d)\setminus \bigcup_{\substack{k\in\mathbb{Z}^n\setminus\{0\} \\ |k|_1 \leq K}} \Delta_{c\hat q}(k,\frac{\beta}{|k|_1^\tau})$ \\
 & $=$ & $(F-d)\setminus \left[\left(\bigcup_{\substack{k\in\mathbb{Z}^n\setminus\{0\} \\ |k|_1 \leq K \\ \bar k \neq 0}} \Delta_{c\hat q}(k,\frac{\beta}{|k|_1^\tau})\right) \right. $\\
 & & $\quad \cup \left. \left(\bigcup_{\substack{k\in\mathbb{Z}^n\setminus\{0\} \\ |k|_1 \leq K \\ \bar k = 0}} \Delta_{c\hat q}(k,\frac{\beta}{|k|_1^\tau})\right)\right]$ \\
 & $=$ & $(F-d)\setminus \left[\left(\bigcup_{\substack{k\in\mathbb{Z}^n\setminus\{0\} \\ |k|_1 \leq K \\ \bar k \neq 0}} \Delta_{c\hat q}(k,\frac{\beta}{|k|_1^\tau})\right)  \right. $\\
 & & $\quad \cup \left. \left(\bigcup_{\substack{k_1\in\mathbb{Z}\setminus\{0\} \\ |k|_1 \leq \frac{\beta}{|k_1|^\tau}\mathcal{K}'} } \mathbb{R}^n \right)\right]$.\\
\end{longtable}
Note that if for some $k_1 \in \mathbb{Z}\setminus \{0\}$, $|k|_1 \geq \frac{\beta}{|k|_1^\tau}\mathcal{K}'$, we take out all the possible frequencies.
Then seems natural to ask $|k|_1 \geq \frac{\beta}{|k|_1^\tau}\mathcal{K}'$ for all $k_1 \in \mathbb{Z}\setminus\{0\}$, which holds if and only if $|k_1|^{1+\tau} \geq \beta K'$ for all $k_1 \in \mathbb{Z}\setminus\{0\}$ or simply $\beta\leq \frac{1}{\mathcal{K}'}$ which we assumed. Then
$$F(d,\beta,K) := (F-d)\setminus \bigcup_{\substack{k\in\mathbb{Z}^n\setminus\{0\} \\ |k|_1 \leq K \\ \bar k \neq 0}} \Delta_{c\hat q}(k,\frac{\beta}{|k|_1^\tau}).$$
Hence we can replicate the proof of \ref{eq:meas_nonr_1} only with the terms $\bar k \neq 0$. And the bound of \ref{eq:meas_nonr_2} can be slightly improved by using that the number of vectors $k\in \mathbb{Z}^n\setminus\{0\}$ such that $|k|_1 \leq K$ and $|\bar k| \neq 0$ is bounded by $2^n K^n - K$, but since it is not a big improve, for the sake of simplicity we assume the bound \ref{eq:meas_nonr_2} at $Z$.

\end{proof}

\begin{lemma}\label{lemma:2.3}
Let $G \subset \mathbb{R}^n$ be compact. $u,\tilde u: G \rightarrow \mathcal{R}^n$ maps of class $\mathcal{C}^2$. $|\tilde u - u|\leq \varepsilon$. Assume that $u$ is one-to-one on $G$, let $F = u(G)$. Consider the following bounds:
$$\left|\frac{\partial u}{\partial I}\right|_G \leq M, \left|\frac{\partial u}{\partial I}(I)\cdot v\right| \geq \mu|v| \quad \forall v\in\mathbb{R}^n, \forall I \in G,$$
$$\left|\frac{\partial \tilde u }{\partial I}\right|_G\leq\tilde M, \left|\frac{\partial \tilde u}{\partial I^2}\right|_G\leq\tilde M_2, \left|\frac{\partial \tilde u}{\partial I}(I) v\right|\geq\tilde\mu|v| \quad \forall v \in \mathbb{R}^n, \forall I \in G,$$
$\tilde \mu < \mu$ and $\tilde M < M$. Assume $\varepsilon \leq \tilde mu^2/(4\tilde M_2)$.
Then, given a subset $\tilde F \subset F - \frac{4M\varepsilon}{\tilde \mu}$ and writing $\tilde G = (\tilde u)^{-1}(\tilde F)$, the map $\tilde u$ is one-to-one from $\tilde G$ to $\tilde F$ and
$$ \tilde G \subset G - \frac{2\epsilon}{\tilde \mu}, \quad u(\tilde G) \supset \tilde F - \varepsilon.$$
Moreover,
$$|(\tilde u)^{-1} - u^{-1}|_{\tilde F} \leq \frac{\varepsilon}{\mu}$$
\end{lemma}

\begin{proof}
The statement is not any different than the classical one, so we are not going to prove it in here. A proof can be found in \cite{D}.
\end{proof}

\begin{lemma}[Inductive lemma]\label{lemma:inductive}
Let $G \subset \mathbb{R}^n$ be a compact.
$$H(\phi,I)=\hat h(I) + R(\phi,I)$$
 where $\hat h$ is defined as in \ref{eq:bm-hamiltonian} in the domain $\mathcal{D}_\rho(G)$,and $R(\phi,I)$ analytic on the same domain.
Let $\hat u = \frac{\partial \hat h}{\partial I}$ and $u = \frac{\partial h}{\partial I}$. Assume that $|\frac{\partial}{\partial I} u|_{G,\rho_2} \leq M'$ and $|u|_G \leq L$. Also, assume that $u$ is non-degenerate:
$$\left|\frac{\partial u}{\partial I}v\right| \geq \mu|v| \quad \forall I \in \mathcal{G}.$$
Let $\tilde M > M$, $\tilde L > L$ and $\tilde \mu < \mu$.
Assume $u$ is one-to-one on $G$ and denote $F = u(G)$. Assume $\tau > 0$, $0 < \beta \leq 1$ and $K$ given. Assume also that
$$F\cap \Delta_{c,\hat q}\left(K,\frac{\beta}{|k|_1^\tau}\right) = \emptyset, \quad \forall k \in \mathbb{Z}^n,|k|_1 \leq K, k\neq0.$$

Denote $\epsilon:=\|DR\|_{G,\rho,c}$, $\eta:=|R_0|_{G,\rho_2}$ and $\xi:=\left|\frac{\partial R_0}{\partial I}\right|_{G,\rho_2}$.
\begin{enumerate}
\item\label{eq:inductive_lemma_1} $\rho_2 \leq \frac{\beta}{2MK^{\tau+1}}$
\item\label{eq:inductive_lemma_2} $\epsilon \leq \min\left(\frac{\beta \hat \delta_c}{74 A K^\tau},\frac{\tilde\mu^2(\rho_2-\delta_2)}{4\tilde M}\right)$
\item\label{eq:inductive_lemma_3} \textcolor{black}{$\xi \leq \min\left((\tilde M - M)\delta_2/\mathcal{R}, (\mu - \tilde \mu)\rho_2\right)$}
%$\xi \leq \min\left((\tilde M - M)\delta_2/\mathcal{R}, \tilde L - L, (\mu - \tilde \mu)\rho_2\right)$
\end{enumerate}
Then there exists a real canonical transformation $$\Phi:\mathcal{D}_{\rho-\frac{\delta}{2}}(G)\rightarrow \mathcal{D}_\rho(G)$$
 and a decomposition $H\circ\Phi = \tilde{\hat h}(I) + \tilde R (\phi,I)$. Writing $\tilde u = \frac{\partial}{\partial I} \tilde h$ one has.
\begin{enumerate}
\item $|\tilde u - u|_{G,\rho_2}=\xi, \quad |\tilde h - h|_{G,\rho_2} = \eta,$
\item $\tilde \epsilon := \|D\tilde R \|_{G,\rho-\delta,c}\leq e^{-K\delta_1}\epsilon + \frac{14AK^\tau}{\beta\hat\delta_c}\epsilon^2,$
\item $\tilde \eta := |\tilde R_0|_{G,\rho_2-\frac{\delta_2}{2}}\leq\frac{7AK^\tau}{c\beta}\epsilon^2,$
\item $|\Phi -\id|_{G,\rho-\frac{\delta}{2},c}\leq\frac{2AK^\tau}{\beta}\epsilon,$
\item $\left|\frac{\partial}{\partial I} \tilde u\right|_{G,\rho_2} \leq \tilde M'$, $|\tilde u|_G\leq \tilde L,$
\item $|\frac{\partial \tilde u}{\partial I}v| \geq \tilde \mu |v| \quad \forall I \in \mathcal{G},$
\item Given a subset $\tilde F \subset F -\frac{4M\epsilon}{\tilde \mu}$, $\tilde G (\tilde u)^{-1}(\tilde F)$ the map $\tilde u$ is one-to-one from $\tilde G$ to $\tilde F$, $\tilde G \subset G -\frac{2\epsilon}{\tilde \mu}$, $u(\tilde G) \supset \tilde F -\epsilon$. Moreover $|\tilde u^{-1} - u^{-1}|_{\tilde F} \leq \epsilon/\mu$.
\end{enumerate}

\end{lemma}

\begin{proof}
%Recall that $A^\omega$ is defined as $A^\omega = \left(\begin{array}{rl}
%\mathcal{A}(I_1) + \mathcal{B}(I_1) & 0  \\
% 0 &  \text{Id}_{n-1,n-1}\\
%\end{array}\right).$
The set $u(I)$ is $\beta/K^\tau, K$-non-resonant with respect to $\omega$. This implies that
\begin{equation}\label{eq:nonres}
|k_1\mathcal{B}(I_1)u_1 + \bar{k}\bar{u} + \mathcal{A}(I_1)u_1| \geq \beta/K^\tau. \geq \frac{\beta}{|k|_1^\tau} \geq \frac{\beta}{K^\tau}.
\end{equation}
Then $\rho_2\leq \frac{\beta/K^\tau}{2MK} = \frac{\beta}{2MK^{\tau+1}}$, $\|DR\|_{G,\rho,c} \leq \frac{\beta/K^\tau \hat{\delta}_c}{74A} = \frac{\beta \hat{\delta}_c}{74AK^\tau}$.
We apply the iterative lemma (Theorem \ref{lemma:iterative}) to obtain $\Phi:\mathcal{D}_{\rho-\frac{\delta}{2}}(G) \rightarrow \mathcal{D}_\rho(G)$, such that $H\circ \Phi = \tilde{h} + \tilde{R}$ where $\tilde{h} = h + R_0$.

We have taken out the points that are not $\beta/K^\tau,K$-non-resonant with respect to $\omega$.
Because of conditions \ref{eq:inductive_lemma_1} and \ref{eq:inductive_lemma_2} we can apply the Iterative lemma.
Now let us prove each of the points in the statement.
\begin{enumerate}
\item We know by definition that $\tilde{u} = \frac{\partial \tilde{h}}{\partial I} = \frac{\partial (h + R_0)}{\partial I} = \frac{\partial h}{\partial I} + \frac{R_0}{\partial I}$, hence:
$$|\tilde u - u|_{G,\rho_2} = |\frac{\partial h}{\partial I} +\frac{\partial R_0}{\partial I} - \frac{\partial h}{\partial I}|_{G,\rho_2} = |\frac{\partial R_0}{\partial I}|_{G,\rho_2} = \xi$$
$$\tilde h = h + R_0 \Rightarrow |\tilde h - h|_{G,\rho_2} = |h + R_0 - h|_{G,\rho_2} = |R_0|_{G,\rho_2} = \eta$$
\item By the iterative lemma:
$$ \begin{array}{rcl}
\|D \tilde R\|_{G,\rho-\delta,c} & \leq & e^{-K \delta_1}\|DR\|_{G,\rho,c} + \frac{14A}{\alpha \hat \delta_c}\|DR\|_{G,\rho,c}\\
 & \leq & e^{-K\delta_1}\varepsilon + \frac{14A}{\alpha \hat \delta_c}\varepsilon^2\\
 & = & e^{-K\delta_1}\varepsilon + \frac{14A K^\tau}{\beta \hat \delta_c}\varepsilon^2,
\end{array}
$$
where we have used that $\alpha = \frac{\beta}{K^\tau}$.
\item At this point we use an inequality used in the proof of the iterative Lemma (theorem \ref{lemma:iterative}).
$$\begin{array}{rcl}
|\tilde R_0|_{G,\rho_2 - \delta_2/2} & \leq & |r_2(h,W,1) + r_1(R,W,1)|_{G,\rho_2 - \delta_2/2}\\
 & \leq & \frac{7A}{\alpha c} \|DR\|^2_{G,\rho,c} = \frac{7AK^\tau}{\beta}\varepsilon^2.
\end{array}$$
\item Also using the the iterative Lemma:
$$|\Phi-\text{id}|_{G,\rho-\delta/2,c} \leq \frac{2A}{\alpha} \|DR\|_{G,\rho,c} = \frac{2AK^\tau}{\beta}\|DR\|_{G,\rho,c}.$$
\item Recall that $|\frac{\partial}{\partial I} A^\omega \tilde u|_{G,\rho_2 - \delta_2}\leq \tilde M$, $|\tilde u|_G \leq \tilde L$, $\tilde h = h + R_0$, $|\frac{\partial}{\partial I} A^\omega u|_{G,\rho_2}\leq M$, $|u|_G \leq L$.
Note that $\mathcal{A}(I_1) \leq m \cdot \max_j(q_j)/\min_j(c_j)$ and $\mathcal{B}(I_1) \leq 1/ \min_j(c_j)$. Hence $\mathcal{A}(I_1) + \mathcal{B}(I_1) \leq \max_j(q_j)/\min_j(c_j) + 1/ \min_j(c_j) := \mathcal{R}$, and we have that $|A^\omega| \leq \mathcal{R}$.
$$
\begin{array}{rcl}
 |\frac{\partial}{\partial I} A^\omega \tilde u|_{G,\rho_2-\delta_2}& = &
 |\frac{\partial}{\partial I} A^\omega \tilde u + \frac{\partial}{\partial I} A^\omega u - \frac{\partial}{\partial I} A^\omega  u|_{G,\rho_2-\delta_2} \\
 & \leq & |\frac{\partial}{\partial I} A^\omega (\tilde u - u)|_{G,\rho_2-\delta_2} + |\frac{\partial}{\partial I} A^\omega  u|_{G,\rho_2-\delta_2}\\
 & \leq & |\frac{\partial}{\partial I} A^\omega R_0|_{G,\rho_2-\delta_2} + M \\
 & \leq & \frac{|A^\omega|_{G,\rho_2}|R_0|_{G,\rho}}{\delta_2} + M\\
 & \leq & \frac{|A^\omega|_{G,\rho_2} \cdot \xi}{\delta_2} + M \\
 & \leq & \frac{\mathcal{R}\xi}{\delta_2} + M\\
 & \leq & \displaystyle \frac{(\frac{(\tilde M - M)\delta_2}{\mathcal{R}})\mathcal{R}}{\delta_2} + M \leq \tilde M - M + M = \tilde M,
\end{array}
$$
where $\xi \leq (\tilde M - M)\delta_2/\mathcal{R}$.
\item We know $|\frac{\partial u}{\partial I}(I) v| \geq \mu |v|$ for all $I \in \mathcal{G}$, then $|\frac{\partial u}{\partial I}(I) v|_G \geq \mu|v|$.
We want to find $|\frac{\partial \tilde u}{\partial I} (I) v|_G \geq \mu' |v|$ if $\mu' < \mu$.
$$
\begin{array}{rcl}
 |\frac{\partial \tilde u}{\partial I} v|_G &  = & |(\frac{\partial \tilde u}{\partial I} + \frac{\partial u}{\partial I} - \frac{\partial u}{\partial I})v|_G \\
 & = & |(\frac{\partial^2 R_0}{\partial I^2} + \frac{\partial u}{\partial I})v|_G\\
  & \geq & -|\frac{\partial^2 R_0}{\partial I^2} v|_G + |\frac{\partial u}{\partial I} v|_G\\
  & \geq & \mu|v| - |\frac{\partial^2 R_0}{\partial I^2}|_G|v|\\
  & \geq & \mu|v| - |\frac{\partial R_0}{\partial I}|_G\frac{1}{\delta_2}|v|\\
    & \geq & \mu|v| - \frac{\xi}{\rho_2}|v| = (\mu - \xi/\rho_2)|v| \geq \mu'|v|,\\
\end{array}
$$
where we have used that $|\frac{\partial^2 R_0}{\partial I^2}|_G \leq |\frac{\partial R_0}{\partial I}|\frac{1}{\rho_2}$, and also that $\mu' < \mu -\xi/\rho_2$, hence $\xi \leq (\mu - \mu')\rho_2$.
\item To apply lemma \ref{lemma:2.3} we only need to check that $\varepsilon \leq \frac{\tilde \mu ^2}{ \tilde M_2}$.
$\tilde M_2$ can be chosen such that $|\frac{\partial^2 u}{\partial I^2}|_G \leq \tilde M_2$. Note that $|\frac{\partial^2 u}{\partial I^2}|_G \leq |\frac{\partial^2 u}{\partial I^2}|_{G,\rho_2-\delta_2}$.
$$
\begin{array}{rcl}
|\frac{\partial u}{\partial I}|_{G,\rho_2-\delta_2} \leq \tilde M & \Rightarrow & |\frac{\partial^2 u}{\partial I^2}|_{G, \rho_2-\delta_2}(\rho_2-\delta_2) \leq |\frac{\partial u}{\partial I}|_{G,\rho_2-\delta_2} \leq \tilde M\\
 & \Rightarrow & |\frac{\partial^2 u}{\partial I ^2}|_{G,\rho_2-\delta_2} \leq \frac{\tilde M}{\rho_2-\delta_2} = \tilde M_2\\
 & \Rightarrow & |\frac{\partial^2 u}{\partial I^2}|_G \leq \tilde M_2
\end{array}
$$
Then $\varepsilon \leq \frac{\tilde M}{4 \tilde M_2}$ if and only if $\varepsilon \leq \mu^2/(4\frac{\tilde M}{(\rho_2 - \delta_2)})$ if and only if $\varepsilon \leq \frac{\mu^2(\rho_2-\delta_2)}{4\tilde M}$ which it is assumed in the statement.
\end{enumerate}

\end{proof}

\section{ A KAM theorem on $b^m$-symplectic manifolds}

\begin{theoremB}[ A $b^m$-KAM theorem]\label{th:bm_kam}
Let $\mathcal{G} \subset \mathbb{R}^n$, $n\geq 2$ be a compact set.
Let $H(\phi, I) = \hat h (I) + f(\phi,I)$, where $\hat h$ is a $b^m$-function $\hat h (I) = h(I) + q_0 \log(I_1) + \sum_{i=1}^{m-1} \frac{q_i}{I_1^i}$ defined on $\mathcal{D}_\rho(G)$, with $h(I)$ and $f(\phi,I)$ analytic.
Let $\hat u = \frac{\partial \hat h}{\partial I}$ and $u = \frac{\partial h}{\partial I}$.
Assume $|\frac{\partial u}{\partial I}|_{G,\rho_2} \leq M$, $|u|_{\mathcal{G}} \leq L$.
Assume that $u$ is $\mu$ non-degenerate ($|\frac{\partial u}{\partial I}v|\geq \mu|v|$ for some $\mu \in \mathbb{R}^+$ and $I \in \mathcal{G}$. Take $a = 16M$.
Assume that $u$ is one-to-one on $\mathcal{G}$ and its range $F = u(\mathcal{G})$ is a $D$-set.
Let $\tau>n-1,\gamma>0$ and $0 < \nu < 1$. Let
\begin{enumerate}
\item \begin{equation}\label{eq:kam1}
\varepsilon:=\|f\|_{\mathcal{G}, \rho} \leq \frac{\nu^2 \mu^2 \hat \rho^{2\tau+2}}{2^{4\tau+32}L^6M^3} \gamma^2,
\end{equation}
\item \begin{equation}\label{eq:kam2}
\gamma \leq \min(\frac{8LM\rho_2}{\nu \hat \rho^{\tau+1}}, \frac{L}{\mathcal{K}'})
\end{equation}
\item \begin{equation}\label{eq:kam3}
\mu \leq \min(2^{\tau+5}L^2 M,2^7\rho_1 L^4 K^{\tau+1},\beta\nu^{\tau+1}2^{2\tau+1}\rho_1^\tau),
\end{equation}
\end{enumerate}
where $\hat \rho := \min \left(\frac{\nu\rho_1}{12(\tau+2)},1\right)$.
Define the set $\hat G = \hat G_\gamma := \{I \in  \mathcal{G}-\frac{2\gamma}{\mu} | u(I) \text{ is } \tau,\gamma,c,\hat q- Dioph.\}$.
Then, there exists a real continuous map $\mathcal{T}: \mathcal{W}_{\frac{\rho_1}{4}}(\mathbb{T}^n)\times \hat G \rightarrow \mathcal{D}_\rho(\mathcal{G})$ analytic with respect the angular variables such that
\begin{enumerate}
\item\label{kam:point1} For all $I \in \hat G$ the set $\mathcal{T}(\mathbb{T}^n\times \{I\})$ is an invariant torus of $H$, its frequency vector is equal to $u(I)$.
\item\label{kam:point2} Writing $\mathcal{T}(\phi,I)=(\phi + \mathcal{T}_\phi(\phi,I), I + \mathcal{T}_I(\phi,I))$ with estimates

\textcolor{black}{
$$|\mathcal{T}_\phi(\phi,I)| \leq \frac{2^{2\tau + 15} M L^2}{\nu^2 \hat \rho^{2\tau+1}}\frac{\varepsilon}{\gamma^2}$$
$$|\mathcal{T}_I(\phi,I))| \leq \frac{2^{10+\tau} L (1+M)}{\nu \hat \rho^{\tau+1}}\frac{\varepsilon}{\gamma}$$
}

\item\label{kam:point3} $\text{meas} [(\mathbb{T}^n\times \mathcal{G})\setminus\mathcal{T}(\mathbb{T}^n\times \hat G)] \leq C \gamma$ where $C$ is
\textcolor{black}{a really complicated constant depending on $n$,  $\mu$,  $D$,  $\text{diam} F$,  $M$, $\tau$, $\rho_1$, $\rho_2$, $K$ and $L$.}

\end{enumerate}

\end{theoremB}

\begin{proof}
This proof, as the one in \cite{D} is going to be structured in six sections. First we define the parameters used in each iteration while building the diffeomorphism. After that, we prove that we can apply the inductive lemma \ref{lemma:inductive} and we exhibit some bound that hold using the results of the inductive lemma. Next, we find that the sequence of frequency vectors and the sequence of diffeomorphisms that we built actually converges. Then we find estimates of the components of the canonical transformation that we have built. Then we find a way to identify the invariant tori and finally we give a bound for the measure of the set of invariant tori.
\begin{enumerate}
\item Choice of parameters

We are going to make iterative use of proposition \ref{lemma:inductive}. So we need to properly define all the parameters in the statement for every iteration.
Let:
$$
\left\{
\begin{array}{rcl}
M_q & = & (2 - \frac{1}{2^q})M, \\
L_q & = & (2 - \frac{1}{2^q})L, \\
\mu_q & = & (1 + \frac{1}{2^q})\frac{\mu}{2}.
\end{array}
\right.
$$

Note that $M_q, L_q$ monotonically increase from $M$ to $2M$ and $L$ to $2L$ when $q\rightarrow\infty$. On the other hand $\mu_q$ monotonically decreases from $\mu$ to $\mu/2$.
Also, let:

$$
\left\{
\begin{array}{rcl}
K_0 & = & 0, \\
K_q & = & K\cdot q^{q-1}, q \geq 1,
\end{array}
\right.
$$

where  $K$ is the minimum natural number greater or equal than $1/\hat \rho$ and \textcolor{black}{greater or equal than $(\frac{\nu\beta}{\mu 2^{2\tau+12}})^{1/\tau}$}. Moreover, $\beta:= \gamma/L \leq 1$, and

$$
\left\{
\begin{array}{rcl}
 \rho^{(q)}& = & (\rho_1^{(q)}, \rho_2^{(q)}), \\
 \rho_1^{(q)}& = & (1+\frac{1}{2^{\nu q}})\frac{\rho_1}{4}, \\
 \rho_2^{(q)}& = & \frac{\nu\beta}{32 M K_{q+1}^{\tau+1}}.
\end{array}
\right.
$$

Notice that $\rho_1^{(q)}$ decreases monotonically from $\rho_1/2$ to $\rho_1/4$. Also, $\rho_2^{(q)}$ decreases to 0. We also denote:

$$
\left\{
\begin{array}{rcl}
 \delta_1^{(q)} & = & \rho_1^{(q-1)} - \rho_1^{(q)}, \\
 \delta_2^{(q)} &= &\rho_2^{(q-1)}-\rho_2^{(q)}, \\
 c_q & = & \frac{\delta_2^{(q)}}{\delta_1^{(q)}}.
\end{array}
\right.
$$

Note that
$$
\begin{array}{rcl}
\delta_1^{(q)} & = & \left(1+\frac{1}{2^{\nu(q-1}}\right)\frac{\rho_1}{4}-\left(1 + \frac{1}{2^{\nu q}}\right)\frac{\rho_1}{4} \\
& = & \left(\frac{1}{2^{\nu(q-1)}} - \frac{1}{2^{\nu q}}\right)\frac{\rho_1}{4} \\
& = & \frac{1 - 1/2^\nu}{2^{\nu(q-1)}}\frac{\rho_1}{4}.\\
\end{array}
$$

Also, since $0 < \nu < 1$ then $\nu/2 \leq 1-1/2^\nu \leq \nu$. Plugging this in the previous equation we obtain:

\begin{equation}\label{eq:delta1ineq}
\frac{\nu \rho_1}{2^{\nu(q-1)}8} \leq \delta_1^{(q)} \leq \frac{\nu \rho_1}{2^{\nu(q-1)}4}.
\end{equation}

Also,

$$
\begin{array}{rcl}
\delta_2^{(q)} & = & \frac{\nu \beta}{32 M K_q^{\tau+1}} - \frac{\nu \beta}{32 M K_{q+1}^{\tau+1}}\\
& = & \frac{\nu \beta}{32M (K2^{q-1})^{\tau+1}} - \frac{\nu \beta}{32 M (K 2^q)^{\tau+1}}\\
& = & \frac{\nu \beta}{32M(K2^{q-1})^{\tau + 1}}\left(1-\frac{1}{2^{\tau+1}}\right) .\\
\end{array}
$$

Also, since $\tau > 0$ then $1/2 \leq (1 - 1/2^{\tau+1}) \leq 1$. Using this in the previous equation:

\begin{equation}\label{eq:delta2ineq}
\frac{\nu \beta}{64 M K_q^{\tau+1}}\leq \delta_2^{(q)} \leq \frac{\nu \beta}{32 M K_q^{\tau+1}}.
\end{equation}

Using equations \ref{eq:delta1ineq} and \ref{eq:delta2ineq} we find bounds for $c_q$

$$
\left\{
\begin{array}{rcl}
c_q & \leq & \frac{\left(\frac{\nu \beta}{32 M K_q^{\tau+1}}\right)}{\left(\frac{\nu \rho_1}{2^{\nu(q-1)}}\right)} = \frac{\beta 2^{\nu(q-1)}}{4 M K_q^{\tau+1}\rho_1},\\
c_q & \geq & \frac{\left(\frac{\nu \beta}{64 M K_q^{\tau+1}}\right)}{\left(\frac{\nu \rho_1}{2^{\nu(q-1)} 4}\right)} = \frac{\beta 2^{\nu(q-1)}}{16 M K_q^{\tau+1}\rho_1}.\\
\end{array}
\right.
$$

Then, we also define

$$
\left\{
\begin{array}{rcl}
\beta_q & = & (1-\frac{1}{2^{\nu q}})\beta,\\
\beta'_q & = & \frac{\beta_q + \beta_{q+1}}{2}.\\
\end{array}
\right.
$$

Observe that both $\beta_q$ and $\beta'_q$ tend to $\beta$. Also observe that $\beta'_q \geq \frac{\nu}{4}\beta$, because:

$$
\begin{array}{rcl}
\beta'_q & = & \frac{\beta_q + \beta_{q+1}}{2} \\
& = & \frac{\left(1 - \frac{1}{2^{\nu q}}\right) + \left(1 - \frac{1}{2^{\nu(q+1)}}\right)}{2}\beta \\
& = & \left(1 - \left(\frac{1 + \frac{1}{2^\nu}}{2^{\nu q}}\right)\frac{1}{2}\right)\beta \\
& \geq & \left(1 - (1 - 1/2^\nu)\frac{1}{2}\right)\beta \geq  \frac{\nu}{4}\beta.
\end{array}
$$

As $K$ is the minimal natural number such that $K \geq 1/\hat\rho$ then $K \leq 2/\hat \rho$. Hence $\hat\rho \leq \frac{2}{K}$. Also $$\frac{1}{\hat\rho^{\tau+1}} \geq \left(\frac{K}{2}\right)^{\tau+1}.$$

Recall that $\hat \rho = \min(\frac{\nu \rho_1}{12(\tau+2)},1)$ and, in particular, $\hat \rho \leq \nu\rho_1$ and $\hat \rho \leq 1$.

By definition $\gamma \leq \frac{8LM\rho_2}{\nu \hat \rho^{\tau+1}}$. And because $\beta = \gamma/L$:
$$\beta L \leq \frac{8L M \rho_2}{\nu \hat \rho^{\tau+1}} \leq \frac{8L M \rho_2 K^{\tau+1}}{\nu}.$$

Because we assumed $\varepsilon \leq \frac{\nu^2 \mu^2 \hat\rho^{2\tau+2}}{2^{4\tau + 32} L ^6 M^3} \gamma^2$ then, using that $\gamma = L\beta$ and $\hat \rho \leq 2/K$:

\begin{equation}\label{eq:kam_epsilon_1}
\varepsilon \leq \frac{\nu^2 \mu^2 \left(\frac{2}{K}\right)^{2\tau+2}}{2^{4\tau + 32} L ^6 M^3} \leq \frac{\nu^2 \mu^2 \beta^2}{2^{4\tau + 30} L ^4 M^3 K^{2\tau+2}}.
\end{equation}

Also using again the assumption that $\varepsilon \leq \frac{\nu^2 \mu^2 \hat \rho^{2\tau+2}}{2^{4\tau+32} L^6 M^3} \gamma^2$ we want to prove that

\begin{equation}\label{eq:kam_epsilon_2}
\varepsilon \leq \frac{\nu^3 \rho_1 \beta^2}{2^{2\tau+22} M K^{2\tau +1}}.
\end{equation}

 It is enough to check that:

$$\frac{\nu^2 \mu^2 \hat \rho^{2\tau+2} L^2 \beta^2}{2^{4\tau + 32} L^6 M^3} \leq \frac{\nu^3 \rho_1 \beta^2}{2^{2\tau + 22} M K^{2\tau+1}}$$

where we used $\gamma = L\beta$. Now observing that $\hat \rho \leq \nu \rho_1$ it suffices to see

$$\frac{\nu^2 \mu^2 \nu^{2\tau + 2} \rho_1^{2\tau +2} L^2 \beta^2}{2^{4\tau + 32} L^6 M^3}\leq\frac{\nu^3 \rho_1 \beta^2}{2^{2\tau+22} M K^{2\tau+1}},$$

which simplifies to

$$\frac{\mu^2 \rho_1^{2\tau+2}}{2^{4\tau+10}L^4M^2} \leq \frac{1}{K^{2\tau+1}}.$$

Using that $K \geq 1/(\nu\rho_1)$ is enough to check that

$$\frac{\mu^2 \rho_1^{2\tau+1}\nu^{2\tau+2}}{2^{2\tau+12} L^4 M^2} \leq (\nu \rho_1)^{2\tau+1},$$

which holds if and only if $\mu \leq 2^{\tau+5}L^2 M$ as we assumed.

\item Induction

Let us take $G_0 = \mathcal{G}$.
Now the goal is to construct a decreasing sequence of compact sets $G_q \subset \mathcal{G}$ and a sequence of real analytic canonical transformations

$$\Phi^{(q)}:\mathcal{D}_{\rho^{(q)}}(G_q) \rightarrow \mathcal{D}_{\rho^{(q-1)}}(G_{q-1}), \quad q \geq 1.$$

Denoting $\Psi^{(q)}= \Phi^{1}\circ \cdots \circ \Phi^{(q)}$ the transformed Hamiltonian functions will be noted by $H^{(q)} = H\circ \Psi^{(q)} = \hat h^{(q)}(I) + R^{(q)}(\phi,I)$. Moreover, $u^{(q)} = \frac{\partial h^{(q)}}{\partial I}$ and $\hat u^{(q)} = \frac{\partial \hat h^{(q)}}{\partial I}$.

We are going to show that the following bounds hold for all $q \geq 0$:

\begin{enumerate}
\item\label{eq:induction1} $\varepsilon_q := \|DR^{(q)}\|_{G_q,\rho^{(q)},c_{q+1}} \leq \frac{8\varepsilon}{\nu \rho_1 2^{(2\tau+2)q}},$
\item\label{eq:induction2} $\eta_q:=|R_0^{(q)}|_{G_q,\rho_2^{(q)}} \leq \textcolor{black}{\frac{\varepsilon}{2^{(2\tau+3)q}}}$ and $\xi_q:=|\frac{\partial R_0^{(q)}}{\partial I}|_{G_q,\rho_2^{(q)}} \leq \textcolor{black}{ \frac{4 M K^{\tau+1 \varepsilon}}{\nu \beta 2^{(\tau+2)q}}},$
\item\label{eq:induction3} $|\frac{\partial^2 h^{(q)}}{\partial I^2}|_{G_q,\rho_2^{(q)}}\leq M_q, \quad |u^{(q)}| \leq L_q \quad \forall I\in G_q,$
\item\label{eq:induction4} $u^{(q)}$ is $\mu_q$-non-degenerate on $G_q$,
\item\label{eq:induction5} $u^{(q)}$ is one-to-one on $G_q$, and $u^{(q)}(G_q)=F_q$ where we define:
$$F_q := (F - \beta_q)\setminus \bigcup_{\substack{k\in\mathbb{Z}^n\setminus\{0\} \\ |k|_1 \leq K}} \Delta_{c_q,\hat q}(K,\frac{\beta_q}{|k|_1^\tau})$$
\end{enumerate}

To prove this we proceed by induction. For $q = 0$:

$$
\left\{
\begin{array}{l}
G_0 = \mathcal{G}, \\
h^{(0)} = h,\hat h^{(0)} = \hat h,\\
R^{(0)} = f.
\end{array}
\right.
$$

Using the definitions from the previous point:

$$
\left\{
\begin{array}{rcl}
\rho_1^{(0)} &=& (1+1)\frac{\rho_1}{4} = \rho_1/2,\\
\rho_2^{(0)} &=& \frac{\nu \beta}{32 M K^{\tau+1}} \leq \frac{\rho_2}{2},
\end{array}
\right.
$$

where in the last inequality we have used that $\beta \leq \frac{8M\rho_2 K^{\tau+1}}{\nu}$ and hence $\rho_2 \geq \frac{\beta\nu}{8 M K^{\tau+1}}$.

Then,
\textcolor{black}{
$$\varepsilon_0 = \|Df\|_{G,\rho(0),c_1} = \|Df\|_{G,\rho(1)+\delta(1)}.$$
Now, let us use that $|D\Upsilon|_{G,\rho-\delta,c}\leq \frac{2|\Upsilon|_{G,\rho,c}}{\hat \delta_c}$ while having in mind that $\hat \delta_{c_1}^{(1)} = \min(c_1\delta_1^{(1)},\delta_2^{(1)})$.
Then,
$$\|D f\|_{G,\rho(1),c_1} \leq \frac{c_1 |f|_{G,\rho(0)}}{\hat \delta_{c_1}} \leq \frac{|f|_{G,\rho(0)}}{\delta_1^{(1)}} \leq \frac{|f|_{G,\rho(0)} 8}{\nu \rho_1} = \frac{8\varepsilon}{\nu \rho_1},$$
where we have used $\delta_1^{(1)}\geq \frac{\nu \rho_1}{8\cdot 2^{\nu(1-1)}} = \frac{\nu\rho_1}{8}$.
}
This proves the first step of the induction  for \ref{eq:induction1}).

Let us prove now the base case for \ref{eq:induction2}).
On one side $\eta_0 = |R_0^{(0)}|_{G_0,\rho_2(0)} \leq \frac{\varepsilon}{2^{(2\tau+3)0}} = \varepsilon$, which holds because $|R_0^{(0)}|_{G_0,\rho_2^{(0)}} \leq |R^{(0)}|_{\mathcal{G},\rho^{(0)}} = |f|_{\mathcal{G},\rho(0)} = \epsilon$.
On the other hand $\xi_0 = |\frac{\partial R_0^{(0)}}{\partial I}|_{G_0,\rho_2^{(0)}} \leq |\frac{\partial R_0^{(0)}}{\partial I}|_{G_0,\rho_2 -\rho_2/2} \leq \frac{1}{\rho_2/2}\|R_0\|_{G,\rho} \leq \frac{2\varepsilon}{\rho_2} \leq \frac{\varepsilon}{\rho_2^{(0)}}$, where we used that $\rho_2(0) \leq \rho_2/2 = \rho_2 - \rho_2/2$.

The base case of \ref{eq:induction3}) is immediate because $|\frac{\partial^2 h^{(0)}}{\partial I^2}|_{G_0,\rho_2^{(0)}} \leq |\frac{\partial^2 h}{\partial I^2}|_{\mathcal{G},\rho_2} = M = M_0$ and also $|u^{(0)}|_{G_0} = |u|_{\mathcal{G}} \leq L = L_0$.

The base case of \ref{eq:induction4} holds because $u^{(0)} = u$ is $\mu$ non-degenerate in $\mathcal{G} = G_0$.

The base case of \ref{eq:induction5} holds because $u^{(0)} = u$ is one-to-one in $G_0 = \mathcal{G}$ by hypothesis. $u^{(0)}(G_0) = F_0$ where $F_0 = (F - \beta_0)\setminus \{\emptyset\} = F$ because $K_0 = 0$ and $\beta_0 = 0$.

For $q \geq 1$, we assume the statements hold for $q-1$ and we prove it for $q$. Let us apply proposition \ref{lemma:inductive} (Inductive Lemma) to $H^{(q-1)} = h^{q-1} + R^{q-1}$ with $K_q$ instead of $K$.

We have to be careful with the condition $F\cap \Delta_{c,\hat q}(k,\frac{\beta}{|k|_1^\tau}) = \emptyset$ $\forall k \in \mathbb{Z}^n, |k|_1 \leq K_q, k\neq 0$ and with the definition
$$F_{q-1}:=(F - \beta_{q-1}) \setminus \bigcup_{\substack{k\in\mathbb{Z}^n\setminus\{0\} \\ |k|_1 \leq K}} \Delta_{c,\hat q}(k,\frac{\beta_{q-1}}{|k|_1^\tau}),$$
 because the resonances have to be removed up to order $K_q$, not $K_{q-1}$.
Let us define
$$F_{q-1}':= (F-\beta_{q-1}) \setminus \bigcup_{\substack{k\in\mathbb{Z}^n\setminus\{0\} \\ |k|_1 \leq K}} \Delta_{c,\hat q}(k,\frac{\beta_{q-1}'}{|k|_1^\tau})$$

where we simply replaced $\beta_{q-1}$ for $\beta_{q-1}'$ because $\Delta_{c,\hat q}(k,\frac{\beta_{q-1}}{|k|_1^\tau})$ makes no sense when $q=1,\beta_{q-1}=0$.

Accordingly let us define $G_{q-1}' := (u^{(q-1)})^{-1}(F_{q-1}')$. The conditions in proposition \ref{lemma:inductive} are going to be satisfied with $F_{q-1}'$, $\beta_{q-1}'$, $K_q$, $M_{q-1}$, $L_{q-1}$, $\mu_{q-1}$,$\rho^{(q-1)}$, $\delta^{(q)}$, $c_q$, $M_q$, $L_q$, $\mu_q$ replacing $F, \beta, K, M,L, \mu, \rho,\delta, c, \tilde M, \tilde L, \tilde \mu$. And also $a = 16M \geq 8M_q$.

We are now going to check that \ref{eq:inductive_lemma_1}, \ref{eq:inductive_lemma_2} and \ref{eq:inductive_lemma_3} are satisfied so we can apply proposition \ref{lemma:inductive}.
 \begin{itemize}

 \item[--] \ref{eq:inductive_lemma_1} We want to see that $\rho_2^{(q-1)}\leq \frac{\beta_{q-1}'}{2M_qK_q^{\tau+1}}$.
 By definition $\rho_2^{q-1} = \frac{\nu \beta}{32 M K_q^{\tau+1}} \leq \frac{4 \beta_{q-1}'}{32 M K_q^{\tau+1}}\leq \frac{\beta_{q-1}'}{8M_q K_q^{\tau+1}} \leq \frac{\beta_{q-1}}{2 M_q K_q^{\tau+1}}$, where we used that $M_q \geq M$.
 \item[--] \ref{eq:inductive_lemma_2} We want to see that $\varepsilon_{q-1} \leq \min \left(\frac{\beta_{q-1}\hat\rho_c^{(q)}}{74 A_q K_{q-1}^\tau}, \frac{\mu^\tau_q(\rho_2^{(q-1)}-\delta_c^{(q-1)})}{4 M_q}\right)$, where $A_q := 1 + \frac{2M_{q-1} c_q K_q^\tau}{\beta_{q-1}'}$.

 Notice that:

$$
\begin{array}{rcl}
A_q & := & 1 + \frac{2 M_{q-1} c_q K_q^\tau}{\beta_{q-1}'} \\
 & \leq & 1 + \frac{8 M_{q-1} c_q K_q^\tau}{\nu \beta}\\
 & \leq & 1 + \frac{8 M_{q-1}\beta 2^{\nu(q-1)}K_q^\tau}{4M K_q^{\tau+1}\rho_1 \nu \beta}\\
 & = & 1 + \frac{2 M_{q-1} 2 ^{\nu(q-1)}}{M K_q \rho_1 \nu} \\
 & = & 1 + \frac{2 M_{q-1} 2^{\nu(q-1)}}{M K 2^{q-1} \rho_1 \nu} \\
 & \leq & 1 + \frac{4 M 2^{q-1}}{M K 2^{q-1} \rho_1 \nu} \\
 & = & 1 + \frac{4}{K \rho_1 \nu} \leq 1 + 4 = 5\\
\end{array}
$$

 First, we  check that $\varepsilon_{q-1} \leq \frac{\beta_{q-1}\hat\rho_c^{(q)}}{74 A_q K_{q-1}^\tau}$.

 By induction hypothesis we know that $\varepsilon_{q-1} \leq \frac{8\varepsilon}{\nu \rho_1 2^{(2\tau+2)(q-1)}}$. Hence it is enough to see
 $$\frac{8\varepsilon}{\nu \rho_1 2^{(2\tau+2)(q-1)}} \leq \frac{\beta_{q-1}' \delta_2^{(q)}}{75\cdot 5 K_q^\tau}.$$

 Notice that
 $$
 \begin{array}{rcl}
 \frac{\beta_{q-1}' \delta_2^{(q)}}{379 K_q^\tau} & \geq & \frac{\nu \beta}{4}\frac{\nu \beta}{64 M K_q^{\tau+1}}\frac{1}{37 K_q^\tau} \\
 & = & \frac{\nu^2 \beta^2}{4\cdot 64 \cdot 370}\frac{1}{M K_q^{2\tau+1}}\\
 & = & \frac{\nu^2\beta^2}{4\cdot 64 \cdot 370 \cdot M 2^{(q-1)(2\tau+1)} K^{2\tau+1}}.
 \end{array}
$$

And this holds if the following is true:

$$\frac{8 \varepsilon}{\nu \rho_1} \leq \frac{\nu^2 \beta^2}{4\cdot 64 \cdot 379 \cdot K^{2\tau + 1} M}\Leftrightarrow \varepsilon \leq \frac{\nu^3 \beta^2 \rho_1}{2^{12} 135 M K^{2\tau+1}}.$$

This is true because in the previous section we have seen that $\varepsilon \leq \frac{\nu^3 \rho_1 \beta^2}{2^{2\tau+30}M K^{2\tau+1}}$

 Let us now prove that $\varepsilon \leq \frac{\mu_q^2 (\rho_2^{(q-1)} - \delta_2^{(q-1)})}{2 M_q}$. First of all observe that $(\rho_2^{(q-1)} - \delta_2^{(q)}) = \rho_2^{(q)}$.
 So, what we want to prove is equivalent to proving $\varepsilon_{q-1}\leq \frac{\mu_q^2 \rho_2^{(q)}}{2 M_q}$.

 On the other hand, we know that $\varepsilon_{q-1}\leq \frac{8\varepsilon}{\nu \rho_1 2^{(2\tau + 2)(q-1)}}$. And observe also that $\frac{\mu_q^2 \rho_2^{(q)}}{2M_q} \geq \frac{(\mu/2)^2\frac{\nu\beta}{32 M K^{\tau+1}}}{2M} = \frac{\mu^2 \nu \beta}{2^8 M^2 K^{\tau+1}}$.

 If are able to check that $\frac{8\varepsilon}{\nu \rho_1 2^{(2\tau+2)(q-1)}} \leq \frac{\mu\nu\beta}{2^8 M^2 K^{\tau+1}}$ we would be fine. The previous equation holds if and only if the following holds,
 $$\varepsilon \leq \frac{\mu \nu^2 \beta \rho_1 2^{(2\tau+2)(q-1)}}{2^{11}M^2 K^{\tau+1}}.$$

If we knew beforehand that $\varepsilon \leq \frac{\mu \nu^2 \beta \rho_1 2^{-(2\tau+2)}}{2^{11}M^2 K^{\tau+1}} = \frac{\mu \nu^2 \beta \rho_1}{2^{2\tau+13}M^2 K^{\tau+1}}$ we would be done.

But we also know that

$$\epsilon \leq \frac{\nu^2 \mu^2 \beta^2}{2^{2\tau+30}L^4 M^3 K^{2\tau+2}}.$$

Then it is enough to check that

$$\frac{\nu^2 \mu^2 \beta^2}{2^{2\tau+30}L^4 M^3 K^{2\tau+2}}\leq \frac{\mu \nu^2 \beta \rho_1}{2^{2\tau+13} M^2 K^{\tau+1}}.$$

And this holds because $\mu \leq 2^7 \rho_1 L^4 K^{\tau+1}$

\item[--] \ref{eq:inductive_lemma_3} Lastly we want to see that
$$\xi_{q-1} \leq \min((M_q-M_{q-1})\frac{\delta_2^{(q)}}{\mathcal{R}},(\mu_{q-1}-\mu_{q})\rho_2^{(q-1)}).$$
%$$\xi_{q-1} \leq \min((M_q-M_{q-1})\frac{\delta_2^{(q)}}{\mathcal{R}},L_q-L_{q-1},(\mu_{q-1}-\mu_{q})\rho_2^{(q-1)}).$$

Observe that $\mathcal{R}$ does not depend on $q$ because at each iteration $\hat{h}^{(q)}$ singular part is not modified. $\hat h^{(q)}= \hat h^{(q)} + R_0^{(q)}$ and $R_0$ is analytic depending only on the action coordinates.
By induction hypothesis, we know that
$$\xi_{q-1} = |\frac{\partial R_0^{(q)}}{\partial I}|_{G_q,\rho_2^{(q)}} \leq \frac{4MK^{\tau+1}\varepsilon}{\nu \beta 2^{(\tau+2)q}}.$$
We are going to check the \textcolor{black}{two} different inequalities separately
\begin{enumerate}
\item $\xi_{q-1} \leq (M_q-M_{q-1})\frac{\delta_2^{(q)}}{\mathcal{R}}$.
Note that $M_q =(2-\frac{1}{2^q})M$, then $M_q-M_{q-1} = \frac{M}{2^q}$.
$$\delta_2^{(q)} \geq \frac{\nu \beta}{64 M (K 2^{q-1})^{\tau+1}} \geq \frac{\nu\beta}{64 M (K 2^q)^{\tau+1}}$$
$$= \frac{\nu \beta}{64 M K^{\tau+1}}\frac{1}{2^{q\tau+q}}.$$
We deduce
$$(M_q-M_{q-1})\delta_2^{(q)} \geq \frac{\nu \beta}{64 K^{\tau+1}}\frac{1}{2^{\tau q + 2 q}}.$$
Hence we only need to check that
$$\frac{4MK^{\tau+1}\varepsilon}{\nu \beta 2^{(\tau+2)q}}\leq\frac{\nu \beta}{64 K^{\tau+1}}\frac{1}{2^{\tau q + 2 q}}.$$
The previous condition holds if and only if
$$\frac{4 M K^{\tau+1}\varepsilon}{\nu\beta} \leq \frac{\nu \beta}{2^6 K^{\tau+1}} \Leftrightarrow \varepsilon \leq \frac{\nu^2 \beta^2}{2 K^{2\tau+2}M}.$$

On the other hand, let us use again that $\varepsilon \leq \frac{\nu^2 \mu ^2 \beta^2}{2^{2\tau+30}L^4 M^3 K^{2\tau+2}}$. If we apply the condition $\mu \leq 2^{\tau+6}L^2 M$ in the last expression we obtain:

$$\varepsilon \leq \frac{\nu^2 \beta^2 2^{2\tau+12}L^4M^2}{2^{2\tau+30}L^4 M^3 K^{2\tau+2}} = \frac{\nu^2\beta^2}{2^8 K^{2\tau+2} M}.$$

%\item $\xi_{q-1} \leq L_q-L_{q-1}$.
%Observe that
%$$L_q = (2-\frac{1}{2^q})L, \quad L_q - L_{q-1} = \frac{1}{2^q}.$$
%Hence it is enough to check that
%$$\frac{2^2 M K^{\tau+1} \varepsilon}{\nu \beta 2^{(\tau+2)q}} \leq \frac{L}{2^q}.$$
%But observe also that

%$$\frac{2^2 M K^{\tau+1} \varepsilon}{\nu \beta 2^{(\tau+2)q}} \leq \frac{2^2 M K^{\tau+1} \varepsilon}{\nu \beta 2^q}.$$

%So, checking the following is enough

%$$\varepsilon \leq \frac{\nu \beta L}{4 M K^{\tau+1}}.$$

%But we know that

%\begin{longtable}{rcl}
%$\varepsilon$ & $\leq$ & $\frac{\nu^3 \rho_1 \beta^2}{2^{2\tau+20 M K^{2\tau +1}}}$\\
%& $\leq$ & $\frac{\nu^3 2 \beta^2}{2^{2\tau+20} M K^{2\tau+1} \nu K}$\\
%& $=$& $\frac{\nu^2 \beta^2}{2^{2\tau+19} M K^{2\tau+2}}$\\
%& $\leq$ & $\frac{\nu \beta}{2^{2\tau+19}}\frac{\nu \beta}{4M K^{2\tau+2}}$\\
%& $\leq$ & $\frac{\nu \beta}{4M K^{2\tau+2}} \leq \frac{\nu \beta}{4M K^{2\tau+2}} \textcolor{black}{L}$\\
%\end{longtable}

%where we have used that $\rho \leq \frac{2}{\nu K}$, $\nu < 1$ and $\beta \leq 1$.

\item $\xi_{q-1} \leq (\mu_{q-1}-\mu_{q})\rho_2^{(q-1)}$.

Observe that

$$\mu_q = (1 + \frac{1}{2^q})\frac{\mu}{2},$$
$$(\mu_{q-1} - \mu_q) = ((1 + \frac{1}{2^{q-1}})-(1 + \frac{1}{2^{q}}))\frac{\mu}{2} = (\frac{1}{2^{q-1}} - \frac{1}{2^{q}})\frac{\mu}{2} $$
$$= \left(\frac{2-1}{2^q}\right)\frac{\mu}{2} = \frac{1}{2^q}\frac{\mu}{2} = \frac{\mu}{2^{q+1}}$$
Also,
$$\rho_2^{(q-1)} = \frac{\nu \beta}{32 M K_q^{\tau+1}} = \frac{\nu\beta}{32 M (K 2^{q-1})^{\tau+1}} $$
$$\geq \frac{\nu \beta}{32 M K^{\tau+1} 2^{q(\tau+1)}}.$$

Then,
$$(\mu_{q-1}-\mu_q)\rho_2^{(q-1)} \geq \frac{\mu}{2^{q+1}}\frac{\nu\beta}{32 M K^{\tau+1} 2^{q(\tau+1)}}.$$

Then we only have to check that

$$
\begin{array}{rcl}
\frac{4 M K^{\tau+1}\varepsilon}{\nu \beta 2^{\tau q + 2q -2}} & \leq & \frac{\mu}{2^{q+1}}\frac{\nu \beta}{32 M K^{\tau+1} 2^{q(\tau+1)}}\\
& = & \frac{\mu}{2^{\tau q + 2q +1}}\frac{\nu \beta}{32 M K^{\tau+1} }.\\
\end{array}
$$

Which holds if and only if

$$\frac{M K^{\tau+1} \varepsilon}{\nu \beta 2^{-2}} \leq \frac{\mu}{2}\frac{\nu \beta}{32 M K^{\tau+1}}.$$

Then,

$$\varepsilon \leq \frac{\mu 2^{-2}}{2}\frac{\nu^2 \beta^2}{32 M^2 K^{2\tau+2}} = \frac{\mu \nu^2 \beta^2}{2^8 M^2 K^{2\tau+2}}.$$

But we know

\begin{longtable}{rcl}
$\varepsilon $ & $\leq$ & $\frac{\nu^2 \mu^2 \beta^2}{2^{\tau+30} L^4 M^3 K^{2\tau+2}}$ \\
& $\leq$ & $\frac{\nu^2 \mu 2^{\tau+5}L^4M\beta^2}{2^{\tau+30}L^4 M^3 K^{2\tau+2}}$ \\
& $=$ & $\frac{\nu^2 \mu \beta^2}{2^25 M^2 K^{2\tau+2}}$ \\
%& $=$ & $\frac{\mu \nu^2 \beta^2}{2^6 M^2 K^{2\tau+2}} \textcolor{black}{\frac{2}{L^2}}$ \\
& $\leq$ & $\frac{\mu \nu^2 \beta^2}{2^8 M^2 K^{2\tau+2}}$ \\
\end{longtable}

as we wanted. In the second inequality we used that $\mu \leq 2^{\tau+5}L^4 M$.%$\mu \leq 2^{\tau+6}L^2 M$

\end{enumerate}

So, finally, we can apply the inductive lemma \ref{lemma:inductive} with the parameters mentioned previously in this section. Hence we obtain a canonical transformation $\Phi^{(q)}$ and a transformed hamiltonian $H^{(q)} = h^{(q)} + R^{(q)}$.
The new domains $G_q \subset G_{q-1}'$ are going to be specified in the following lines. So now we are going to prove \ref{eq:induction1},\ref{eq:induction2},\ref{eq:induction3},\ref{eq:induction4},\ref{eq:induction5}.

\begin{itemize}
\item \ref{eq:induction1}. We want to see $\varepsilon_q := \|DR^{(q)}\|_{G_q, \rho^{(q)}, c_{q+1}} \leq \frac{8 \varepsilon}{\nu \rho_1 2^(2\tau+2)q}$.

By the second result of proposition \ref{lemma:inductive} we have:

\begin{equation}\label{eq:1qtobound}
\varepsilon_q \leq e^{-K_q \delta_1^{(q)}}\varepsilon_{q-1} + \frac{14 A_q K_q^\tau}{\beta_{q-1}'\delta_2^{(q)}} \varepsilon_{q-1}^2.
\end{equation}

Now we are going to bound each term of the right hand of the expression at a time.

Recall that $\delta_1^{(q)} \geq \frac{\nu \rho_1}{8 2^{\nu(q-1)}}$.

$$
\begin{array}{rcl}
  K_q \delta_1^{(q)} &  \geq &  K 2^{q-1} \frac{\nu \rho_1}{8} 2^{-\nu(q-1)} \\
  &  =  &  \frac{\nu \rho_1}{8}K 2^{(1-\nu)(q-1)} \\
  &  \geq &  \frac{12(\tau+2)}{8} 2^{(1-\nu)(q-1)} \\
  &  =  &  (3/2\tau + 3)2(1-\nu)(q-1) \\
  &  \geq &  (2\tau+3)\frac{3}{4} \geq (2\tau+3)\ln 2, \\
\end{array}
$$

where we used that $K\hat\rho\geq 1$ and hence $K \geq \frac{12(\tau+2)}{\nu\rho_1}$.
So we conclude that $e^{-K_q \delta_1^{(q)}} \leq \frac{1}{2^{2\tau+3}}$, and we have bounded the first term of \ref{eq:1qtobound}. Let us bind the second one.

On one hand, we have that

$$\frac{14 A_q K_q^\tau}{\beta_{q-1}'} \leq \frac{14\cdot5 K_q^\tau}{\frac{\nu \beta}{4}} \leq \frac{2^9 K_q^2}{\nu \beta} $$%\textcolor{black}{\leq \frac{2^8 K_q^2}{\nu \beta}},

where we have used that $\beta_q'\geq \frac{\nu\beta}{4}$ and $ A_q \leq 5$.

Now we are going to apply that $\varepsilon_{q-1} \leq \frac{8 \varepsilon}{\nu \rho_1 2^{(2\tau+2)(q-1)}}$, $\delta_2^{(q)}\geq \frac{\nu \beta}{64 M K_q^{\tau+1}}$ and $\epsilon \leq \frac{\nu^3 \rho_1 \beta^2}{2^{2\tau+22} M K^{2\tau+1}}$ to obtain

\begin{longtable}{rcl}
 $\ \frac{14 A_q K_q^\tau}{\beta_{q-1}' \delta_2^{(q)}} \varepsilon_{q-1}$ & $ =$ & $ \frac{14 A_q K_q^\tau}{\beta_{q-1}'}\frac{1}{\delta_2^{(q)}}\varepsilon_{q-1}$ \\
 & $ \leq$ & $ \frac{2^9 K_q^\tau}{\nu\beta} \frac{64 M K_q^{\tau+1}}{\nu \beta}\frac{8\varepsilon}{\nu \rho_1 2^{(2\tau+2)(q-1)}}$ \\
  & $ \leq$ & $ \frac{2^{18} M K_q^{2\tau+1}}{\nu^3 \beta^2 \rho_1 2^{(2\tau+2)(q-1)}}\frac{\nu^3 \rho_1 \beta^2}{2^{2\tau+22}M K^{2\tau + 1}}$ \\
  & $ \leq$ & $ 2^{18} 2^{(q-1)(2\tau+1)-(2\tau+2)(q-1)-(2\tau+22)}$ \\
  & $ =$ & $ 2^{(1-q)}2^{-2\tau-4} = \frac{1}{2^{2\tau+3}2^{q-1}}.$ \\
\end{longtable}
This gives us the bound of the second term of \ref{eq:1qtobound}. Now we put both bounds together:

$$\varepsilon_q \leq \frac{1}{2^{2\tau+3}}\varepsilon_{q-1} + \frac{1}{2^{2\tau+3}}\frac{1}{2^{q-1}}\varepsilon_{q-1} \leq \frac{1}{2^{2\tau+2}}\varepsilon_{q-1}.$$

That implies $\varepsilon_q \leq \frac{\epsilon}{2^{(2\tau+2)(q-1)}}$
as we wanted. \textcolor{black}{Because we can assume $\nu\rho_1 \leq 1$}.

\item \ref{eq:induction2}

Let us write $\sigma_2^{(q)} = \rho_2^{(q-1)} - \delta_2^{(q)}/2 = \rho_2^{(q)} + \delta_2^{(q)}/2 \geq \rho_2^{(q)}$, then $\eta_q = |R_0^{(q)}|_{G_q, \rho_2^{(q)}} \leq |R_0^{(q)}|_{G_q,\sigma_2^{(q)}}$.

By the inductive lemma \ref{lemma:inductive}:

\begin{longtable}{rcl}
 $\eta_q$ & $\leq$ & $\frac{7 A_q K_q^\tau}{c_q \beta_{q-1}'} \varepsilon_{q-1}^2 $\\
 & $\leq$ & $\frac{7 A_q K_q^\tau}{\beta_{q-1}'}\varepsilon_{q-1}^2 \frac{\delta_1^{(q)}}{\delta_2^{(q)}}$\\
 & $=$ & $\frac{14 A_q K_q^\tau}{\beta_{q-1}' \delta_2^{(q)}}\varepsilon_{q-1}^2\frac{\delta_1^{(q)}}{2}$\\
 & $\leq$ & $\frac{1}{2^{2\tau+3}2^{q-1}}\varepsilon_{q-1}\frac{\delta_1^{q)}}{2}$\\
 & $\leq$ & $\frac{1}{2}\frac{\delta_1^{(q)}}{2^{2\tau+3}2^{q-1}}\frac{\varepsilon}{\nu\rho_1 2^{(2\tau+2)(q-1)}}$ \\
 & $\leq$ & $\frac{1}{2}\frac{\nu \rho_1}{4\cdot 2^{\nu(q-1)}}\frac{1}{2^{2\tau+3}2^{q-1}}\frac{8\varepsilon}{\nu\rho_1 2^{(2\tau+2)(q-1)}}$ \\
 & $\leq$ & $\frac{\varepsilon}{2^{(2\tau + 3)q}}$. \\
\end{longtable}

For the second part we only need to apply  Cauchy inequalities:

$$\xi_q \leq \frac{2}{\delta_2^{(q)}} |R_0^{q}|_{G_q,\rho_2^{(q)}} \leq \frac{2}{\delta_2^{(q)}}\frac{\varepsilon}{2^{(2\tau+3)q}}.$$

\item \ref{eq:induction3} and \ref{eq:induction4} are direct from lemma \ref{lemma:inductive}.
\item \ref{eq:induction5} We need to consider again the results from lemma \ref{lemma:inductive} with $F_q$ as $F$.
We have to check the condition $F_q \subset  F_{q-1}' - \textcolor{black}{\frac{4M_{q-1}\varepsilon_{q-1}}{\mu_q}}$.
Let us define $d_q:=\frac{\beta_q - \beta_{q-1}}{2 K_q^{\tau+1}}$.
Using that $F_{q-1}' := (F-\beta_{q-1}) \setminus \bigcup_{\substack{k\in\mathbb{Z}^n\setminus\{0\} \\ |k|_1 \leq K}} \Delta_{c,\hat q}(k.\frac{\beta_{q-1}'}{|k|_1^\tau})$ we have
$$F_{q-1}' - d_q \supset (F-(\beta_{q-1}+d_q))\setminus \bigcup_{\substack{k\in\mathbb{Z}^n\setminus\{0\} \\ |k|_1 \leq K}}\Delta_{c,\hat q}(k,\frac{\beta_{q-1}'}{|k|_1^\tau} + |k|d_q).$$

Moreover,

$$
\left\{
\begin{array}{rcl}
\beta_{q-1} + d_q & \leq & \beta_q, \text{ and}\\
\frac{\beta_{q-1}'}{|k|_1^\tau} + |k|d_q & = & \frac{\beta_{q-1}' + |k|_1^\tau|k|\frac{\beta_q - \beta_{q-1}}{2 K_q^{\tau+1}}}{|k|_1^\tau} \\
& \leq & \frac{\beta_{q-1}' + K_q^{\tau+1}\frac{\beta_q - \beta_{q-1}}{2 K_q^{\tau+1}}}{|k|_1^\tau} = \frac{\beta_{q-1}' + \frac{\beta_q}{2} - \frac{\beta_{q-1}}{2}}{|k|_1^\tau} = \frac{\beta_q}{|k|_1^\tau}.
\end{array}
\right.
$$

Now if we see that $\frac{4 M_{q-1} \varepsilon_{q-1}}{\mu_q} \leq d_q$ we will have the inclusion we want.
Observe that $\frac{4 M_{q-1}}{\mu_q} \leq \frac{4\cdot 2M}{\mu/2} = \frac{16M}{\mu}$. So, it is enough to check that $\frac{16M}{\mu} \varepsilon_{q-1} \leq d_q$.

\begin{longtable}{rcl}
$\varepsilon_{q-1}$ & $\leq$ & $\frac{8\varepsilon}{\nu \rho_1 2^{(2\tau+2)q}}$ \\
& $\leq$ & $\frac{8 \nu^3 \rho_1 \beta^2}{\nu \rho_1 2^{(2\tau+2)q}2^{(2\tau+22)}M K^{2\tau+1}}$\\
& $\leq$ & $\frac{8 \nu^2 \beta^2}{2^{(2\tau+2)q + 2\tau + 20}M K^{2\tau+1}}$ \\
& $\leq$ & $\frac{8 \nu^2 \beta\frac{2(\beta_q-\beta_{q-1})}{\nu}}{2^{(\tau+1)q + (\tau+1)q + 2\tau + 20} M K^{2\tau+1}}$\\
& $=$ & $\frac{8\nu\beta 2(\beta_q - \beta_{q-1})}{2^{(\tau+1) + (\tau+1)q + 2\tau + 20}M K_{q}^{\tau+1}K^\tau}$\\
& $=$ & $\frac{8\nu\beta 2}{2^{(\tau+1) + (\tau+1)q + 2\tau + 19}M K^\tau}\frac{(\beta_q-\beta_{q-1})}{2 K_q^{\tau+1}}$\\
& $=$ & $\frac{\nu \beta}{2^{(\tau+1) + (\tau+1)q + 2\tau + 15} M K^\tau} d_q$\\
& $\leq$ & $\frac{\nu\beta}{2^{3\tau+16}M K^\tau}dq.$
\end{longtable}

Hence, it is enough to prove the following:

$$\frac{16 M}{\mu}\frac{\nu \beta}{2^{3\tau+16} M K^\tau}d_q \leq d_q.$$

Wich holds if and only if

$$\frac{16 M}{\mu}\frac{\nu\beta}{2^{\tau+16} M K^\tau} \leq 1 \Leftrightarrow K^\tau \geq \frac{\nu\beta}{\mu 2^{\tau+12}},$$

which we assumed when choosing $K$.

\end{itemize}

 \end{itemize}

\item Convergence of diffeomorphisms

Now we are going to prove the convergence of the successive maps

$u^{(q)}: G_q \rightarrow F_q$

i.e. we want to see that exist proper sets $G^*, F^*$ and an analytical map $u^*$ such that $u^{(q)}:G_q \rightarrow F_q$ converge to $u^*: G^* \rightarrow F^*$.

Let us use lemma \ref{lemma:inductive} as before.

For $q \geq 1$ we obtain

$$|u^{(q)} - u^{(q-1)}|_{G_q} \leq \xi_q \quad \text{ and } \quad |(u^{(q)})^{-1} -(u^{(q-1)})^{-1}|_{F_q}\leq \frac{\varepsilon_q}{\mu_q}.$$

Now, because the following two inequalities hold

$$
\left\{
\begin{array}{rcl}
\xi_q & \leq & \frac{4M K^{\tau+1}\varepsilon}{\nu \beta 2^{(\tau+2)q}}\\
\frac{\varepsilon_q}{\mu_q} & \leq & \frac{8\varepsilon}{\nu \rho_1 2^{2\tau+2}q} \frac{1}{(1+\frac{1}{2^q})\frac{\mu}{2}} = \frac{8\varepsilon 2^{q-1}}{\nu \beta 2^{(2\tau+2)q}(2^q+1)\mu}
\end{array}
\right.
$$

the sequences $u^{q}$ and $(u^{(q)})^{-1}$ converge to maps $u^*$ and $\Upsilon$ respectively.
These maps are defined on the following sets:

$$
\begin{array}{rcl}
G^* & := & \bigcap_{q\geq 0} G_q,\\
F^* & := & \bigcap_{q\geq 0} F_q = (F - \beta)\setminus \bigcup_{\substack{k\in\mathbb{Z}^n\setminus\{0\} \\ |k|_1 \leq K}}\Delta_{c,\hat q}(k,\frac{\beta}{|k|_1^\tau}).\\

\end{array}
$$

The second equality holds because $F^*$ is a compact for being the intersection of compact sets. We can now deduce that

\begin{longtable}{rcl}
$|u^* - u^{(q)}|_{G^*}$ & $\leq$ & $\sum_{s\geq q} |u^{(q)} - u^{(q-1)}|_{G^*}$ \\
& $\leq$ & $\sum_{s\geq q} |u^{(q)} - u^{(q-1)}|_{G}$ \\
& $\leq$ & $\sum_{s\geq q} \xi_q$. \\

\end{longtable}

with the same argument we see that $|\Upsilon - (u^{(q)})^{-1}|_{F^*} \leq \ldots \leq \sum_{s \geq q} \frac{\varepsilon_q}{\mu_q}$.

The next steps are going to be to prove that $G_q \subset G_{q-1} - \frac{2\varepsilon_{q-1}}{\mu_{q-1}}$ and $F_q \subset F_{q-1} - \frac{4 M_{q-1}\varepsilon_{q-1}}{\mu_{q-1}}$. If we check it and take the limit we would have:
$$G^* \subset G_q - \sum_{s\geq q} \frac{2\varepsilon_q}{\mu_q} \quad \text{ and } \quad F^* \subset F_q - \sum_{s\geq q} \frac{4 M_q \varepsilon_q}{\mu_q}.$$

Let us first check $F_q \subset F_{q-1} - \frac{4 M_{q-1} \varepsilon_{q-1}}{\mu_{q-1}}$. Let us define $x := \frac{4M_{q-1}}{\mu_{q-1}}$.

$$
\begin{array}{rcl}
 F_{q-1} -x& \supset & (F-(\beta_{q-1}+x))\setminus\bigcup_{\substack{k\in\mathbb{Z}^n\setminus\{0\} \\ |k|_1 \leq K}} \Delta_{c,\hat q}(k,\frac{\beta_{q-1}}{|k|_q^\tau} + |k|x) \\
 & \supset & (F-(\beta_{q-1}+x))\setminus\bigcup_{\substack{k\in\mathbb{Z}^n\setminus\{0\} \\ |k|_1 \leq K}} \Delta_{c,\hat q}(k,\frac{\beta_{q-1} + K_q^{\tau+1}x}{|k|_q^\tau}). \\
\end{array}
$$

To have the inclusion we want, we have to check that:

\begin{enumerate}
\item $\beta_{q-1}+x \leq \beta_q.$
\item $\frac{\beta_{q-1} + K_q^{\tau+1}x}{|k|_1^\tau} \leq \frac{\beta_q}{|k|_1^\tau} \Leftrightarrow \beta_{q-1} + K_q^{\tau+1}x.$
\end{enumerate}

Since the second one implies the first we will only check the second one.

\begin{longtable}{rcl}
$\beta_{q-1} + K_q^{\tau+1}x$ & $=$ & $\beta_{q-1} + K_q^{\tau+1}\frac{4M_{q-1}\varepsilon_{q-1}}{\mu_{q-1}}$\\
& $\leq$ & $\beta_{q-1} + K_q^{\tau+1}\frac{16M\varepsilon_{q-1}}{\mu}$\\
& $\leq$ & $\beta_{q-1} + K_q^{\tau+1}d_q$\\
& $=$ & $\beta_{q-1} + K_q^{\tau+1}\frac{\beta_q-\beta_{q-1}}{2K_q^{\tau+1}}$\\
& $=$ & $\beta_{q-1} - \beta_{q-1}/2 + \beta_{q}/2$\\
& $=$ & $\frac{\beta_{q-1} + \beta_{q}}{2}$\\
& $=$ & $\beta_{q}$\\
\end{longtable}

Where we have used that $16M\varepsilon_q/\mu\leq d_q$ and that $\beta_q$ is monotonically increasing with $q$.

The inclusion $G_q \subset G_{q-1} - \frac{2\varepsilon_{q-1}}{\mu_{q-1}}$ is given as a result of the lemma \ref{lemma:inductive}.

So we proved what we wanted. We are now going to see that $u^*$ is one-to-one on $G^*$ and that $u^*(G^*) = F^*$.

Tale $I\in G^*$, we have that $u^{(q)}(I)\in F_q$ for every $q$. Hence $u^*(I) \in F^*$, and we deduce that $u^*(G^*) \subset F^*$. With the same argument, we see $\Upsilon(F^*) \subset G^*$.
Let us prove that $\Upsilon(u^*(I))=I$.

$$
\begin{array}{rcl}
 |\Upsilon(u^*(I)) - I|& \leq & |\Upsilon(u^*(I)) - (u^{(q)})^{-1}(u^* (I)) \\
 & & \quad + (u^{(q)})^{-1}(u^*(I)) - (u^{(q)})^{-1}(u^{(q)}(I))|\\
 & \leq & |\Upsilon(u^*(I)) - (u^{(q)})^{-1}(u^* (I))| \\
 & & \quad + |(u^{(q)})^{-1}(u^*(I)) - (u^{(q)})^{-1}(u^{(q)}(I))|\\
 & \leq & |\Upsilon - (u^{(q)})^{-1}|_{F^*} + \frac{1}{\mu_q}|u^* -u^{(q)}|_{G^*}.\\
\end{array}
$$

Where to bound the second term we used the mean value theorem, i.e. $|u^{(q)}(x) - u^{(q)}(y)|_{G_q} \leq |\frac{\partial}{\partial I} u^{(q)}|_{G_q}|x-y|$, and the fact that because of the $\mu_q$-non-degeneracy, $|\frac{\partial u^{(q)}}{\partial I}| \geq \mu_q|v|$, $\forall v \in \mathbb{R}^n$ and $\forall I' \in G_q$.
Note that we can use the mean value theorem because $u^*(I) - u^{(q)}(I)$ belongs to $F_q$ because $\frac{4 M_q \varepsilon_q}{\mu_q} \geq \xi_q$. Let us prove this inequality. If we want to see $\frac{4 M_q \varepsilon_q}{\mu_q} \geq \xi_q$, it is enough to see $\frac{4 M \varepsilon_q}{\mu} \geq \xi_q$.

\begin{longtable}{rcl}
$\xi_q$ & $\leq$ & $\frac{2}{\delta_2^{(q)}} |R_0^{(q)}|_{G_q,\sigma_2^{(q)}}$\\
& $\leq$ & $\frac{2}{\delta_2^{(q)}} \frac{\delta_1^{(q)}\varepsilon_{q-1}}{2}\frac{1}{2^{2\tau+3}2^{q-1}}$\\
& $=$ & $\frac{1}{c_q}\frac{1}{2^{2\tau+2}2^{q-1}}\varepsilon_{q-1} \leq \frac{4M}{\mu}\varepsilon_{q-1}$\\
\end{longtable}

The last inequality holds true if and only if

\begin{longtable}{rcl}
$\mu$ & $\leq$ & $\frac{\beta 2^{\nu(q-1)}2^{2\tau+3}2^{q-1}}{K_q^{\tau+1}\rho_1 4}$ \\
& $=$ & $\frac{\beta 2^{\nu(q-1)}2^{2\tau+3}2^{q-1}}{K^{\tau+1} 2^{(\tau+1)(q-1)}\rho_1 4}$ \\
& $\leq$ & $\frac{\beta 2^{2\tau+3}}{K^{\tau+1}\rho_1 4}$ \\
& $=$ & $\frac{\beta 2^{2\tau+1}}{K^{\tau+1}\rho_1}$ \\
& $\leq$ & $\frac{\beta 2^{2\tau+1}}{(\frac{1}{\nu \rho_1})^{\tau+1}\rho_1} = \beta \nu^{\tau+1}2^{2\tau+1}\rho_1^tau$ \\
\end{longtable}

as we assumed in the statement of the theorem.
Since the bound obtained tends to 0, we have $\Upsilon(u^*(I)) = I$ and hence $u^*$ is one-to-one. Analogously we obtain $u^*(\Upsilon(J)) = J \quad \forall J \in F^*$. Finally, $u^*$ is one-to-one and $u^*(G^* ) = F^*$. Note also that from the inductive lemma we obtain $|h^{(q)}-h^{(q-1)}|_{G_q,\rho_2^{(q-1)}}\leq \eta_{q-1}$.
Also, observe the following bound that we are going to use in the next sections.

$$|u^* - u^{(q)}|_{G^*} \leq \sum_{s\geq q} \frac{4 M K^{\tau+1}\varepsilon}{\nu\beta 2^{(\tau+2)s}}.$$

\item Convergence of the canonical transformations

Let $\sigma^{(q)} = \rho^{(q-1)} - \delta_2^{(q)}/2$. Observe that this definition implies that $\sigma^{(q)} - \rho^{(q)} = \delta_2^{(q)}$ and $\sigma^{(q)} - \delta_2^{(q)} =  \rho^{(q)}$.
Observe that applying the inductive lemma \ref{lemma:inductive}:

\begin{longtable}{rcl}\label{eq:conv_can_trans}
$|\Phi^{(q)} - \text{id}|_{G_q,\sigma^{(q)}, c_q}$ & $\leq $ & $\frac{2A_{q-1} K_q^{\tau}}{\beta_{q-1}'}\varepsilon_{q-1}$ \\
& $\leq$ & $\frac{2\cdot 5 \cdot 4}{\nu\beta}\frac{8\varepsilon}{\nu\rho_1 2^{(2\tau+2)(q-1)}}$ \\
& $\leq$ & $\frac{2^9 K^\tau \varepsilon}{\nu^2 \rho_1 \beta 2^{(\tau+2)(q-1)}}$ \\
& $\leq$ & $\frac{2^9 K^\tau \nu^3 \rho_1 \beta^2}{\nu^2 \rho_1 \beta 2^{(\tau+2)(q-1)}2^{2\tau+22}M K^{2\tau+1}}$ \\
& $\leq$ & $\frac{2^9 \nu \beta}{2^{(\tau+2)(q-1)}2^{2\tau+20} M K^{\tau+1}}$ \\
& $=$ & $\frac{\nu\beta}{2^6 M (K 2^{q-1})^{\tau+1}}\frac{2^9}{2^{(q-1)}2^{2\tau+14}}$ \\
& $\leq$ & $\delta_2^{(q)}\frac{1}{2^{(q-1)} 2^{2\tau+5}}$ \\
& $\leq$ & $\frac{\delta_2^{(q)}}{2^{(q-1)} 32},$ \\
\end{longtable}

where we have used that $\delta_2^{(q)}\geq \frac{\nu \beta}{84 M K_qP {\tau+1}}$, $\varepsilon \leq \frac{\nu^3 \rho_1 \beta^2}{2^{2\tau+20}M K^{2\tau+1}}$, $\beta \leq \frac{8M K^{\tau+1}\rho_2}{\nu}$ and $\beta_{q-1}' \geq \frac{\nu \beta}{4}$.

Now, recall that $\hat\delta_c = \min(c\delta_1, \delta_2)$, then $\hat{\delta}_{c_q} = \min(c_q \delta_1^{(q)}, \delta_2^{(q)}) = \min(\delta_2^{(q)},\delta_2^{(q)}) = \delta_2^{(q)}$.

Now using that $|D \Upsilon|_{G,\rho-\delta,c} \leq \frac{|\Upsilon|_{G,\rho,c}}{\hat\delta_c}$, we can obtain:

\begin{longtable}{rcl}
$|D\Phi^{(q)} - \text{Id}|_{G_q,\rho^{(q)}, c_q}$ & $=$ & $|D(\Phi^{(q)}) - \text{id}|_{G_q, \rho^{(q)}, c_q}$ \\
 & $\leq$ & $|D(\Phi^{(q)}) - \text{id}|_{G_q, \sigma^{(q)} - \delta_2^{(q)}, c_q}$ \\
 & $\leq$ & $\frac{|\Phi^{(q)}-\text{id}|_{G_q,\sigma^{(q)},c_q}}{\hat\delta_{c_q}}$ \\
 & $\leq$ & $\frac{|\Phi^{(q)}-\text{id}|_{G_q,\sigma^{(q)},c_q}}{\delta^{(q)}_{2}}$ \\
 & $\leq$ & $\frac{2 |\Phi^{(q)} - \text{id}|_{G_q,\sigma^{(q)},c_q}}{\delta_2^{(q)}}$ \\
 & $\leq$ & $\frac{2}{\delta_2^{(q)}}\frac{\delta_2^{(q)}}{2^{(q-1)}\cdot 32} \leq \frac{1}{2^{q-1}16} \leq \frac{1}{2^{(q-1)}4}$ \\
\end{longtable}

Let $x,y$ be such that the segment joining them is contained in $\mathcal{D}_{\rho^{(q)}}(G_q)$. Using the mean value theorem one can deduce the following bound:

$$|\Phi^{q}(x) - \Phi^{q}(y)|_{c_q} \leq |D \Phi^{(q)}|_{G_q,\rho^{(q)}, c_q}\cdot|x-y|_{c_q}.$$

By \ref{eq:conv_can_trans}, in particular $|\Phi^{(q)}(x)-x|_{c_q} \leq \delta_2^{q}$ and $|\Phi^{(q)}(y)-y|_{c_q} \leq \delta_2^{q}$. Then the segment that join $\Phi^{(q)}(x)$ and $\Phi^{(q)}(y)$ is contained in $\mathcal{D}_{\rho^{(q-1)}}(G_{q-1}) = \mathcal{D}_{\rho^{(q)}+\delta^{(q)}}$, because $G_q \subset G_{q-1} - \frac{2 \varepsilon_{q-1}}{\mu_{q-1}}$ and because $\rho^{(q)}-\rho^{(q-1)} \leq \delta_2^{(q)}$ because $\rho^{(q)}-\rho^{(q-1)} = \delta_2^{(q)}$.

Therefore we can apply the mean value theorem once again:

$$
\begin{array}{lcl}
|\Phi^{(q-1)}(\Phi^{(q)}(x)) - \Phi^{(q-1)}(\Phi^{(q)}(y))|_{c_{q-1}} \\
  \leq  |D\Phi^{(q-1)}|_{G_{q-1},\rho^{q-1},c_{q-1}}|\Phi^{(q)}(x) - \Phi^{(q)}(y)|_{c_{q-1}}\\
  \leq  2^{\tau+1-\nu}|D\Phi^{(q-1)}|_{G_{q-1},\rho^{q-1},c_{q-1}}|\Phi^{(q)}(x) - \Phi^{(q)}(y)|_{c_{q}},\\
\end{array}
$$

where we have used that $c_{q-1}/c_q = \frac{\delta_2^{(q-1)}/\delta_1^{(q-1)}}{\delta_2^{(q)}/\delta_1^{(q)}} = \frac{\delta_2^{(q-1)}}{\delta_2^{(q)}}\frac{\delta_1^{(q)}}{\delta_1^{(q-1)}} = 2^{\tau+1}\frac{1}{2^\nu} = 2^{\tau+1-\nu}$.

Using the previous bounds and iterating by $q$, we obtain the following:

\begin{longtable}{lcl}
$|\Psi^{(q)}(x) - \Psi^{(q)}(y)|_{c_1}$ \\
 $\leq$  $2^{(\tau+1-\nu)(q-1)}|D\Phi^{(1)}|_{G_1,\rho^{(1)},c_1}\cdot\ldots\cdot|D\Phi^{(q)}|_{G_q,\rho^{(q)},c_q}|x-y|_{c_q}$ \\
 $\leq$  $2^{(\tau+1-\nu)(q-1)}(1+\frac{1}{4})(1+\frac{1}{4\cdot 2})\cdot\ldots\cdot(1 + \frac{1}{4\cdot 2^{q-1}})|x-y|_{c_q}$\\
 $\leq$  $2^{(\tau+1-\nu)(q-1)}e^{1/2}|x-y|_{c_q} \leq 2^{(\tau+1-\nu)(q-1)}\cdot 2|x-y|_{c_q}.$\\
\end{longtable}

Which holds for $q\geq 1$ and for every $x,y$ such that the segment joining them is contained in $\mathcal{D}_{\rho^{(q)}}(G_q)$. Now, given $q\geq 2$ and $x\in\mathcal{D}_{\rho^{(q)}}(G_q)$ let $y = \Phi^{(q)}(x)$:

\begin{longtable}{rcl}\label{eq:conv_can_trans2}
$|\Psi^{(q)}(x) - \Psi^{(q-1)}(x)|_{c_1}$ & $=$ & $|\Psi^{(q-1)}(\Phi^{(q)}(x)) - \Psi^{(q-1)}(x)|_{c_1}$ \\
& $\leq$ & $2^{(\tau+1+\nu)(q-2)}2|\Phi^{(q)}(x) - x|_{c_{q-1}}$\\
& $\leq$ & $2^{(\tau+1+\nu)(q-1)}2|\Phi^{(q)}(x) - x|_{c_{q}}$\\
& $\leq$ & $2^{(\tau+1+\nu)(q-1)}2\delta_2^{(q)}$\\
& $\leq$ & $2^{(\tau+1+\nu)(q-1)}2\frac{2^8 K^\tau \varepsilon}{\nu^2\rho_1\beta 2^{(\tau+2)(q-1)}}$\\
& $=$ & $\frac{2^9 K^\tau \varepsilon}{\nu^2 \rho_1 \beta 2^{(1+\nu)(q-1)}}.$\\
\end{longtable}

Which holds even for $q=1$ by setting $\Psi^{(0)} = \text{id}$ by \ref{eq:conv_can_trans}. Hence \ref{eq:conv_can_trans2} implies that $\Psi^{(q)}$ converges to a map

$$\Psi^*: D_{(\rho_1/4, 0)}(G^*) = \mathcal{W}_{\frac{\rho_1}{4}}(\mathbb{T}^n)\times G^* \rightarrow \mathcal{D}_\rho(G).$$

And we deduce for every $q\geq 0$ that

$$|\Psi^* -\Psi^{(q)}|_{G^*,(\frac{\rho_1}{4},0),c_1} \leq \frac{2^10 K^\tau \varepsilon}{\nu^2 \rho_1 \beta 2^{(1+\nu)q}}.$$

Moreover by taking the limit to the equation

$$H\circ \Psi^{(q)} = h^{(q)} + R^{(q)}$$

we see that $H\circ \Psi^* = h^*(I)$ on $\mathcal{D}_{(\frac{\rho_1}{4},0)}(G^*)$.

\item Stability estimates

Next we see that for $q\rightarrow \infty$, the motions associated to the transformed hamiltonian $\hat H^{(q)} = \hat h^{(q)} + R^{(q)}$ and the quasi-periodic motions of $\hat h^{(q)}$ get closer and closer.

Let us denote
$$
\left\{
\begin{array}{rcl}
x^{(q)}(t) = ( \phi^{(q)}(t), I ^{(q)}(t)) & \quad & \text{the trajectory of }  H^{(q)},\\
\hat x^{(q)}(t) = (\hat \phi^{(q)}(t),\hat I ^{(q)}(t)) & \quad & \text{the trajectory of } \hat H^{(q)}\\
\end{array}
\right.
$$

corresponding to a given initial condition $x^{(q)}(0) = x_0^* = (\phi_0^*, I_0^* ) \in \mathbb{T}^n\times G_q$. Let

$$
\left\{
\begin{array}{rcl}
\tilde x^{(q)}(t) & := & (\tilde \phi^{(q)}(t), I_0^* ) = (\phi_0^*  + u^{(q)}(I_0^*))t, I_0^*,\\
\hat{ \tilde x}^{(q)}(t) & := &  (\hat{\tilde \phi}^{(q)}(t), I_0^* ) = (\phi_0^*  + u'^{(q)}(I_0^*))t, I_0^*\\
\end{array}
\right.
$$

the corresponding trajectories of the integrable parts of $h^{(q)}$ and $\tilde h^{(q)}$ respectively. Recall that $\hat h^{(q)}(I) = h^{(q)}(I) + \zeta^{(q)}(I_1) = h^{(q)}(I) + q_0 \log(I_1) + \sum_{i=1}^{m-1} q_i\frac{1}{I_1^i}$ and $u'^{(q)} = \bar{\mathcal{B}}u^{(q)}+ \bar{\mathcal{A}}(I_1)$. It is clear that $\tilde x^{(q)}(t)$ and $\hat{\tilde x}^{(q)}(t)$ are defined for all $t\in\mathbb{R}$.

Let us denote:

$ T_q = \inf\{t>0: | I^{(q)}(t) - I_0^* | > \delta_2^{(q+1)} \text{ or } | \phi^{(q)}(t) - {\tilde \phi}^{(q)}(t)|_\infty > \delta_1^{(q+1)}\}.$
$\hat T_q = \inf\{t>0: |\hat I^{(q)}(t) - I_0^* | > \delta_2^{(q+1)} \text{ or } |\hat \phi^{(q)}(t) - \hat{\tilde \phi}^{(q)}(t)|_\infty > \delta_1^{(q+1)}\}.$

Observe that $x^{(q)}(t)$ and $\hat x^{(q)}(t)$ are defined and belong do $\mathcal{D}_{\rho^{(q)}}(G_q)$, for $0\leq t\leq T_q$ and $0\leq t \leq \hat T_q$ respectively, because $\delta^{(q)} \leq \rho^{(q)}$. Also recall the Hamiltonian equations. Let us first state the motion equations for our Hamiltonian function $\hat H^{(q)}$:

$$ \iota_{X_{\hat H^{(q)}}} \omega = d\hat H^{(q)} , \quad \text{ or } \quad X_{\hat H^{(q)}}  = \Pi(d \hat H^{(q)}, \cdot).$$

Let us write

$$X_{\hat H^{(q)}} = \dot{\hat{I}}_1^{(q)}\frac{\partial}{\partial I_1} + \ldots \dot{\hat{I}}_n^{(q)}\frac{\partial}{\partial I_n} + \dot{\hat{\phi}}_1^{(q)}\frac{\partial}{\partial \phi_1} + \ldots + \dot{\hat{\phi}}_n^{(q)}\frac{\partial}{\partial \phi_n}.$$

Moreover

$$
\begin{array}{rcl}
d \hat H^{(q)} & = & d \hat h^{(q)} + dR^{(q)} \\
 & = & d\zeta^{(q)} + dh^{(q)} + dR^{(q)} \\
 & = & \sum_{i=1}^n \frac{\partial \zeta^{(q)}}{\partial I_i} +
\underbrace{\sum_{i=1}^n \frac{\partial \zeta^{(q)}}{\partial \phi_i}}_{=0} +
\sum_{i=1}^n \frac{\partial h^{(q)}}{\partial I_i} \\
&& \quad +
\underbrace{\sum_{i=1}^n \frac{\partial h^{(q)}}{\partial \phi_i}}_{=0} +
\sum_{i=1}^n \frac{\partial R^{(q)}}{\partial I_i} +
\sum_{i=1}^n \frac{\partial R^{(q)}}{\partial \phi_i}.
\end{array}
$$

Recall

$$\omega = \left(\sum_{j=1}^m\frac{c_j}{I_1^j}\right)dI_1\wedge d\phi_1 + \sum_{i=2}^{n} dI_i \wedge d\phi_i, $$
$$ \Pi = \frac{1}{\left(\sum_{j=1}^{m}\frac{c_j}{I_1^J}\right)}\frac{\partial}{\partial I_1} \wedge \frac{\partial}{\partial \phi_1} + \sum_{i=2}^{n} \frac{\partial}{\partial I_i}\wedge\frac{\partial}{\partial \phi_i}.$$

Then:

$$
\left\{
\begin{array}{rcl}
 \dot{\hat{I}}_j^{(q)}  & = &  -\frac{\partial R^{(q)}}{\partial \phi_j}(\hat x^{(q)}(t)), \quad \text{ if }  j \neq 1 \text{ and }\\
 \dot{\hat{I}}_1^{(q)} & = & -\frac{1}{\left(\sum_{i=1}^{m}\frac{c_i}{I_1^i}\right)}\frac{\partial R^{(q)}}{\partial \phi_j}(\hat x^{(q)}(t)) = -\mathcal{B}(I_1) \frac{\partial R^{(q)}}{\partial \phi_1}(\hat x^{(q)}(t)).\\
\end{array}
\right.
$$

Observe that

\begin{equation}\label{eq:stab_est_1}
|\dot{\hat I}_1^{(q)}(t)| \leq \left|\frac{\partial R^{(q)}}{\partial \phi_1}(\hat x^{(q)}(t))\right|.
\end{equation}

Moreover,

$$
\left\{
\begin{array}{rcl}
 \dot{\hat \phi}_j^{(q)} & = & \hat{u}_j^{(q)}(\hat I^{(q)}(t)) + \frac{\partial R^{(q)}}{\partial I_j}(\hat x^{(q)}(t)) \\ 
 &=& u_j^{(q)}(\hat I^{(q)}) + \frac{\partial R^{(q)}}{\partial I_j}(\hat x^{(q)}(t)) \quad \text{ if }  j \neq 1 \\
 \dot{\hat{\phi}}_1^{(q)} & = & \underbrace{(\mathcal{B}(I_1)u_1^{(q)} + \mathcal{A}(I_1))}_{u_1^{(q)}}(\hat I^{(q)}(t)) + \mathcal{B}(I_1)\frac{\partial R^{(q)}}{\partial I_1}(\hat x^{(q)}(t)), \\
\end{array}
\right.
$$

where we have used that $\hat u_j^{(q)} = u_j'^{(q)}$ if $j \neq 1$.
Using \ref{eq:stab_est_1} we obtain

$$|\dot{\hat I}^{(q)}(t)| \leq \left\|\frac{\partial R^{(q)}}{\partial \phi}\right\|_{G_q,\rho^{(q)}} \leq \varepsilon_q.$$

Hence,

\begin{longtable}{rcl}
 $|\dot{\hat \phi}^{(q)} - u'^{(q)}(I_0^*)|_\infty$ & $=$ & $|u'^{(q)}(\hat I^{(q)}(t)) + \bar{\mathcal{B}}\frac{\partial R^{(q)}}{\partial I_1}(\hat x^{(q)}(t)) - u'^{(q)}(I_0^*)|_\infty$\\
 & $\leq $ & $|u'^{(q)}(\hat I ^{(q)})(\hat I^{(q)}(t)) - u'^{(q)}(I_0^* )|_\infty + |\frac{\partial R^{(q)}}{\partial I}(\hat x^{(q)}(t))|_\infty$\\
 & $\leq$ & $M_q'|\hat I^{(q)}(t) - I_0^*| + \|\frac{\partial R^{(q)}}{\partial I}\|_{F_q, \rho^{(q)},\infty}$\\
 & $\leq$ & $M_q|\hat I^{(q)}(t) - I_0^*| + \frac{\varepsilon_q}{c_{q+1}}$\\
 & $\leq$ & $2M\delta_2^{(q+1)} + \frac{\varepsilon_q}{c_{q+1}} \leq 3M\delta_2^{(q+1)}.$\\
\end{longtable}
Where in the last bound we used that

\begin{equation}\label{eq:sta_est_2}
\frac{\varepsilon_q}{c_{q+1}} \leq M \delta_2^{(q+1)},
\end{equation}

that holds because:

\begin{longtable}{rcl}
$\frac{\varepsilon_q}{c_{q+1}}$ & $\leq$ & $\frac{16 M K_{q+1}^{\tau+1} \rho_1}{\beta 2^{\nu q}}\frac{8 \varepsilon}{\nu \rho_1 2^{(2\tau+2)q}}$ \\
 & $\leq$ & $\frac{16 M K_{q+1}^{\tau+1}\rho_1}{\beta e^{\nu q}}\frac{8}{\nu\rho_1 2^{(2\tau+2)q}}\frac{\nu^2 \mu^2 \beta^2}{2^{\tau+30}L^4 M^2 K^{2\tau+2}}$ \\
 & $\leq$ & $\frac{2^7 K_{q+1}^{\tau+1}\nu\mu^2\beta}{2^{(2\tau+2)q + \nu q + \tau + 30} L^4 M^2 K^{2\tau+2}}$\\
 & $\leq$ & $\frac{\nu\beta\mu^2}{K_{q+1}^{\tau+1} 2^{\nu q + \tau + 23} L^4 M^2}$\\
 & $=$ & $\frac{\mu^2}{2^{\nu q + \tau + 17} L^4 M}\frac{\nu\beta}{2^6 M K_{q+1}^{\tau+1}}$\\
 & $\leq$ & $\frac{\mu^2}{2^{\tau+17+\nu q} L^4 M}\delta_2^{(q+1)}$\\
 & $\leq$ & $\frac{2^{2\tau+12}L^4 M^2}{2^{\tau+17+\nu q} L^4 M}\delta_2^{(q+1)}$\\
 & $\leq$ & $2^{\tau-5-\nu q} \delta_2^{(q+1)}$\\
 & $\leq$ & $\frac{2^\tau}{2^{5+\nu q}} M \delta_2^{(q+1)}$ \\
 & $\leq$ & $M \delta_2^{(q+1)}\quad $ if $q$ is large enough.\\
\end{longtable}

Thus, since one of the inequalities defining $\hat T_q$ has to be an equality for $t = T_q$ we obtain,

$$
\begin{array}{rcl}
\delta_2^{(q+1)} & = & |\hat I^{(q)}(T_q) - I_0^*| \leq T_q \varepsilon_q, \quad \text{ or}\\
\delta_1^{(q+1)} & = & |\hat \phi^{(q)}(T_q) - \hat{\tilde \phi}^{(q)}(T_1)|_\infty \leq T_q 3M\delta_2^{(q+1)}.\\
\end{array}
$$

Hence, $\hat T_q \geq \min(\frac{\delta_2^{(q+1)}}{\varepsilon_q},\frac{\delta_1^{(q+1)}}{3 M \delta_2^{(q+1)}}) \geq \frac{1}{3 M c_{q+1}}$, where we used again \ref{eq:sta_est_2}.

Let us denote $T_q' := \frac{1}{3 M c_{q+1}}$, then $\hat T_q \geq T_q'$. This implies

$$|\hat x^{(q)}(t) - \hat{\tilde x}^{(q)}(t)|_{c_{q+1}} \leq \delta_2^{(q+1)} \quad \text{ for }  |t| \leq T_q'.$$

Since $\hat H^{(q)} = \hat H \circ \Psi^{(q)}$ and $\Psi^{(q)}$ is canonical it turns out that $\Psi^{(q)}(\hat x^{(q)}(t))$ is a trajectory of $\hat H$ defined for $t \leq T_q'$. It is important to observe that for $q$ big enough this trajectory remains near the torus $\Psi^{(q)}(\mathbb{T}^n\times \{I_0^*\})$. Moreover $T_q'$ tends to infinity when $q \rightarrow \infty$.

\item Invariant tori

Assume now that $x_0^* \in \mathbb{T}^n\times G^*$ and let us write

$$
\left\{
\begin{array}{rcl}
x^* (t) &  = &  (\phi^*_0 + u^*(I_0^*)t, I_0^*) \\
\hat x^* (t) & = & (\phi^*_0 + u'^*(I_0^*)t, I_0^*)
\end{array}
\right.
\quad \text{ for } t\in\mathbb{R}.
$$

Note that

$$
\begin{array}{rcl}
|\hat{\tilde{x}}^{(q)}(t) - \hat{x}^*(t)|_{c_{q+1}} & \leq & c_{q+1}|u'^{(q)}(I_0^*) - u'^* (I_0^*)|_\infty |t|\\
 & \leq & c_{q+1}|u'^{(q)} - u'^* |_{G^*, \infty} |t|.\\
\end{array}
$$

And observe that if $|t| \leq \frac{\delta_1^{(q+1)}}{|u'^{(q)} - u'^*|_{G^* ,\infty}} =: T_q''$ then,

$$
\begin{array}{rcl}
|\hat{\tilde{x}}^{(q)}(t) - \hat{x}^*(t)|_{c_{q+1}} & \leq &  c_{q+1}|u'^{(q)} - u'^* |_{G^*, \infty} \frac{\delta_1^{(q+1)}}{|u'^{(q)} - u'^*|_{G^* ,\infty}}\\
 & \leq& \frac{\delta_2^{q+1}}{\delta_1^{q+1}}\delta_1^{q+1} = \delta_2^{q+1}.\\
\end{array}
$$

Observe that 
$$|u'^*  - u'^{(q)}|_{G^*} = |\bar{\mathcal{B}} u^* +\bar{\mathcal{A}} - \bar{\mathcal{B}} u^{(q)} - \bar{\mathcal{A}}|_{G^*}$$
$$ = |\mathcal{B}(u^* - u^{(q)})|_{G^*} \leq |u^* -u^{(q)}|_{G^*},$$
 close enough to $Z$.

Hence the bound obtained for $|u^* -u^{(q)}|_{G^*}$ also holds for $|u'^* -u'^{(q)}|_{G^*}$.

$$|u'^* -u'^{(q)}|_{G^*} \leq \sum_{s\geq q} \frac{4 M K^{\tau+1}\varepsilon}{\nu \beta 2^{(\tau+2)s}} \leq \frac{8 M K^{\tau+1}\varepsilon}{\nu\beta 2^{(\tau+2)q}}.$$

Using this bound, we see that $T_q''$ tends to infinity because

$$T_q'' \geq \left(\frac{\nu \rho_1}{8\cdot 2^{\nu q}}\right)\left(\frac{\nu \beta 2^{(\tau+2)q}}{8 M K^{\tau+1}\varepsilon}\right) = \frac{\nu^2 \beta \rho_1}{64 M K^{\tau+1}\varepsilon} 2^{(\tau+2-\nu)q}.$$

Then

$$ |\hat x^{(q)}(t) - \hat x^* (t)|_{c_{q+1}} \leq |\hat x^{(q)}(t) - \hat{\tilde x}^{(q)}(t)|_{c_{q+1}} + |\hat{\tilde x}^{(q)}(t) - \hat x^* (t)|_{c_{q+1}} \leq 2 \delta_2^{(q+1)}.$$

when $t \leq T_q''':= \min(T_q', T_q'')$.

Next, we see that the trajectory $\Psi^{(q)}(x^{(q)}(t))$ is very close to $\Psi^*(x^*(t))$ for large values of $q$. This is true because when $|t| \leq T_q'''$.

\begin{longtable}{lcl}\label{eq:inv_tor_1}
$|\Psi^{(q)}(\hat x^{(q)} - \Psi^* (\hat x^* (t)))|_{c_1}$\\
 $\leq$  $|\Psi^{(q)}(\hat x^{(q)}(t)) - \Psi^{(q)}(\hat x^* (t))|_{c_1} + |\Psi^{(q)}(\hat x^*(t)) - \Psi^* (\hat x^* (t))|_{c_1}$\\
 $\leq$  $2^{(\tau+1-\nu)(q-1)}\cdot 2|\hat x^{(q)}(t) - \hat x^* (t)|_{c_q} + |\Psi^{(q)} - \Psi^*|_{G^* , (\rho_1/4,0),c_1}$\\
 $\leq$  $2^{(\tau+1-\nu)(q-1)}\cdot 4 \delta_2^{(q+1)} + |\Psi^{(q)} - \Psi^*|_{G^*,(\rho_1/4,0),c_1}$\\
 $\leq$  $2^{(\tau+1-\nu)(q-1)}\cdot 4 \delta_2^{(q+1)} + \frac{2^{10} K^\tau \varepsilon}{\nu^2\rho_1\beta 2^{(1+\nu)q}}$\\
 $\leq$  $\frac{c_1}{c_{q+1}}\frac{4\delta_2^{(q+1)}}{2^{(\tau+1-\nu)}} + \frac{2^{10} K^\tau \varepsilon}{\nu^2\rho_1\beta 2^{(1+\nu)q}}$\\
 $\leq$  $\frac{c_1 4}{2^{(\tau+1-\nu)}}\frac{\delta_1^{(q+1)}}{\delta_2^{(q+1)}}\delta_2^{(q+1)} + \frac{2^{10} K^\tau \varepsilon}{\nu^2\rho_1\beta 2^{(1+\nu)q}}$\\
 $\leq$  $\frac{c_1 4}{2^{(\tau+1-\nu)}}\delta_1^{(q+1)} + \frac{2^{10} K^\tau \varepsilon}{\nu^2\rho_1\beta 2^{(1+\nu)q}}$\\
\end{longtable}

where we used that $c_{q-1}/c_q = 2^{\tau+1-\nu}$ then $c_1/c_{q+1} = 2^{(\tau+1-\nu)q}$.

The bound \ref{eq:inv_tor_1} tends to zero. So we deduce, for every fixed $t$, $\Psi^{(q)}(\hat x^{(q)}(t))$ exits or $q$ large enough and its limit is $\Phi^* (\hat x^* (t))$.
This fact and the continuity of the flow of $\hat H$ imply that $\Psi^* (\hat x^* (t))$ is also a trajectory of $\hat H$, which is defined for all $t\in\mathbb{R}$.

This holds for every initial condition $x_0^*=(\phi_0^* , I_0^* ) \in \mathbb{T}^n \times G^* $ for this reason $\Psi^* (\mathbb{T}^n\times\{I_0^*\})$ is an invariant torus of $\hat H$, with frequency vector $u'^*(I_0^* )$. Observe that the energy on the torus is $\hat H(\Psi^* (\phi_0^* , I_0^* )) = h^* (I_0^* )$.

The preserved invariant tori are completely determined by the transformed actions $I_0^* \in G^* $. We are now going to characterize the preserved tori by the original action coordinates.

First, let us see that $u(\hat G) \subset F^* $. Recall that:

$$\Delta_{c,\hat q}(k,\alpha) = \{J \in \mathbb{R} \text{ such that } |k\bar{\mathcal{B}}u(I) + k\bar{\mathcal{A}}| < \alpha \},$$
$$\hat G = \{I \in \mathcal{G} - \frac{2\gamma}{\mu} \text{ such that } |k\bar{\mathcal{B}}u(I) + k\bar{\mathcal{A}}|< \frac{\beta}{|k|_1^\tau} \}.$$

With this definition is obvious that if $I\in \hat G$ then $u(I)$ is $\frac{\beta}{|k|_1^\tau},K,c,\hat q$-non-resonant. Hence $u(I) \notin \Delta_{c,\hat q}(k,\frac{\beta}{|k|^\tau_1})$ for all $k \neq 0$. Then $u(\hat G) \subset F^* $.

We want to find a correspondence between the invariant tori of $\hat h$ and the invariant tori of the perturbed system $\hat H = \hat h + R$, or in the new coordinates $\hat h^*$.

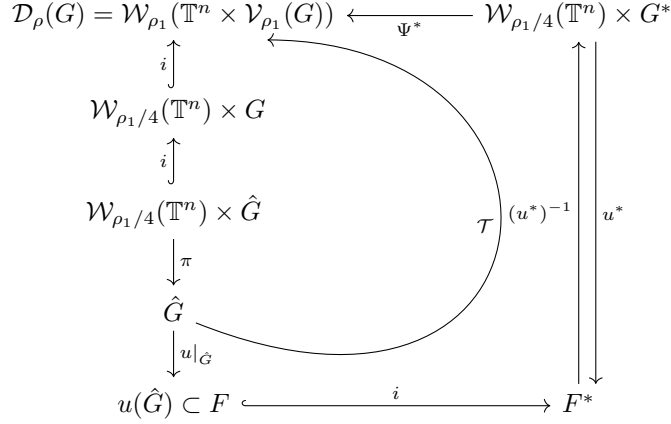
\begin{figure}
\centering
\begin{tikzcd}
\mathcal{D}_\rho(G) = \mathcal{W}_{\rho_1}(\mathbb{T}^n\times\mathcal{V}_{\rho_1}(G)) & & \mathcal{W}_{\rho_1/4}(\mathbb{T}^n)\times G^* \arrow[ll, "\Psi^* "] \arrow[dddd,  shift left=1.5ex, "u^*"]\\\
\mathcal{W}_{\rho_1/4}(\mathbb{T}^n)\times  \arrow[u, hook, "i"] G & & \\
\mathcal{W}_{\rho_1/4}(\mathbb{T}^n)\times \arrow[u, hook, "i"] \arrow[d, "\pi"] \hat G & & \\
\hat G \arrow[d, "u|_{\hat G}"]
 \arrow[uuu, controls={+(5,-2) and +(4.5,0)},end anchor={[yshift=0ex,xshift=-7ex]south east},"\mathcal{T}"]
& & \\
u(\hat G) \subset F \arrow[rr,hook,"i"] & & F^* \arrow[uuuu,"(u^*)^{-1}"]
\end{tikzcd}
\caption{Diagram of the different maps and sets used in the proof.}
\label{fig:diagram_kam}
\end{figure}

Recall 
$$u' = \bar{\mathcal{B}} u + \bar{\mathcal{A}},$$ 
$$u'^* = \bar{\mathcal{B}} u^* + \bar{\mathcal{A}} = (\frac{1}{\sum_{i=1}^m\frac{c_i}{I_1^i}} u_1^* +\frac{\sum_{i=1}^m\frac{\hat q_i}{I_1^i}}{\sum_{i=1}^m\frac{c_i}{I_1^i}}, u_2^*, \ldots, u_n^*).$$
Observe \textcolor{black}{$u'^*(0,I_2,\ldots,I_n) = \frac{\hat q_m}{c_m} = \frac{1}{\mathcal{K}'}$} the inverse of the modular period, hence $u'^*$ and $u'$ are not one-to-one at $Z$ because they project the first component of $u^*$ and $u$ to \textcolor{black}{$\frac{1}{\mathcal{K}'}$}.

Let us define $I_0^* = (u^* )^{-1}(u(I_0))$, recall that $u$ and $u^* $ are indeed one-to-one even though $u'$ and $u'^*$ are not, so $I_0^*$ is properly defined.

With this definition $u^*(I_0^*) = u(I_0)$ and this implies $u'^*(I_0^*) = u'(I_0)$.
Now, let us define $\mathcal{T}(\phi_0,I_0) = \Psi^*(\phi_0, I_0^*)$.

We obtain \ref{eq:kam1} because the set $\mathcal{T}(\mathbb{T}^n\times\{I_0\})$ is an invariant torus of the hamiltonian flow of $\hat H$ with frequency vector $u'^*(I_0^*)$ because $\mathbb{T}^n\times\{I_0^*\}$ is an invariant torus for the hamiltonian flow of $\hat h^*$. And we have seen that  $u'^*(I_0^*) = u'(I_0)$.
In a nutshell, the original frequencies (of the unperturbed system) $u(I_0)$ for $I_0 \in \hat G$ are in $F^*$ and hence can be seen as frequencies of the unperturbed system in the new coordinates $u^* (I_0^*)$. Hence we can conclude that for this $I_0 \in \hat G$ its new (perturbed) solution is also linear in a torus $(\phi_0 + u'^* t, I_0^*) \in \Psi^*(\mathbb{T}^n\times \{I_0^*\}) = \mathcal{T}(\mathbb{T}^n\times\{I_0\})$. And the new frequency vector $u'^* $ is such that $u'^* = u'$.

Let us now prove \ref{eq:kam2}. Let us write, for $(\phi_0,I_0^* ) \in \mathcal{W}_{\frac{\rho_1}{4}}(\mathbb{T}^n) \times G^*$.

$$\Psi^*(\phi_0,I_0^* ) = (\phi_0 + \Psi_\phi^*(\phi_0, I_0^*), I_0^* + \Psi_I^* (\phi_0,I_0^*)).$$

And for $(\phi_0,I_0) \in \mathcal{W}_{\frac{\rho_1}{4}(\mathbb{T}^n)\times \hat G}$.

$$\mathcal{T}(\phi_0,I_0) = (\phi_0 + \mathcal{T}_\phi(\phi_0,I_0), I_0 + \mathcal{T}_I(\phi_0,I_0)).$$

Then, for $(\phi_0,I_0) \in \mathcal{W}_{\frac{\rho_1}{4}(\mathbb{T}^n)\times \hat G}$:

$$ \mathcal{T}_\phi(\phi_0,I_0) = \Psi^*_\phi(\phi_0, I_0^*), \quad \text{ and} \quad \mathcal{T}_I(\phi_0,I_0) = \Psi^*_I(\phi_0,I_0^*) + I_0 - I_0^*.$$

Let us bound the norms of these terms:

$$
\begin{array}{rcl}
 |\Psi_\phi^* (\phi_0,I_0^* )|_\infty & \leq & \frac{1}{c_1}|\Psi^* -\text{id}|_{G^*,(\frac{\rho_1}{4},0),c_1}\\
 & \leq & \frac{16 M K^{\tau+1} \rho_1}{\beta}\frac{2^{10} K^\tau \varepsilon}{\nu^2 \rho_1 \beta}\\
 & \leq & \frac{2^{14} M K^{2\tau+1} \varepsilon}{\nu^2 \beta^2},\\
\end{array}
$$
where we used that $c_1 \geq \frac{\beta}{16 M K^{\tau+1}\rho_1}$. Then,

$$
\begin{array}{rcl}
 \Psi_I^*(\phi_0,I_o^*) & \leq & |\Phi^* - \text{id}|_{G^*,(\frac{\rho_1}{4},0,c_1)}\\
 & \leq & \frac{2^{10} k^\tau \varepsilon}{\nu^2 \rho_1 \beta}. \\
\end{array}
$$

Now it only remains the term $I_0^* -I_0$:

$$|I_0^* - I_0| \leq |(u^*)^{(-1)} - (u)^{(-1)}|_{F^*} \leq \sum_{s\geq 0} \xi_s $$
$$\leq \sum_{s\geq 0} \frac{4 M K^{\tau+1}\varepsilon}{\nu \beta 2^{(\tau+2)s}} \leq \frac{8M K^{\tau+1}\varepsilon}{\nu \beta 2^{(\tau+2)}}.$$

Let us put everything together and use $\hat \rho \leq \nu \rho_1$, $K \leq 2/\hat \rho$ and $\beta = \gamma/L$.

\begin{longtable}{rcl}
 $|\Psi_\phi^* (\phi_0, I_0^* )|_\infty$ & $\leq$ & $\frac{2^{14} M (\frac{2}{\hat \rho})^{2\tau+1}\varepsilon}{\nu^2 (\frac{\gamma}{L})^2}$\\
 &$\leq$& $\frac{2^{2\tau + 15} M L^2}{\nu^2 \hat \rho^{2\tau+1}}\frac{\varepsilon}{\gamma^2}$\\
\end{longtable}

\begin{longtable}{rcl}
 $|\Psi^*(\phi_0,I_0^*)| + |I_0^* - I_0| $ & $\leq$ & $\frac{2^{10}(\frac{2}{\hat \rho})^\tau\varepsilon}{\nu \hat \rho (\frac{\gamma}{L})} + \frac{8M (\frac{2}{\hat \rho})^{\tau+1} \varepsilon}{\nu (\frac{\gamma}{L}) 2^{(\tau+2)}}$\\
 & $=$ & $\frac{2^{10+\tau} L\varepsilon}{\nu \hat \rho^{\tau+1} \gamma} + \frac{8 M 2^{\tau+1} L \varepsilon}{\nu \hat \rho^{\tau+1}\gamma 2^{(\tau+2)}}$\\
 & $\leq$& $\frac{2^{20+\tau} L\varepsilon + M 2^{\tau+4}L\varepsilon}{\nu \hat \rho^{\tau+1}\gamma} \leq \frac{2^{10+\tau} L (1+M)}{\nu \hat \rho^{\tau+1}}\frac{\varepsilon}{\gamma}$\\
\end{longtable}

\item Estimate of the measure

Finally, we carry out the estimate of part \ref{kam:point3}. Let us write $$\hat G^* = (u^*)^{-1}(u(\hat G)).$$
The invariant tori fill the set 
$$\mathcal{T}(\mathbb{T}^n\times \hat G) = \Psi^* (\mathbb{T}^n\times \hat G^*)$$
 i.e. all the tori inside $\mathcal{T}(\mathbb{T}^n\times \hat G)$ are invariant although there are more of them.
Because $\Psi^{(q)}$ are hamiltonian transformations, in particular, they preserve the volume:

$$\text{meas}[\Psi^{(q)}(\mathbb{T}^n\times \hat G^* )] = \text{meas}(\mathbb{T}^n\times \hat G^*) = (2\pi)^{n}\text{meas}(\hat G^*).$$

Now, let us consider the measure of the limit:

$$\text{meas}[\Psi^*(\mathbb{T}^n\times \hat G^*)].$$

To do this we use the superior limit of sets:

$$\bigcap_{n=q}^{\infty} \bigcup_{j=q}^{\infty} (\Psi^{(j)}(\mathbb{T}^n \times \hat G^*)).$$

Because $\Psi^{(j)}(\mathbb{T}^n\times \hat G^*)$ are compact and we have the bound 
$$|\Psi^* - \Psi^{(q)}|_{G^* , (\frac{\rho_1}{4},0), c_1} \leq \frac{2^{10} K^\tau \varepsilon}{\nu^2 \rho_1 \beta 2^{(1+\nu)q}},$$
 $\bigcup_{j=q}^\infty(\Psi^{(j)}(\mathbb{T}^n\times \hat G^*))$ is also compact. All the measures are well-defined and we can say that

$$\text{meas}[\Psi^*(\mathbb{T}^n\times \hat G^* )] \geq (2\pi)^n\text{meas}(\hat G^*).$$

Then, to bound the measure of the complement of the invariant set it is enough to bound the measure of $\mathcal{G}\setminus \hat G^*$.

But first, we are going to define some auxiliary sets. Let $\tilde \beta = \frac{2 \gamma M}{\mu}$, $\tilde\beta_q = (1-\frac{1}{2^{\nu q}})\tilde \beta$. Note that \textcolor{black}{$\tilde \beta \geq \beta$ if and only if $\mu \leq 2ML$ and we assumed $\mu \leq 2^{\tau+6}L^2 M$.}

Then, for $q \geq 0$ we define

$$\tilde F_q = (F - \tilde\beta_q) \setminus \bigcup_{\substack{k\in\mathbb{Z}^n\setminus\{0\} \\ |k|_1 \leq K}} \Delta_{c,\hat q}(k,\frac{\tilde \beta_q}{|k|_1^\tau}), \quad \tilde G_q = (u^{(q)})^{-1}(\tilde F_q)$$

and

$$\tilde F^* = \bigcap_{q\geq 0} \tilde F_q = (F - \tilde\beta) \setminus \bigcup_{\substack{k\in\mathbb{Z}^n\setminus\{0\} \\ |k|_1 \leq K}} \Delta_{c,\hat q}(k,\frac{\tilde \beta}{|k|_1^\tau}), \quad \tilde G^* = \bigcap_{q\geq 0} \tilde G_q.$$

In order to prove the bounds, we need to prove previously the inclusions $\tilde G^* \subset \hat G^*$ and $\tilde G_0 \subset \mathcal{G}$.

\begin{enumerate}
\item $\mathcal{G} \supset \tilde G_0 = (u^{(0)})^{-1}(\tilde F_0) = (u)^{-1}(F-\tilde\beta)$, but we know $u(\mathcal{G}) = F$.

\item $\tilde G^* \subset \hat G^*$. Take $I\in \tilde G^*$, then $I\in \tilde G_q \forall q\geq 0$. Hence $\exists J \in \tilde \tilde F_q \forall q$ such that $u^{(q)}(I)$. Then $\exists J \in \tilde F^*$ such that $u^*(J) = I$.

If we check that $J \in u(\tilde G)$ we obtain $(u^*)^{-1}(J) = I \in \hat G^*$ and we will be done.
We want $\tilde F^* \subset u(\hat G)$. Because we are taking out all the resonances in $\tilde F^*$ it is enough to see $(F - \tilde \beta) \subset u(\mathcal{G} - \frac{2\gamma}{\mu})$. We only need to use that $|\frac{\partial u}{\partial I}|_{\mathcal{G},\rho_2} \leq M$. Then $F - \tilde\beta \subset u(\mathcal{G}-\frac{2\gamma}{\mu})$. This holds if and only if $\frac{\tilde \beta}{M} \leq \frac{2\gamma}{\mu}$ which is true because $\hat \beta \leq \frac{2\gamma M}{\mu}$.

\end{enumerate}

Then, we proceed as follows

\begin{longtable}{rcl}

$\text{meas}(\mathcal{G}\setminus \hat G^*)$ & $\leq$ & $\text{meas}(\mathcal{G}\setminus \tilde G^*)$\\

& $\leq$ & $\text{meas}(\tilde G_0\setminus \tilde G^*)$\\
& $\leq$ & $\sum_{q=1}^{\infty}\text{meas}(\tilde G_{q-1} \setminus  \tilde G_q).$\\

\end{longtable}

For $q \geq 1$ we obtain the following estimate:

$$\text{meas}(\tilde G_{q-1}\setminus \tilde G_q) \leq \frac{1}{|\det(\frac{\partial u^{(q-1)}}{\partial I}(I))|} \text{meas}(\overbrace{\tilde F_{q-1}}^{u^{(q-1)}(\tilde G_{q-1})}\setminus(\overbrace{\tilde F_q - \varepsilon_{q-1})}^{u^{(q-1)}(\tilde G_q)}).$$

Where we have used lemma \ref{lemma:inductive}. Also $\det(\frac{\partial u^{(q-1)}}{\partial I}(I)) \geq \mu_{q-1}^n$ because of the $\mu_{q-1}$-non-degeneracy condition all the eigenvalues have to be greater than $\mu_{q-1}$.

$$
\begin{array}{rcl}
\text{meas}(\tilde G_{q-1} \setminus \tilde G_q) & \leq & \frac{1}{\mu_{q-1}^n}\text{meas}(\tilde F_{q-1} \setminus(\tilde F_q - \varepsilon_{q-1}))\\
 & \leq& \frac{2^n}{\mu^n} \text{meas}(\tilde F_{q-1} \setminus(\tilde F_q - \varepsilon_{q-1})).\\
\end{array}
$$

Now, we are going to apply lemma \ref{lemma:measures_nonresonant} with 
$$\tilde F_{q-1} = F(\tilde \beta_{q-1}, \tilde \beta_{q-1}, K_q-1)$$
and $\tilde F_q = F(\tilde \beta_q, \tilde \beta_q, K_q)$.

%Recall that $|k|_{2,\omega} = \sqrt{\mathcal{B}(I_1)k_1^2 + k_2^2 + \ldots k_n^2}$ and observe that at $Z$, $|k|_{\tau,\omega}$ vanishes for $k$ parallel to $(1,0,\ldots,0)$,

Applying the lemma:

$$
\begin{array}{lcl}
\displaystyle \text{meas}(\tilde F_{q-1} \setminus \tilde F_q)  \leq  \displaystyle D(\tilde \beta_q - \tilde \beta_{q-1})\\
 \displaystyle + 2(\text{diam} F)^{n-1}\left(\sum_{\substack{k\in\mathbb{Z}^n\setminus\{0\} \\ |k|_1 \leq K_{q-1}}} \frac{\tilde \beta_q - \tilde \beta_{q-1}}{|k|_1^\tau |k|_{2,\omega}} + \sum_{\substack{k\in\mathbb{Z}^n\setminus\{0\} \\ K_{q-1} \leq |k|_1 \leq K_{q}}} \frac{\tilde \beta_q}{|k|_1^\tau |k|_{2,\omega}} \right)\\
\end{array}
$$

and

$$
\text{meas}(\tilde F_q \setminus (\tilde F_q - \varepsilon_q)) \leq (D + 2^{n+1}(\text{diam} F)^{n-1} K^n)\varepsilon_q.
$$

Putting everything together (and using that $\tilde \beta_0 = 0$), we get

\begin{equation}\label{eq:final_bound}
\begin{array}{rcl}
 \text{meas}(\mathcal{G} \setminus \hat G^*) & \leq & \displaystyle \frac{2^n}{\mu^n}\left( D \tilde \beta + 2(\text{diam}F)^{n-1}\sum_{\substack{k\in\mathbb{Z}^n\setminus\{0\} \\ }}\frac{\tilde \beta}{|k|_1^\tau|k|_{2,\omega}}\right.\\
 & & \quad \displaystyle \left. + D\sum_{q=1}^\infty \varepsilon_{q-1} + 2^{n+1}(\text{diam} F)^{n-1} \sum_{q=1}^{\infty} K_q^n \varepsilon_{q-1} \right).\\
\end{array}
\end{equation}

We now only have to check that the series  in the previous expression converge. Let us check that they converge at $Z$ first and then outside of $Z$. Recall that at $Z$ we take the vectors $\bar k \neq 0$.

\begin{longtable}{rcl}
$\sum_{\substack{k\in\mathbb{Z}^n\setminus\{0\} \\ \textcolor{black}{\bar k \neq 0}}} \frac{1}{|k|_1^\tau |k|_{2,\omega}}$ & $\leq$ & $\sum_{\substack{k\in\mathbb{Z}^n\setminus\{0\} \\ \textcolor{black}{\bar k \neq 0}}} \frac{1}{|k|_1^\tau|\bar k|}$\\
 & $\leq$ & $\sum_{\substack{\bar k\in\mathbb{Z}^{n-1}\setminus\{0\} \\ \textcolor{black}{\bar k \neq 0}}}\sum_{k_n \in \mathbb{Z}} \frac{\sqrt{n}}{(|\bar k|_1 + |k_n|)^\tau|\bar k|_1}$\\
  & $\leq$ & $\sqrt{n} 2^{n-1} \sum_{j=1}^{\infty}\sum_{k_n \in \mathbb{Z}} \frac{j^{n-3}}{j + |k_n|)^\tau}$\\
\end{longtable}
where we used that the number of vectors $\bar k \in \mathbb{Z}^{n-1}$ with $|\bar k|_1 = j \geq 1$ can be bounded by $2^{n-1} j^{n-2}$. This series can be bounded by comparing it to an integral:

$$\sum_{k_n \in \mathbb{Z}} \frac{1}{(j + |k_n|)^\tau} \leq \frac{1}{j^\tau} + 2 \int_0^\infty \frac{dx}{(j+x)^\tau} $$
$$= \frac{1}{j^\tau} + \frac{2}{(\tau+1) j^{\tau-1}} \leq \frac{\tau+1}{\tau-1}\frac{1}{j^{\tau+1}}.$$

Where we used that $\tau > 1$ because $n\geq 2$. Then

$$\sum_{\substack{k\in\mathbb{Z}^{n}\setminus\{0\} \\ \textcolor{black}{\bar k \neq 0}}} \frac{1}{|k|_1^\tau|k|_{2,\omega}}\leq \frac{\sqrt{n} 2^{n-1}(\tau+1)}{\tau-1} \sum_{j=1}^\infty \frac{1}{j^{\tau-n+2}}$$

which converges by the condition $\tau > n-1$.

Now let us check that it converges outside of $Z$.

$$
\begin{array}{rcl}
\sum_{k\in \mathbb{Z}^n\setminus \{0\}}\frac{1}{|k|_1^\tau |k|_{2,\omega}} & = & \sum_{\substack{k \in \mathbb{Z}^n \setminus \{0\} \\ \bar k \neq 0}} \frac{1}{|k|_1^\tau |k|_{2,\omega}} + \sum_{\substack{k \in \mathbb{Z}^n \setminus \{0\} \\ \bar k = 0}} \frac{1}{|k|_1^\tau |k|_{2,\omega}}\\
& = & \sum_{\substack{k \in \mathbb{Z}^n \setminus \{0\} \\ \bar k \neq 0}} \frac{1}{|k|_1^\tau |\bar k|} + \sum_{k_1 \in \mathbb{Z}} \frac{1}{|k|_1^\tau |k_1^2\mathcal{B}(I_1)^2|}\\
\end{array}
$$

We have seen before that the first term converges.
The second term:

$$\sum_{k_1 \in \mathbb{Z}}\frac{1}{|k_1|^\tau|k_1^\tau \mathcal{B}(I_1)^2|} = \frac{1}{\mathcal{B}(I_1)^2}\sum_{k_1 \in \mathbb{Z}}\frac{1}{k_1^{\tau+2}},$$

which converges $\forall I_1 \neq 0$, i.e. outside of $Z$.

Now we go back to the expression \ref{eq:final_bound}.
The other terms of that expression can be bounded simultaneously inside and outside $Z$.
Now we only have to check that the third series converges, because if the third converges so does the second.
We only have to check that $\sum_{q=1}^\infty K_q^n \varepsilon_{q-1}$ converges. We will use that $\varepsilon \leq \frac{8 \varepsilon}{\nu \rho_1 2^{(2\tau+2)(q-1)}}$.

\begin{longtable}{rcl}
$\sum_{q=1}^\infty K_q^n \varepsilon_{q-1}$ & $=$ & $K^n \sum_{q=1}^\infty 2^{n(q-1)} \varepsilon_{q-1}$ \\
 & $=$ & $K^n \sum_{q=1}^\infty \frac{8\varepsilon 2^{n(q-1)}}{\nu \rho_1 2^{(2\tau+2)(q-1)}}$ \\
  & $=$ & $K^n \frac{8\varepsilon}{\nu \rho_1} \sum_{q=1}^\infty \frac{1}{2^{(2\tau+2-n)(q-1)}}.$ \\
\end{longtable}
Which converges if and only if $2\tau+2-n \geq 1$. And we are done because $2\tau \geq n-1$ since $\tau  \geq n-1$ by hypothesis.

Putting everything together:

$$
\begin{array}{llc}
\text{meas}(\mathcal{G}\setminus \hat G^*) \\
 \leq  \displaystyle \frac{2^n}{\mu^n}\left( D^2\frac{2\gamma M}{\mu} + 2(\text{diam}F)^{n-1}\frac{2\gamma M}{\mu}\frac{\sqrt{n} 2^{n-1} (\tau+1)}{\tau-1} \sum_{j=1}^\infty \frac{1}{j^{\tau - n + 2}}\right.\\
 \displaystyle  D\frac{8\varepsilon}{\nu \rho_1}\sum_{q=1}^{\infty}\frac{1}{2^{}(2\tau+2)(q-1)} \\
\displaystyle \left. + 2^{n+1}(\text{diam} F)^{n-1} K^n \frac{8\varepsilon}{\nu \rho_1} \sum_{q=1}^\infty \frac{1}{2^{(2\tau+2-n)(q-1)}}\right)\\
\end{array}
$$

Now using that

$$
\varepsilon \leq \frac{\nu^2 \mu^2 \beta^2}{2^{\tau+30} L^4 M^3 K^{2\tau+1}} \leq \frac{2^{\tau-18}\cdot 8 M K^{\tau+1} \rho_2}{L M K^{2\tau+2}}\gamma \leq \frac{2^{\tau-15}\rho_2}{L K^{\tau+1}}\gamma
$$

We can write $\text{meas}(\mathcal{G}\setminus \hat G^*) \leq C' \gamma$ where $C'$ depends only on $n$,  $\mu$,  $D$,  $\text{diam} F$,  $M$,  $\tau$,  $\rho_1$,  $\rho_2$,  $L$,  $K$ and if we efine $C = (2\pi)^n C'$. Hence,

$$\text{meas}[(\mathbb{T}^n \times \mathcal{G}) \setminus \mathcal{T}(\mathbb{T}^n \hat G)] \leq C \gamma.$$

\end{enumerate}
\end{proof}

\chapter[Desingularization of $b^m$-integrable systems]{Desingularization of $b^m$-integrable systems}

In this chapter, we follow \cite{GMW17}, for the definition of the desingularization of the $b^m$-symplectic form.

\begin{definition}The \textbf{$f_\epsilon$-desingularization} $\omega_\epsilon$ form of $\omega= \frac{dx}{x^{m}}\wedge \left(\sum_{i=0}^{m-1}x^i\alpha_{m-i}\right) + \beta$ is:
$$\omega_\epsilon = df_\epsilon \wedge \left(\sum_{i=0}^{m-1}x^i\alpha_{m-i}\right) + \beta.
$$
Where in the even case, $f_\epsilon(x)$ is defined as $\epsilon^{-(2k -1)}f(x/\epsilon)$.
And $f \in \mathcal{C}^\infty(\mathbb{R})$ is an odd smooth function satisfying $f'(x) > 0$ for all $x \in \left[-1,1\right] $ and satisfying outside that
\begin{equation}
f(x) = \begin{cases}
\frac{-1}{(2k-1)x^{2k-1}}-2& \text{for} \quad x < -1,\\
\frac{-1}{(2k-1)x^{2k-1}}+2& \text{for} \quad x > 1.\\
\end{cases}
\end{equation}
And in the odd case, $f_\epsilon(x) = \epsilon^{-(2k)}f(x/\epsilon)$. And $f \in \mathcal{C}^\infty(\mathbb{R})$ is an even smooth positive function which satisfies: $f'(x) < 0$ if $x < 0$, $f(x) = -x^2 + 2$ for $x \in [-1,1]$, and
\begin{equation}
f(x) = \begin{cases}
\frac{-1}{(2k+2)x^{2k+2}}-2& \text{if } k > 0, x \in \mathbb{R}\setminus[-2,2]\\
\log(|x|)& \text{if } k = 0, x \in \mathbb{R}\setminus[-2,2].\\
\end{cases}
\end{equation}
\end{definition}

\begin{remark}
With the previous definition, we obtain smooth symplectic (in the even case) or smooth folded symplectic (in the odd case) forms that agree outside an $\epsilon$-neighbourhood with the original $b^m$-forms. Moreover, there is a convergence result in terms of $m$. See \cite{GMP17} for the details.
\end{remark}

To simplify notation, we introduce $F_{\epsilon}^{m-i}(x) = (\frac{d}{dx}f_\epsilon(x))x^i$, and hence $F_{\epsilon}^{i}(x) = (\frac{d}{dx}f_\epsilon(x))x^{m-i}$. With this notation the desingularization $\omega_\epsilon$ is written:

$$\omega_\epsilon = \sum_{i = 0}^{m-1} F_{\epsilon}^{m-i}(x) dx\wedge \alpha_{m-i} + \beta. $$

\begin{definition}
The desingularization for $(M,\omega,\mu)$ is the triple $(M,\omega_\varepsilon,\mu_\epsilon)$ where $\omega_\varepsilon$ is defined as above and $\mu_\varepsilon$ is:

$$
\mu    \mapsto \mu_\epsilon = \left(f_{1\epsilon} = \sum_{i = 1}^{m} \hat{c}_i G_{\epsilon}^i(x), f_2(\tilde I,\tilde \phi), \ldots, f_n(\tilde I,\tilde \phi)\right),
$$

where 
$$\mu = \left(f_1 = c_0 \log(x) + \sum_{i = 1}^{m-1} c_i \frac{1}{x^i},f_2(I,\phi) \ldots, f_n(I,\phi)\right)$$ 
$$G_\epsilon^i(x) = \int_0^x F_\epsilon^i(\tau)d\tau,$$
and $\hat{c}_1 = c_0$ and $\hat{c}_{i-1} = -ic_i$ if $i \neq 0$. Also

$$
\left\{
\begin{array}{lrcl}
\tilde I = (\tilde I_1, I_2,\ldots, I_n), &\quad \tilde I_1 & = & \int_0^{I_1}\left(\frac{\sum_{i=1}^m \mathcal{K} \hat c_i F_\varepsilon^i(\tau)}{\sum_{i=1}^m \frac{\mathcal{K} \hat c_j}{\tau^j}}\right) d\tau \\

\tilde \phi_1 = (\tilde \phi_1, \phi_2,\ldots, \phi_n), & \quad \tilde \phi_1 & = & \left(\frac{\sum_{i=1}^m \mathcal{K} \hat c_i F_\varepsilon^i(I_1)}{\sum_{i=1}^m \frac{\mathcal{K} \hat c_j}{I_1^j}}\right) \phi_1
\end{array}
\right.
$$

\end{definition}

\begin{remark}
Observe that with the last definition, when $\epsilon$ tends to $0$, $\mu_\epsilon$ tends to $\mu$.
\end{remark}

\begin{theoremC}
The desingularization transforms a $b^m$-integrable system into an integrable system for $m$ even on a symplectic manifold. For $m$ odd the desingularization transforms it into a folded integrable system. The integrable systems are such that:
$$X_{f_j}^\omega = X_{f_{j\epsilon}}^{\omega_\epsilon}.$$
\end{theoremC}

\begin{proof}
Let us first check the singular part, i.e. let us check that that $X_{f_1}^\omega = X_{f_{1\epsilon}}^{\omega_\epsilon}$. Let us compute the two equations that define each one of the vector fields. We have to impose $-df_1 = \iota_{X_{f_1}^\omega}\omega$ and $-df_{1\epsilon} = \iota_{X_{f_{1\epsilon}}^{\omega_\epsilon}}\omega_\epsilon$.
But observe first that we can rewrite $\omega = \sum_{i=1}^m \frac{1}{x^i}dx\wedge\alpha_i + \beta$ and $\omega_\epsilon = \sum_{i=1}^m F_\epsilon^idx\wedge\alpha_i + \beta$. The conditions translate as:

$$-\sum_{i = 1}^{m}\hat{c}_{i} \frac{1}{x^i}dx = \iota_{X_{f_1}^\omega}\left(\sum_{i = 1}^{m} \frac{1}{x^i} dx\wedge \alpha_i + \beta\right),$$
$$-\sum_{i = 0}^{m-1}\hat{c}_i F_\epsilon^i(x)dx = \iota_{X_{f_{1\epsilon}}^{\omega_\epsilon}} \left(\sum_{i = 0}^{m-1} F_{\epsilon}^i(x) dx\wedge \alpha_i + \beta\right).$$

Since the toric action leaves the form $\omega$ invariant, in particular, the singular set is invariant, and then $X_{f_{1\epsilon}}^{\omega_\epsilon}$ and $X_{f_{1}}^{\omega}$ are in the kernel of $dx$. Moreover, since $\beta$ is a symplectic form in each leaf of the foliation and $X_{f_{1\epsilon}}^{\omega_\epsilon}$ and $X_{f_{1}}^{\omega}$ are transversal to this foliation, they are also in the kernel of $\beta$.

$$-\sum_{i = 0}^{m-1}\hat{c}_i \frac{1}{x^i}dx = \sum_{i = 0}^{m-1} \frac{1}{x^i} dx\wedge \alpha_i(X_{f_1}^\omega),$$
$$-\sum_{i = 0}^{m-1}\hat{c}_i F_\epsilon^i(x)dx = \sum_{i = 0}^{m-1} F_{\epsilon}^i(x) dx\wedge \alpha_i(X_{f_{1\epsilon}}^{\omega_\epsilon}).$$

Then, the conditions over $X_{f_1}^\omega$ and $X_{f_{1\epsilon}}^{\omega_\epsilon}$ are respectively:

$$-\hat{c}_i = \alpha_i(X_{f_1}^\omega),$$
$$-\hat{c}_i = \alpha_i(X_{f_{1\epsilon}}^{\omega_\epsilon}).$$

Then, the two vector fields have to be the same.

Let us now see $X_{f_j}^\omega = X_{f_{j\epsilon}}^{\omega_\epsilon}$ for $j > 1$.
Assume now we have the $b^m$-symplectic form in action-angle coordinates $\omega = \sum_{i=1}^m \frac{\mathcal{K}\hat c_i}{I_1^i}dI_1\wedge d\phi_1 + \sum_{i=1}^n dI_i\wedge d\phi_i$.

The differential of the functions are

$$
\begin{array}{rcl}
d f_i^\varepsilon & = & \frac{\partial f_i^\varepsilon}{\partial I_1} dI_1 + \frac{\partial f_i^\varepsilon}{\partial \phi_1} d\phi_1 + \sum_{j= 2}^n\left(\frac{\partial f_i^\varepsilon}{\partial I_j} dI_j + \frac{\partial f_i^\varepsilon}{\partial \phi_j} d\phi_j\right)\\
 & = & \frac{\partial f_i}{\partial I_1}\left(\frac{\sum_{i=1}^m \mathcal{K} \hat c_i F_\varepsilon^i(\tau)}{\sum_{i=1}^m \frac{\mathcal{K} \hat c_j}{\tau^j}}\right) dI_1 + \frac{\partial f_i}{\partial \phi_1}\left(\frac{\sum_{i=1}^m \mathcal{K} \hat c_i F_\varepsilon^i(\tau)}{\sum_{i=1}^m \frac{\mathcal{K} \hat c_j}{\tau^j}}\right) d\phi_1 \\
 & & \quad + \sum_{j= 2}^n\left(\frac{\partial f_i^\varepsilon}{\partial I_j} dI_j + \frac{\partial f_i^\varepsilon}{\partial \phi_j} d\phi_j\right).\\
\end{array}
$$

On the other hand, the desingularized form is:

$$\omega^\varepsilon = \sum_{j=1}^m \mathcal{K} \hat c_i F_\varepsilon^j (I_1) d I_1 \wedge d\phi_1 + \sum_{j=2}^m dI_j \wedge d\phi_j.$$

Hence, one can see that the expression for both $X_{f_j}^\omega$ and $X_{f_{j\epsilon}}^{\omega_\epsilon}$ is

$$X_{f_j}^\omega = X_{f_{j\epsilon}}^{\omega_\epsilon} = \frac{\frac{\partial f_i}{\partial I_1}}{\sum_{i=1}^m\frac{\mathcal{K}\hat c_i}{I_1^i}}\frac{\partial}{\partial \phi_1} - \frac{\frac{\partial f_i}{\partial \phi_1}}{\sum_{i=1}^m\frac{\mathcal{K}\hat c_i}{I_1^i}}\frac{\partial}{\partial I_1} +\sum_{j= 2}^n\left(\frac{\partial f_i^\varepsilon}{\partial I_j} dI_j + \frac{\partial f_i^\varepsilon}{\partial \phi_j} d\phi_j\right)$$

\end{proof}

\begin{remark}
The previous lemma tells us that the dynamics of the desingularized system are identical to the dynamics of the original $b^m$-integrable system in the $b^m$-symplectic manifold.
\end{remark}

Hence the desingularized $b^m$-form goes to a folded symplectic form in the case $m=2k+1$ and to symplectic for $m=2k$. And the $b^m$-integrable system goes to a folded integrable system (see \cite{EvaRobert}) in the case $m=2k+1$ and to a standard integrable system for $n=2k$.

%\begin{remark} The previous lemma works for any $\epsilon$ small enough: we are not taking the limit. In this way, we have constructed a set of smooth hamiltonian vector fields in a $b^m$-symplectic manifold with a $b^m$-integrable system. Moreover these vector fields coincide with the hamiltonian vector fields associated to the desingularized symplectic manifold and the desingularized integrable system. Then, we can see this our $b^m$-integrable system as an integrable system in the standard sense. Hence, we can apply to it results for standard integrable systems, for instance the KAM theorem.
%\end{remark}

\chapter[$b^m$-KAM Desingularization]{Desingularization of the KAM theorem on $b^m$-symplectic manifolds}

The idea of this section is to recover some version of the classical KAM theorem by ``desingularizing the $b^m$-KAM theorem", as well as a new version of a KAM theorem that works for folded symplectic forms. Observe that no KAM theorem is known for folded symplectic forms. The best that is known is a KAM theorem for presymplectic structures that was done in \cite{presymplectic}. Desingularizing the KAM means applying the $b^m$-KAM in the $b^m$-manifold and then translating the result to the desingularized setting.

To be able to obtain proper desingularized theorems we need to identify which integrable systems can be obtained as a desingularization of a $b^m$-integrable system. To simplify computations we are going to use a particular case of $b^m$-integrable systems, where $f_1=\frac{1}{I_1^{m-1}}$. We call these systems \emph{simple}. Observe that by taking a particular case of $b^m$-integrable systems we will not get all the systems that can be obtained by desingularizing a $b^m$-integrable system, but some of them.

\begin{enumerate}
\item \textbf{Even case} $m = 2k$.

$F=(f_1 = \frac{1}{I_1^{2k-1}},f_2 ,\ldots, f_n)$, $\omega = \frac{1}{I_1^m} d I_1 \wedge d\phi_1 + \sum_{j=1}^n dI_j \wedge d\phi_j$.
Observe that close to $Z$ in the even case we can assume $f(I_1) = c I_1$ for some $c > (2 -\frac{1}{2^{2k-1}})$. Then $f_\varepsilon(I_1) = \frac{1}{\varepsilon^{(2k-1)}}\frac{c I_1}{\varepsilon}= c' I$, hence $\omega_\varepsilon = c' dI_1\wedge d\phi_1 + \sum_{j=1}^n dI_j \wedge d\phi_j$. Also $F_\varepsilon^m(I_1) = c'$, $G_\varepsilon^m(I_1) = c' I_1$.
Then,

$$
\left\{
\begin{array}{rcl}
\tilde I_1 & = & \int_0^{I_1}\frac{c'}{1/\tau^m}d \tau = \int_0^{I_1} c' \tau^m d\tau = c'\frac{I_1^{m+1}}{m+1},\\
\tilde \phi_1 & = & \frac{c'}{1/I_1^m}\phi_1 = c' I_1^m \phi_1
\end{array}
\right.
$$

\begin{equation}\label{eq:desingularized_even}
F^\varepsilon = ((m-1)c_{m-1}c' I_1, f_2(\tilde I, \tilde \phi), \ldots f_n(\tilde I, \tilde \phi)).
\end{equation}

Hence, the systems in this form can be viewed as a desingularization of a $b^m$-integrable system.

\begin{theoremD}[Desingularized KAM for symplectic manifolds]
Consider a neighborhood of a Liouville torus of an integrable system $F_\varepsilon$ as in \ref{eq:desingularized_even} of a symplectic manifold $(M, \omega_\varepsilon)$ semilocally endowed with coordinates $(I,\phi)$, where $\phi$ are the angular coordinates of the torus, with $\omega_\varepsilon = c' dI_1 \wedge d\phi_i + \sum_{j= 1}^n dI_j\wedge d\phi_j$. Let $H=(m-1)c_{m-1}c' I_1 + h(\tilde I) + R(\tilde I,\tilde \phi)$ be a nearly integrable system where
$$
\left\{
\begin{array}{rcl}
\tilde I_1 & = & c'\frac{I_1^{m+1}}{m+1},\\
\tilde \phi_1 & = & c' I_1^m \phi_1 ,
\end{array}
\right.
$$
and
$$
\left\{
\begin{array}{rcl}
\tilde I & = & (\tilde I_1, I_2, \ldots, I_n),\\
\tilde \phi & = & (\tilde \phi_1, \phi_2, \ldots, \phi_n).
\end{array}
\right.
$$
Then the results for the $b^m$-KAM theorem \ref{th:bm_kam} applied to $H_{\text{sing}} = \frac{1}{I_1^{2k-1}} + h(I) + R(I,\phi)$ hold for this desingularized system.
\end{theoremD}

\begin{remark} This theorem is not as general as the standard KAM, but we also know extra information about the dynamics. For instance, the perturbation of trajectories in tori inside of $Z$ will be trajectories lying inside of $Z$. In this sense, the theorem is new because it leaves invariant an hypersurface of the manifold.
\end{remark}

\item \textbf{Odd case} $m = 2k+1$.

$F = (f_1 = \frac{1}{I_1^{2k}}, f_2,\ldots,f_n)$ and $\omega = \frac{1}{I_1^{2k+1}} dI_1 \wedge d\phi_1 + \sum_{j=1}^n dI_j\wedge d \phi_j$.
Before continuing we need the following notions defined in \cite{EvaRobert}.

\begin{definition}
A function $f:M\rightarrow \mathbb{R}$ in a folded symplectic manifold $(M,\omega)$ is folded if $df|_Z(v)=0$ for all $v\in V=\textup{ker}\omega|_Z$.
\end{definition}

\begin{definition}
An integrable system in a folded symplectic manifold $(M,\omega)$ with critical surface $Z$ is a set of functions $F =(f_1,\ldots,f_n)$ such that they define Hamiltonian vector fields which are independent ($df_1\wedge \ldots \wedge d f_n \neq 0$ in the folded cotangent bundle) on a dense subset of $Z$ and $M$, and commute with respect to $\omega$.
\end{definition}

Note that we need to prove that the desingularized functions in this case are folded.

Observe that close to $Z$ in the odd case we can assume $f(I_1) = -I_1^2 + 2$. Then $f_\varepsilon(I_1) = \varepsilon^{-(2k)} f(\frac{I_1}{\varepsilon}) = \frac{1}{\varepsilon^{2k}}(-(\frac{I_1}{\varepsilon})^2 + 2) = c I_1^2 + \frac{2}{\varepsilon^{2k}}$. Then

$$\omega_\varepsilon = 2cI_1 dI_1 \wedge d\phi_1 + \sum_{j=1}^n dI_j \wedge d\phi_j.$$

Also $F_\varepsilon^m(I_1) = 2c I_1$, $G_\varepsilon^m(I_1) = cI_1^2$. Then,

$$
\left\{
\begin{array}{rcl}
\tilde I_1 & = & \int_0^{I_1}\frac{2 c \tau}{1/\tau^m} d\tau = 2c \frac{I_1^{(m+2)}}{(m+2)},\\
\tilde \phi_1 & = & 2c I_1^{m+1}\phi_1
\end{array}
\right.
$$

Then the desingularized moment map becomes

\begin{equation}\label{eq:desingularized_odd}
F^\varepsilon = ((m-1)c_{m-1}c I_1^2, f_2(\tilde I, \tilde \phi), \ldots f_n(\tilde I, \tilde \phi)).
\end{equation}

It is a simple computation to check that these functions are actually folded and hence they form a folded integrable system. Note that the systems of the form \ref{eq:desingularized_odd} can be viewed as a desingularization of a $b^m$-integrable system. Then, as we proceeded in the even case:

\begin{theoremE}[Desingularized KAM for folded symplectic manifolds]
Consider a neighborhood of a Liouville torus of an integrable system $F_\varepsilon$ as in \ref{eq:desingularized_odd} of a folded symplectic manifold $(M, \omega_\varepsilon)$ semilocally endowed with coordinates $(I,\phi)$, where $\phi$ are the angular coordinates of the Torus, with $\omega_\varepsilon = 2cI_1 dI_1 \wedge d\phi_1 + \sum_{j=2}^m dI_j \wedge d\phi_j$.
Let $H = (m-1)c_{m-1} cI_1^2 + h(\tilde I) + R(\tilde I, \tilde \phi)$ a nearly integrable system with
$$
\left\{
\begin{array}{rcl}
\tilde I_1 & = &  2c\frac{I_1^{m+2}}{m+2},\\
\tilde \phi_1 & = &  2c I_1^{m+1} \phi_1 ,
\end{array}
\right.
$$
and
$$
\left\{
\begin{array}{rcl}
\tilde I & = & (\tilde I_1, I_2, \ldots, I_n),\\
\tilde \phi & = & (\tilde \phi_1, \phi_2, \ldots, \phi_n).
\end{array}
\right.
$$
Then the results for the $b^m$-KAM theorem \ref{th:bm_kam} applied to $H_{\text{sing}} = \frac{1}{I_1^{2k}} + h(I) + R(I,\phi)$ hold for this desingularized system.
\end{theoremE}

\begin{remark}
The last two theorems can be improved if we consider $b^m$-integrable systems not necessarily \emph{simple}.
\end{remark}

\end{enumerate}

\chapter{ Applications to Celestial mechanics}\label{ch:binvitation}

The theory presented in this monograph establishes a formalized approach to perturbation theory in singular situations that manifest in real-world physical systems. We offer numerous illustrations from Celestial mechanics to support this theory, and conclude by outlining the potential applications of our KAM theory in detecting periodic trajectories.

Within this chapter, we showcase a variety of examples from Celestial Mechanics that involve the occurrence of singular symplectic forms. A few of these instances are elaborated upon in \cite{binvitation}. The majority of these singularities arise as a result of implementing "regularization" methods. We encourage readers to consult the book \cite{knauf} for a pedagogical approach to the study of regularization. 

This compilation of examples is of particular significance for this booklet, as the theoretical outcomes we attain, such as action-angle coordinates or KAM, can be readily employed in the range of problems outlined below.

Structures that are symplectic almost everywhere can emerge from coordinate transformations  that do not preserve the canonical symplectic structure. 

For instance: For the Kepler problem given a configuration space $\mathbb{R}^2$ and phase space $T^*\mathbb{R}^2$,  the traditional (canonical) Levi-Civita transformation described as follows: identify $\mathbb{R}^2\cong\mathbb{C}$ so that $T^*\mathbb{R}^2\cong T^*\mathbb{C}\cong \mathbb{C}^2$ and treat $(q,p)$ as complex variables $(q_1+iq_2:=u,p_1+ip_2:=v)$. Take the following change of coordinates $(q,p) = (u^2/2, v/\bar{u})$, where $\bar{u}$ denotes the complex conjugation of~$u$. The resulting coordinate change can easily be seen to preserve the canonical symplectic form. 
Nevertheless, this canonical transformation can increase the complexity of the Hamiltonian equations, making it more challenging to analyze the dynamical aspects of the system. Hence, it becomes intriguing to explore alternative coordinate transformations that do not preserve the symplectic form. Some of them induce new singular forms where our geometrical and dynamical techniques can be applied.

Other examples are discussed in \cite{DKM17}.

\section{The Kepler Problem}

In suitable coordinates in $T^*\left(\mathbb{R}^2\setminus\{0\}\right)$, the Kepler problem has Hamiltonian
\begin{equation}
 H(q,p)=\frac{\|p\|^2}{2}-\frac{1}{\|q\|}.
\end{equation}
With the canonical Levi-Civita transformation $(q,p) = (u^2/2, v/\bar{u})$, this expression becomes
\begin{equation}
 H(u,v)=\frac{\|v\|^2}{2\|\bar{u}\|^2}-\frac{1}{\|u\|^2}.
\end{equation}

To avoid the complications arising from preserving the canonical symplectic form, we suggest an alternative approach. By retaining the momentum as is, we can consider a transformation of the form $(q,p) = (u^2/2, p)$, which may lead to a simpler Hamiltonian. However, this transformation is not a symplectomorphism, and as a result, the symplectic form on $T^*\mathbb{R}^2$ pulls-back under the transformation to a two-form  which is symplectic almost everywhere but degenerates on a hypersurface of $T^*\mathbb{R}^2$

Namely, the Liouville one-form $p_1 dq_1+p_2 dq_2=\Re(p d\bar{q})$ pulls back to
\begin{eqnarray*}
    \theta=\Re\left(p d\left(\frac{\bar{u}^2}{2}\right)\right)&=&\Re\left(p \bar{u}d\bar{u}\right)\\
    &=&p_1(u_1du_1-u_2du_2)+p_2(u_2du_1+u_1du_2)
\end{eqnarray*}
and the associated $2$-form  $-d\theta$ yields a form that is almost everywhere symplectic 
\begin{eqnarray*}
   \omega=u_1du_1\wedge dp_1 - u_2 du_1 \wedge dp_2 + u_2 du_2 \wedge dp_1 + u_1 du_2\wedge dp_2.
\end{eqnarray*}

To examine the characteristics of this form, we take its wedge product with itself, yielding:
\begin{eqnarray*}
   \omega\wedge\omega=(u_1^2-u_2^2) du_1 \wedge dp_1 \wedge du_2\wedge dp_2
\end{eqnarray*}
which is degenerate along the hypersurface given by $u_1=\pm u_2$.

We now consider the restriction of the form to the critical set. It does not have maximal rank so it is not a folded symplectic structure.
This form is  degenerately folded and the folding hypersurface is not regular and is described by the equations $u_1=\pm u_2$.

\section{The Problem of Two Fixed Centers}

We now regularize the problem of two fixed centers. 

The problem of two fixed centers is associated to the motion of a satellite moving in a gravitational potential generated by two fixed massive bodies. Additionally, we assume that the satellite's motion is limited to a plane in $\mathbb{R}^3$ that includes the two massive bodies.

The Hamiltonian function in suitable coordinates reads:
\begin{equation}
    H=\frac{p^2}{2m}-\frac{\mu}{r_1}-\frac{1-\mu}{r_2}
\end{equation}
where $\mu$ is the mass ratio of the two bodies (i.e. $\mu=\frac{m_1}{m_1+m_2})$.

Euler was the first to demonstrate the integrability of this problem, using \emph{elliptic} coordinates in which the coordinate lines are confocal ellipses and hyperbolas.

Explicitly, consider a coordinate system in which the two centers are placed at $(\pm 1, 0)$, in which the (Cartesian) coordinates are given by $(q_1, q_2)$. Then the elliptic coordinates of the system are given by
\begin{align}
q_1&=\sinh\lambda\cos\nu\\
q_2&=\cosh\lambda\sin\nu
\end{align}
for $(\lambda, \nu)\in \mathbb{R}\times S^1$. Thus lines of $\lambda=c$ and $\nu=c$ are given by confocal hyperbola and ellipses in the plane, respectively. Similar to the Levi-Civita transformation this results in a double-branched covering with branch points at the centers of attraction.

Pulling back the canonical symplectic structure $\omega=dq \wedge dp$ we find
\begin{equation}
\omega= \cosh\lambda \cos\nu (d\lambda \wedge dp_1+d\nu \wedge dp_2) -\sinh\lambda\sin\nu(d\nu \wedge dp_1 + d\lambda \wedge d p_2)
\end{equation}
which is degenerate along the hypersurface $(\lambda,\nu)$ satisfying $\cosh\lambda\cos\nu=\sinh\lambda\sin\lambda$.

\section{Double Collision and McGehee coordinates}\label{Sec:Doublecollision}

In this section, we describe another example of $b$-symplectic structure appearing quite naturally in physical dynamical systems. From this example, it would seem natural that a collection of different examples for $b^m$-symplectic models or even $b^m$-folded models would follow. But one finds a major problem while pursuing these examples.  By examining why this example cannot be extended to construct $b^m$-symplectic models or $b^m$-folded models for any $m$, we can identify a general pattern.

First, let us introduce the McGehee coordinate change for the problem of double collision.

The system of two particles moving under the influence of the generalized potential $U(x) = -|x|^{-\alpha}$, $\alpha > 0$, where $|x|$ is the distance between the two particles, is studied by McGehee in \cite{McGehee}. We fix the center of mass at the origin and hence can simplify the problem to the one of a single particle moving in a central force field.

The equation of motion can be written as,
\begin{equation}
\ddot{x} = -\nabla U(x) = -\alpha |x|^{-\alpha-2}x
\end{equation}
where the dot represents the derivative with respect to time. In the Hamiltonian formalism, this equation becomes
\begin{equation}
\begin{array}{rcl}
\dot{x} & = &  y, \\
\dot{y} & = & -\alpha |x|^{-\alpha-2}x.
\end{array}
\end{equation}
To study the behavior of this system, the following change of coordinates is suggested in \cite{McGehee}:
\begin{equation}\label{eq:mcgeheechange}
\begin{array}{rcl}
x & = & r^\gamma e^{i\theta}, \\
y & = & r^{-\beta\gamma}(v + iw)e^{i\theta}
\end{array}
\end{equation}
where the parameters $\beta$ and $\gamma$ are related with $\alpha$ as follows:
\begin{equation}\label{eq:relations}
\begin{array}{rcl}
\beta & = & \alpha/2, \\
\gamma & = & 1/(1 + \beta).
\end{array}
\end{equation}
Identifying the plane $\mathbb{R}^2$ with the complex plane $\mathbb{C}$, we can write the symplectic form of this problem as $\omega = \Re (dx\wedge  d\overline{y})$.

\begin{remark}
To check that a form $\omega$ is actually a $b^m$-symplectic form, it is not enough to check that the multi-vector field dual to $\omega\wedge\omega$ is a section of $\bigwedge^{2n}(^{b^m}TM)$  which is transverse to the zero section. One has to check additionally that the Poisson structure dual to $\omega$ itself is a proper section of $\bigwedge^{2}(^{b^m}TM)$.
\end{remark}

\begin{proposition} Under the coordinate change (\ref{eq:mcgeheechange}), the symplectic form $\omega$ is sent to a $b$-symplectic structure for $\alpha = 2$.
\end{proposition}

\begin{proof}
The proof of this proposition is a straightforward computation.
Observe that the change is not a smooth change, so we are not working with standard De Rham forms. But, at the end of the computation, it will become clear that the form is a $b$-symplectic form, and hence the computations are legitimate.
If one does the change of variables, we obtain:
\begin{equation}
\begin{array}{rcl}
 \bar{y} & = & r^{\beta\gamma}(v - iw)e^{-i\theta}. \\
 dx & = & \gamma r^{\gamma-1} e^{i\theta}dr + r^{\gamma}e^{i\theta}id\theta. \\
 d\bar{y} & = & r^{-\beta\gamma-1}(-\beta\gamma)(v - iw)e^{-i\theta}dr + r^{-\beta\gamma}e^{-i\theta}dv \\
 & & \quad + r^{\beta\gamma}(v - iw)e^{-i\theta}(-i)d\theta. \\
\end{array}
\end{equation}
By wedging the previous two forms, we obtain:
\begin{equation}
\begin{array}{rcl}
dx \wedge d\bar{y} & = & dr\wedge dv (\gamma r^{\gamma-1-\beta\gamma})\\
& + & dr\wedge dw (\gamma r^{\gamma-1-\beta\gamma})\\
& + & dr\wedge d\theta (\gamma r^{\gamma-1-\beta\gamma}(-iv-w))\\
& + & d\theta \wedge dr (ir^{\gamma-1-\beta\gamma}(-\beta\gamma)(v-iw))\\
& + & d\theta\wedge dv(ir^{\gamma-\beta\gamma})\\
& + & d\theta\wedge dw(ir^{\gamma-\beta\gamma}(-i)).\\
\end{array}
\end{equation}
Now we can take the real part of this form and use that $\gamma - 1 - \beta\gamma = -\alpha\gamma$.
 In the new coordinates, the form reads.
\begin{equation}\label{eq:omegachanged}
\begin{array}{rcl}
\omega = \Re(dx \wedge d\bar{y}) & = & \gamma r^{-\beta\gamma + \gamma - 1} dr\wedge dv - \gamma(1 - \beta)r^{-\beta \gamma + \gamma -1}wdr\wedge d\theta \\
&-& r^{-\beta\gamma + \gamma} dw \wedge d\theta.
\end{array}
\end{equation}
Moreover, we can use that $\gamma(1+\beta) = 1$ to simplify the previous expression further to:
\begin{equation}
\omega = (dr\wedge dv + dr\wedge dw)\gamma r^{-\alpha\gamma} + dr\wedge d\theta(wr^{-\alpha\gamma}) + d\theta\wedge dw (r^{-\alpha\gamma + 1}).
\end{equation}
In order to classify this structure,  we wedge it with itself and look at the structure of the form in the singular set.
Wedging this form, we obtain
\begin{equation}
\begin{array}{rcl}
\omega\wedge \omega  & = &  -\gamma r^{-2\beta\gamma + 2\gamma - 1}dr\wedge dv \wedge d\theta\wedge dw\\
& = & \displaystyle -\gamma r^{\frac{2 -3\alpha}{2 + \alpha}}dr\wedge dv \wedge d\theta\wedge dw.
\end{array}
\end{equation}
where we use (\ref{eq:relations}). Let us set $f(\alpha) = \frac{2 -3\alpha}{2 + \alpha}$. This function does not take values lower than $-3$ or higher than $1$. 
When $\alpha = 2$ this gives us a $b$-symplectic structure:
$$\omega\wedge\omega = -\gamma r dr\wedge dv \wedge d\theta \wedge dw.$$
The section of $ \bigwedge^4(^bTM)$ given by the dual structure of $\omega\wedge\omega$ is clearly transverse to the zero section.

On the other hand if $\alpha = 2$, then $\beta = 1$ and hence:

$$\omega = \gamma r^{-1} dr\wedge \omega \wedge dv,$$

and its dual Poisson structure is clearly also a proper section of $ \bigwedge^2(^bTM)$.

\end{proof}

\begin{remark}
One may ask if for other values of $\alpha$ it is possible to obtain other $b^m$-symplectic structures for different $m$. For example for $\alpha = 6$,  as $\omega\wedge\omega = -\gamma r^{-2}dr\wedge dv \wedge d\theta \wedge dw$, so it seems likely to obtain a $b^2$-symplectic form. But from the expression of $\omega$ it becomes clear that it is not a proper section of $\bigwedge^2(^{b^2}T^*M)$
\end{remark}

\section[Applications]{The restricted three-body problem}\label{Sec:Escapesingularity}

In this last section of the monograph, we catch up with the circular planar restricted three-body problem.

The restricted elliptic $3$-body problem is a simplified version of the $3$-body problem. It describes the trajectory of a body with negligible mass moving in the gravitational field of two massive bodies called primaries, orbiting in elliptic Keplerian motion. The restricted planar version assumes that all motion occurs in a plane.

The associated Hamiltonian of the particle can be written as:
\begin{equation}
H(q,p)=\frac{\|p\|^2}{2}+\frac{1-\mu}{\|q-q_1\|}+\frac{\mu}{\|q-q_2\|}=T+U
\end{equation}
wit $\mu$  the reduced mass of the system.

\begin{figure}
	
	\tikzset{>=latex}
	\centering
	\definecolor{xdxdff}{rgb}{0.49019607843137253,0.49019607843137253,1.}
	\definecolor{qqqqff}{rgb}{0.,0.,1.}
	\begin{tikzpicture}[line cap=round,line join=round,x=0.8cm,y=0.8cm]
	\clip(-6,-2.8) rectangle (7,4.5);
	\draw [->] (5.96,-0) -- (-0.3,3.92);
	\draw [->] (-3.78,0) -- (-0.3,3.92);
	\draw (5.96,-0)-- (-3.78,0);
	\draw [->] (0,0) -- (-0.3,3.92);
	\draw[color=black] (-3.7,-1.1) node {$m_1 = 1 - \mu$};
	\draw[color=black] (5.9,-1.1) node {$m_2 = \mu$};
	\draw [fill=black] (-0.3,3.92) circle (2pt);
	\draw [fill=black] (-3.78,0) circle (1.5pt);
	\draw [fill=black] (5.96,-0) circle (1.5pt);
	\draw[color=black] (-0.4,4.29) node {$q$};
	\draw[color=black] (2.3,3.2) node {$r_2 = q - q_2$};
	\draw[color=black] (-2.54,3.2) node {$r_1 = q - q_1 $};
	\draw [fill=black] (0,0) circle (2pt);
	\draw[color=black] (-0.3,-0.3) node {$\text{Center of mass}$};
	\draw[color=black] (0.2,1.8) node {$r$};
	\draw [fill=cyan,draw=none,fill opacity=1] (-3.78,0) circle (0.5cm);
	\draw [fill=magenta,draw=none,fill opacity=1] (5.96,0) circle (0.3cm);
	\draw[color=black] (-3.78,0) node {$q_1$};
	\draw[color=black] (5.96,0) node {$q_2$};
	\end{tikzpicture}
	\caption{Scheme of the three-body problem.}
\end{figure}
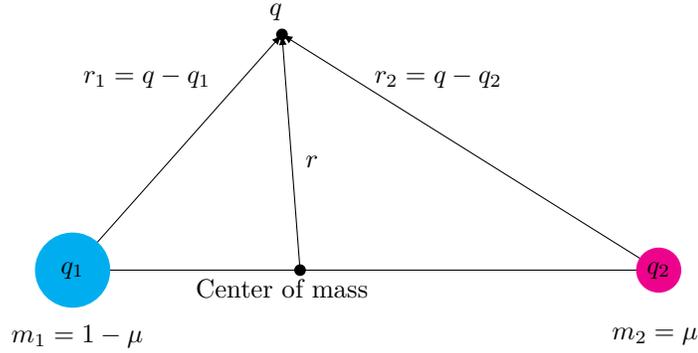

As it was observed in \cite{kiesenhofermirandascott}, it is possible to associate a singular structure to this problem. Consider the symplectic form on $\mathrm{T}^{\ast} \mathbb{R}^2$ in polar coordinates,
After making a change to polar coordinates $(q_1,q_2)=(r\cos\alpha,r\sin\alpha)$ and the corresponding canonical change of momenta we find the Hamiltonian function
\begin{equation}
H(r,\alpha,P_r,P_\alpha)=\frac{P_r^2}{2}+\frac{P_\alpha ^2}{2r^2}+U(r\cos\alpha,r\sin\alpha)
\end{equation}
where $P_r,P_\alpha$ are the associated canonical momenta and with potential energy:
$U(r\cos\alpha,r\sin\alpha)$

The McGehee change of coordinates is used to examine the behavior of orbits near infinity, see also \cite{delshams2015global}:
\begin{equation}\label{eqn:McGehee}
r=\frac{2}{x^2}.
\end{equation}
The corresponding change for the canonical momenta is easily seen to be
\begin{equation}
P_r=-\frac{x^3}{4}P_x.
\end{equation}
The Hamiltonian is transformed to
\begin{equation}
H(r,\alpha,P_r,P_\alpha)=\frac{x^6P_x^2}{32}+\frac{x^4P_\alpha^2}{8}+U(x,\alpha).
\end{equation}
By  transforming the position coordinate~(\ref{eqn:McGehee}) without modifying the momentum associated to $r$, we are left with a simpler Hamiltonian, however, the pull-back of the symplectic form  is no longer symplectic, but exhibits a singularity of order $3$ and it is called $b^3$-symplectic:
\begin{equation}
\omega= \frac{4}{x^3} dx \wedge dP_r + d\alpha\wedge d P_\alpha.
\end{equation}

Adding the line at infinity provides a description of the dynamics within the critical set $Z=\{x=0\}$.
From the change of coordinates implemented, we might think that the dynamics within  $Z$ may have  no physical meaning, but its interplay with the dynamics close to $Z$ gives information about the behaviour of escape orbits sometimes identified as \emph{singular periodic orbits} (see \cite{MO20} and \cite{cedricdanieleva}).

Given an autonomous Hamiltonian system of a symplectic manifold  of dimension $2n$, the level sets of the Hamiltonian function are often endowed with a contact structure
( a contact structure is given by a one form $\alpha$ satisfying a condition of type $\alpha\wedge (d\alpha)^{n-1}\neq 0$).

In \cite{MO18, MO20}  applications of the $b$-apparatus are discussed in this context. In particular, the notion of $b^m$-contact structures is introduced by translating the condition above for $b^m$-forms. The classical notions in the contact realm such as  Reeb vector fields can also be introduced in this set-up.

By considering the McGehee change as we did in the contact context,  the following theorem is proved in \cite{MO20}:

\begin{theorem}\label{thm:bcontact3bp}
After the McGehee change, the Liouville vector field $Y=p\frac{\partial}{\partial p}$ is a $b^3$-vector field that is everywhere transverse to the level sets of the Hamiltonian $\Sigma_c$ for $c>0$ and the level-sets $(\Sigma_c,\iota_Y \omega)$ for $c>0$ are $b^3$-contact manifolds. Topologically, the critical set of this contact manifold is a cylinder (which can be interpreted as a subset of the line at infinity) and the Reeb vector field admits infinitely many non-trivial periodic orbits on the critical set.
\end{theorem}

The  KAM theorem in this monograph can be applied to find new periodic orbits of the restricted three-body problem close to infinity by perturbing the periodic orbits described above (see also \cite{MO20}).  This perturbation technique is an old method in perturbation theory, possibly originating from Poincaré himself (known as Poincaré's continuation method, as mentioned in \cite{meyeroffin}).  This paves the way for further research that will be pursued elsewhere.

\section{Escape orbits in Celestial mechanics and Fluid dynamics}
Reeb vector fields and their dynamics are closely related to Beltrami vector fields, which provide stationary solutions of the Euler equations. Indeed, Etnyre and Ghrist \cite{EG} revealed the existence of a "mirror" that reflects Reeb vector fields as Beltrami vector fields. Thus, the applications of KAM theory to find periodic orbits can be exported to understand periodic orbits of stationary fluid flows.  Also, this correspondence yields the possibility to translate concepts in celestial mechanics into  fluid dynamics. Among these concepts is the notion of escape orbits.
The Reeb-Beltrami correspondence was extended to the $b$-setting in \cite{danielevarobert}.

In \cite{MO20} we introduced the notion of singular periodic orbit of a $b$-Reeb vector field $R_{\alpha}$:
\begin{definition}
		Let $(M,\alpha)$ be $b$-contact manifold manifold with critical hypersurface $Z$. Denote by $R_{\alpha}$ its $b$-Reeb vector field.  A \emph{singular periodic orbit} $\gamma$ is an orbit such that $\lim_{t \to  \pm \infty} \gamma(t) =p_{\pm} \in Z$ where $R_{\alpha}(p_{\pm})=0$.
	\end{definition}

	\begin{figure}[hbt!]\label{fig:singularorbit}
\begin{center}
\begin{tikzpicture}[scale=2.2]
 \draw[color=blue](1,0) arc (0:180:1 and 1);
% \draw (-1,0) -- (1,0);
 \draw[dashed][color=red] (-2,0) --  (2,0);
 \fill (1,0) circle[radius=0.5pt];
 \fill (-1,0) circle[radius=0.5pt];
 %\draw (-0.75,0.2) ..controls +(0,0.5) and +(0,0.5).. node {\midarrow} (0.75,0.2); % arc semicircle up
 %\draw (-0.75,0.2) ..controls +(0,0.3) and +(-0.5,0).. (0,0.7); % arc1 upleft
% \draw (0,0.7) ..controls +(0.5,0) and +(0,0.3).. (0.75,0.2); % arc1 upright
% \draw (-0.75,0.2) ..controls +(0,-0.2) and +(0,-0.2).. (0.75,0.2); %arc1 down
%  \draw (-0.6,0.3) ..controls +(0,0.05) and +(-0.5,0).. (0,0.6); % arc2 upleft
% \draw (0,0.6) ..controls +(0.5,0) and +(0,0.05).. (0.6,0.3); % arc2 upright
% \draw (-0.6,0.3) ..controls +(0,-0.2) and +(0,-0.2).. (0.6,0.3); %arc2 down
%\flecha[shift={(0,0)},black,scale=2,rotate=];
%\fill [shift={(0.05,0)},scale=0.05,rotate=0]   (0,0) -- (-1,-0.7) -- (-1,0.7) -- cycle;
\fill [shift={(-0.05,1)},scale=0.05,rotate=180]   (0,0) -- (-1,-0.7) -- (-1,0.7) -- cycle; %arrow up
%\fill [shift={(-0.05,0.7)},scale=0.05,rotate=180]   (0,0) -- (-1,-0.7) -- (-1,0.7) -- cycle; %arrow up (2nd)
%\fill [shift={(0.05,0.15)},scale=0.05,rotate=0]   (0,0) -- (-1,-0.7) -- (-1,0.7) -- cycle;
%\fill [shift={(0.05,0.24)},scale=0.05,rotate=0]   (0,0) -- (-1,-0.7) -- (-1,0.7) -- cycle;
%  \draw[shift={(0,0.15)},scale=0.6] (-0.6,0.3) ..controls +(0,0.05) and +(-0.5,0).. (0,0.6); % arc2 upleft
% \draw[shift={(0,0.15)},scale=0.6] (0,0.6) ..controls +(0.5,0) and +(0,0.05).. (0.6,0.3); % arc2 upright
% \draw[shift={(0,0.15)},scale=0.6] (-0.6,0.3) ..controls +(0,-0.2) and +(0,-0.2).. (0.6,0.3); %arc2 down
\end{tikzpicture}

\caption{A singular periodic orbit}

\end{center}
\end{figure}
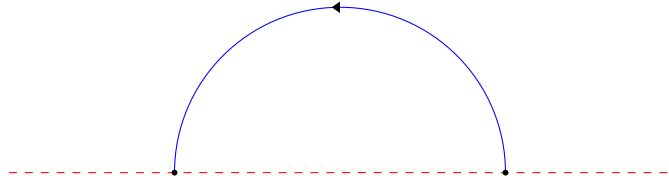

The singular Weinstein conjecture was conjecture in \cite{MO20}. 
In \cite{cedricdanieleva} the existence of  singular periodic orbits  and  generalizations
such as oscillatory motions is investigated for singular structures. Singular periodic orbits are a particular case of escape orbits.

Escape orbits for $b$-Beltrami or $b$-Reeb vector fields, are orbits whose $\alpha$- or $\omega$-limit set lies on the critical set associated to the $b$-structure. The $b$-Reeb Beltrami correspondence together with results of Uhlenbeck \cite{uhlenbeck} on Laplacian eigenfunctions yields that for the majority of asymptotically exact $b$-metrics, $b$-Beltrami vector fields have escape orbits.

 The following theorem proved in \cite{cedricdanielevajosep} gives a lower bound on the number of escape orbits for generic classes of $b$-Beltrami or $b$-Reeb vector fields. The lower bound depends on the number of connected components of the critical set of $Z$ but can often be infinite.

\begin{theorem}{\cite{cedricdanielevajosep}}
    Let $(M,Z)$ be a $3$-dimensional $b$-manifold. Then for a generic asymptotically exact $b$-metric, any $b$-Beltrami vector field has either at least $2N$ or infinitely many escape orbits, where $N$ is the number of connected components of $Z$.
\end{theorem}

In view of the $b$-Reeb-Beltrami correspondence, this implies that generic $b$-Reeb vector fields within a special class of $b$-contact forms also have at least $2N$ or infinitely many escape orbits.

Poincaré continuation method and our KAM results thus can be used to localize the singular counterparts or periodic orbits, either escape orbits or singular periodic orbits in these new scenarios described in \cite{eva}, \cite{MO20},\cite{cedricdanieleva}, and \cite{cedricdanielevajosep}.

\backmatter

\bibliographystyle{amsalpha}
% \typeout{}
\bibliography{bibliarnau}

\newcommand{\etalchar}[1]{$^{#1}$}
\providecommand{\bysame}{\leavevmode\hbox to3em{\hrulefill}\thinspace}
\providecommand{\MR}{\relax\ifhmode\unskip\space\fi MR }
\providecommand{\MRhref}[2]{%
  \href{http://www.ams.org/mathscinet-getitem?mr=#1}{#2}
}
\providecommand{\href}[2]{#2}
\begin{thebibliography}{DKdlRS19}

\bibitem[AdlL12]{presymplectic}
Hassan~Najafi Alishah and Rafael de~la Llave, \emph{Tracing {KAM} tori in
  presymplectic dynamical systems}, J. Dynam. Differential Equations
  \textbf{24} (2012), no.~4, 685--711. \MR{3000600}

\bibitem[Arn81]{ARNOLD75}
V.~I. Arnol'd, \emph{Singularity theory}, London Mathematical Society Lecture
  Note Series, vol.~53, Cambridge University Press, Cambridge-New York, 1981,
  Selected papers, Translated from the Russian, With an introduction by C. T.
  C. Wall. \MR{631683}

\bibitem[Arn89]{ARNOLD85}
\bysame, \emph{Poisson structures on the plane and other powers of volume
  forms}, Journal of Soviet Mathematics \textbf{47} (1989), no.~3, 2509--2516.

\bibitem[BDM{\etalchar{+}}19]{binvitation}
Roisin Braddell, Amadeu Delshams, Eva Miranda, C\'{e}dric Oms, and Arnau
  Planas, \emph{An invitation to singular symplectic geometry}, Int. J. Geom.
  Methods Mod. Phys. \textbf{16} (2019), no.~suppl. 1, 1940008, 16.
  \MR{3904653}

\bibitem[BKM23]{BKM}
Roisin Braddell, Anna Kiesenhofer, and Eva Miranda, \emph{A {$b$}-symplectic
  slice theorem}, Bull. Lond. Math. Soc. \textbf{55} (2023), no.~1, 90--112.
  \MR{4561237}

\bibitem[BMO22]{cedricjoaquimeva}
Joaquim Brugués, Eva Miranda, and C\'{e}dric Oms, \emph{The arnold conjecture
  for singular symplectic manifolds}, arXiv:2212.01344 (2022).

\bibitem[CM22]{EvaRobert}
Robert Cardona and Eva Miranda, \emph{Integrable {S}ystems on {S}ingular
  {S}ymplectic {M}anifolds: {F}rom {L}ocal to {G}lobal}, Int. Math. Res. Not.
  IMRN (2022), no.~24, 19565--19616. \MR{4523256}

\bibitem[CMPS19]{danielevarobert}
Robert Cardona, Eva Miranda, and Daniel Peralta-Salas, \emph{Euler flows and
  singular geometric structures}, Philos. Trans. Roy. Soc. A \textbf{377}
  (2019), no.~2158, 20190034, 15. \MR{4036383}

\bibitem[DG96]{D}
Amadeu Delshams and Pere Guti\'{e}rrez, \emph{Effective stability and {KAM}
  theory}, J. Differential Equations \textbf{128} (1996), no.~2, 415--490.
  \MR{1398328}

\bibitem[DKdlRS19]{delshams2015global}
Amadeu Delshams, Vadim Kaloshin, Abraham de~la Rosa, and Tere~M. Seara,
  \emph{Global instability in the restricted planar elliptic three body
  problem}, Communications in Mathematical Physics \textbf{366} (2019), no.~3,
  1173--1228.

\bibitem[DKM17]{DKM17}
Amadeu Delshams, Anna Kiesenhofer, and Eva Miranda, \emph{Examples of
  integrable and non-integrable systems on singular symplectic manifolds}, J.
  Geom. Phys. \textbf{115} (2017), 89--97. \MR{3623614}

\bibitem[EG00]{EG}
John Etnyre and Robert Ghrist, \emph{Contact topology and hydrodynamics. {I}.
  {B}eltrami fields and the {S}eifert conjecture}, Nonlinearity \textbf{13}
  (2000), no.~2, 441--458. \MR{1735969}

\bibitem[FMOPS23]{cedricdanielevajosep}
Josep Fontana, Eva Miranda, C\'{e}dric Oms, and Daniel Peralta-Salas, \emph{2n
  or infinitely many escape orbits}, arXiv:2303.17690 (2023).

\bibitem[GLPR17]{GUAL14}
Marco Gualtieri, Songhao Li, Álvaro Pelayo, and Tudor Ratiu, \emph{The
  tropical momentum map: a classification of toric log symplectic manifolds},
  Mathematische Annalen \textbf{367} (2017).

\bibitem[GMP10]{GMP10}
Victor Guillemin, Eva Miranda, and Ana Pires, \emph{Codimension one symplectic
  foliations and regular poisson structures}, Bulletin of the Brazilian
  Mathematical Society, New Series \textbf{42} (2010).

\bibitem[GMP14]{GMP14}
Victor Guillemin, Eva Miranda, and Ana~Rita Pires, \emph{Symplectic and poisson
  geometry on b-manifolds}, Advances in Mathematics \textbf{264} (2014),
  864--896.

\bibitem[GMPS15a]{GMPS}
Victor Guillemin, Eva Miranda, Ana~Rita Pires, and Geoffrey Scott, \emph{Toric
  actions on {$b$}-symplectic manifolds}, Int. Math. Res. Not. IMRN (2015),
  no.~14, 5818--5848. \MR{3384459}

\bibitem[GMPS15b]{GMP15}
\bysame, \emph{Toric actions on {$b$}-symplectic manifolds}, Int. Math. Res.
  Not. IMRN (2015), no.~14, 5818--5848. \MR{3384459}

\bibitem[GMPS17]{GMP17}
\bysame, \emph{Convexity for {H}amiltonian torus actions on {$b$}-symplectic
  manifolds}, Math. Res. Lett. \textbf{24} (2017), no.~2, 363--377.
  \MR{3685275}

\bibitem[GMW17]{GMW17}
Victor Guillemin, Eva Miranda, and Jonathan Weitsman, \emph{Desingularizing
  $\boldsymbol{b^m}$-symplectic structures}, International Mathematics Research
  Notices \textbf{2019} (2017), no.~10, 2981--2998.

\bibitem[GMW18a]{GMWbmconvexity}
Victor~W. Guillemin, Eva Miranda, and Jonathan Weitsman, \emph{Convexity of the
  moment map image for torus actions on {$b^m$}-symplectic manifolds}, Philos.
  Trans. Roy. Soc. A \textbf{376} (2018), no.~2131, 20170420, 6. \MR{3868423}

\bibitem[GMW18b]{GMWbquant}
\bysame, \emph{On geometric quantization of {$b$}-symplectic manifolds}, Adv.
  Math. \textbf{331} (2018), 941--951. \MR{3804693}

\bibitem[GMW21]{GMWbmquant}
\bysame, \emph{On geometric quantization of {$b^m$}-symplectic manifolds},
  Math. Z. \textbf{298} (2021), no.~1-2, 281--288. \MR{4257086}

\bibitem[GS90]{GS90}
V.~Guillemin and S.~Sternberg, \emph{Symplectic techniques in physics},
  Cambridge University Press, 1990.

\bibitem[KM17]{KM17}
Anna Kiesenhofer and Eva Miranda, \emph{Cotangent models for integrable
  systems}, Comm. Math. Phys. \textbf{350} (2017), no.~3, 1123--1145.
  \MR{3607471}

\bibitem[KMS16a]{KMS16}
Anna Kiesenhofer, Eva Miranda, and Geoffrey Scott, \emph{Action-angle variables
  and a {KAM} theorem for {$b$}-{P}oisson manifolds}, J. Math. Pures Appl. (9)
  \textbf{105} (2016), no.~1, 66--85. \MR{3427939}

\bibitem[KMS16b]{kiesenhofermirandascott}
\bysame, \emph{Action-angle variables and a {KAM} theorem for {$b$}-{P}oisson
  manifolds}, J. Math. Pures Appl. (9) \textbf{105} (2016), no.~1, 66--85.
  \MR{3427939}

\bibitem[Kna18]{knauf}
Andreas Knauf, \emph{Mathematical physics: classical mechanics}, Unitext, vol.
  109, Springer-Verlag, Berlin, 2018, Translated from the 2017 second German
  edition by Jochen Denzler, La Matematica per il 3+2. \MR{3752660}

\bibitem[LGMV08]{LMV11}
Camille Laurent-Gengoux, Eva Miranda, and Pol Vanhaecke, \emph{Action-angle
  coordinates for integrable systems on poisson manifolds}, International
  Mathematics Research Notices \textbf{2011} (2008).

\bibitem[Mar19]{Marle}
Charles-Michel Marle, \emph{{Projection st{\'e}r{\'e}ographique et moments}},
  working paper or preprint, June 2019.

\bibitem[McG81]{McGehee}
{Richard P} McGehee, \emph{Double collisions for a classical particle system
  with nongravitational interactions}, Commentarii Mathematici Helvetici
  \textbf{56} (1981), no.~1, 524--557 (English (US)).

\bibitem[Mel93]{Melrose93}
Richard~B. Melrose, \emph{The {A}tiyah-{P}atodi-{S}inger index theorem},
  Research Notes in Mathematics, vol.~4, A K Peters, Ltd., Wellesley, MA, 1993.
  \MR{1348401}

\bibitem[Mir20]{eva}
Eva Miranda, \emph{Looking for periodic orbits}, Gac. R. Soc. Mat. Esp.
  \textbf{23} (2020), no.~3, 631--654. \MR{4176121}

\bibitem[MM22]{anastasiaeva}
Anastasia Matveeva and Eva Miranda, \emph{Reduction theory for singular
  symplectic manifolds and singular forms on moduli spaces}, arXiv:2205.12919
  (2022).

\bibitem[MO17]{meyeroffin}
Kenneth~R. Meyer and Daniel~C. Offin, \emph{Introduction to {H}amiltonian
  dynamical systems and the {N}-body problem}, third ed., Applied Mathematical
  Sciences, vol.~90, Springer, Cham, 2017. \MR{3642697}

\bibitem[MO18]{MO18}
Eva Miranda and C\'{e}dric Oms, \emph{Contact structures with singularities:
  from local to global}, arXiv:1806.05638 (2018).

\bibitem[MO21]{MO20}
Eva Miranda and C{\'e}dric Oms, \emph{The singular weinstein conjecture},
  Advances in Mathematics \textbf{389} (2021), 107925.

\bibitem[MOPS22]{cedricdanieleva}
Eva Miranda, C\'{e}dric Oms, and Daniel Peralta-Salas, \emph{On the singular
  {W}einstein conjecture and the existence of escape orbits for
  {$b$}-{B}eltrami fields}, Commun. Contemp. Math. \textbf{24} (2022), no.~7,
  Paper No. 2150076, 25. \MR{4476313}

\bibitem[MOT14]{MO2}
Ioan Marcut and Boris Osorno~Torres, \emph{Deformations of log-symplectic
  structures}, Journal of the London Mathematical Society \textbf{90} (2014),
  no.~1, 197--212.

\bibitem[MS21]{evageoff}
Eva Miranda and Geoffrey Scott, \emph{The geometry of {$E$}-manifolds}, Rev.
  Mat. Iberoam. \textbf{37} (2021), no.~3, 1207--1224. \MR{4236806}

\bibitem[NT96]{NT96}
Ryszard Nest and Boris Tsygan, \emph{Formal deformations of symplectic
  manifolds with boundary}, J. Reine Angew. Math. \textbf{481} (1996), 27--54.
  \MR{1421945}

\bibitem[NT01]{enest}
\bysame, \emph{Deformations of symplectic {L}ie algebroids, deformations of
  holomorphic symplectic structures, and index theorems}, Asian J. Math.
  \textbf{5} (2001), no.~4, 599--635. \MR{1913813}

\bibitem[P{\"o}s93]{JP93}
J{\"u}rgen P{\"o}schel, \emph{Nekhoroshev estimates for quasi-convex
  hamiltonian systems}, Mathematische Zeitschrift \textbf{213} (1993), no.~1,
  187--216.

\bibitem[Rad02]{Radko02}
Olga Radko, \emph{A classification of topologically stable {P}oisson structures
  on a compact oriented surface}, J. Symplectic Geom. \textbf{1} (2002), no.~3,
  523--542. \MR{1959058}

\bibitem[Sco16]{Scott16}
Geoffrey Scott, \emph{The geometry of {$b^k$} manifolds}, J. Symplectic Geom.
  \textbf{14} (2016), no.~1, 71--95. \MR{3523250}

\bibitem[Swa62]{SWAN}
Richard~G. Swan, \emph{Vector bundles and projective modules}, Transactions of
  the American Mathematical Society \textbf{105} (1962), 264--277.

\bibitem[Uhl76]{uhlenbeck}
K.~Uhlenbeck, \emph{Generic properties of eigenfunctions}, Amer. J. Math.
  \textbf{98} (1976), no.~4, 1059--1078. \MR{464332}

\end{thebibliography}

\end{document}